\documentclass[11pt,twoside, reqno]{amsart}
\usepackage{hyperref}
\hypersetup{
    colorlinks=true,
    linkcolor=blue,
    filecolor=magenta,      
    urlcolor=cyan,
    pdftitle={Backward Forward Attainable Set},
    }
\usepackage[normalem]{ulem}
\usepackage{comment}
\usepackage{enumitem}
\usepackage{tikz}
\usepackage{comment}
\usepackage{amsmath}
\usepackage{bbm}
\usepackage{mathrsfs}
\usepackage{stmaryrd}
\usepackage{subcaption}
\usepackage{amssymb}
\usepackage{mathrsfs}
\usepackage{appendix}
\usepackage{soul}
\newtheorem{thm}{Theorem}[section]
\newtheorem{prop}[thm]{Proposition}
\newtheorem{coro}[thm]{Corollary}
\newtheorem{lemma}[thm]{Lemma}
\theoremstyle{definition}
\newtheorem{exmp}[thm]{Example}
\newtheorem{defi}[thm]{Definition}
\theoremstyle{remark}
\newtheorem{remark}[thm]{Remark}
\usepackage[margin=1in]{geometry}
\numberwithin{equation}{section}
\newcommand*\dif{\mathop{}\!\mathrm{d}}

\newcommand{\second}{\prime \prime}
\newcommand{\R}{\mathbb R}
\newcommand{\mc}[1]{\mathcal{#1}}

\newcommand{\mr}[1]{\mathrm{#1}}
\newcommand{\bs}[1]{\boldsymbol{#1}}
\newcommand{\ms}[1]{\mathsf{#1}}
\newcommand{\msc}[1]{\mathscr{#1}}

\newcommand{\sabp}{\mc S^{\, [A B]+}}
\newcommand{\sabpt}{\mc S_t^{\, [A B]+}}

\newcommand{\sabpT}{{\mc S}_{\it T}^{\, [{\it A B}\,]+}}\newcommand{\sabm}{\mc S^{\, [A B]-}}
\newcommand{\sabmt}{\mc S_t^{\, [A B]-}}
\newcommand{\sabmT}{\mc S_T^{\, [A B]-}}
\newcommand{\sobap}{\overline{\mc S}^{\, [\overline B\, \overline A\,]+}}
\newcommand{\sobapt}{\overline{\mc S}_t^{\, [\overline B\, \overline A\,]+}}
\newcommand{\sobapT}{\overline{\mc S}_T^{\, [\overline B\, \overline A\,]+}}

\newcommand\blu[1]{{\color{black}#1}}

\title[Backward-forward characterization of attainable set for conservation laws]{Backward-forward characterization of attainable set\\ for conservation laws with spatially discontinuous flux}

\author{
Fabio Ancona \and Luca Talamini
}
\thanks{Dipartimento di Matematica ``Tullio Levi-Civita", Universit\`a di Padova, Italy, ~~E-mail: ancona@math.unipd.it, \\ luca.talamini@math.unipd.it}

\begin{document}

\maketitle

\begin{abstract}
Consider a scalar conservation law  
with a spatially discontinuous flux at a single point $x=0$, and assume that the flux is uniformly convex when $x\neq 0$.
Given an interface connection $(A,B)$, 
we define a backward solution operator consistent with the concept of $AB$-entropy solution~\cite{Adimurthi2005,MR2807133,BKTengquist}. We then analyze the
family $\mc A^{[AB]}(T)$ of profiles that
can be attained at time $T>0$ by $AB$-entropy solutions
with ${\bf L^\infty}$-initial data. 
We provide a  characterization of 
$\mc A^{[AB]}(T)$ as fixed points of
the backward-forward solution operator.
As an intermediate step we establish 
for the first time a full characterization of $\mc A^{[AB]}(T)$
in terms of 
unilateral constraints and Ole\v{\i}nik-type estimates, valid for all connections.
Building on such a characterization we derive uniform $BV$
bounds on the flux of $AB$-entropy solutions, which in turn yield the
$\bf L^1_{\mr{loc}}$-Lipschitz continuity 
in time of these solutions.

\end{abstract}

\tableofcontents

\section{Introduction}
Consider the initial value problem for the scalar conservation law in one space dimension
\begin{align}
    \label{conslaw}
&u_t+f(x,u)_x=0,
\quad x \in \mathbb{R}, \quad t \geq 0,
\\
\noalign{\smallskip}
    \label{initdat}
    &u(x,0) = u_0(x), \qquad x \in \mathbb{R},
\end{align}
where $u = u(x,t)$ is the state variable and the flux $f$ is a space discontinuous function given by 
\begin{equation}\label{discflux}
f(x,u) =\begin{cases}
f_l(u), \qquad x < 0, \\
f_r(u), \qquad x > 0\,.
\end{cases}
\end{equation}
We assume that $f_l, f_r: \R\to \R$
are twice continuously differentiable, uniformly  convex maps that satisfy
\begin{equation}
\label{eq:flux-assumption-1}
f_l''(u),\ f_r''(u) \geq a >0\,.
\end{equation}
Conservation
laws with discontinuous flux 
serve as mathematical models 
for: 
oil reservoir simulation~\cite{GimseRisebro1992,GR93};
 traffic flow dynamics
 with roads of varying amplitudes or surface conditions~\cite{Mochon1987};
 radar shape-from-shading problems~\cite{MR1912069};
blood flow in endovascular treatments~\cite{MR1901663,Ca};
and for many other different applications
(see~\cite{anconachiri} and references therein).

We recall that problems of this type
do not  posses classical solutions globally defined in time (even in the continuous flux case when $f_l=f_r$), since, regardless of how smooth the initial data are, they can develop discontinuities (shocks) in finite time
because of the nonlinearity of the equation. To achieve existence results, one has to look for 
weak distributional
solution
that, for sake of uniqueness, satisfy 
the classical Kru\v{z}kov entropy inequalities
away from the point of flux discontinuity, and a further interface entropy condition
at the flux-discontinuity interface $x=0$. 

Various type
of interface-entropy conditions have
been introduced in the literature
according with the different physical phenomena modelled by~\eqref{conslaw} (see~\cite{MR3416038,MR3369104}).
Here, 
as in~\cite{anconachiri}, 
for modellization and control treatment reasons
we employ an admissibility criterion involving the so-called {\it interface connection} $(A,B)$,
which yields the Definition~\ref{defiAB} of {\it $AB$-entropy solution}
(cfr.\cite{Adimurthi2005,BKTengquist}).
A connection $(A, B)$ is a pair of states connected by a stationary weak solution
of~\eqref{conslaw},
taking values $A$ for $x<0$, and $B$ for $x>0$, which has characteristics diverging
from (or parallel to) the flux-discontinuity interface $x=0$
(see Definition~\ref{def:connect}).
The admissibility criterion 
for an {\it $AB$-entropy solution} 
can be equivalently formulated in terms of an interface entropy condition
or of Kru\v{z}kov-type entropy inequalities adapted to the particular connection $(A, B)$ taken into account
(cfr.~\cite{Adimurthi2005,MR2807133,BKTengquist}). Relying on \blu{these} extended entropy inequalities
and using  an adapted version of the Kru\v{z}kov doubling of variables argument, one can establish ${\bf L^1}$-stability and uniqueness of $AB$-entropy solutions to the Cauchy problem~\eqref{conslaw}-\eqref{initdat}
 (see \cite{BKTengquist, Garavellodiscflux}). We shall adopt the semigroup notation 
$u(x,t)\doteq \sabpt u_0 (x)$ for the unique solution of~\eqref{conslaw}-\eqref{initdat}.
%
%

In this paper we are concerned as in~\cite{adimurthi2020exact,anconachiri}
with a controllability problem
for~\eqref{conslaw}
where one regards the initial data as controls and study 
the corresponding {\it attainable set} 
at a fixed time $T>0$:
\begin{equation}
\label{eq:attset}
    \mc A^{[AB]}(T) \doteq  \big\{\mc \sabpT u_0 \; : \; u _0 \in {\bf L}^{\infty}(\mathbb R)\big\}\,.
\end{equation}
In the same spirit of~\cite{ZuazuainverseproblemHJ,ZZreachableset,MR3643881,zuazua2020} 
we  introduce a {\it backward solution operator} 
(see Definition~\ref{def:backop})
\begin{equation}
    \sabmT: {\bf L^\infty}(\R) \to
   {\bf L^\infty}(\R),
   \qquad\quad \omega\mapsto 
   \sabmT \omega\,,
   \end{equation}
and we 
characterize the attainable targets for~\eqref{conslaw} at a time horizon $T>0$
as fixed-points of the composition {\it backward-forward operator} $\sabpT \circ \sabmT$, as stated in  our first main result: 
\begin{thm}\label{thm:backfordiscflux}
Let $f$ be a flux as in~\eqref{discflux}
satisfying the assumption~\eqref{eq:flux-assumption-1},
and let 
 $(A,B)$
 be a connection. Then,
 for every $T>0$,
 and for any $\omega\in  {\bf L^\infty}(\R)$,
 the following conditions are equivalent. 
\begin{enumerate}
    \item $\omega \in \mc A^{[AB]}(T)$, 
    \smallskip
        \item 
        $\omega=
        \mc \sabpT \circ \sabmT \omega\,.$\
\end{enumerate}
Moreover, if $(A,B)$ is a non critical connection, i.e. if $A\neq \theta_l, B\neq \theta_r$, then the condition (2) is equivalent to
\begin{enumerate}
    \item[(1)']
    $\omega \in \mc A_{bv}^{[AB]}(T)$, where
    \begin{equation}
    \mc A_{bv}^{[AB]}(T) \doteq  \big\{\mc \sabpT u_0 \; : \; u _0 \in BV_{loc}(\mathbb R)\big\}\,,
\end{equation}
\end{enumerate}
and it holds true
\begin{equation}
    \mc A^{[AB]}(T)=
    \mc A_{bv}^{[AB]}(T)\,.
\end{equation}
\end{thm}
Clearly the main content of Theorem~\ref{thm:backfordiscflux} are the implication (1) $\Longrightarrow$ (2) 
and (1)' $\Longrightarrow$ (2),
since the reverse implications
are straightforward once we define the backward operator 
$\sabmT$ and verify that, in the case
of a non critical connection, one has $ \sabmT u_0\in BV_{\mr{loc}}(\R)$
for all $u_0\in {\bf L^\infty}(\R)$. \blu{This last property is an immediate consequence of 
the uniform BV bounds
on $AB$-entropy solutions established in Proposition \ref{BVbound}, since the backward operator $\sabmT$ is defined in terms of the forward operator $\sobapt$ (see Definition \ref{def:backop}).}

\noindent
The proof of (1) $\Longrightarrow$ (2) 
and (1)' $\Longrightarrow$ (2)
are obtained in two steps:
\begin{itemize}
\item[(I)] 
First, we show that any attainable profile $\omega \in \mc A^{[AB]}(T)$
belongs
to a class of functions
$\msc A\subset BV_{\mr{loc}}(\mathbb R \setminus \{0\})$
which satisfy suitable Ole\v{\i}nik-type inequalities and pointwise constraints
related to the $(A,B)$-connection
in intervals containing the origin
(see Theorem~\ref{thm:attprofiles}, 
\ref{thm:attprofilescrit}, \ref{thm:attprofiles2},  \ref{thm:attprofiles3}).
\blu{
We classify the different type of
pointwise constraints satisfied by the attainable profiles in $\msc A$ highlighting the ones that can be recovered as limiting
cases (see Remarks~\ref{rem:cases-thm-4.1}, \ref{rem:cases-thm-4.8}, \ref{rem:threecases1}, \ref{rem:constr-LR=0},
\ref{rem:denseconds}).
}
\smallskip
\item[(II)] 
Next, we prove that any element of
$\msc A$ is a fixed point of
the composition  backward-forward operator $\sabpT \circ \sabmT$.
Namely, for any given $\omega\in \msc A$
we construct an initial datum $u_0\in{\bf L^\infty(\R)}$ such that $\omega=\sabpT u_0$, and then we show that indeed $u_0=\sabmT\omega$.
\end{itemize}
These two steps are firstly carried out 
in the case of a non critical 
connection $(A,B)$
and of attainable profiles
$\omega\in \mc A^{[AB]}(T)\cap BV_{\mr{loc}}(\mathbb R)$.
The proofs are obtained
exploiting as in~\cite{anconachiri}  the theory of generalized characteristics
 by Dafermos~\cite{dafermosgenchar},
 applied to the setting of discontinuous flux, and relying on the duality property of the backward and forward solution operators. Next, we address the case of a critical connection and
 of attainable profiles $\omega\in \mc A^{[AB]}(T)$ relying on the $\bf L^1_{\blu{\mr{ loc}}}$-stability of the map
 $(A,B,u_0)\mapsto \sabpt u_0$
 (see Theorem~\ref{theoremsemigroup}).
\medskip

Some remarks are here in order.
\begin{itemize}
[leftmargin=14pt]
    \item
    The results of Theorem~\ref{thm:backfordiscflux} extend to the present setting of space discontinuous fluxes the similar {\it characterization of attainable profiles in terms of the backward solution operator} obtained in~\cite[Theorem 3.1, Corollary 3.2]{COLOMBO2020} and~\cite[Corollary 1]{MR3643881}
for conservation laws with strictly convex flux independent on the space variable.
\smallskip
\item
The characterization of $\mc A^{[AB]}(T)$
obtained in this paper  unveils the presence of 
{\it \blu{two classes} of attainable states
for \blu{critical and non critical} connections}
that were {\it not detected in~\cite{adimurthi2020exact, anconachiri}},  see Remarks~\ref{rem:a-g}, ~\ref{rem:otherresults}.
\item The characterization of
attainable profiles for~\eqref{conslaw},
\eqref{discflux} in terms of
unilateral constraints and Ole\v{\i}nik-type estimates provides a powerful tool to investigate regularity properties 
of the solutions to ~\eqref{conslaw},
\eqref{discflux}. In particular, 
we build on such a characterization to 
derive {\it uniform BV bounds} 
on $AB$-entropy solutions
with initial datum in ${\bf L^\infty}$ (in the case of
non critical connections), and {\it on the flux
of $AB$-entropy solutions} (for general connections).
This is a fairly non-trivial result since it is well known~\cite{MR2743877} that the total variation of $AB$-entropy solutions
may well blow up in a neighborhood of the
flux-discontinuity interface $x=0$.
Thanks to these uniform BV bounds,
we can then establish the 
{\it $\bf L^1_{\mr{loc}}$-Lipschitz continuity 
in time of $AB$-entropy solutions}.
\item The proof that Theorem~\ref{thm:backfordiscflux}
holds for critical connections once we know that Theorem~\ref{thm:backfordiscflux}
is verified by non critical connections
relies on a perturbation argument 
for attainable profiles.
This construction yields an {\it approximate controllability} result since it 
provides a general explicit procedure to  approximate an attainable profile for a critical connection by attainable profiles for non critical connections.
    \end{itemize}

Note furthermore that, by the \blu{backward non-uniqueness} of~\eqref{conslaw}
(due to the possible presence of shocks
in its solutions),
there may exist in general multiple initial data $u_0$
that are \blu{steered} by~\eqref{conslaw} to
$\omega \in \mc A^{[AB]}(T)$.
In fact, an important control problem related to the one considered in this paper is the inverse design,
which has the goal to reconstruct
the set of initial data $u_0$ evolving to a given
attainable target $\omega$
(see~\cite{COLOMBO2020,MR3643881,zuazua2020, LZ23} 
for conservation laws with  convex flux independent on the space variable, and~\cite{ZZreachableset}
for Hamilton-Jacobi equations with convex Hamiltonian).
On the other hand, when
a target state $\omega$ is not attainable at time $T>0$, the image of $\omega$
through the backward-forward operator $\sabpT \circ \sabmT$ represents a natural candidate to construct a reachable function which is ``as
close as possible" (in an appropriate sense) to the observed state $\omega$ (see~\cite{ZuazuainverseproblemHJ}
in the case of Hamilton-Jacobi 
and Burgers equations).

The results of the present paper
provide a key building block 
to address both of these problems, namely
the characterization of the aforementioned set of initial data leading to a given
attainable target $\omega$
for~\eqref{conslaw},
and  the analysis of the properties
of the backward-forward operator $\sabpT \circ \sabmT$
related to optimization problems for unattainable target profiles, which are
pursued in the
forthcoming paper~\cite{talamini_ancona_initialdata}.

In the  case
of non-convex flux, an explicit characterizations of the attainable set in terms of
Ole\v{\i}nik-type estimates seems difficult to obtain and only partial results are present in the literature, see for example \cite{BDM17}. 
For systems of conservation laws, the problem has been considered in~\cite{BDGSU15} (triangular systems) and in \cite{CDN23} (chromatography system).
For a characterization of the attainable set in terms of fixed points of a backward-forward operator, 
a key point would be to 
provide a proper 
definition of backward operator
in these more general contexts,
which is lacking at the moment, making also the analysis of the inverse design problem nontrivial.

The paper is organized as follows. In \S~\ref{sec:basic-def}
we recall the definitions of interface connection $(A,B)$, of $AB$-entropy solution and of $AB$-backward solution operator.
We also collect the stability properties 
of the ${\bf L^1}$-contractive semigroup of $AB$-entropy solutions.
In \S~\ref{sec:building-blocks}
we establish the duality property
of the backward and forward solution operators, which constitutes a fundamental ingredient 
of the proof of Theorem~\ref{thm:backfordiscflux}.
\S~\ref{sec:statement-main}
collects the precise statements of the
results on the characterization
of the attainable set $\mc A^{[AB]}(T)$
 via
 Ole\v{\i}nik-type inequalities and state constraints. We also include 
  the statement of Theorem~\ref{thm:backfordiscfluxcycle}
 which contains the equivalence of 
 conditions (1), (2) of
 Theorem~\ref{thm:backfordiscflux}
 with the characterization of 
 $\mc A^{[AB]}(T)$ in terms of
 Ole\v{\i}nik-type inequalities and unilateral constraints.
In \S~\ref{sec:proof-main-thm}
we carry out the rather technical and involved proof of Theorem~\ref{thm:backfordiscfluxcycle}.
At the beginning of the section, 
for reader's convenience, 
we  provide a roadmap of the proof of {Theorem~\ref{thm:backfordiscfluxcycle}},
where we also highlight the key innovative parts of the paper.
In \S~\ref{sec:BVboundsABsol}
we  derive uniform BV bounds
on $AB$-entropy solutions 
in the case of non critical connections,
and on the flux of $AB$-entropy solutions
for general connections.
In Appendix~\ref{app:stabconn}
we establish the ${\bf L^1}$-stability
properties of the semigroup of $AB$-entropy solutions with respect to time and with respect to the connections.
In Appendix~\ref{app:no-rarefaction}
we provide, for sake of completeness, a simple proof 
of the non existence of rarefactions
emanating from the interface $x=0$, 
which is a distinctive feature of 
$AB$-entropy solutions.
Finally, in Appendix~\ref{app:uplwsmicsolns}
we derive some lower/upper ${\bf L^1}$-semicontinuity property for solutions to conservation laws, used to recover the
proof of { Theorem~\ref{thm:backfordiscfluxcycle}}
in the case of critical connections
once we know the validity of { Theorem~\ref{thm:backfordiscfluxcycle}}
for non critical connections.

\section{Basic definitions and general setting}
\label{sec:basic-def}

\subsection{Connections and $AB$-entropy solutions}

We recall here the definitions and properties of interface connection and of entropy admissible solution introduced in~\cite{Adimurthi2005}.

\begin{defi}[{\bf Interface Connection}]\label{ABsol}
\label{def:connect}
Let $f$ be a flux as in~\eqref{discflux}
satisfying the assumption~\eqref{eq:flux-assumption-1},
{ and let $\theta_l$, $\theta_r$ denote the unique critical points 
 of $f_l, f_r$, respectively.}
A pair of values $(A,B)\in \mathbb R^2$ is called a \textit{connection} if
\begin{enumerate}
\item $f_l(A) = f_r(B)$,
\item $A\leq \theta_l$ and $B \geq \theta_r$.
\end{enumerate} 
We will say that connection $(A,B)$ is  \textit{critical} if $A=\theta_l$
or $B=\theta_r$.
\end{defi}

\begin{figure}
\centering
\begin{tikzpicture}[scale = 0.7]
\draw[->] (-6,0)--(4,0)node[right]{$u$};
\draw[scale = 0.5, domain=-1:5, smooth, variable=\x] plot ({\x}, {(\x-2)*(\x-2)+1});
\draw[scale = 0.5, domain=-9:3, smooth, variable=\x] plot ({\x}, {(0.5*\x+1)*(0.5*\x+2)+1});

\draw (-5,5) node{$f_l$};
\draw (3,5) node{$f_r$};

\draw[very thick, blue] (-3.8,3)--(2.13,3);
\draw[dotted] (-3.8,3)--(-3.8,0)node[below]{$A$};
\draw[dotted] (2.13,3)--(2.13,0)node[below]{$B$};

\draw[dotted] (-1.5,0.3)--(-1.5,0)node[below]{$\theta_l$};
\draw[dotted] (1,0.5)--(1,0)node[below]{$\theta_r$};
\end{tikzpicture}
\caption{An example of connection $(A,B)$ with $f_l, f_r$ strictly convex fluxes}
\label{Connectionpic}
\end{figure}

\begin{figure}[ht]
\centering

\tikzset{every picture/.style={line width=0.75pt}} 

\begin{tikzpicture}[x=0.75pt,y=0.75pt,yscale=-1,xscale=1, scale =0.9]

\draw [line width=1.5]    (171.42,211.04) -- (425.58,210.63) ;
\draw [shift={(429.58,210.63)}, rotate = 179.91] [fill={rgb, 255:red, 0; green, 0; blue, 0 }  ][line width=0.08]  [draw opacity=0] (6.97,-3.35) -- (0,0) -- (6.97,3.35) -- cycle    ;
\draw [line width=1.5]    (300.02,86.25) -- (300.5,210.83) ;
\draw [shift={(300,82.25)}, rotate = 89.78] [fill={rgb, 255:red, 0; green, 0; blue, 0 }  ][line width=0.08]  [draw opacity=0] (6.97,-3.35) -- (0,0) -- (6.97,3.35) -- cycle    ;
\draw  [dash pattern={on 3pt off 3pt}]  (180,90.58) -- (300.5,210.83) ;
\draw  [dash pattern={on 3pt off 3pt}]  (180,110.75) -- (280.33,210.83) ;
\draw  [dash pattern={on 3pt off 3pt}]  (180,130.25) -- (260,210.5) ;
\draw  [dash pattern={on 3pt off 3pt}]   (180.83,149.92) -- (240.67,211.17) ;
\draw [dash pattern={on 3pt off 3pt}]    (180.5,170.75) -- (222,212) ;
\draw [dash pattern={on 3pt off 3pt}]    (180.33,190.75) -- (199.83,208.5) ;
\draw [line width=0.75]  [dash pattern={on 3pt off 3pt}]  (200,90.5) -- (300.33,190.83) ;
\draw [color={rgb, 255:red, 0; green, 0; blue, 0 }  ,draw opacity=1 ][line width=0.75]  [dash pattern={on 3pt off 3pt}]  (220.67,90.75) -- (300.33,170.83) ;
\draw  [dash pattern={on 3pt off 3pt}]  (239.83,90.58) -- (299.67,150.17) ;
\draw  [dash pattern={on 3pt off 3pt}]  (260,90.08) -- (300,130.5) ;
\draw  [dash pattern={on 3pt off 3pt}]  (279.83,90.08) -- (299.67,110.17) ;
\draw   [dash pattern={on 3pt off 3pt}]  (419.33,115.83) -- (300.33,170.83) ;
\draw [dash pattern={on 3pt off 3pt}]    (419.67,176.17) -- (341.33,210.17) ;
\draw  [dash pattern={on 3pt off 3pt}]   (350,90.17) -- (299.67,110.17) ;
\draw [dash pattern={on 3pt off 3pt}]    (419.67,136.5) -- (301.67,189.17) ;
\draw [dash pattern={on 3pt off 3pt}]    (390.67,90.83) -- (300,130.5) ;
\draw [dash pattern={on 3pt off 3pt}]    (420,95.5) -- (299.67,150.17) ;
\draw [dash pattern={on 3pt off 3pt}]    (419.67,157.5) -- (300.5,210.83) ;
\draw [dash pattern={on 3pt off 3pt}]    (418.67,194.5) -- (379.67,210.5) ;

\draw (232.67,216.73) node [anchor=north west][inner sep=0.75pt]  [font=\scriptsize]  {$A$};
\draw (355,215.9) node [anchor=north west][inner sep=0.75pt]  [font=\scriptsize]  {$B$};
\draw (419,216.4) node [anchor=north west][inner sep=0.75pt]  [font=\scriptsize]  {$x$};
\draw (305.5,74.9) node [anchor=north west][inner sep=0.75pt]  [font=\scriptsize]  {$t$};

\end{tikzpicture}

\caption{The stationary undercompressive solution $c^{AB}$.}
\label{statsolutionpic}
\end{figure}

Observe that condition (2) is equivalent 
to: $f'_l(A)\leq 0$, \ $f'_r(B)\geq 0$.
Therefore, if $(A,B)$ is a connection,
then
the function 
\begin{equation}\label{eq:cABdef}
c^{AB}(x) \stackrel{\cdot}{=}
\begin{cases}
A, & x \leq 0, \\
B, & x \geq 0
\end{cases}
\end{equation}
is a weak stationary undercompressive (or marginally undercompressive) solution of \eqref{conslaw}, since the characteristics diverge from, or are parallel to, the flux-discontinuity interface (see Figure \ref{statsolutionpic}).
In relation to the function $c^{AB}$  the \, {\it adapted entropy} \,
$\eta_{AB}(x,u)=\big|u-c^{AB}(x)\big|$ is introduced  in~\cite{BKTengquist}. Then, in the spirit of~\cite{audusse2005uniqueness},  the entropy $\eta_{AB}$  is employed in~\cite{BKTengquist} to select a unique
solution of the Cauchy problem~\eqref{conslaw}-\eqref{initdat} that satisfies the interface entropy inequality
\begin{equation}\label{adaptedABentropy}
\left|u-c^{AB}\right|_t + \left[\mathrm{sgn}(u-c^{AB})(f(x,u)-f(x,c^{AB}))\right]_x \leq 0, \quad \textrm{in}\ \ \mathcal{D}^{\prime}\,,
\end{equation}
in the sense of distributions, which leads to the following definition.

\begin{defi}[{\bf $AB$-entropy solution}]\label{defiAB}
Let $(A,B)$ be a connection
and let $c^{AB}$ be the function defined in~\eqref{eq:cABdef}. A function $u \in \mathbf L^{\infty}(\mathbb R \times [0,+\infty[\,){ ~\cap~ \mc C^0([0,+\infty), \; \mathbf L^1_{loc}(\mathbb R)) }$ is said to be an \textit{$AB$-entropy solution} of the problem \eqref{conslaw},\eqref{initdat} if the following holds:
\begin{enumerate}
[leftmargin=25pt]
\item $u$ is a distributional solution of \eqref{conslaw} on $\R\times \,]0,+\infty[$, that is, for all test functions $\phi \in \mathcal{C}^1_c$
with compact support 
contained in $\mathbb{R}\times \,]0,+\infty[$,
it holds true
\begin{equation}\label{AB1}
\int_{-\infty}^{\infty}\int_0^{\infty}  \big\{u \phi_t+f(x,u)\phi_x \big\}\dif x \dif t = 0\,.
\end{equation}
\item $u$ is a Kru\v{z}kov entropy weak solution of \eqref{conslaw},\eqref{initdat} on $(\mathbb R \setminus \{0\}) \times \,]0,+\infty[$, that is
 the initial condition~\eqref{initdat} is satisfied almost everywhere,
and:
\begin{itemize}
[leftmargin=27pt]
\item[(2.a)]
for any non-negative test function $\phi\in\mathcal{C}^1_c$ with compact support 
contained in $]-\infty,0[\,\times $ $\,]0,+\infty[$, it holds true
\begin{equation}\label{kruz-}
\int_0^{\infty} \int_{-\infty}^0  \big\{|u-k| \phi_t + \mathrm{sgn}(u-k)\left(f_l(u) - f_l(k)\right) \phi_x \big\}
\dif x \dif t \geq 0, \quad \forall k \in \mathbb R\,;
\end{equation}
\item[(2.b)]
for any non-negative test function $\phi\in\mathcal{C}^1_c$ with compact support 
contained in $]0,+\infty[\,\times$ $]0,+\infty[$, it holds true
\begin{equation}\label{kruz+}
\int_0^{\infty}\int_0^{\infty} \big\{
|u-k| \phi_t + \mathrm{sgn}(u-k)\left(f_r(u) - f_r(k)\right) \phi_x
\big\}\dif x \dif t \geq 0, \quad \forall k \in \mathbb R\,.
\end{equation}
\end{itemize}

\item $u$ satisfies a Kru\v{z}kov-type entropy inequality relative to the connection $(A,B)$, that is, 
for any non-negative test function $\phi\in\mathcal{C}^1_c$ with compact support 
contained in $\R\times ]0,+\infty[$, it holds true
\begin{equation}\label{adaptedentropyinequality}
\int_{-\infty}^{\infty}\int_0^{\infty} \big\{\left|u-c^{AB}\right| \phi_t + \mathrm{sgn}(u-c^{AB})\left(f(x,u) - f(x,c^{AB})\right)\phi_x \big\}\dif x \dif t \geq 0\,.
\end{equation}
\end{enumerate}
\end{defi}

%

\begin{remark}
\label{rem:abentr-sol-prop1}
Since
an $AB$-entropy solution $u$ is in particular an entropy weak solution 
of a scalar conservation law
with uniformly convex flux,
on $]-\infty, 0[\,\times\,]0,+\infty[$,
and on $]0,+\infty[\,\times\,]0,+\infty[$
(by property (2) of Definition~\ref{defiAB}
and assumption~\eqref{eq:flux-assumption-1}),
it follows that $u(\cdot\,,t)\in {BV}_{\mathrm{loc}}(\mathbb R \setminus \{0\})$ for any $t>0.$
Here $BV_{\mathrm{loc}}(\mathbb R\setminus\{0\})$
denotes the set of functions that have finite total variation on compact subsets of $\mathbb R \setminus \{0\}$.
On the other hand, 
relying on  a result in \cite{panov}
(see also~\cite{MR1869441}), 
we deduce that
$u$ admits left and right strong traces at $x = 0$ for a.e. $t>0$, i.e. that 
there exist the one-sided limits \begin{equation}\label{traces}
u_l(t)\doteq u(0-,t),
\qquad 
u_r(t)\doteq  u(0+,t),
\qquad \text{for a.e. $t > 0$}\,.
\end{equation}
We point out that a consequence of 
the characterizatin of attainable profiles
provided by our results (Theorems \ref{thm:attprofiles}, \ref{thm:attprofilescrit}, \ref{thm:attprofiles2}, \ref{thm:attprofiles3})
will be that these limits are actually defined at every time $t>0$ (not only at almost every time).
Moreover, 
since $u$ is also a distributional solution of \eqref{conslaw} on $\R\times \,]0,+\infty[$
(by property (1) of Definition~\ref{defiAB}), we deduce that
$u$ must satisfy the Rankine-Hugoniot condition at the interface $x=0:$
\begin{equation}\label{RHtraces}
f_l(u_l(t)) = f_r(u_r(t)), \qquad \text{for a.e. $t>0$}.
\end{equation}
\end{remark}

\noindent
In~\eqref{traces} and throughout the paper, for the one-sided limits of a function $u(x)$ we use the notation
\begin{equation}
    u(x\,\pm)\doteq \lim_{y\to x\,\pm} u(y).
\end{equation}
\medskip

In relation to a connection $(A,B)$
consider the function
\begin{equation}
\label{IABdef}
I^{AB}(u_l,u_r)  \stackrel{\cdot}{=} \mathrm{sgn}(u_r-B)\left(f_r(u_r)-f_r(B)\right) \\- \mathrm{sgn}(u_l-A)\left(f_l(u_l)-f_l(A)\right),
\qquad u_l,u_r\in\R\,,
\end{equation}
which is useful to characterize 
the interface entropy admissibility criterion. In fact, by the analysis in ~\cite[Lemma 3.2]{BKTengquist} and ~\cite[Section~4.8]{MR2807133},
it follows that, because of
condition (1) of Definition~\ref{defiAB}
and assumption~\eqref{eq:flux-assumption-1}, 
the following holds.

\begin{lemma}
\label{lem:traceseq}
Let $u \in \mathbf L^{\infty}(\mathbb R \times [0,+\infty))$
be a function satisfying conditions (1)-(2) of Definition~\ref{defiAB}.
Then, condition (3) is equivalent to the
$AB$ interface entropy condition
\begin{itemize}
    \item[(3)'] 
    \begin{equation}
    \label{interfaceentropy}
I^{AB}(u_l(t),u_r(t)) \leq 0 \qquad\quad  \text{for a.e. $t>0$\,.}
\end{equation}
\end{itemize}
%
\end{lemma}

\begin{lemma}
\label{lem:traceseq2}
Let $(A,B)$ be a connection. Then, for any pair $(u_l,u_r)\in\R^2$, the conditions
\begin{equation}
        f_l(u_l)=f_r(u_r),
        \qquad\quad 
        I^{AB}(u_l, u_r)\leq 0\,,
\end{equation}
are equivalent to the conditions
\begin{equation}
\label{ABtraces}
\begin{aligned}
    &f_l(u_l) = f_r(u_r) \geq f_l(A) = f_r(B), \\ 
    \noalign{\smallskip}
&\big(u_l \leq \theta_l, \quad u_r \geq \theta_r\big) \ \ \Longrightarrow \ \ u_l = A, \ \ u_r = B\,.
\end{aligned}
\end{equation}
\end{lemma}
\medskip

The first condition in~\eqref{ABtraces}
tells us that, when we choose
a connection $(A,B)$
and we employ the concept of
$AB$-entropy solution, we are 
imposing a constraint (from below) on the flux at the interface $x=0$.
In order to achieve existence, we need to compensate for this constraint with an  additional freedom in the admissibility criteria.
In fact, the second condition in~\eqref{ABtraces} 
prescribes the admissibility of exactly one undercompressive wave at the interface, given by $c^{AB}$ in~\eqref{eq:cABdef}. 
This rule corresponds to the 
 {\it $(A,B)$ characteristic condition} in~\cite[Definition~1.4]{BKTengquist}.
\medskip

\begin{remark}
\label{rem:char-ABsol}
 In view of~\eqref{RHtraces}, 
 we can extend the classical concept of genuine characteristic for solutions to conservation laws with continuous fluxes
 (see~\cite{dafermosgenchar})
 by considering also characteristics that are refracted by the discontinuity interface $x=0$. Thus, we will say  that a polygonal line $\vartheta : [0,T]\to\R$ is a {\it genuine characteristic for an $AB$-entropy solution~$u$} if  one of the following cases occurs:
 \begin{itemize}
 [leftmargin=25pt]
     \item[(i)] $\vartheta(t)<0$ for all $t\in\,]0,T[\,$, and $\vartheta$ is a  characteristic for the restriction of $u$ on  $]-\infty, 0[\,\times\,]0, T[\,$;
      \item[(ii)] $\vartheta(t)>0$ for all $t\in\,]0,T[\,$, and $\vartheta$ is a characteristic for the restriction of $u$ on  $]0,+\infty[\,\times\,]0, T[\,$;
      \item[(iii)] there exists $\tau\in \,]0, T[$\,, such that: 
      \begin{itemize}
      [leftmargin=10pt]
          \item[-]
          $\vartheta(t)<0$ for all $t\in\,]0,\tau [\,$, and $\vartheta$ is a characteristic for the restriction of $u$ on  $]-\infty, 0[\,\times\,]0, \tau[\,$,
          \item[-]
      $\vartheta(t)>0$ for all $t\in\,]\tau,T[\,$, and $\vartheta$ is a characteristic for the restriction of $u$ on  $]0,+\infty[\,\times\,]\tau, T[\,$,\\
      or viceversa.
      \item[-]
      $f_l(u_l(\tau))=f_r(u_r(\tau))$ and $I^{AB}(u_l(\tau),u_r(\tau)) \leq 0 $,
      \end{itemize}
      %
 \end{itemize}
where we are using the term ``characteristic'' for a classical genuine characteristic of a solution to the conservation law $u_t+f_l(u)_x=0$
on $\{x<0\}$, or to the conservation law $u_t+f_r(u)_x=0$
on $\{x>0\}$.
\end{remark}

{ 
\begin{remark}[Local solutions]
 Throughout the  paper  we say that 
 a function $u \in \mathbf L^{\infty}(\Omega)$  is 
 a (local) \textit{$AB$-entropy solution} of~\eqref{conslaw}
 on a domain $$\Omega \doteq \big\{(t,x) \;  | \; t \in [a, b], \quad \gamma_1(t) < x <  \gamma_2(t) \big\}\subset  \mathbb R\times\,]0,+\infty[$$ where $\gamma_1< \gamma_2 :[a,b] \to \mathbb R$ are Lipschitz curves if it satisfies the conditions of Definition~\ref{defiAB} localized on $\Omega$.
 Namely, if the following holds:
 \begin{enumerate}
     \item For any test functions $\phi \in \mathcal{C}^1_c(\Omega)$
with compact support 
contained in $\Omega$,
it holds true~\eqref{AB1}.
      \item 
      The map $t \mapsto u(\cdot,t)$ is continuous  from $I\doteq\{t>0\,:\, (x,t)\in\Omega\}$ to ${\bf L^1_{\mr{loc}}}(\Omega_t)$,
      $\Omega_t\doteq\{x\,:\, (x,t)\in\Omega\}$,
      and it holds:
     \begin{itemize}
     [leftmargin=27pt]
\item[(2.a)] for any non-negative test function $\phi\in\mathcal{C}^1_c(\Omega)$ with compact support 
contained in $\Omega\cap\big(]-\infty,0[\,\times\,]0,+\infty[\big)$, it holds true~\eqref{kruz-};
\item[(2.b)] or any non-negative test function $\phi\in\mathcal{C}^1_c(\Omega)$ with compact support 
contained in $\Omega\cap\big(]0, +\infty[\,\times\,]0,+\infty[\big)$, it holds true~\eqref{kruz+}.
\end{itemize}
     \item For any test functions $\phi \in \mathcal{C}^1_c(\Omega)$
with compact support 
contained in $\Omega $,
it holds true~\eqref{adaptedentropyinequality}.
 \end{enumerate}
We will sometimes implicitly use the following fact: assume that two local $AB$-entropy solutions~$u_1, u_2$, of~\eqref{conslaw}
are given on two
disjoint domains $\Omega_1, \Omega_2$,
such that $\partial \Omega_1 \cap \partial \Omega_2 = \Gamma$,
where $\Gamma$ is the graph of a Lipschitz curve $\gamma:[a,b] \to \mathbb R$, with $\Omega_1 \subset \{x \leq  \gamma(t)\}$, $\Omega_2 \subset \{x \geq  \gamma(t)\}$.
Moreover, assume that 
$u_1(t,\gamma(t)-) = u_2(t, \gamma(t)+)$ for a.e. $t \in [a, b]$ such that $\gamma(t) \neq 0$, and that 
$$f_l(u_1(0-,t))=f_r(u_2(0+,t)),\qquad\quad I^{AB}(u_1(0-,t),u_2(0+,t)) \leq 0,$$
for a.e. $t\in [a,b]$ such that $\gamma(t)=0$.
Then, by standard arguments one can deduce  that the function 
$$
u(x,t) = \begin{cases}
    u_1(x,t) & \text{if $(x,t) \in \Omega_1$},\\
    u_2(x,t) & \text{if $(x,t) \in \Omega_2$}
\end{cases}
$$
is an $AB$-entropy solution of~\eqref{conslaw} on $\Omega$.
\end{remark}
}
\medskip

It was proved in~\cite{Adimurthi2005,BKTengquist} (see also~\cite{MR2807133,Garavellodiscflux})
that  $AB$-entropy solutions of~\eqref{conslaw},\eqref{initdat}
with bounded initial data are unique and form an $\mathbf{L}^1$-contractive semigroup. Moreover, we will show that they are ${\bf L^1}$-stable also with respect to the
values $A,B$ of the connection.
This type of stability, beside 
being used to extend our main results
from the case of  non critical connections to the critical one,
has 
an interest on its own.
We will also
prove that $AB$-entropy solutions of~\eqref{conslaw},\eqref{initdat}
are ${\bf L^1}$-Lipschitz continuous in time. This property is an immediate consequence of the BV regularity
of such solutions in the case of
non critical connections. 
Instead,  
in the case of critical connections where
$A=\theta_l$ or $B=\theta_r$, and $f_l(\theta_l)\neq f_r(\theta_r)$, the total variation
of an $AB$-entropy solution
may well blow up in a neighborhood of the
flux-discontinuity interface $x=0$, 
as shown in~\cite{MR2743877}. 
However, we recover the ${\bf L^1}$-Lipschitz continuity in time also in this case exploiting the BV regularity of the flux of an $AB$-entropy solution, which
is established relying on the analysis pursued in this paper.
We collect all these (old and new) results in the following statement:
\begin{thm}\label{theoremsemigroup}
{\bf (Semigroup of $\textbf{\textit{AB}}$-Entropy Solutions)}
Let $f$ be a flux as in~\eqref{discflux}
satisfying the assumption~\eqref{eq:flux-assumption-1},
and let 
 $(A,B)$
 be a connection. Then there exists a map
 \begin{equation}
 \label{eq:f-AB-op}
    \sabp: [0,+\infty[\,\times\,{\bf L^\infty}(\R) \to
   {\bf L^\infty}(\R),
   \qquad\quad (t,u_0)\mapsto 
   \sabpt u_0\,,
   \end{equation}
enjoying the following properties:   
\begin{enumerate}
\item[(i)] For each $u_0 \in {\bf L}^{\infty}(\R)$, the function $u(x,t) \doteq \sabpt u_0(x)$ provides the unique bounded, $AB$-entropy  solution of the Cauchy problem \eqref{conslaw}, \eqref{initdat}.
    \item[(ii)] 
    $\sabp_0  u_0 = u_0$, \ $\sabp_s \circ \sabp_t u_0= \sabp_{s+t}u_0$, \ for all \ $t,s \geq 0$, $u_0\in {\bf L^\infty}(\R)$.
    \smallskip
    \item[(iii)] For any $u_0, v_0\in {\bf L^\infty}(\R)$, there exists a constant $L>0$,
depending on  $f$ and 
on $\|u_0\|_{\bf L^\infty}, \|v_0\|_{\bf L^\infty}$,
such that, for any $R > 0$, it holds: \\
$\big\Vert \sabpt u_0-\sabpt v_0\big\|_{{\bf L}^1([-R,R])}
\leq \big\Vert u_0-v_0\big\|_{{\bf L}^1([-R-Lt, R+Lt])}$,
    \ for all \ $t\geq 0$.
    
    \smallskip
\item[(iv)] For any $u_0
\in {\bf L^\infty}(\R)$, 
and for any $R > 0$, it holds: \\
$\big\Vert \sabpt u_0-\mathcal{S}_t^{[A^{\prime}B^{\prime}]+}u_0
\big\Vert_{{\bf L}^1([-R ,R])} \leq 2\,t\, \big| f_r(B)-f_r(B^{\prime}) \big|$,\\
for all connections $(A,B)$, $(A^{\prime}, B^{\prime})$, and for all $t \geq 0$.
\smallskip
{\item[(v)] 
For any
$u_0\in {\bf L^\infty}(\R)$,
and for any $R>0$,
there exists a constant $C_{R}>0$ 
depending on $f$, $\|u_0\|_{\bf L^1}$, $R$,
and on the connection $(A,B)$,
such that it holds:\\
$\big\Vert \sabpt u_0-\mc S^{[AB]+}_s u_0\big\Vert_{{\bf L}^1{([-R,R])} } \leq \dfrac{C_{R}}{t} 
|s-t|$, \  for all $s> t>0$.
}
\end{enumerate}
\end{thm}

\noindent
The proof of the new properties (iv)-(v) is given in  Appendix~\ref{app:stabconn}.

    { 
\begin{remark}
Most of the literature on conservation laws with discontinuous flux as in~\eqref{discflux}
    considers fluxes $f_l, f_r$ satisfying
    \begin{equation}
\label{eq:flux-assumption-2}
f_l(0) = f_r(0), \quad f_l(1) = f_r(1)\,, \qquad\quad \theta_l \geq 0, \quad \theta_r \leq 1.
\end{equation}
However 
the existence of an $\mathbf{L}^1$-contractive semigroup of $AB$-entropy  solutions of the Cauchy problem \eqref{conslaw}, \eqref{initdat},
remains valid also without this assumption as shown for example in~\cite{Garavellodiscflux}.
On the other hand, 
by a reparametrization of the fluxes, one can always 
    reduce the problem to the setting where 
    the critical points 
    of $f_l, f_r$ satisfy~\eqref{eq:flux-assumption-2}.
In fact, given $f_l, f_r$, 
    for any pair of invertible affine maps $\phi_l, \phi_r : \mathbb R \to \mathbb R$, we can observe that a map $u(x,t)$
    is an $AB$-entropy  solution
    of~\eqref{conslaw}, \eqref{discflux} with fluxes $f_l, f_r$ if and only if 
$$\widetilde u(x,t) \doteq \begin{cases}
    \phi_l^{-1}(u) & \text{if $x <0,$} \\
    \phi_r^{-1}(u) & \text{if $x > 0,$}
\end{cases}
$$
is an $AB$-entropy  solution of~\eqref{conslaw}, \eqref{discflux} with  fluxes $f_l\circ \phi_l, \, f_r\circ \phi_r$.
\end{remark}
}
%
\smallskip
\begin{remark}
\label{rem:linf-bound-ABsol}
By the analysis in~\cite[\S 3.1]{Garavellodiscflux}
(see also~\cite[Remark~4.1]{anconachiri})
it follows that, for every $M>0$,
there exists  $C_M>0$ such that, if $\|u_0\|_{{\bf L}^\infty}\leq M$, and $A,B\leq M$, then 
$\|\sabpt u_0\|_{{\bf L}^\infty}\leq C_M$, for all~$t>0$.
\end{remark}
\smallskip
\begin{coro}
\label{cor:fluxtraces-stab}
Let $\{(A_n, B_n)\}_n$ be a sequence of connections that converges in $\mathbb R^2$ to a connection~$(A, B)$, and let $\{u_{n,0}\}_n$
be a sequence of functions in  ${\bf L}^\infty(\R)$
that converges in~$\bf L^1_{\mr{loc}}$ to $u_0\in {\bf L}^{\infty}(\R)$.
Let $u_{n,l}, u_{n,r}$ denote, respectively, 
the left and right traces 
at $x=0$ of $u_n(x,t)\doteq \mc S^{[A_nB_n]+}_t u_{n,0}(x)$,
defined as in~\eqref{traces}.
Similarly, let $u_l, u_r$ denote the left and right traces 
at~$x=0$ of $u(x,t)\doteq \mc S^{[AB]+}_t u_0(x)$.
Then, we have:
\begin{equation}
\label{eq:flux-trace-conv1}
     f_l(u_{n,l})
    \ \  \rightharpoonup  \ \
     f_l(u_{l}),
     \qquad 
     f_r(u_{n,r})
    \ \  \rightharpoonup \ \
     f_r(u_{r})
    \quad\text{weakly \ in}\quad
    \mathbf L^1(\mathbb R^+)\,.
\end{equation}
\end{coro}
\noindent
The proof of the Corollary is given in  Appendix~\ref{app:stabconn}.

\begin{remark}
    We point out that, differently
    from~\eqref{eq:flux-trace-conv1},
    in general 
    the ${\bf L^1}$-convergence 
    $u_{n, l} \to u_l$
    and $u_{n, r} \to u_r$
    fails due to the possible formation of stationary boundary layers
    at the interface $x=0$, as one can see in the following
    \end{remark}
    \begin{exmp}
    Consider a  non critical connection 
    $(A, B)$ and the sequence of initial data
    $$
    u_{n,0}(x) = \begin{cases}
        \overline A, & \text{if $x \leq -n^{-1}$}, \\
        A, & \text{if $x \in \,]-n^{-1},0[$}, \\
        \overline B, & \text{if $x \geq 0$},
    \end{cases}
    $$
    with
   \begin{equation}
\label{eq:bar-AB-def}
     \overline A \doteq ({f_l}_{\mid [\theta_l, +\infty[\,})^{-1}\circ f_l(A)\,,
     \qquad\quad 
     \overline B \doteq 
    ({f_r}_{|\,]-\infty, \theta_r]})^{-1}\circ f_r(B),
\end{equation}
where $f_{|\,I}$
denotes the restriction of the function $f$ to the interval $I$.
    One can immediately check that the $AB$-entropy  solution
    of~\eqref{conslaw}, \eqref{initdat}, with initial datum $u_{n,0}$ is the stationary solution $u_n(t, x) = \mc S_t^{[AB]+}u_{n,0}(x) = u_{n,0}(x)$ and that $u_{n}$ converges in ${\mathbf L^1}$ to
    $$u(x,t)=\begin{cases}
        \overline A\ \ &\text{if}\quad x<0,
        \\
        \noalign{\smallskip}
        \overline B\ \ &\text{if}\quad x>0\,.
    \end{cases}
    $$
    Moreover, one has $u_{n,l}(t) = A$ and $u_l(t) = \overline A$ for every $t > 0$.
    \end{exmp}

\medskip

\subsection{Backward 
solution operator} 
In this section we shall first  review quickly the concept of 
backward solution
operator 
for conservation laws with flux
depending only on the state variable, and then
we will introduce the
definition of  
backward solution
operator associated to a connection~$(A,B)$,
for spatially discontinuous flux as in~\eqref{discflux}.

\subsubsection{Backward 
solution operator for conservation laws with space independent flux}\label{sec:backoopconv}
The use of the backward-forward method to characterize 
the attainable set for
conservation laws   was first proposed in~\cite{MR3643881,zuazua2020}  (see also~\cite{ZuazuainverseproblemHJ} in the framework of Hamilton-Jacobi equations). 
Because of the regularizing effect of the nonlinear dynamics 
of a conservation law 
\begin{equation}
\label{eq:convexconslaw}
    u_t+f(u)_x = 0,\qquad  x\in\R\,,
\end{equation}
with 
uniformly convex flux $f(u)$,
the only restriction to
controllability of~\eqref{eq:convexconslaw} at a fixed time $T>0$, when one regards as controls the initial data
\begin{equation}
\label{eq:indatum}
   u(x,0) = u_0(x), \qquad\quad\, x \in \mathbb R\,,
\end{equation}
is the decay of positive waves.
Therefore it is by now well known 
the characterization of the
attainable set
\begin{equation}
\label{eq:convvattsett}
    \mc A(T) = \big\{\omega \; : \; \omega = u(\cdot,T), \; \text{$u$ entropy weak solution of \eqref{eq:convexconslaw}-\eqref{eq:indatum} with $u_0\in{\bf L^\infty}$}\big\},
\end{equation}
in terms of the 
Ole\v{\i}nik-type inequality
\begin{equation}\label{eq:convexoleinik}
    D^+ \omega(x) \leq \frac{1}{Tf^{\prime \prime}(\omega(x))}, \quad \forall x \in \mathbb R\,,
\end{equation}
where $D^+ \omega$ denotes the upper Dini derivative of $\omega$
(see~\eqref{eq:dini_der_def}).
Similar results in the case of boundary controllability were obtained in~\cite{AGGcontrolconvex,anconamars,anconamars99,MR1612027}).

A different perspective to address this controllability problem 
was introduced in~\cite{MR3643881,zuazua2020},
and consists in constructing
initial data leading to 
 attainable targets $\omega$
 at a time horizon $T>0$,
through the definition
of an
appropriate concept of backward solution to~\eqref{eq:convexconslaw}.
Namely, letting $\mc S_t^+u_0(x)$ denote the (forward) entropy weak solution of the Cauchy problem~\eqref{eq:convexconslaw}-\eqref{eq:indatum}
evaluated at $(x,t)$,
it was defined in~\cite{zuazua2020} an appropriate \textit{backward operator} $\mc S_T^-: {\bf L^{\infty}} \to {\bf L^{\infty}}$, and proved that a profile $\omega$ belongs to $\mc A(T)$
if and only if $\omega = \mc S_{T}^+ \circ \mc S_{T}^- \omega$, i.e. if and only if it is a fixed point of the backward-forward operator $\mc S_{T}^+ \circ \mc S_{T}^-$ (see ~\cite[Corollary 1]{MR3643881}).
Moreover, for $\omega \in \mc A(T)$, the solution defined as 
\begin{equation}
\label{eq:u^*class}
    u^*(x,t) \doteq  \mc S_t^+ (\mc S_{T}^- \omega)(x),
    \qquad x\in\R,\ t\in [0,T]\,,
\end{equation}
is the unique  solution to \eqref{eq:convexconslaw} that
is locally Lipschitz on the strip $\mathbb R\times ]0,T[$, and yields $\omega$ at time $T$. Equivalently, 
\begin{equation}
\label{eq:indatum-lip}
    u^*_0 \doteq \mc S_{T}^- \omega
\end{equation} 
is the unique initial datum that produces a  solution to \eqref{eq:convexconslaw}
locally Lipschitz
on $]0,T[$,
yielding $\omega$ at time~$T$. 
The operator $\mc S_{t}^-$, for $t \geq 0$, is defined as follows
\begin{equation}
\label{eq:backw-sol-def}
    \mc S_{t}^-\omega(x)  \doteq  \mc S_t^+ \big(\omega(-\ \cdot\,)\big)(-x)\qquad x\in\R,\ t\geq 0\,.
\end{equation}
In words, we use $\omega(-\ \cdot\,)$ as initial datum for the forward operator $\mc S_t^+$, we compute  the (forward) solution
to \eqref{eq:convexconslaw}, and then we reverse the space variable.
{ 
\begin{remark}[Classical  solutions]
Throughout the  paper by a \emph{classical solution} to a conservation law with space independent flux $u_t + f(u)_x = 0$ we mean a locally Lipschitz function $u : \Omega \to \mathbb R$, $\Omega \subset \mathbb R\times\,]0,+\infty[$, such that 
$$
u_t(t, x) + f(u(x,t))_x = 0 \qquad \text{for a.e. $(t,x) \in \Omega$}.
$$
Any classical solution is an entropy admissible weak solution.
The function~\eqref{eq:u^*class} is a classical solution to~\eqref{eq:convexconslaw}. 
Sometimes in the literature classical solutions are denoted as {\it strong solutions}.
\end{remark}
}
\begin{remark}
    \label{charact-backw_sol}
    One can easily verify that
the function $\text{w}(x,t)\doteq \mc S_{t}^-\omega(x)$ is the entropy weak solution of the Cauchy problem
\begin{equation}
\label{eq:backw-cp}
    \begin{cases}
        \text{w}_t-f(\text{w})_x=0, \quad &x \in \mathbb{R}, \quad t \geq 0,
\\
\noalign{\smallskip}
    \text{w}(x,0)=\omega(x), \quad &x \in \mathbb{R}.
    \end{cases}
\end{equation}
In fact, by definition~\eqref{eq:backw-sol-def} it follows that $\text{w}(x,t)$
is a distributional solution 
 of~\eqref{eq:backw-cp}, since it is obtained from the distributional solution $\mc S_t^+ \big(\omega(-\ \cdot\,)\big)(x)$ of~\eqref{eq:convexconslaw} by the change of variable $x\mapsto -x$. 
 On the other hand, 
since every shock discontinuity
of $\mc S_t^+ \big(\omega(-\ \cdot\,)\big)(x)$, 
connecting a left state $u^-$ with a right state $u^+$, must satisfy the Lax condition $u^->u^+$
(equivalent to the entropy admissibility criterion since the  flux $f(u)$ in~\eqref{eq:convexconslaw} is convex, e.g. see~\cite{Dafermoscontphysics,MR0267257}),
it follows
that the left and right states 
$u^-, u^+$ of
every shock discontinuity in $\text{w}(x,t)$
must satisfy the reverse condition
$u^-<u^+$, which is the Lax admissibility condition for~\eqref{eq:backw-cp}, since the flux $-f(\text{w})$ is concave.
 Finally, we can observe that
 $\text{w}(x,0)=\mc S_0^+ \big(\omega(-\ \cdot\,)\big)(-x)=\omega (x)$, for all $x\in\R$,
 which completes the proof of our claim.
\end{remark}
This procedure to characterize the attainable profiles is motivated by the following observation.
Given a target profile 
$\omega$, if we know that
for any $t\in \,]0,T[\,$,
the map $x\mapsto v(x,t)\doteq \mc S_t^+ \big(\omega(-\ \cdot\,)\big)(x)$
is locally Lipschitz on $\R$,
it would follow that 
$u(x,t)\doteq v(-x,T-t) =
\mc S_{T-t}^-\,\omega(x) $
is a classical solution 
of~\eqref{eq:convexconslaw}
which attains the target profile
$\omega$ at time $t=T$,
and starts with the initial datum $u_0^*$ in~\eqref{eq:indatum-lip}.
Since classical solutions
of~\eqref{eq:convexconslaw}
are 
entropy admissible, 
by uniqueness of entropy weak solutions of the Cauchy problem for~\eqref{eq:convexconslaw}
it would follow that $u(x,t)=
\mc S_t^+ u_0^*(x)=u^*(x,t)$.
However, if $v$ admits shock
discontinuities, 
the function $v(-x,T-t)$
fails to be an entropy admissible solution of~\eqref{eq:convexconslaw},
despite still being a weak distributional solution of~\eqref{eq:convexconslaw}.
The 
one-sided Lipschitz condition~\eqref{eq:convexoleinik}
is precisely equivalent to the property that the map $x\mapsto v(x,t)\doteq \mc S_t^+ \big(\omega(-\ \cdot\,)\big)(x)$
is locally Lipschitz on $\R$,
for all $t\in \,]0,T[\,$
(e.g. see~\cite{AGN2012,anconamars}),
and thus one obtains
the characterization of
the elements of $\mc A(T)$
as fixed points of
 the  backward-forward operator.

\subsubsection{Backward 
solution operator in the spatially-discontinuous flux setting}\label{sec:backopdisc}
Given a flux  $f$  as in~\eqref{discflux}
satisfying the assumption~\eqref{eq:flux-assumption-1},
and a connection
 $(A,B)$,
  let 
 $\sabp$  be the \textit{forward semigroup operator} associated to the connection $(A,B)$, as in Theorem \ref{theoremsemigroup}.
Observe that, 
letting $\overline{A}, \overline{B}$
be as in~
\eqref{eq:bar-AB-def},
the pair $(\overline B, \overline A\,)$
turns out to be a connection for the symmetric flux
\begin{equation}
\label{eq:symm-flux}
    \overline{f}(x,u) = \begin{cases}
    f_r(u) , & x \leq 0, \\
    f_l(u), & x \geq 0\,,
    \end{cases}
\end{equation}
(see Figure \ref{Connectionpicsym}).\\

\begin{figure}[h]
\centering
\begin{tikzpicture}[scale = 0.7]
\draw[->] (-6,0)--(4,0)node[right]{$u$};
\draw[scale = 0.5, domain=-1:5, smooth, variable=\x] plot ({\x}, {(\x-2)*(\x-2)+1});
\draw[scale = 0.5, domain=-9:3, smooth, variable=\x] plot ({\x}, {(0.5*\x+1)*(0.5*\x+2)+1});

\draw (-5,5) node{$f_l$};
\draw (3,5) node{$f_r$};

\draw[very thick, blue] (-3.8,3)--(2.13,3);
\draw[dotted] (-3.8,3)--(-3.8,0)node[below]{$A$};
\draw[dotted] (2.1,3)--(2.1,0)node[below]{$B$};
\draw[dotted] (0.8,3)--(0.8,0)node[below]{$\overline A$};
\draw[dotted] (-0.1,3)--(-0.1,0)node[below]{$\overline B$};

\end{tikzpicture}
\caption{The connection $(\overline B, \overline A)$ of the symmetric flux $\overline f(x, u)$ defined in \eqref{eq:symm-flux}.}
\label{Connectionpicsym}
\end{figure}

Then, letting
$\sobapt u_0(x)$ denote
the unique $\overline B\,\overline A$-entropy solution
of
\begin{equation}\label{eq:invertedproblem}
    \begin{cases}
        u_t+\overline{f}(x,u)_x = 0 & x \in \mathbb{R}, \quad t \geq 0, \\
            \noalign{\smallskip}
        u(x,0) = u_0(x) &  x \in \mathbb R,
    \end{cases}
\end{equation}
evaluated at $(x,t)$,
we shall define the backward 
solution operator associated to the connection  $(A,B)$ in terms of the operator $\sobapt$ as follows.

\begin{defi}[{\bf $AB$-Backward solution operator}]\label{def:backop}
Given a connection $(A,B)$, the \textit{backward solution operator}
associated to $(A,B)$ is the map
$\mc S_{(\cdot)}^{\, [A B]-} :
[0,+\infty)\times {\bf L^{\infty}(\R)} \to {\bf L^{\infty}(\R)}$,  defined by
\begin{equation}
     \label{eq:backw-conn-sol-def}
    \sabmt \omega (x) \doteq  \sobapt 
    \big(\omega(-\ \cdot\,)\big)
     (-x)\qquad x\in\R, \ t\geq 0\,. 
\end{equation}
\end{defi}
\begin{remark}
\label{charact-backw_sol-disc}
One can show that the function
$\text{w}(x,t)\doteq \sabmt \omega (x)$ is the $\overline{A}\,\overline{B}$-entropy solution of the Cauchy problem
\begin{equation}
\label{eq:backw-cp-disc}
    \begin{cases}
        \text{w}_t-f(x,\text{w})_x=0, \quad &x \in \mathbb{R}, \quad t \geq 0,
\\
\noalign{\smallskip}
    \text{w}(x,0)=\omega(x), \quad &x \in \mathbb{R}.
    \end{cases}
\end{equation}
Notice that in~\eqref{eq:backw-cp-disc} the flux is $-f(x,\text{w})$,
which is a discontinuous function that
coincides with the
 uniformly strictly concave maps $-f_l(\text{w}), -f_r(\text{w})$, on the left and on the right, respectively, of $x=0$. 
 As observed in~\cite[\S 7]{anconachiri}, in the case
 of a two-concave flux as $-f(x,\text{w})$, one replaces
 the $\leq$~sign with the $\geq$ sign, and viceversa,
 in the Definition~\ref{ABsol} of 
 interface connection.
Thus, $(\overline{A},\,\overline{B})$
is indeed a connection for the flux 
$-f(x,\text{w})$.
The 
$\overline{A}\,\overline{B}$
interface entropy admissibility condition for $\text{w}(x,t)$ is formulated 
as in~\eqref{interfaceentropy}.
 In order to verify the claim that 
$\text{w}(x,t)$ is the $\overline{A}\,\overline{B}$-entropy solution of the Cauchy problem~\eqref{eq:backw-cp-disc}
we proceed as in Remark~\ref{charact-backw_sol}.
We first observe that
$\text{w}(x,t)$ is a distributional solution of~\eqref{eq:backw-cp-disc},
and that it is entropy admissible
in the regions $\{x<0\}$, $\{x>0\}$.
In fact,  by definition~\eqref{eq:backw-conn-sol-def}, $\text{w}(x,t)$
is obtained from $\sobapt 
    \big(\omega(-\ \cdot\,)\big)(x)$ with the  
    change of variable $x\mapsto -x$, and we have $\text{w}(x,0)=\sobap_0
    \big(\omega(-\ \cdot\,)\big)(-x)=\omega(x)$,
    for all $x\in\R$.
    Next, we check that $\text{w}(x,t)$ satisfies the $\overline A\, \overline B$
    entropy  condition~\eqref{interfaceentropy}
    for the two-concave flux $-f(x,\text{w})$, i.e. that,
    letting $\text{w}_l(t), \text{w}_r(t)$ denote the left and right traces of $\text{w}(x,t)$ at $x=0$, it holds true
    \begin{equation}
        \label{eq:entr-conc-w}
         \mathrm{sgn}(\text{w}_r(t)-\overline{B}\,  )\left(-f_r(\text{w}_r(t))+f_r(\,\overline{B}\,)\right)-
         \mathrm{sgn}(\text{w}_l(t)-\overline{A}\,)\left(-f_l(\text{w}_l(t))+f_l(\,\overline{A}\,)\right)
         \leq 0\quad  \text{for a.e. $t>0$}\,.
    \end{equation}
    Observe that 
    the left and right traces
    $u_l(t), u_r(t)$ 
    of $\sobapt 
    \big(\omega(-\ \cdot\,)\big)(x)$ at $x=0$, 
    satisfy
    the $\overline{B}\,\overline{A}$ entropy condition~\eqref{interfaceentropy} for the flux $\overline f$ in~\eqref{eq:symm-flux}, that
    reads
    \begin{equation}
        \label{eq:entr-conv-u}
         \mathrm{sgn}(u_r(t)-\overline{A}\,  )\left(f_l(u_r(t))-f_l(\,\overline{A}\,)\right)-
         \mathrm{sgn}(u_l(t)-\overline{B}\,)\left(f_r(u_l(t))-f_r(\,\overline{B}\,)\right)
         \leq 0\qquad  \text{for a.e. $t>0$}\,.
    \end{equation}
    On the other hand, 
    since one obtains $\text{w}(x,t)$
     from $\sobapt 
    \big(\omega(-\ \cdot\,)\big)(x)$
    reversing the space variables,
    we have $u_l(t)=\text{w}_r(t)$,
    $u_r(t)=\text{w}_l(t)$
    for all $t>0$.
    Hence, we recover~\eqref{eq:entr-conc-w}
    from~\eqref{eq:entr-conv-u},
    thus completing the proof of the claim.
\end{remark}
\smallskip

\begin{remark}
We observe that
if $\omega$ is an attainable state in $\mc A^{[AB]}(T)$, it will follow from our results that
the solution $v(x,t)\doteq \sobapt  \big(\omega(-\ \cdot\,)\big)(x)$
related to the backward solution operator
may well contain a shock discontinuity exiting from the interface $x=0$ at a time $\tau< T$.
As a consequence here,
differently from the space-independent flux setting,
the map
$x\mapsto v(x,t)$
is in general not
locally Lipschitz outside the interface $\{x = 0\}$.
In turn, this implies that,
for $\omega\in \mc A^{[AB]}(T)$,
the (forward) $AB$-entropy solution defined by
\begin{equation}
    u^*(x,t)\doteq 
    \sabpt \big(\sabmT \omega\big)(x),\qquad x\in\R,\ t\in [0,T]\,,
\end{equation}
will be in general different 
from $v(-x,T-t)$ on $\R\times [0,T[\,$.
However, exploiting the duality property enjoyed by the forward and backward solution operators
(see \S~\ref{sec:building-blocks}), one can still prove
that $u^*(x,T)= v(-x,0)=\omega(x)$ for all $x\in\R$, which shows that
$\omega$ is a fixed point 
of the backward-forward operator $\sabpT \circ \sabmT$ as stated in Theorem~\ref{thm:backfordiscflux}.

\end{remark}

\section{Technical tools for characterization
of the near-interface wave structure}
\label{sec:building-blocks}

In this section we 
introduce some 
technical tools needed to 
characterize the pointwise constraints satisfied by 
the attainable profiles
of~\eqref{conslaw}
in intervals containing the origin.
Throughout the section, 
 $f:\R\to\R$ will be a
twice continuously differentiable, uniformly  convex map,
and we let $\theta$ be its unique critical point, $f'(\theta)=0$.
Set 
\begin{equation}
\label{eq:lambda-def}
    \lambda(u,v)\doteq 
    \frac{f(v)-f(u)}{v-u},\qquad u,v\in\R, \ u\neq v\,,
\end{equation}
and observe that, by the convexity of $f$ one has 
\begin{equation}
\label{eq:der-mon}
    f'(u)<\lambda(u,v)<f'(v)\qquad \forall~ u<v\in\R\,.
\end{equation}

\subsection{Left backward shock (Figure \ref{fig:Rrare}, left)}\label{defi:ur}

For every $B >\theta$, \
$0<\ms R< T\cdot f^{\prime}(B)$,
we define here:
\begin{itemize}
[leftmargin=25pt]
\item[-]two constants 
$\bs t[\ms R, B, f]$, $\bs u[\ms R,B, f]$;
\smallskip
\item[-]a function 
$t\mapsto \bs y[\ms R,B, f](t)$, \ $t\in [\bs t[\ms R, B, f],\,T]$;
\end{itemize}
which enjoy the following properties
that will be justified in the sequel (see~\S~\ref{def:rsr-block}, \ref{subsec:part3b}), but that we highlight here to clarify the purpouse of their introduction.
Let $(A,B)$ be a connection for a flux 
as in~\eqref{conslaw},
and let 
$\overline A, \overline B$
be as in~
\eqref{eq:bar-AB-def}.
Then, the map $t\mapsto \bs y[\ms R,B, f_r](t)$
identifies the location of a shock curve
in a $\overline B \overline A$-entropy solution 
of $u_t+\overline f(x,u)_x=0$, with $\overline f$ as in~\eqref{eq:symm-flux},
defined on some domain $\Omega\subset \,]-\infty, 0]\times [0, +\infty[\,$.
Since a $\overline B \overline A$-entropy solution of~\eqref{eq:invertedproblem} is associated to the backward solution operator~\eqref{eq:backw-conn-sol-def}, 
we will say that $\bs y[\ms R,B, f_r]$
identifies the location of
a {\it left backward shock.}

This curve starts from the interface $\{x=0\}$
at time $t=\bs t[\ms R, B, f_r]$,
and reaches 
the point $x=\bs y[\ms R,B, f_r](T)$
at time $t=T$.
Such a shock discontinuity
has, at time $t=T$, 
 left state $\bs u[\ms R,B, f_r]$
 and right state $\overline{B}$.
The point $(-\ms R,0)$
is the center of a rarefaction fan 
located on the left of the curve~$x\mapsto(\bs y[\ms R,B, f_r](t), t)$.
\smallskip

We proceed to introduce these definitions as follows. Set
\begin{equation}
\label{eq:trbf-def}
         \bs t[\ms R, B, f] \doteq 
        \frac{\ms R}{f^{\prime}(B)},
        \qquad\quad\overline B \doteq 
    ({f}_{|\,]-\infty, \theta]})^{-1}\circ f(B)\,.
     \end{equation}
Then, consider
the Cauchy problem
      \begin{equation}\label{eq:yR}
\begin{cases}
    {y}'(t) = \lambda \Big(\!(f^{\prime})^{-1}\big(\frac{y(t)+\ms R}{t}\big), \overline{B}\, \Big), \qquad t\geq \bs t[\ms R, B, f]\,,
    \\
    \noalign{\smallskip}
    y( \bs t[\ms R, B, f] ) = 0.
    \end{cases}
\end{equation}
By~\eqref{eq:lambda-def}, the differential equation in~\eqref{eq:yR}
ensures that, for all $t\geq \bs t[\ms R, B, f]$, the pair \linebreak
$\big((f')^{-1}(\frac{y(t)+\ms R}{t}),\, \overline{B}\,\big)$
satisfies the Rankine-Hugoniot condition with 
slope $y'(t)$ for the conservation law $u_t+f(u)_x=0$.
Observe that, since
$g(t,y)\doteq \lambda\big((f')^{-1}\big(\frac{y+\ms R}{t}\big),\, \overline{B}\,\big)$ is locally
Lipschitz continuous in $y$,
by classical arguments 
it admits a unique solution $\bs y(t)$
defined on some maximal interval
$[t[\ms R, B, f],\,\tau[$\,.
On the other hand, 
because of~\eqref{eq:der-mon} we have
\begin{equation}
    g(t,y)>f'(\overline B)\qquad 
    \forall~t\in 
    \big[\bs t[\ms R, B, f],\,\min\{\tau,T\}\big[\,,\ y> - \ms R + T\cdot f'(\overline{B})\,,
\end{equation}
and hence, 
since $f'(\overline{B})<0$,
it follows that 
\begin{equation}
\label{eq:est-sol-cp-bs}
\begin{aligned}
     \bs y(t)&>\big(\min\{\tau,T\}-t[\ms R, B, f]
    \big)\cdot f'(\overline B)
    \\
    &\geq \min\{\tau,T\}\cdot f'(\overline B)
\end{aligned}
    \qquad\quad 
    \forall~t\in 
    \big[t[\ms R, B, f],\,\min\{\tau,T\}\big[\,.
\end{equation}
In turn, \eqref{eq:est-sol-cp-bs} implies that
$\tau>T$. 
Then, 
we will 
denote by
\begin{equation*}
    \bs y[\ms R,B, f] : \big[ \bs t[\ms R, B, f]  ,\, T\big]\ \to\ \,]-\infty,0[\,,
    \qquad t\mapsto \bs y[\ms R,B, f](t)\,,
\end{equation*}
the unique solution to~\eqref{eq:yR}
defined on the interval $\big[ \bs t[\ms R, B, f]  ,\, T\big]$.
Notice that $t\mapsto  \frac{d}{dt}{\bs y}[\ms R,B, f](t)$ is strictly decreasing, and $\frac{d}{dt}{\bs y}[\ms R,B, f](t)\leq 0$
for all $t\in [ \bs t[\ms R, B, f]  ,\, T]$.
Hence, by~\eqref{eq:est-sol-cp-bs}
with $\min\{\tau,T\}=T$,
the terminal point satisfies $\bs y[\ms R,B, f](T)\in\,]T\cdot f'(\overline{B}), 0[\,.$
Next, we set
\begin{equation}
\label{eq:urbf-def} 
    \bs u[\ms R,B, f]\doteq
    (f^{\prime})^{-1}\left(\frac{\ms R+\bs y[\ms R, B, f](T)}{T}\right)\,.
\end{equation}
Observe that, by construction, $\bs y[\ms R, B, f](T)$ and $\bs u[\ms R,B, f]$ depend continuously on the
parameters $\ms R$ and $B$,
and that we have
\begin{equation}
\label{eq:uB-ineq}
{ \overline B<~}
    \bs u[\ms R,B, f]< B\,.
\end{equation}


\smallskip

\subsection{Right backward shock (Figure~\ref{fig:Lrare}, right)}
\label{defi:vl}

Symmetrically to \S~\ref{defi:ur}, 
for every $A <\theta$, \linebreak
$T\cdot f^{\prime}(A)<\ms L< 0$,
we define here:
\begin{itemize}
[leftmargin=25pt]
\item[-]two constants 
$\bs s[\ms L, A, f]$, $\bs v[\ms L,A, f]$;
\smallskip
\item[-]a function 
$t\mapsto \bs x[\ms L,A, f](t)$, \ $t\in [\bs s[\ms L, A, f],\,T]$;
\end{itemize} 
which enjoy the following properties
that we highlight here as in \S~\ref{defi:ur} to clarify the purpouse of their introduction
(but we will justify them in the sequel, see~\S~\ref{def:lrs-block}, \ref{subsec:part3b}).
The map $t\mapsto \bs  x[\ms L,A, f_l](t)$
identifies the location of a shock curve
in a $\overline B \overline A$-entropy solution 
of $u_t+\overline f(x,u)_x=0$, with $\overline f$ as in~\eqref{eq:symm-flux},
defined on some domain $\Omega\subset [0,+\infty[\, \times [0, +\infty[\,$.
Since a $\overline B \overline A$-entropy solution of~\eqref{eq:invertedproblem} is associated to the backward solution operator~\eqref{eq:backw-conn-sol-def}, 
we will say that $\bs  x[\ms L,A, f_l]$
identifies the location of
a {\it right backward shock.}

This curves starts from the interface $\{x=0\}$
at time $t=\bs s[\ms L, A, f_l]$,
and reaches 
 the point $x=\bs x[\ms L,A, f_l](T)$
 at time $t=T$.
Such a shock discontinuity
has, at time $t=T$,  left state $\overline{A}$ and
 right state 
 $\bs v[\ms L,A, f_l]$.
The point $(-\ms L,0)$
is the center of a rarefaction fan 
located on the right of the curve~$t\mapsto (\bs x[\ms L,A, f_l](t), t)$.
\smallskip

We proceed to introduce these definitions as follows. Set
\begin{equation}
\label{eq:slaf-def}
         \bs s[\ms L, A, f] \doteq
        \frac{\ms L}{f^{\prime}(A)}, 
        \qquad\quad\overline A \doteq 
    ({f}_{|\,[\theta,+\infty[})^{-1}\circ f(A)\,.
     \end{equation}
Then, let
$\bs x[\ms L,A, f] : [ \bs s[\ms L, A, f]  , T] \ \to\ \,]0,+\infty[$
denote the unique solution to the Cauchy problem
      \begin{equation}\label{eq:xL}
\begin{cases}
    {x}'(t) = \lambda \left(\!(f^{\prime})^{-1}\big(\frac{x(t)+\ms L}{t}\big), \overline{A}\,\right), \qquad t \in \big[ \bs s[\ms L, A, f] ,\, T\big], \\
     \noalign{\smallskip}
    x( \bs s[\ms L, A, f] ) = 0.
    \end{cases}
\end{equation}
By~\eqref{eq:lambda-def}, the differential equation in~\eqref{eq:xL}
ensures that, for all $t\geq \bs s[\ms L, A, f]$, the pair \linebreak
$\big( \overline{A},\, (f')^{-1}(\frac{x(t)+\ms L}{t})\,\big)$
satisfies the Rankine-Hugoniot condition with 
slope $x'(t)$ for the conservation law $u_t+f(u)_x=0$.
The terminal point $\bs x[\ms L, A, f](T)$ 
depends continuously on the parameters~$\ms L$,~$A$, and
satisfies $\bs x[\ms L, A, f](T) \in \big]0,T\cdot f^{\prime}(\overline{A})\big[$\,.
Moreover,
 the map $t\mapsto \frac{d}{dt}{\bs x}[\ms L,A, f](t)$ is strictly increasing. Next, we define the quantity \begin{equation}
    \label{eq:vlaf-def} 
    \bs v[\ms L,A, f]\doteq
    (f^{\prime})^{-1}\left(\frac{\ms L+\bs x[\ms L,A, f](T)}{T}\right)\,,
\end{equation}
which depends continuously on $\ms L$
and $A$, and satisfies
\begin{equation}
\label{eq:vA-ineq}
   A< \bs v[\ms L,A, f]{ ~<
   \overline A}\,.
\end{equation}
\smallskip

\begin{figure}
\centering

\tikzset{every picture/.style={line width=0.75pt}} 

\begin{tikzpicture}[x=0.75pt,y=0.75pt,yscale=-0.65,xscale=0.7]

\draw  [color={rgb, 255:red, 0; green, 0; blue, 0 }  ,draw opacity=1 ][fill={rgb, 255:red, 208; green, 2; blue, 27 }  ,fill opacity=0.43 ] (513.45,73.94) -- (566.67,254.89) -- (420,224.75) -- (419.65,116.31) -- cycle ;
\draw  [color={rgb, 255:red, 0; green, 0; blue, 0 }  ,draw opacity=1 ][fill={rgb, 255:red, 74; green, 144; blue, 226 }  ,fill opacity=0.33 ] (106.98,71.85) -- (177.85,254.06) -- (247.33,193.67) -- (246.64,115.27) -- cycle ;
\draw    (65.98,256.14) -- (271.11,255.42) ;
\draw [shift={(274.11,255.41)}, rotate = 179.8] [fill={rgb, 255:red, 0; green, 0; blue, 0 }  ][line width=0.08]  [draw opacity=0] (5.36,-2.57) -- (0,0) -- (5.36,2.57) -- cycle    ;
\draw    (65.29,68.73) -- (271.33,70.78) ;
\draw    (247.23,19.64) -- (246.87,256.24) ;
\draw [shift={(247.24,16.64)}, rotate = 90.09] [fill={rgb, 255:red, 0; green, 0; blue, 0 }  ][line width=0.08]  [draw opacity=0] (5.36,-2.57) -- (0,0) -- (5.36,2.57) -- cycle    ;
\draw  [dash pattern={on 4.5pt off 4.5pt}]  (106.98,71.85) -- (177.85,254.06) ;
\draw    (247.33,193.67) -- (177.85,254.06) ;
\draw  [dash pattern={on 4.5pt off 4.5pt}]  (180.63,109.34) -- (177.85,254.06) ;
\draw  [dash pattern={on 4.5pt off 4.5pt}]  (202.17,127.04) -- (177.85,254.06) ;
\draw  [dash pattern={on 4.5pt off 4.5pt}]  (159.09,94.76) -- (177.85,254.06) ;
\draw [line width=0.75]  [dash pattern={on 4.5pt off 4.5pt}]  (137.55,84.35) -- (177.85,254.06) ;
\draw  [dash pattern={on 4.5pt off 4.5pt}]  (220.93,147.86) -- (177.85,254.06) ;
\draw  [dash pattern={on 4.5pt off 4.5pt}]  (180.63,109.34) -- (247.05,136.44) ;
\draw  [dash pattern={on 4.5pt off 4.5pt}]  (234.82,169.72) -- (177.85,254.06) ;
\draw    (106.98,71.85) -- (246.64,115.27) ;
\draw  [dash pattern={on 4.5pt off 4.5pt}]  (230.66,160.35) -- (248.72,170.45) ;
\draw  [dash pattern={on 4.5pt off 4.5pt}]  (215.37,140.57) -- (245.94,153.07) ;
\draw [color={rgb, 255:red, 208; green, 2; blue, 27 }  ,draw opacity=1 ]   (106.98,71.85) .. controls (179.24,97.88) and (223.01,132.24) .. (247.33,193.67) ;
\draw  [dash pattern={on 0.84pt off 2.51pt}]  (331.41,193.67) -- (247.33,193.67) ;
\draw    (602.86,257.17) -- (507.61,256.84) -- (397.74,256.45) ;
\draw [shift={(605.86,257.18)}, rotate = 180.2] [fill={rgb, 255:red, 0; green, 0; blue, 0 }  ][line width=0.08]  [draw opacity=0] (5.36,-2.57) -- (0,0) -- (5.36,2.57) -- cycle    ;
\draw    (601,69.77) -- (394.96,71.82) ;
\draw    (419.05,20.68) -- (419.42,257.29) ;
\draw [shift={(419.05,17.68)}, rotate = 89.91] [fill={rgb, 255:red, 0; green, 0; blue, 0 }  ][line width=0.08]  [draw opacity=0] (5.36,-2.57) -- (0,0) -- (5.36,2.57) -- cycle    ;
\draw  [dash pattern={on 4.5pt off 4.5pt}]  (473.5,111.25) -- (566.67,254.89) ;
\draw  [dash pattern={on 4.5pt off 4.5pt}]  (454.5,141.25) -- (566.67,254.89) ;
\draw  [dash pattern={on 4.5pt off 4.5pt}]  (497.89,85.18) -- (566.67,254.89) ;
\draw  [dash pattern={on 4.5pt off 4.5pt}]  (440.5,169.75) -- (566.67,254.89) ;
\draw  [dash pattern={on 4.5pt off 4.5pt}]  (454.5,141.25) -- (420,154.25) ;
\draw  [dash pattern={on 4.5pt off 4.5pt}]  (430,198.75) -- (566.67,254.89) ;
\draw    (513.45,73.94) -- (419.65,116.31) ;
\draw    (430,198.75) -- (419.06,202.49) ;
\draw  [dash pattern={on 4.5pt off 4.5pt}]  (440.5,169.75) -- (420,179.75) ;
\draw [color={rgb, 255:red, 208; green, 2; blue, 27 }  ,draw opacity=1 ]   (513.45,73.94) .. controls (449.94,115.38) and (444.32,163.32) .. (420,224.75) ;
\draw  [dash pattern={on 0.84pt off 2.51pt}]  (420,224.75) -- (595.51,225.58) ;
\draw  [dash pattern={on 4.5pt off 4.5pt}]  (420.5,122.75) -- (497.89,85.18) ;
\draw  [dash pattern={on 4.5pt off 4.5pt}]  (419.23,137.48) -- (473.5,111.25) ;

\draw (291.37,172.02) node [anchor=north west][inner sep=0.75pt]  [font=\tiny]  {$\mathbf{t}[\mathsf{R} ,\ B,\ f]$};
\draw (207.98,34.63) node [anchor=north west][inner sep=0.75pt]  [font=\tiny]  {$\overline{B}$};
\draw (90.95,34.58) node [anchor=north west][inner sep=0.75pt]  [font=\tiny]  {$\mathbf{y}[\mathsf{R} ,\ B,\ f]$};
\draw (22.91,54.37) node [anchor=north west][inner sep=0.75pt]  [font=\tiny]  {$t\ =\ T$};
\draw (226.05,232.41) node [anchor=north west][inner sep=0.75pt]  [font=\tiny]  {$B$};
\draw (167.68,274.05) node [anchor=north west][inner sep=0.75pt]  [font=\tiny]  {$-\mathsf{R}$};
\draw (254.53,12.72) node [anchor=north west][inner sep=0.75pt]  [font=\tiny]  {$t$};
\draw (273.99,267.8) node [anchor=north west][inner sep=0.75pt]  [font=\tiny]  {$x$};
\draw (444.05,236.36) node [anchor=north west][inner sep=0.75pt]  [font=\tiny]  {$\overline{B}$};
\draw (613.76,253.54) node [anchor=north west][inner sep=0.75pt]  [font=\tiny]  {$x$};
\draw (403.36,20.01) node [anchor=north west][inner sep=0.75pt]  [font=\tiny]  {$t$};
\draw (457.56,51.24) node [anchor=north west][inner sep=0.75pt]  [font=\tiny]  {$B$};
\draw (537.31,275.09) node [anchor=north west][inner sep=0.75pt]  [font=\tiny]  {$-\mathbf{y}[\mathsf{R} ,\ B,\ f]$};
\draw (593.74,212.14) node [anchor=north west][inner sep=0.75pt]  [font=\tiny]  {$\mathbf{\tau }[\mathsf{R} ,\ B,\ f]$};
\draw (621.15,46.04) node [anchor=north west][inner sep=0.75pt]  [font=\tiny]  {$t\ =\ T$};
\draw (507.45,43.95) node [anchor=north west][inner sep=0.75pt]  [font=\tiny]  {$\mathsf{R}$};

\end{tikzpicture}

\caption{The dual solutions $\bs y[\ms R, B, f](\cdot)$ (left) and $\bs x[\bs y[\ms R, B, f](T), \overline{B}, f](\cdot)$ (right) of the Cauchy problems \eqref{eq:yR}, \eqref{eq:xL}. This represents the statement of Lemma \ref{lemma:dualshocks}}
\label{fig:Rrare}
\end{figure}

\begin{figure}
\centering

\tikzset{every picture/.style={line width=0.75pt}} 

\tikzset{every picture/.style={line width=0.75pt}} 

\begin{tikzpicture}[x=0.75pt,y=0.75pt,yscale=-0.75,xscale=0.75]

\draw  [color={rgb, 255:red, 0; green, 0; blue, 0 }  ,draw opacity=1 ][fill={rgb, 255:red, 208; green, 2; blue, 27 }  ,fill opacity=0.43 ] (526.83,77.7) -- (578.1,240.58) -- (436.2,168.55) -- (436.1,101.55) -- cycle ;
\draw  [color={rgb, 255:red, 0; green, 0; blue, 0 }  ,draw opacity=1 ][fill={rgb, 255:red, 74; green, 144; blue, 226 }  ,fill opacity=0.33 ] (135.23,75.83) -- (154.89,123.06) -- (203.51,239.83) -- (270.45,185.48) -- (269.78,114.91) -- cycle ;
\draw    (95.74,241.71) -- (293.24,241.06) ;
\draw [shift={(296.24,241.05)}, rotate = 179.81] [fill={rgb, 255:red, 0; green, 0; blue, 0 }  ][line width=0.08]  [draw opacity=0] (5.36,-2.57) -- (0,0) -- (5.36,2.57) -- cycle    ;
\draw    (95.07,73.02) -- (293.57,74.86) ;
\draw    (270.36,29.13) -- (270,241.8) ;
\draw [shift={(270.36,26.13)}, rotate = 90.09] [fill={rgb, 255:red, 0; green, 0; blue, 0 }  ][line width=0.08]  [draw opacity=0] (5.36,-2.57) -- (0,0) -- (5.36,2.57) -- cycle    ;
\draw  [dash pattern={on 4.5pt off 4.5pt}]  (135.23,75.83) -- (203.51,239.83) ;
\draw    (270.45,185.48) -- (203.51,239.83) ;
\draw  [dash pattern={on 4.5pt off 4.5pt}]  (206.19,109.57) -- (203.51,239.83) ;
\draw  [dash pattern={on 4.5pt off 4.5pt}]  (226.94,125.5) -- (203.51,239.83) ;
\draw  [dash pattern={on 4.5pt off 4.5pt}]  (185.44,96.45) -- (203.51,239.83) ;
\draw [line width=0.75]  [dash pattern={on 4.5pt off 4.5pt}]  (164.69,87.08) -- (203.51,239.83) ;
\draw  [dash pattern={on 4.5pt off 4.5pt}]  (245.01,144.24) -- (203.51,239.83) ;
\draw  [dash pattern={on 4.5pt off 4.5pt}]  (206.19,109.57) -- (270.18,133.96) ;
\draw  [dash pattern={on 4.5pt off 4.5pt}]  (258.4,163.92) -- (203.51,239.83) ;
\draw    (135.23,75.83) -- (269.78,114.91) ;
\draw  [dash pattern={on 4.5pt off 4.5pt}]  (254.39,155.49) -- (271.79,164.58) ;
\draw  [dash pattern={on 4.5pt off 4.5pt}]  (239.66,137.68) -- (269.11,148.93) ;
\draw [color={rgb, 255:red, 208; green, 2; blue, 27 }  ,draw opacity=1 ]   (135.23,75.83) .. controls (204.85,99.26) and (247.02,130.19) .. (270.45,185.48) ;
\draw  [dash pattern={on 0.84pt off 2.51pt}]  (351.45,185.48) -- (270.45,185.48) ;
\draw    (612.86,242.63) -- (521.2,242.33) -- (415.35,241.99) ;
\draw [shift={(615.86,242.64)}, rotate = 180.19] [fill={rgb, 255:red, 0; green, 0; blue, 0 }  ][line width=0.08]  [draw opacity=0] (5.36,-2.57) -- (0,0) -- (5.36,2.57) -- cycle    ;
\draw    (611.17,73.96) -- (412.67,75.8) ;
\draw    (435.89,30.07) -- (436.24,242.74) ;
\draw [shift={(435.88,27.07)}, rotate = 89.91] [fill={rgb, 255:red, 0; green, 0; blue, 0 }  ][line width=0.08]  [draw opacity=0] (5.36,-2.57) -- (0,0) -- (5.36,2.57) -- cycle    ;
\draw  [dash pattern={on 4.5pt off 4.5pt}]  (492.42,102.82) -- (578.1,240.58) ;
\draw  [dash pattern={on 4.5pt off 4.5pt}]  (475.6,120.05) -- (578.1,240.58) ;
\draw  [dash pattern={on 4.5pt off 4.5pt}]  (511.83,87.83) -- (578.1,240.58) ;
\draw  [dash pattern={on 4.5pt off 4.5pt}]  (460.29,137.49) -- (578.1,240.58) ;
\draw  [dash pattern={on 4.5pt off 4.5pt}]  (492.42,102.82) -- (436.89,118.78) ;
\draw  [dash pattern={on 4.5pt off 4.5pt}]  (445.7,155.55) -- (578.1,240.58) ;
\draw    (452.29,142.99) -- (437.2,148.05) ;
\draw  [dash pattern={on 4.5pt off 4.5pt}]  (475.6,120.05) -- (438.1,132.55) ;
\draw [color={rgb, 255:red, 208; green, 2; blue, 27 }  ,draw opacity=1 ]   (526.83,77.7) .. controls (513.13,89.14) and (501.48,95.24) .. (492.42,102.82) .. controls (472.6,122.55) and (458.52,131.22) .. (436.2,168.55) ;
\draw  [dash pattern={on 0.84pt off 2.51pt}]  (436.2,168.55) -- (599.2,168.05) ;

\draw (319.4,160.1) node [anchor=north west][inner sep=0.75pt]  [font=\tiny]  {$\mathbf{\sigma }[ L,\ A,\ f]$};
\draw (224.29,44.83) node [anchor=north west][inner sep=0.75pt]  [font=\tiny]  {$A$};
\draw (133.25,44.83) node [anchor=north west][inner sep=0.75pt]  [font=\tiny]  {$\mathsf{L}$};
\draw (53.79,59.83) node [anchor=north west][inner sep=0.75pt]  [font=\tiny]  {$t\ =\ T$};
\draw (248.39,217.21) node [anchor=north west][inner sep=0.75pt]  [font=\tiny]  {$\overline{A}$};
\draw (571.03,263.19) node [anchor=north west][inner sep=0.75pt]  [font=\tiny]  {$-\mathsf{L}$};
\draw (277.17,22.34) node [anchor=north west][inner sep=0.75pt]  [font=\tiny]  {$t$};
\draw (295.91,251.94) node [anchor=north west][inner sep=0.75pt]  [font=\tiny]  {$x$};
\draw (460.72,217.27) node [anchor=north west][inner sep=0.75pt]  [font=\tiny]  {$A$};
\draw (623.25,239.11) node [anchor=north west][inner sep=0.75pt]  [font=\tiny]  {$x$};
\draw (420.55,28.9) node [anchor=north west][inner sep=0.75pt]  [font=\tiny]  {$t$};
\draw (472.77,57.89) node [anchor=north west][inner sep=0.75pt]  [font=\tiny]  {$\overline{A}$};
\draw (180.03,264.13) node [anchor=north west][inner sep=0.75pt]  [font=\tiny]  {$-\mathbf{x}[\mathsf{L} ,\ A,\ f]( T)$};
\draw (602.16,162.29) node [anchor=north west][inner sep=0.75pt]  [font=\tiny]  {$\mathbf{s}[\mathsf{L} ,\ A,\ f]$};
\draw (630.13,52.33) node [anchor=north west][inner sep=0.75pt]  [font=\tiny]  {$t\ =\ T$};
\draw (519.9,42.96) node [anchor=north west][inner sep=0.75pt]  [font=\tiny]  {$\mathbf{x}[\mathsf{L} ,\ A,\ f]( T)$};

\end{tikzpicture}

\caption{The dual solutions $\bs x[\ms L, A, f](\cdot)$ (right) and $\bs y[\bs x[\ms L, A, f](T), \overline{A}, f](\cdot)$ (left) of the Cauchy problems \eqref{eq:yR}, \eqref{eq:xL}. This represents the statement of 
Lemma~\ref{lemma:dualshocks}}
\label{fig:Lrare}
\end{figure}

\subsection{Duality of forward and
backward shocks}
\label{sec:dual-prop}
The definitions of backward shocks given in
\S~\ref{defi:ur}-\ref{defi:vl} turn out to  be dual one of the other, as clarified by the following:

\begin{lemma}\label{lemma:dualshocks}
With the notations introduced in \S~\ref{defi:ur}-\ref{defi:vl},
for every $B>\theta$, the following holds.
The maps 
\begin{equation}
\label{eq;yx-def-37}
\begin{aligned}
     &\bs y[\cdot\,, B, f](T)
     : \ ]0, T\cdot f^{\prime}(B)[\,\ \to\ ]T \cdot f^{\prime}(\overline{B}), 0[\,,\qquad
    \ms R \mapsto \bs y[\ms R, B, f](T), 
    \\
    &\bs x[\cdot\, ,\overline{B}, f](T)
    : \ ]T \cdot f^{\prime}(\overline{B}), 0[\, \ \to\ ]0, T\cdot f^{\prime}(B)[\,,\qquad 
   \ms L  \mapsto \bs x[\ms L, \overline{B}, f](T) 
    \end{aligned}
\end{equation}
are 
increasing,
and one is the inverse of the other, i.e.
it holds true
\begin{equation}\label{eq:yxinvertible}
    \begin{aligned}
        \ms R &= \bs x \big[ \bs y[\ms R, B, f](T), \overline{B}, f\big](T), \qquad \forall~\ms R \in  \,]0, T\cdot f^{\prime}(B)[, \\
        \noalign{\smallskip}
        \ms L &= \bs y \big[ \bs x[\ms L, \overline B, f](T), B, f\big](T), \qquad \forall~\ms L \in \,]T \cdot f^{\prime}(\overline{B}), 0[\,.
    \end{aligned}
\end{equation}
Moreover, one has
\begin{equation}
\label{eq:xy-1sdlim}
    \begin{aligned}
        \lim_{\ms R \to 0+}
        \,\bs y[R, B, f](T) &= T\cdot f^{\prime}(\overline{B}), \qquad  \lim_{\ms R \to T\cdot f^{\prime}(B)-}
        \bs y[\ms R, B, f](T) = 0,
        \\
        \noalign{\smallskip}
        \lim_{\ms L \to 0-}
       \bs x[\ms L, \overline{B}, f](T) &= T\cdot f^{\prime}( B),  \qquad 
       \lim_{\ms L \to T\cdot f^{\prime}(\overline{B})+}
       \bs x[\ms L, \overline{B}, f](T) = 0.
    \end{aligned}
\end{equation}
\end{lemma}

\begin{proof}
\hfill\\
{\bf 1.} \
We will prove only the first equality in \eqref{eq:yxinvertible}, 
the proof of the second one being entirely similar.
Fix $\ms R \in \,]0, T \cdot f^{\prime}(B)[$, and consider the polygonal
region  (the blue set in Figure~\ref{fig:Rrare}) defined by
\begin{equation}
\label{eq:delta12def-0}
    \begin{aligned}
   \Delta &\doteq \Delta_1\cup\Delta_2,
    \\
    \noalign{\medskip}
      \Delta_1&\doteq
      \Big\{(x, t) \in\,]-\infty,0[\, \times \,]0, T[\,   \; : \; 
      \ms L \!-\!(T-t) \cdot f^{\prime}(\bs u)<x
      <\ms L\!-\!(T-t)\cdot f^{\prime}(\bar B),
       \ \bs t<t < T \Big\},
       \\
      \noalign{\bigskip}
      \Delta_2&\doteq
      \Big\{(x, t) \in\,]-\infty,0[\, \times \,]0, T[\,   \; : \; 
      \ms L \!-\!(T-t) \cdot f^{\prime}(\bs u)
       <x< (t-\bs t)\cdot\! f^{\prime}(\,B\,),
      \ 0 \leq t \leq \bs t
       \Big\},
    \end{aligned}
\end{equation}
where $\bs u\doteq \bs u[\ms R, B, f]$
is the constant in~\eqref{eq:urbf-def}, $\bs t \doteq \bs t[\ms R, B, f]$
is defined as in~\eqref{eq:trbf-def} and $\ms L \doteq  \bs y[\ms R, B, f](T)$.
Observe that the function $v :  \Delta \to\R$ defined by
\begin{equation}
    v(x,t)\doteq 
     \begin{cases}
     \overline{B}  &\text{if}\qquad   \gamma(t)< x < 0,
\\
\noalign{\smallskip}
(f^{\prime})^{\strut -1}\bigg(\dfrac{x+\ms R}{t}\bigg) \quad &\text{otherwise,}
     \end{cases}
\end{equation}
is locally Lipschitz continuous 
and satisfies the equation~\eqref{eq:convexconslaw}
at every point  $(x,t)\in\Delta$ outside the curve $\gamma(\cdot )\doteq  \bs y[\ms R, B, f](\cdot)$.
Moreover, because of the construction of $\bs y[\ms R, B, f ](\cdot)$,
$u$ satisfies the Rankine-Hugoniot conditions along 
the curve~$\gamma$.
Therefore $v(x,t)$ is a distributional solution
of~\eqref{eq:convexconslaw} on $\Delta$. 
Hence, applying the divergence theorem 
to the piecewise smooth vector field
$(v,f(v))$ on $\Delta$, 
and setting $\tau_1\doteq T-\bs y[\ms R, B, f]/f'(\overline B)$, we find
\begin{equation*}
\begin{aligned}
        0&=\big(f(\overline B )-\overline  Bf'(\overline B)\big)(T-\tau_1)+
        \big(\bs u\, f'(\bs u) -f(\bs u)\big)T+
        \\&\big(f(\,B\,)-Bf'(\,B\,)\big)\, \bs t + f(B)(\tau_1-\bs t)\,. 
\end{aligned}
\end{equation*}
Then, observing that $f(B)=f(\overline B)$ and that $f'(\,B\,) \bs t = \ms R$, we find
\begin{equation}
\label{eq:cons-gamma}
    \overline{B}\,\bs y[\ms R, B, f](T)+ B\, \ms R
    -\big(\bs u\, f'(\bs u) - f(\bs u)\big) T - f(B) T=0\,.
\end{equation}
Since $f'(\bs u)=(\bs y[\ms R, B, f](T)+\ms R)/T$,
and because the Legendre transform $f^*$ of $f$
satisfies the identity $$
f^*(f'(u))=u\,f'(u) - f(u)\qquad\forall~u\,,
$$  
(e.g. see~[\S A.2]\cite{MR2041617}), we derive
from~\eqref{eq:cons-gamma} the identity
\begin{equation}
\label{eq:massconsyR}
\qquad
\overline{B}\, \bs y[\ms R, B, f](T)+ B\, \ms R  -f^{*}\left(\frac{\bs y[\ms R, B, f](T)+\ms R}{T}\right)\,T-f(B)\,T=0\qquad \forall~\ms R \in \,]0, T\cdot f^{\prime}(B))[\,.
\end{equation}
\smallskip

\noindent
{\bf 2.}
Next, consider the polygonal
region  (the red set in Figure~\ref{fig:Lrare} with $A=\overline{B}$ and $\overline{A}=B$) defined by
\begin{equation}
\label{eq:gamma12def-0}
    \begin{aligned}
   \Gamma &\doteq  \Gamma_1\cup  \Gamma_2,
    \\
    \noalign{\medskip}
       \Gamma_1&\doteq
      \Big\{(x, t) \in\,]0,+\infty[\, \times \,]0, T[\, \; : \; 
      \bs x\big[\bs  y[\ms R, B, f](T), \overline{B}, f\big](T)\!-\!(T-t)\cdot f^{\prime}(B) < x <
      \\
      &\hspace{2.4in} < \bs x\big[\bs  y[\ms R, B, f](T), \overline{B}, f\big](T) \!-\!(T-t) \cdot f^{\prime}(\bs v),
       \ \bs s <t < T \Big\},
       \\
      \noalign{\bigskip}
        \Gamma_2&\doteq
      \Big\{(x, t) \in
      \,]0,+\infty[\, \times \,]0, T[\,  \; : \; 
       (t-\bs s)\cdot\! f^{\prime}(\,\overline{B}\,) < x <
      \\
      &\hspace{2.4in}< \bs x\big[\bs  y[\ms R, B, f](T), \overline{B}, f\big](T) \!-\!(T-t) \cdot f^{\prime}(\bs v),
      \ 0 \leq t \leq \bs s
       \Big\},
    \end{aligned}
\end{equation}
where $\bs v\doteq \bs v\big[\bs y[\ms R, B, f](T), \overline{B}, f\big]$
is the constant defined as in~\eqref{eq:vlaf-def}, with 
$\ms L=\bs y[\ms R, B, f](T)$, $A=\overline{B}$ and $\bs s = \bs s[\ms L, A, f]$.
Observe that the function $u :  \Gamma\to\R$ defined by
\begin{equation}
    u(x,t)\doteq 
     \begin{cases}
     B  &\text{if}\qquad  0< x < \gamma(t),
\\
\noalign{\smallskip}
(f^{\prime})^{\strut -1}\bigg(\dfrac{x+\bs y[\ms R, B, f](T)}{t}\bigg) \quad &\text{otherwise,}
     \end{cases}
\end{equation}
with $\gamma(t)\doteq \bs x\big[\bs  y[\ms R, B, f](T), \overline{B}, f\big](t)$,
is a distributional solution
of~\eqref{eq:convexconslaw} on $ \Gamma$
for the symmetric arguments of the previous point.
Then, repeating the same type of
analysis of above for the piecewise smooth vector field
$(u,f(u))$ on $ \Gamma$, one finds
the identity
\begin{equation}
\label{eq:massconsyR2}
    \begin{aligned}
        &B\, \bs x \big[ \bs y[\ms R, B, f](T), \overline{B}, f\big](T) +\overline{B}\,\bs y[\ms R, B, f](T) +
        \\
        \noalign{\smallskip}
        &\quad -f^*\left(\frac{\bs x \big[ \bs y[\ms R, B, f](T), \overline{B}, f\big](T)+\bs y[\ms R, B, f](T)}{T}\right)T- f(B)\, T =0
    \end{aligned}
    \qquad \forall~\ms R \in \,]0, T\cdot f^{\prime}(B))[\,.
\end{equation}
Notice that,  by 
definition of the function
$\bs y[\ms R, B, f](\cdot)$
in \S~\ref{defi:ur},
the terminal value satisfies\linebreak  
$\bs y[\ms R, B, f](T) \in \,]T\cdot f^{\prime}(\overline{B}),0[$\,,
for all $\ms R \in \,]0, T\cdot f^{\prime}(B))[$\,. In turn, from the
definition of
$\bs x[\ms L, A, f]$
in \S\ref{defi:vl},
with $A = \overline{B}$, 
and $\ms L = \bs y[\ms R, B, f](T)$,
it follows that  
\begin{equation}
\label{eq:xyinint}
    \bs x \big[ \bs y[\ms R, B, f](T), \overline{B}, f\big](T) \in \,]0 ,T \cdot f^{\prime}(B)[\,,\
    \qquad \forall~\ms R \in \,]0, T\cdot f^{\prime}(B))[\,.
\end{equation}
\smallskip

\noindent
{\bf 3.}
We fix now $\ms R \in \,]0, T\cdot f^{\prime}(B))[\,$, and we
consider the map
$\Upsilon : \ ]0, T \cdot f^{\prime}(B)[\ \to \R$, defined by
\begin{equation}
\Upsilon(x) \doteq   \overline{B}\, \bs y[\ms R, B, f](T) + B\, x -f^*\left(\frac{\bs y[\ms R, B, f](T)+x}{T}\right)T  - f(B)\,T\,.
\end{equation}
Observe that, by~\eqref{eq:massconsyR}, \eqref{eq:massconsyR2}, \eqref{eq:xyinint}, one has
\begin{equation}
    \Upsilon(\ms R)=
    \Upsilon\Big(\bs x \big[ \bs y[\ms R, B, f](T), \overline{B}, f\big](T)\Big)=0\,.
\end{equation}
Hence,  it is sufficient to show that $\Upsilon$ admits only one zero in the interval $]0, T\cdot f^{\prime}(B)[$ to conclude the proof of the first equality in \eqref{eq:yxinvertible}.
To this end, differentiating $\Upsilon$ 
and recalling the well known property of the Legendre transform (e.g. see~\cite[\S A.2]{MR2041617}),
$$
(f^*)^{\prime}(p) = (f^{\prime})^{-1}(p)
\qquad\forall~p\,,
$$
we find
\begin{equation}
\label{eq:diffups}
    \begin{aligned}
        \Upsilon'(x)&=B-(f^{\prime})^{-1}\left(\frac{ \bs y[\ms R, B, f](T) +x}{T}\right) \\   
    & = (f^{\prime})^{-1}\left(\frac{0+Tf^{\prime}(B)}{T}\right)-(f^{\prime})^{-1}\left(\frac{ \bs y[\ms R, B, f](T) +x}{T}\right)\,.
    \end{aligned}
\end{equation}
Since $\bs y[\ms R, B, f](T)  < 0$, $ x < T\cdot f^{\prime}(B)$, and because $f^{\prime}$ is strictly increasing
as $f'$, we deduce from~\eqref{eq:diffups}
that $\Upsilon'(x)>0$ for all $x\in \,]0, T\cdot f^{\prime}(B))[\,$.
Therefore $\Upsilon$ is strictly increasing in the interval $\,]0, T\cdot f^{\prime}(B))[\,$, completing the proof of of the first equality in \eqref{eq:yxinvertible}.
\smallskip

\noindent
{\bf 4.}
We show now that the map $\ms R \mapsto \bs y(\ms R)\doteq \bs y[\ms R, B, f](T)$ is strictly increasing in the interval $\,]0, T\cdot f^{\prime}(B)[$\,. Differentiating \eqref{eq:massconsyR} with respect to $\ms R$, we obtain
\begin{equation}\label{eq:consdiffereq}
\left[\overline{B}-(f^{\prime})^{-1}\Big( \frac{y(\ms R)+\ms R}{T} \Big) \right] \bs y'(\ms R) 
= (f^{\prime})^{-1}\Big( \frac{\bs y(\ms R)+\ms R}{T} \Big)-B\qquad \forall~\ms R \in \,]0, T\cdot f^{\prime}(B))[\,.
\end{equation}
Since $T\cdot f'(\,\overline{B}\,)<\bs y(\ms R) <0$ and $0<\ms R <T\cdot f'(B)$, 
because $f'$ is strictly increasing we deduce
$$
\overline{B} < (f')^{-1}\Big( \frac{\bs y(\ms R)}{T} \Big)<
(f^{\prime})^{-1}\Big( \frac{\bs y(\ms R)+\ms R}{T} \Big) < 
(f')^{-1}\Big( \frac{\ms R}{T} \Big)
<B\,,
$$
which, together with~\eqref{eq:consdiffereq}, implies that $\bs y'(\ms R) >0$
for all $R\in \,]0, T\cdot f^{\prime}(B))[$\,, as wanted. 
In turn, since $\ms L \mapsto \bs x[\ms L, \overline{B}, f](T)$ is the inverse of $\ms R \mapsto \bs y[\ms R, B, f](T)$, this implies that $\ms L \mapsto \bs x[\ms L, \overline{B}, f](T)$ is strictly increasing as well in its domain, and 
that the image of the maps 
$\bs y[\cdot\,, B, f](T)$, $\bs x[\cdot\,, \overline{B}, f](T)$, in~\eqref{eq;yx-def-37} are the sets $\,]0, T\cdot f^{\prime}(B))[\,$ and $\,]0, T\cdot f^{\prime}(B)[$, respectively.
This, together with the monotonicity of the maps $\bs y[\cdot\,, B, f](T)$, $\bs x[\cdot\,, \overline{B}, f](T)$, 
in particular implies the one-sided limits in~\eqref{eq:xy-1sdlim},
thus concluding the proof of the Lemma.
\end{proof}

\begin{remark}\label{rem:monotonicityxy}
    As a consequence of Lemma \ref{lemma:dualshocks}
    and of the monotonicity of $f'$, we find that
    the maps 
    \begin{equation}
        \ms R \mapsto \bs u[\ms R, B, f],\qquad\quad
        \ms L \mapsto \bs v[\ms L, A, f],
    \end{equation}
    defined as in~\eqref{eq:urbf-def}
 and~\eqref{eq:vlaf-def}, 
    are strictly increasing, and that we have
    \begin{equation}\label{eq:limitsuv}
        \begin{aligned}
           & 
           \lim_{\ms R \to 0+}  \bs u[\ms R, B, f]  = \overline{B}, \qquad  
           \lim_{\ms R \to T\cdot f^{\prime}(B)-}\bs u[\ms R, B, f] = B,\\
           \noalign{\smallskip}
            & 
            \lim_{\ms L \to 0-}  \bs v[\ms L, A, f] = \overline{A}, \qquad\  \lim_{\ms L \to T\cdot f^{\prime}(A)+}\bs v[\ms L, A, f] = A.
        \end{aligned}
    \end{equation}
    This implies that
    the functions 
    \begin{equation*}
    \begin{aligned}
        &\bs u[\cdot, \cdot, f]: \ ]0, T\cdot f'(B)[\,
 \times \,]\theta,+\infty[\, \to \mathbb R,
 \\
 &\bs v[\cdot, \cdot, f]: \ ]T\cdot f'(A),0[\, \times \,]\theta,+\infty[\, \to \mathbb R
 \end{aligned}
    \end{equation*}
can be extended to continuous function on $[0, T\cdot f'(B)]\times$ $\,]\theta,+\infty[$\, and
 $[T\cdot f'(A),0] \times \,]\theta,+\infty[$, setting
 \begin{equation}
 \label{eq:u-v-cont-ext}
 \begin{aligned}
     \bs u[0, B, f]  &= \overline{B},\qquad\quad \bs u[T\cdot f^{\prime}(B), B, f]  =B,
     \\
     \noalign{\smallskip}
     \bs v[0, A, f] &= \overline{A},
     \qquad\quad\
     \bs v[T\cdot f^{\prime}(A), A, f] =
     A.
 \end{aligned}
 \end{equation}
 Moreover, one has
 \begin{equation}
 \label{eq:uv-ba-ineq}
     \begin{aligned}
        \bs u[\ms R, B, f] < B\qquad\quad &\forall~\ms R\in \,]0, T\cdot f'(B)[\,,
         \\
         \noalign{\smallskip}
        \bs v[\ms L, A, f] > A\qquad\quad &\forall~\ms L\in\,]T\cdot f'(A),0[\,.
     \end{aligned}
 \end{equation}
\end{remark}

\subsection{Right forward shock-rarefaction wave pattern (Figure~\ref{fig:Rrare}, right)}
\label{def:rsr-block}
For every $B >\theta$, \
$0<\ms R< T\cdot f^{\prime}(B)$,
we define now:
\begin{itemize}
[leftmargin=25pt]
\item[-]a constant $\bs \tau[\ms R, B, f]$;
\smallskip
\item[-]a function 
$(x,t)\mapsto u[\ms R, B, f](x,t),$ $(x,t)\in \Gamma[\ms R, B, f]$;
\end{itemize} 
with the following properties.
When $f=f_r$, the function
$u[\ms R, B, f](x,t)$ defines a (forward) solution
associated to the operator
$\sabp$, which contains 
a shock 
starting 
from the interface $\{x=0\}$ at time
$t=\bs \tau[\ms R, B, f_r]$.
%
The  location of such a shock is given by the map $t\mapsto \bs x\big[\bs y[\ms R, B, f_r](T), \overline{B}, f_r\big](t)$,
where 
$\bs y[\ms R,B, f_r]$ and $\bs x[\ms L,\overline{B}, f_r]$
with $\ms L=\bs y[\ms R, B, f_r](T)$,
are the backward shocks 
of a backward solution 
associated to the operator
$\sabm$
introduced in \S~\ref{defi:ur}-\ref{defi:vl}. 
Because of Lemma~\ref{lemma:dualshocks},
the shock
$t\mapsto \bs x\big[\bs y[\ms R, B, f_r](T), \overline{B}, f_r\big](t)$
reaches the point $x=R$
at time $t=T$.
We can regard 
$\bs x\big[\bs y[\ms R, B, f_r](T), \overline{B}, f_r\big]$ as
the ``dual shock''
of the backward shock $\bs y[\ms R,B, f_r]$.
\smallskip

We proceed to introduce these definitions as follows. 
With the same notations of \S~\ref{defi:ur}-\ref{defi:vl},
for every $B >\theta$, \
$0<\ms R< T\cdot f^{\prime}(B)$,
we set
\begin{equation}
\label{eq:tau-rbf-def}
    \bs \tau[\ms R, B, f]\doteq
    \bs s \big[\bs y[\ms R,B,f](T), \; \overline{B}, \; f\big]=\frac{\bs y[\ms R,B,f](T)}{f'(\,\overline{B}\,)}\,.
\end{equation}
Notice that, by the construction in \S~\ref{defi:ur}, 
and because of Lemma~\ref{lemma:dualshocks},
$\bs \tau[\ms R, B, f]$ depends continuously on the parameters $\ms R, B$,
 the image of the map $\ms R\mapsto \bs \tau[\ms R, B, f]$, $\ms R\in\,]0, T\cdot f'(B)[\,$, is the set~$\,]0,T[\,$, and $\ms R\mapsto \bs \tau[\ms R, B, f]$ is decreasing.

Next, we denote by $ \Gamma[\ms R, B, f] \subset (0, T) \times \mathbb R$ the polygonal set (the pink set in Figure~\ref{fig:Rrare})
\begin{equation}
\label{eq:gammadef}
    \Gamma[\ms R, B, f]\doteq \Gamma_1[\ms R, B, f]\cup\Gamma_2[\ms R, B, f]\,,
\end{equation}
with
\begin{equation}
\label{eq:gamma12def}
    \begin{aligned}
      \Gamma_1[\ms R, B, f]&\doteq 
      \Big\{(x, t) \in\,]0,+\infty[\, \times \,]0, T[\,   \; : \; \ms R-(T-t)\cdot f^{\prime}(B) < x < \ms R -(T-t) \cdot f^{\prime}(\bs u[\ms R, B, f]),
      \\
    \noalign{\smallskip}
      &\hspace{3in}
      \bs \tau[\ms R, B, f] < t < T \Big\},
      \\
      \noalign{\bigskip}
      \Gamma_2[\ms R, B, f]&\doteq 
      \Big\{(x, t) \in\,]0,+\infty[\, \times \,]0, T[\,  \; : \; -(\bs \tau[\ms R, B, f]\!-t) \cdot\! f^{\prime}(\,\overline{B}\,) < x < \ms R \!-\!(T-t) \cdot f^{\prime}(\bs u[\ms R, B, f]),
      \\
    \noalign{\smallskip}
      &\hspace{3in}
      0 \leq t \leq \bs \tau[\ms R, B, f]\Big\}.
    \end{aligned}
\end{equation}
%
Then, set $\gamma(t)\doteq \bs x\big[\bs y[\ms R, B, f](T), \overline{B}, f\big](t),$
and denote by $\ms u[\ms R, B, f]: \Gamma[\ms R, B, f] \to \mathbb R$ the function defined by 
\begin{equation}
\label{eq:urbff-def}
     \ms u[\ms R, B, f] (x,t)\doteq
     \begin{cases}
     \!B  &\text{if}\qquad  0< x < \gamma(t),
\\
\noalign{\smallskip}
\!(f^{\prime})^{\strut -1}\bigg(\dfrac{x-\ms R+T \cdot f^{\prime}(\bs u[\ms R, B, f])}{t}\bigg) &\text{otherwise.}
     \end{cases}
\end{equation}
Notice that, by~\eqref{eq:xL} and~\eqref{eq:yxinvertible},
one has $\gamma(\bs \tau[\ms R, B, f])=0$, 
$\gamma(T)=\ms R$.
Moreover, by the same arguments of the
proof of Lemma~\ref{lemma:dualshocks} it follows that
$u[\ms R, B, f] (x,t)$
is a distributional solution of~\eqref{eq:convexconslaw}
on $\Gamma[\ms R, B, f]$. Furthermore, 
since
$t\mapsto \gamma'(t)$
is strictly increasing
as observed in \S~\ref{defi:vl}, it follows
that also the
map 
$$
t\mapsto 
\dfrac{\gamma(t)-\ms R+T \cdot f^{\prime}(\bs u[\ms R, B, f])}{t}
$$
is strictly increasing. 
Therefore,
by virtue of~\eqref{eq:yxinvertible},
and relying on~\eqref{eq:uv-ba-ineq}, we find
\begin{equation}
\label{eq:lax-rfs-cond}
    \begin{aligned}
         \lim_{x\to\gamma(t)+}
         \ms u[\ms R, B, f] (x,t)&\leq 
         \lim_{x\to\gamma(T)+}
         \ms u[\ms R, B, f] (x,T)
         \\
         &= 
         \lim_{x\to \ms R+}\ms u[\ms R, B, f] (x,T)
         \\
         &=\bs u[\ms R, B, f]<B\\
         &=\lim_{x\to\gamma(t)-}
         \ms u[\ms R, B, f] (x,t)
         \qquad \forall~t\in
    [\bs \tau[\ms R, B, f],T]\,,
    \end{aligned}
\end{equation}
which shows that the
Lax entropy condition 
is satisfied along the curve $(t,\gamma(t))$, $t\in [\bs\tau[\ms R, B, f],T]$.
Since the flux in~\eqref{eq:convexconslaw} is strictly convex,
this proves that
$u[\ms R, B, f] (x,t)$
provides an entropy weak solution of~\eqref{eq:convexconslaw}
on the region $\Gamma[\ms R, B, f]$.
Notice that, by~\eqref{eq:der-mon}, from~\eqref{eq:lax-rfs-cond} we deduce in particular that $f'(B)>\lambda\big(\bs u[\ms R, B, f], B\big)=\gamma'(T)$,
which in turn, by the strict monotonicity of 
$\dot\gamma(t)$, yields
\begin{equation}
\label{eq:lax-rfs-cond-2}
    f'(B)>\gamma'(t)\qquad \forall~t\in
    [\bs \tau[\ms R, B, f],T]\,.
\end{equation}
Hence, relying on~\eqref{eq:lax-rfs-cond-2}, we find
\begin{equation}
    \label{eq:lax-rfs-cond-3}
    f'(B)>\frac{\gamma(T)-\gamma(\bs\tau[\ms R, B, f])}{T-\bs\tau[\ms R, B, f]}=
    \frac{\ms R}{T-\bs\tau[\ms R, B, f]}\,.
\end{equation}

\subsection{Left forward rarefaction-shock wave pattern (Figure~\ref{fig:Lrare}, left)}
\label{def:lrs-block}
Symmetrically to~\S~\ref{def:rsr-block}, 
for every $A <\theta$, \
$T\cdot f^{\prime}(A)<\ms L< 0$,
we define here:

\begin{itemize}
[leftmargin=25pt]
\item[-]a constant $\bs \sigma[\ms L, A, f]$;
\smallskip
\item[-]a function 
$(x,t)\mapsto \ms v[\ms L, A, f](x,t),$ $(x,t)\in \Delta[\ms L, A, f]$;
\end{itemize} 
with the following properties.
When $f=f_l$, the function
$\ms v[\ms L, A, f](x,t)$ defines a (forward) solution
associated to the operator
$\sabp$, which contains 
a shock 
starting 
from the interface $\{x=0\}$ at time
$t=\bs \sigma[\ms L, A, f]$.
The  location of such a shock is given by the map $t\mapsto \bs y \big[\bs x[\ms L,A,f_l](T), \; \overline{A}, \; f_l\big](t)$,
where 
$\bs x[\ms L,A, f_l]$ and $\bs y[\ms R,\overline{A}, f_l]$
with $\ms R=\bs x[\ms L,A,f_l](T)$, 
are the backward shocks 
of a backward solution 
associated to the operator
$\sabm$
introduced in \S~\ref{defi:ur}-\ref{defi:vl}. 
Because of Lemma~\ref{lemma:dualshocks},
the shock
$t\mapsto \bs y \big[\bs x[\ms L,A,f_l](T), \; \overline{A}, \; f_l\big](t)$
reaches the point $x=L$
at time $t=T$.
We can regard \linebreak 
$\bs y \big[\bs x[\ms L,A,f_l](T), \; \overline{A}, \; f_l\big]$ as
the ``dual shock''
of the backward shock $\bs x[\ms L,A,f_l]$.
\smallskip

We proceed to introduce these definitions as follows. 
With the same notations of \S~\ref{defi:ur}-\ref{defi:vl},
for every $A <\theta$, \
$T\cdot f^{\prime}(A)<\ms L< 0$
we set
\begin{equation}
\label{eq:sigma-laf-def}
    \bs \sigma[\ms L, A, f]\doteq
     \bs t \big[\bs x[\ms L,A,f](T), \; \overline{A}, \; f\big]=
     \frac{\bs x[\ms L,A,f](T)}{f'(\,\overline{A}\,)}\,.
\end{equation}
By the construction in \S~\ref{defi:vl}, 
and because of Lemma~\ref{lemma:dualshocks},
$\bs\sigma[\ms L, A, f]$ depends continuously on the parameters $\ms L, A$, the image of the map
$\ms L \mapsto \bs\sigma[\ms L, A, f]$, $\ms L\in \,]T\cdot f'(A), 0[\,$, is the set $\,]0,T[\,$,
and $\ms L \mapsto \bs\sigma[\ms L, A, f]$
is increasing.

Next,  we denote by $\Delta[\ms L, A, f] \subset (0, T) \times \mathbb R$ the polygonal set (the blue set in Figure~\ref{fig:Lrare})
\begin{equation}
\label{eq:deltadef}
    \Delta[\ms L, A, f]\doteq \Delta_1[\ms L, A, f]\cup\Delta_2[\ms L, A, f]\,,
\end{equation}
with
\begin{equation}
\label{eq:delta12def}
    \begin{aligned}
      \Delta_1[\ms L, A, f]&\doteq 
      \Big\{(x, t) \in\,]-\infty,0[\, \times \,]0, T[\,   \; : \; \ms L -(T-t) \cdot f^{\prime}(\bs v[\ms L, A, f])<x<\ms L-(T-t)\cdot f^{\prime}(A),
      \\
    \noalign{\smallskip}
      &\hspace{3in}
      \bs \sigma[\ms L, A, f] < t < T \Big\},
      \\
      \noalign{\bigskip}
      \Delta_2[\ms L, A, f]&\doteq 
      \Big\{(x, t) \in\,]-\infty,0[\, \times \,]0, T[\,   \; : \; \ms L\! -\!(T-t) \cdot f^{\prime}(\bs v[\ms L, A, f])< x < -(\bs \sigma[\ms L, A, f]\!-t) \cdot\! f^{\prime}(\,\overline{A}\,), 
      \\
    \noalign{\smallskip}
      &\hspace{3in}
      0 \leq t \leq \bs \sigma[\ms L, A, f]\Big\}.
    \end{aligned}
\end{equation}
Then, set $\gamma(t)\doteq \bs y\big[\bs x[\ms L, A, f](T), \overline{A}, f\big](t)$,
and denote by $\ms v[\ms L, A, f]: \Delta[\ms L, A, f]\to \mathbb R$ the function defined by 
\begin{equation}
\label{eq:vlaff-def}
     \ms v[\ms L, A, f] (x,t)\doteq
     \begin{cases}
     \!A  &\text{if}\qquad  \gamma(t)< x < 0,
\\
\noalign{\smallskip}
\!(f^{\prime})^{\strut -1}\bigg(\dfrac{x-\ms L+T \cdot f^{\prime}(\bs v[\ms L, A, f])}{t}\bigg) &\text{otherwise.}
     \end{cases}
\end{equation}
Observe that, by~\eqref{eq:yR} and~\eqref{eq:yxinvertible},
one has $\gamma(\bs \sigma[\ms L, A, f])=0$, 
$\gamma(T)=\ms L$.
With the same arguments of \S~\ref{def:rsr-block}, it follows that $v[\ms L, A, f] (x,t)$ provides an entropy weak solution of~\eqref{eq:convexconslaw}
on the region $\Delta[\ms L, A, f]$,
and that we have
\begin{equation}
    \label{eq:lax-rfs-cond-4}
    f'(A)<\frac{\gamma(T)-\gamma(\bs \sigma[\ms L, A, f])}{T-\bs \sigma[\ms L, A, f]}=
    \frac{\ms L}{T-\bs \sigma[\ms L, A, f]}\,.
\end{equation}
\begin{remark}
\label{rem:caratt-u}
The constant $\bs u[\ms R, B, f]$ defined in \S~\ref{defi:ur} 
is crucial to characterize the 
jump of an attainable profile
$\omega\in \mc A^{[AB]}(T)$
at the point
\begin{equation*}
\ms R \doteq  \inf \big\{ R > 0 \; : \; x-T \cdot f_r^{\prime}(\omega(x+)) \geq 0  \quad \forall \; x \geq R\big\},
\end{equation*}
when $\ms R \in \,]0, T \cdot f_r^{\prime}(B){ [}$.
The state $\bs u[\ms R, B, f]$ is constructed so to be the largest right state that one can achieve
at $(\ms R,T)$ with a shock that \textit{isolates} the interface $\{x=0\}$ from the semiaxis $\{x>0\}$. In fact, 
the constant $\bs u[\ms R, B, f]$ with $f=f_r$, identifies
a unique state $\bs u$ { $< B$} that has the
 property:
\begin{itemize}
[leftmargin=25pt]
    \item[-] If $\omega =\sabp_T u_0$,  and $u(t,x)=\sabp_t u_0(x)$ admits a shock 
    generated in $\{x\geq 0\}$
    at some time $t=\tau$,
    and reaching the point $(\ms R,T)$, then letting 
    $\gamma(t)$, $t\in [\tau, T]$, denote the location of such a shock, one has
    \begin{equation}
    \label{eq:char-u-1}
    u_\gamma\doteq 
        \lim_{t\to T-} u(t,\gamma(t)+)\leq (f_r^{\prime})^{-1}(\mathsf R/T) \quad \Longrightarrow\quad  u_{\gamma}  \leq  \bs u\,.
    \end{equation}
\end{itemize} 
In particular, one has $u_\gamma=\bs u$
in~\eqref{eq:char-u-1} only in the case 
where $\sabp_t u_0$ 
coincides in the polygonal region $\Gamma[\ms R, B, f]$ with the right forward shock-rarefaction pattern described in section~\ref{def:rsr-block}.
By definition of $u_\gamma$ it follows that either 
$\omega(\ms R+)=u_\gamma$,
or else there is another jump
connecting $u_\gamma$
with $\omega(\ms R+)$ which
must satisfy the Lax entropy condition $\omega(\ms R+)< u_\gamma$.
Therefore, as a consequence of~\eqref{eq:char-u-1}
we find a necessary condition
for the attainability of $\omega$ at time $T$ given by 
\begin{equation}
\label{eq:lax-om-r-cond}
    \omega(\ms R+)\leq  \bs u[\ms R, B, f_r]\,,
\end{equation}
(see \eqref{eq:2a} of Theorem~\ref{thm:attprofiles}
and the proof in \S~\ref{sec:2a}).
The interesting fact is that,
in the case $\ms R \in \,]0, T \cdot f_r^{\prime}(B){ [}$,
condition~\eqref{eq:lax-om-r-cond}, 
together with the condition
\begin{equation}
\label{eq:omega-cond-1}
    \omega(x)\geq B \qquad \forall~x\in \,]0, \ms R[\,,
\end{equation}
(see~\eqref{eq:2bl-r1}, \eqref{eq:2b-lr2}, 
of Theorem~\ref{thm:attprofiles}),
is also sufficient to
guarantee the existence of
an $AB$-entropy solution $u(x,t)$ 
that satisfies
\begin{equation}
u(x,T)=\omega(x)\qquad\quad
\forall~x\in ]0, \ms R], 
\qquad
u(\ms R,T)=\omega(\ms R+).
\end{equation}
To illustrate this claim, 
in view of the definitions introduced in the previous sections {  we  proceed as follows.}
    \begin{itemize}
    [leftmargin=25pt]
    \item[-]
        By  solving \eqref{eq:yR} one determines the end point $\bs y[\ms R, B, f_r](T)$ of a ``left backward shock'' (Figure~\ref{fig:Rrare}, left). The map $t \mapsto \bs y[\ms R, B, f_r](t)$ represents the position of a shock in a 
        $\overline{B}\overline{A}$-entropy solution
        (which is associated to the backward solution operator
        $\sabm$, 
        see Definition~\ref{def:backop});
        
        \item[-] given the final position $\bs y[\ms R, B, f_r]$
        of the ``backward shock" , one considers the solution $t  \mapsto \gamma(t)\doteq\bs x[\bs y[\ms R, B, f_r], \overline{B}, f_r](t)$ to~\eqref{eq:xL}, when $\ms L =\bs y[\ms R, B, f_r]$, $A = \overline{B}$, $f = f_r$ (see Figure~\ref{fig:Rrare}, right). This map represents the position of a shock in a ``forward solution", i.e. in an $AB$-entropy solution associated to the (forward) operator 
        $\sabp$ in~\eqref{eq:f-AB-op}. Actually, {  we will show in \S \ref{subsec:part3b}, using the results of this section, that} $(t, \gamma(t))$
        is the location
        of a shock of $\sabpt u_0$, with $u_0=\sabmt\omega$.

        \item[-] once determined the point $\bs y[\ms R, B, f_r](T)$, one defines $\bs u[\ms R, B, f_r]$ as the state realizing the slope $(\bs y[\ms, B, f_r](T)+\ms R)/T$ (see~\eqref{eq:urbf-def}):
        $$
        f_r^{\prime}(\bs u[\ms R, B, f_r]) = \frac{\bs y[\ms R, B, f_r](T)+\ms R}{T};
        $$
        
        \item[-] thanks to Lemma \ref{lemma:dualshocks}, we know that the final position at time $T$ of the shock $\gamma(t)$
        satisfies 
        $$
        \gamma(T)=
        \ms R. 
        $$
        Using this procedure, if a profile $\omega$ satisfies the conditions~\eqref{eq:lax-om-r-cond}-\eqref{eq:omega-cond-1},
        we will show in \S~\ref{sec:(3)-(2)}-\ref{subsec:part3b} that we can construct admissible $AB$-shocks 
        that produce at time $T$ the given jump in the profile $\omega$ at position~$\ms R$.
    \end{itemize}
Entirely symmetric considerations hold for the  state $\bs v[\ms L, A, f]$ defined in \S~\ref{defi:vl} (see Figure~\ref{fig:Lrare}).
As a byproduct of this analysis we will obtain that attainable profiles are fixed points of the backward forward solution operator, as stated in Theorem~\ref{thm:backfordiscflux}.
\end{remark}

\section{Statement of the main results}
\label{sec:statement-main}
Conditions (1), (2) of Theorem~\ref{thm:backfordiscflux} will be shown to be equivalent by proving
that they are both equivalent to a  characterization of the
attainable set $\mc A^{[AB]}(T)$ in~\eqref{eq:attset} via
 Ole\v{\i}nik-type inequalities and state constraints.
To present these results we need 
to introduce some furhter notations.

Given a flux $f(x,u)$ as in~\eqref{discflux}, we
will use the notations
$f_{l,-}^{-1}\doteq
({f_l}_{\mid (-\infty,\theta_l]})^{-1}$,\,
$f_{r,-}^{-1}\doteq
({f_r}_{\mid (-\infty,\theta_r]})^{-1}$, for the inverse of the restriction of 
$f_l$, $f_r$  to their decreasing part, respectively, 
and 
$f_{l,+}^{-1}\doteq ({f_l}_{\mid [\theta_l,+\infty)})^{-1}$, \,
$f_{r,+}^{-1}\doteq ({f_r}_{\mid [\theta_{r},+\infty)})^{-1}$, for the inverse of the restriction of 
$f_l$, $f_r$ to their increasing part, respectively.
Then, we set 
\begin{equation}
\label{pimap-def}
\pi_{l,\pm}\doteq f_{l,\pm}^{-1} \circ f_l\,,
\qquad\quad
\pi_{r,\pm}\doteq f_{r,\pm}^{-1} \circ f_r\,,
\qquad\quad
\pi_{l,\pm}^{r}\doteq f_{l,\pm}^{-1} \circ f_r\,,
\qquad\quad
\pi_{r,\pm}^{l}\doteq f_{r,\pm}^{-1} \circ f_l\,.
\end{equation}
Moreover, in connection with a
function $\omega:\R\to\R$    
we define the quantities
%
\begin{equation}
\label{eq:LR-def}
    \begin{aligned}
    \ms R[\omega, f_r] &\doteq  \inf \big\{ R > 0 \; : \; x-T \cdot f_r^{\prime}(\omega(x)) \geq 0  \quad \forall \; x \geq R\big\},
    \\
    \noalign{\smallskip}
    \ms L[\omega, f_l] &\doteq  \sup \big\{ L < 0 \; : \; x-T \cdot f_l^{\prime}(\omega(x)) \leq 0 \quad \forall \; x \leq L\big\},
\end{aligned}
\end{equation}
and, if $\ms L[\omega, f_l] \in \,]T\cdot f_l^{\prime}(A), 0[$\,,
we set
\begin{equation}
            \label{eq:Rtildedef}
             \widetilde{\ms R}
            [\omega,f_l,f_r,A,B]
            \doteq  \big(T-\bs \sigma\big[\ms L[\omega,f_l], A, f_l\big]\big) \cdot f_r^{\prime}(B),
            \end{equation}
while, if $\ms R[\omega, f_r] \in \,]0, T\cdot f^{\prime}_r(B)[\,$, we set
\begin{equation}
            \label{eq:Ltildedef}
            \widetilde{\ms L}[\omega,f_l,f_r,A,B]
            \doteq \big(T-\bs \tau\big[\ms R[\omega,f_r], B, f_r\big]\big) \cdot f_l^{\prime}(A).
            \end{equation}
where $\bs \sigma[\ms L, A, f_l]$, $\bs \tau[\ms R, B, f_r]$, denote the shock starting
times introduced in
\S~\ref{def:rsr-block}-\ref{def:lrs-block}. 
Recalling~\eqref{eq:xy-1sdlim}, \eqref{eq:tau-rbf-def}, \eqref{eq:sigma-laf-def},  we can extend by continuity the definitions~\eqref{eq:Rtildedef}, \eqref{eq:Ltildedef},
setting
\begin{equation}
    \label{eq:tildeLR-bis}
    \begin{aligned}
        \widetilde{\ms R}
            [\omega,f_l,f_r,A,B]
            &\doteq 0,\qquad\text{if}\qquad
            \ms L[\omega, f_l]=0\,,
            \\
            \noalign{\smallskip}
            \widetilde{\ms L}[\omega,f_l,f_r,A,B]
            &\doteq 0,\qquad\text{if}\qquad
            R[\omega, f_r]=0\,.
    \end{aligned}
\end{equation}
Such quantities are used to express the pointwise constraints satisfied by $\omega$  in intervals containing the origin
whenever $\omega$ is attainable.
Next, to express the Ole\v{\i}nik-type inequalities satisfied by 
the attainable profiles it is useful to introduce the functions, 
\begin{equation}
\label{eq:ghdef}
    \begin{aligned}
        g[\omega,f_l,f_r](x)&\doteq
        \!\dfrac{f'_l(\omega(x))\left[  f'_r \circ 
        \pi_{r,-}^{l}
        (\omega(x)) \right]^2 }
     {  \left[f''_r\circ 
      \pi_{r,-}^{l}
      (\omega(x))\right]\!\left[f'_l(\omega(x))\right]^2\!\big(T\!\cdot\! f'_l(\omega(x))-x \big)+x\big[f'_r\circ 
       \pi_{r,-}^{l}
       (\omega(x))\big] ^2\!  f''_l(\omega(x)) },
        \\
        \noalign{\medskip}
        h[\omega,f_l,f_r](x)&\doteq\!\dfrac{f'_r(\omega(x))\big[  f'_l \circ 
        \pi_{l,+}^{r}
        (\omega(x)) \big]^2 }
     {  \big[f''_l\circ 
     \pi_{l,+}^{r}
     (\omega(x))\big]\!\left[f'_r(\omega(x))\right]^2\big( T\!\cdot\!f'_r(\omega(x))-x \big)+x\big[f'_l\circ 
     \pi_{l,+}^{r}
     (\omega(x))\big] ^2\!  f''_r(\omega(x)) },
    \end{aligned}
\end{equation}
{  defined for $x \in \,]\ms L[\omega,f_l], 0[$\,, 
$\omega(x)\leq A$, and
for 
$x \in \, ]0, \ms R[\omega,f_r][$\,,
$\omega(x)\geq B$, respectively}.
{ 
\begin{remark}
The definitions of the functions $g, h$ are meaningful in their domains. 
In fact, the maps $\pi_{r,-}^{l},$
$\pi_{l,+}^{r}$ in~\eqref{pimap-def} 
(that appear in the definitions of $g, h$)
are well defined if
$\omega(x)\leq A$, and $\omega(x)\geq B$, respectively.
Moreover, by definition~\eqref{eq:LR-def}, we have
$$
\begin{aligned}
    & Tf_l^\prime(\omega(x))-x < 0, \qquad f_l^\prime(\omega(x)) < 0 \qquad \forall \; x \in \,]\ms L[\omega,f_l],\, 0[\,,
    \\
    & Tf_r^\prime(\omega(x))-x> 0, \qquad f_r^\prime(\omega(x)) > 0 \qquad \forall \; x \in \,]0,\, \ms R[\omega,f_r][\,.
\end{aligned}
    $$
Hence, relying also on~\eqref{eq:flux-assumption-1},  we deduce that the denominator of $g$ is strictly negative for \linebreak $x \in \,]\ms L[\omega,f_l], 0[$\,, while 
the denominator of $h$ is strictly positive for $x \in \, ]0, \ms R[\omega,f_r][$\,.
The functions $g, h$ will provide 
a one-sided upper bound for the derivative of $\omega$ only in the interval $\,]\ms L[\omega,f_l], 0[$,
assuming $\omega(x)\leq A$, and
on the interval $]0, \ms R[\omega,f_r][$\,, assuming
$\omega(x)\geq B$, respectively.
\end{remark}
}
Since by Remark~\ref{rem:abentr-sol-prop1} we know that 
$\mc A^{[AB]}(T) \subset BV_{\mr{loc}}(\R\setminus\{0\})$, 
we can partition the attainable set as
\begin{equation}
\label{eq:att-set-decomp}
    \mc A^{[AB]}(T)= \bigcup_{\ms L \leq 0,\, \ms R \geq 0} \big( \mc A^{[AB]}(T) \cap \msc A^{\ms L, \ms R}\big),
\end{equation}
where
\begin{equation}
\label{eq:ALR-def1}
\begin{aligned}
    \msc A^{\ms L, \ms R} &\doteq  \Big\{ \omega \in 
     (\mathbf{L^\infty}\cap BV_{\mr{loc}})
    (\mathbb R\setminus\{0\}) \; : \; \
    \ms L[\omega, f_l] = \ms L, \quad  \ms R[\omega, f_r] = \ms R\Big\}\,.
\end{aligned}
\end{equation}
The characterization of the attainable profiles 
in $\msc A^{\ms L, \ms R}$ will be given in:
\begin{itemize}
\item[-]Theorem~\ref{thm:attprofiles}, \ \ if \ \ 
$\ms L < 0$, $\ms R > 0$\,, and $(A,B)$ is non critical;
\smallskip
\item[-]Theorem~\ref{thm:attprofilescrit}, \ \ if \ \ 
$\ms L < 0$, $\ms R > 0$\,,  and $(A,B)$ is  critical;
\smallskip
\item[-]Theorem~\ref{thm:attprofiles2}, \ \ if \ \
$\ms L < 0$, $\ms R = 0$ \ or \ $\ms L = 0$, $\ms R > 0$;
\smallskip
\item[-]Theorem~\ref{thm:attprofiles3}, \ if \ \
$\ms L = 0$, $\ms R = 0$.
\end{itemize}
\smallskip

\begin{remark}
\label{rem:BVclass}
Any element of $\msc A^{\ms L, \ms R}$
is an equivalence class of functions that 
admit one-sided limit at any point $x\in\R$,
and that have at most countably many discontinuities.
Therefore, for any element of $\msc A^{\ms L, \ms R}$, we can always choose a representative  which is left or right continuous. For sake of uniqueness, throughout the paper we will 
consider a representative of $\omega$ that is right continuous.
\end{remark}

Throughout the following
\begin{equation}
\label{eq:dini_der_def}
D^- \omega(x)= \liminf_{h\rightarrow0}   \frac{\omega(x+h)-\omega(x)}{h} ,\quad \quad 
D^+ \omega(x)= \limsup_{h\rightarrow0}   \frac{\omega(x+h)-\omega(x)}{h},
\end{equation}
will denote, respectively, the lower and
the  upper Dini derivative of a function $\omega$ at $x$.

\begin{thm}\label{thm:attprofiles}
    In the same setting of Theorem~\ref{thm:backfordiscflux}, 
    let $(A, B)$ be a {non critical}  connection, 
    let $\mc A^{[AB]}(T)$, $T>0$, be the set
    in~\eqref{eq:attset}, and
    let $\omega$ be an element of the
    set $ \msc A^{\ms L, \ms R}$  
    in~\eqref{eq:ALR-def1},
    with $\ms L < 0$, $\ms R > 0$.
Then, 
$\omega \in \mc A^{[AB]}(T)$ if and only if the limits $\omega(0\pm)$ exist, and there
hold:

\begin{enumerate}
[leftmargin=20pt]
     \item[(i)]      
    the following Ole\v{\i}nik-type inequalities are satisfied 
    \begin{equation}\label{eq:1a}
    \begin{aligned}
    D^+ \omega(x) &\leq \frac{1}{T \cdot f_l^{\second}(\omega(x))} \qquad \forall \; x \in\,]\!-\!\infty,\;  \ms L[\,,
    \\
    \noalign{\smallskip}
       D^+ \omega(x) &\leq \frac{1}{T \cdot f_r^{\second}(\omega(x))} \qquad \forall \; x \in \,]\ms R, +\infty[\,.
        \end{aligned}
    \end{equation}
    Moreover, 
    letting $g, h$ be the functions in~\eqref{eq:ghdef},
    and letting
    $\widetilde {\ms L}\doteq \widetilde{\ms L}[\omega,f_l,f_r,A,B]$, \linebreak $\widetilde {\ms R}\doteq \widetilde{\ms R}
            [\omega,f_l,f_r,A,B]$,
    be the constants in~\eqref{eq:Rtildedef},
    \eqref{eq:Ltildedef},
    if $\ms R \in \,]0, T \cdot f_r^{\prime}(B)[$\,, and if  \ $ \widetilde{\ms L}>{\ms L}$,  then one has
    \begin{equation}
    \label{eq:1b1}
        D^+\omega(x) \leq g[\omega, f_l, f_r](x) \qquad \forall \; x \in\,]\ms L,\, \widetilde{\ms L}\,[\,,
    \end{equation}
    while, if $\ms L \in \,]T\cdot f_l^{\prime}(A), 0[$\,, and if \ $ \widetilde{\ms R}<{\ms R}$,  then one has
    \begin{equation}
    \label{eq:1b12}
        D^+\omega(x) \leq h[\omega, f_l, f_r](x) \qquad \forall \; x \in \,]\widetilde{\ms R }, \; \ms R[\,.
    \end{equation}  

\item[(ii)] letting $\bs u[\ms R, B, f_r]$, $\bs v[\ms L, A, f_l]$,  be 
constants defined as
in~\eqref{eq:urbf-def},
\eqref{eq:vlaf-def},
the following pointwise state constraints are satisfied
\smallskip

\begin{equation}
\label{eq:2a}
\begin{aligned}
\ms L \in \,]T\cdot f_l^{\prime}(A), 0[
  \qquad  &\Longrightarrow \qquad
  \omega(\ms L-) \geq \bs v[\ms L, A, f_l]
\geq  \omega(\ms L+)\,,
\\
\noalign{\medskip}
 \ms R \in \,]0, T \cdot f_r^{\prime}(B)[
  \qquad  &\Longrightarrow \qquad 
  \omega(\ms R+) \leq \bs u[\ms R, B, f_r] \leq  \omega(\ms R-)\,.
\end{aligned}
\qquad
\end{equation}
\smallskip
\begin{equation}
\label{eq:2bl-r1}
\begin{aligned}
\Big[\ms L \in \,]T\cdot f_l^{\prime}(A), 0[\,\;  \ \mathrm{and} \;  \ \ms R \leq  \widetilde{\ms R}\,
\Big]
\quad \mathrm{or} \quad\ \ms L \leq T\!\cdot\!f_l^{\prime}(A)
\qquad
&\Longrightarrow\qquad 
\omega(x) = B\qquad \forall~x\in\,]0, \ms R[\,,
\qquad
%
\\
\noalign{\smallskip}
\Big[\ms R \in \,]0, T\cdot f_r^{\prime}(B)[\,\; \ \mathrm{and} \; \  \widetilde{\ms L}\leq \ms L\,\Big] \quad \mathrm{or} \quad \ms R \geq T\!\cdot\!f_r^{\prime}(B)
\qquad&\Longrightarrow\qquad 
\omega(x) = A\qquad \forall~x\in\,]\ms L, 0[\,,
\qquad
\end{aligned}
\end{equation}
\smallskip
\begin{align}
\label{eq:2b-lr2}
\hspace{0.6in}\ms L \in \,]T\cdot f_l^{\prime}(A), 0[\; \  \mathrm{and}
 \; \ \widetilde{\ms R}<\ms R\qquad
 &\Longrightarrow\qquad   
 \left\{
 \begin{aligned}
     \omega(x)&=B\qquad\forall~x\in\,]0, 
      \widetilde{\ms R}\,],
      \\
      \omega( \widetilde{\ms R}+)&= B,
      \\
      \omega(x)&\geq B\qquad\forall~x\in\,]\,\widetilde{\ms R},\,\ms R[\,,
 \end{aligned} 
 \right.
%
\\
\noalign{\medskip}
\label{eq:2b-lr2-2}
\hspace{0.6in}\ms R \in \,]0, T\cdot f_r^{\prime}(B)[\, \; \ \mathrm{and}
\; \ \ms L <
 \widetilde{\ms L}\qquad
 &\Longrightarrow\qquad   
 \left\{
 \begin{aligned}
     \omega(x)&=A\qquad\forall~x\in [\, 
      \widetilde{\ms L},\,0[\,,
      \\
      \omega( \widetilde{\ms L}-)&= A,
      \\
      \omega(x)&\leq A\qquad\forall~x\in\,]\ms L, \widetilde{\ms L}\,[\,,
 \end{aligned} 
 \right.
 \end{align}
%
\smallskip
\begin{equation}
\label{eq:3b-lr2}
    \quad\ \begin{aligned}
        \ms L \leq T\cdot f'_l(A)
        \qquad
 &\Longrightarrow\qquad   \omega(\ms L-)\geq \omega(\ms L+),
 \\
 \noalign{\smallskip}
        \ms R \geq T\cdot f'_r(B)
        \qquad
 &\Longrightarrow\qquad   \omega(\ms R-)\geq \omega(\ms R+).
    \end{aligned}
\end{equation} 

\end{enumerate}
\end{thm}

\begin{figure}[ht]
\centering

\tikzset{every picture/.style={line width=0.75pt}} 

\begin{tikzpicture}[x=0.75pt,y=0.75pt,yscale=-0.7,xscale=0.7]

\draw  [fill={rgb, 255:red, 208; green, 2; blue, 27 }  ,fill opacity=0.45 ] (455.87,52.53) -- (411.87,237.87) -- (335.87,142.53) -- (335.87,74.53) -- cycle ;
\draw [line width=1.5]    (335.59,237.62) -- (335.59,18.87) ;
\draw [shift={(335.59,14.87)}, rotate = 90] [fill={rgb, 255:red, 0; green, 0; blue, 0 }  ][line width=0.08]  [draw opacity=0] (6.97,-3.35) -- (0,0) -- (6.97,3.35) -- cycle    ;
\draw [line width=1.5]    (90.8,239.53) -- (590.13,238.3) ;
\draw [shift={(594.13,238.29)}, rotate = 179.86] [fill={rgb, 255:red, 0; green, 0; blue, 0 }  ][line width=0.08]  [draw opacity=0] (6.97,-3.35) -- (0,0) -- (6.97,3.35) -- cycle    ;
\draw    (88.13,51.13) -- (590.2,52.17) ;
\draw  [dash pattern={on 4.5pt off 4.5pt}]  (247.2,51.87) -- (334.53,85.87) ;
\draw  [dash pattern={on 4.5pt off 4.5pt}]  (316.83,51.95) -- (332.53,56.53) ;
\draw  [dash pattern={on 4.5pt off 4.5pt}]  (369.87,51.87) -- (332.53,56.53) ;
\draw  [dash pattern={on 4.5pt off 4.5pt}]  (293.2,51.87) -- (337.87,65.2) ;
\draw  [dash pattern={on 4.5pt off 4.5pt}]  (206.53,51.87) -- (334.53,107.87) ;
\draw  [dash pattern={on 4.5pt off 4.5pt}]  (223.2,51.2) -- (335.2,97.2) ;
\draw  [dash pattern={on 4.5pt off 4.5pt}]  (271.2,52.53) -- (332.53,75.2) ;
\draw  [dash pattern={on 4.5pt off 4.5pt}]  (335.67,219.33) -- (346,236.85) ;
\draw  [dash pattern={on 4.5pt off 4.5pt}]  (335.2,153.2) -- (394.53,236.53) ;
\draw  [dash pattern={on 4.5pt off 4.5pt}]  (337.67,203.33) -- (359.33,238.18) ;
\draw    (335.67,184) -- (373.33,238.33) ;
\draw  [dash pattern={on 4.5pt off 4.5pt}]  (190.53,51.87) -- (335.43,118.6) ;
\draw  [dash pattern={on 4.5pt off 4.5pt}]  (175.2,51.87) -- (333.87,129.2) ;
\draw    (138.67,52) -- (166.67,236.18) ;
\draw  [dash pattern={on 4.5pt off 4.5pt}]  (335.2,165.2) -- (383.53,236.71) ;
\draw    (166.53,51.87) -- (335.87,142.53) ;
\draw  [dash pattern={on 4.5pt off 4.5pt}]  (138.67,52) -- (272,236.85) ;
\draw  [dash pattern={on 4.5pt off 4.5pt}]  (138.67,52) -- (196.67,237.51) ;
\draw  [dash pattern={on 4.5pt off 4.5pt}]  (138.67,52) -- (214,236.85) ;
\draw  [dash pattern={on 4.5pt off 4.5pt}]  (138.67,52) -- (335.67,219.33) ;
\draw  [dash pattern={on 4.5pt off 4.5pt}]  (138.67,52) -- (233.33,236.85) ;
\draw  [dash pattern={on 4.5pt off 4.5pt}]  (138.67,52) -- (182.67,238.85) ;
\draw  [dash pattern={on 4.5pt off 4.5pt}]  (138.67,52) -- (335.59,237.62) ;
\draw  [dash pattern={on 4.5pt off 4.5pt}]  (138.67,52) -- (315.33,238.18) ;
\draw  [dash pattern={on 4.5pt off 4.5pt}]  (138.67,52) -- (252.67,238.85) ;
\draw  [dash pattern={on 4.5pt off 4.5pt}]  (146.53,52.53) -- (335.2,165.2) ;
\draw  [dash pattern={on 4.5pt off 4.5pt}]  (138.67,52) -- (334.67,199.51) ;
\draw    (138.67,52) -- (334,179.51) ;
\draw  [dash pattern={on 4.5pt off 4.5pt}]  (138.67,52) -- (294,238.85) ;
\draw  [dash pattern={on 4.5pt off 4.5pt}]  (157.87,54.53) -- (335.2,153.2) ;
\draw  [dash pattern={on 4.5pt off 4.5pt}]  (119.61,52.51) -- (122.67,241.51) ;
\draw  [dash pattern={on 4.5pt off 4.5pt}]  (93.33,52.18) -- (101.33,238.18) ;
\draw  [dash pattern={on 4.5pt off 4.5pt}]  (127.27,50.68) -- (152,238.85) ;
\draw  [dash pattern={on 4.5pt off 4.5pt}]  (106.86,51.68) -- (115.33,236.18) ;
\draw  [dash pattern={on 4.5pt off 4.5pt}]  (123.33,52.18) -- (138,238.18) ;
\draw  [dash pattern={on 4.5pt off 4.5pt}]  (363.2,102.53) -- (335.87,108.53) ;
\draw  [dash pattern={on 4.5pt off 4.5pt}]  (415.2,52.53) -- (337.87,65.2) ;
\draw  [dash pattern={on 4.5pt off 4.5pt}]  (385.87,85.2) -- (335.2,97.2) ;
\draw  [dash pattern={on 4.5pt off 4.5pt}]  (425.87,66.53) -- (334.53,85.87) ;
\draw  [dash pattern={on 4.5pt off 4.5pt}]  (367.2,103.2) -- (411.87,237.87) ;
\draw  [dash pattern={on 4.5pt off 4.5pt}]  (343.2,129.2) -- (411.87,237.87) ;
\draw  [dash pattern={on 4.5pt off 4.5pt}]  (343.87,125.87) -- (333.87,129.2) ;
\draw  [dash pattern={on 4.5pt off 4.5pt}]  (353.87,113.87) -- (411.87,237.87) ;
\draw  [dash pattern={on 4.5pt off 4.5pt}]  (425.87,66.53) -- (411.87,237.87) ;
\draw  [dash pattern={on 4.5pt off 4.5pt}]  (410.53,73.2) -- (411.87,237.87) ;
\draw  [dash pattern={on 4.5pt off 4.5pt}]  (395.2,81.2) -- (411.87,237.87) ;
\draw  [dash pattern={on 4.5pt off 4.5pt}]  (380.53,91.87) -- (411.87,237.87) ;
\draw  [dash pattern={on 4.5pt off 4.5pt}]  (441.2,59.2) -- (411.87,237.87) ;
\draw [color={rgb, 255:red, 208; green, 2; blue, 27 }  ,draw opacity=1 ]   (335.87,142.53) .. controls (348.53,105.2) and (393.2,72.53) .. (455.87,52.53) ;
\draw  [dash pattern={on 4.5pt off 4.5pt}]  (353.87,113.87) -- (335.43,118.6) ;
\draw  [dash pattern={on 4.5pt off 4.5pt}]  (472.97,53.52) -- (486.53,236.53) ;
\draw    (472.97,53.52) -- (499.2,238.53) ;
\draw  [dash pattern={on 4.5pt off 4.5pt}]  (472.97,53.52) -- (457.87,235.87) ;
\draw    (472.97,53.52) -- (450.53,238.53) ;
\draw  [dash pattern={on 4.5pt off 4.5pt}]  (480.53,53.87) -- (509.87,236.53) ;
\draw  [dash pattern={on 4.5pt off 4.5pt}]  (465.89,54.04) -- (433.2,238.53) ;
\draw  [dash pattern={on 4.5pt off 4.5pt}]  (472.97,53.52) -- (471.2,238.53) ;
\draw  [dash pattern={on 0.84pt off 2.51pt}]  (335.87,142.53) -- (593.87,142.53) ;

\draw (128.67,28.4) node [anchor=north west][inner sep=0.75pt]  [font=\footnotesize]  {$\mathsf{L}$};
\draw (569.33,34.07) node [anchor=north west][inner sep=0.75pt]  [font=\footnotesize]  {$\omega $};
\draw (600,231.4) node [anchor=north west][inner sep=0.75pt]  [font=\footnotesize]  {$x$};
\draw (346,14.4) node [anchor=north west][inner sep=0.75pt]  [font=\footnotesize]  {$t$};
\draw (451.33,28.4) node [anchor=north west][inner sep=0.75pt]  [font=\footnotesize]  {$\mathsf{R}$};
\draw (584,120.07) node [anchor=north west][inner sep=0.75pt]  [font=\scriptsize]  {$\mathbf{\tau }[\mathsf{R} ,\ B,\ f_{r}]$};
\draw (255.33,28.4) node [anchor=north west][inner sep=0.75pt]  [font=\footnotesize]  {$A$};
\draw (376,28.73) node [anchor=north west][inner sep=0.75pt]  [font=\footnotesize]  {$B$};
\draw (161.33,24.4) node [anchor=north west][inner sep=0.75pt]  [font=\footnotesize]  {$\tilde{\mathsf{L}}$};

\end{tikzpicture}

\caption{Case 1. }
\label{case1}
\end{figure}

\begin{figure}[ht]
\centering

\tikzset{every picture/.style={line width=0.75pt}} 

\begin{tikzpicture}[x=0.75pt,y=0.75pt,yscale=-0.7,xscale=0.7]

\draw  [fill={rgb, 255:red, 74; green, 144; blue, 226 }  ,fill opacity=0.49 ] (315.2,177.53) -- (118.67,64.33) -- (118.67,64.33) -- (295.33,250.51) -- (315.67,231.67) -- cycle ;
\draw  [fill={rgb, 255:red, 208; green, 2; blue, 27 }  ,fill opacity=0.45 ] (435.87,64.87) -- (391.87,250.2) -- (315.87,154.87) -- (315.87,86.87) -- cycle ;
\draw [line width=1.5]    (315.59,249.96) -- (315.59,31.2) ;
\draw [shift={(315.59,27.2)}, rotate = 90] [fill={rgb, 255:red, 0; green, 0; blue, 0 }  ][line width=0.08]  [draw opacity=0] (6.97,-3.35) -- (0,0) -- (6.97,3.35) -- cycle    ;
\draw [line width=1.5]    (70.8,251.87) -- (570.13,250.63) ;
\draw [shift={(574.13,250.62)}, rotate = 179.86] [fill={rgb, 255:red, 0; green, 0; blue, 0 }  ][line width=0.08]  [draw opacity=0] (6.97,-3.35) -- (0,0) -- (6.97,3.35) -- cycle    ;
\draw    (68.13,63.46) -- (570.2,64.5) ;
\draw  [dash pattern={on 4.5pt off 4.5pt}]  (227.2,64.2) -- (314.53,98.2) ;
\draw  [dash pattern={on 4.5pt off 4.5pt}]  (296.83,64.29) -- (312.53,68.87) ;
\draw  [dash pattern={on 4.5pt off 4.5pt}]  (349.87,64.2) -- (312.53,68.87) ;
\draw  [dash pattern={on 4.5pt off 4.5pt}]  (273.2,64.2) -- (317.87,77.53) ;
\draw  [dash pattern={on 4.5pt off 4.5pt}]  (186.53,64.2) -- (314.53,120.2) ;
\draw  [dash pattern={on 4.5pt off 4.5pt}]  (203.2,63.53) -- (315.2,109.53) ;
\draw  [dash pattern={on 4.5pt off 4.5pt}]  (251.2,64.87) -- (312.53,87.53) ;
\draw  [dash pattern={on 4.5pt off 4.5pt}]  (315.67,231.67) -- (326,249.18) ;
\draw  [dash pattern={on 4.5pt off 4.5pt}]  (315.2,165.53) -- (374.53,248.87) ;
\draw  [dash pattern={on 4.5pt off 4.5pt}]  (317.67,215.67) -- (339.33,250.51) ;
\draw  [dash pattern={on 4.5pt off 4.5pt}]  (315.67,196.33) -- (353.33,250.67) ;
\draw  [dash pattern={on 4.5pt off 4.5pt}]  (170.53,64.2) -- (315.43,130.93) ;
\draw  [dash pattern={on 4.5pt off 4.5pt}]  (155.2,64.2) -- (313.87,141.53) ;
\draw    (118.67,64.33) -- (146.67,248.51) ;
\draw    (315.2,177.53) -- (367.2,250.53) ;
\draw    (146.53,64.2) -- (315.87,154.87) ;
\draw  [dash pattern={on 4.5pt off 4.5pt}]  (118.67,64.33) -- (252,249.18) ;
\draw  [dash pattern={on 4.5pt off 4.5pt}]  (118.67,64.33) -- (176.67,249.85) ;
\draw  [dash pattern={on 4.5pt off 4.5pt}]  (118.67,64.33) -- (194,249.18) ;
\draw  [dash pattern={on 4.5pt off 4.5pt}]  (118.67,64.33) -- (213.33,249.18) ;
\draw  [dash pattern={on 4.5pt off 4.5pt}]  (118.67,64.33) -- (162.67,251.18) ;
\draw  [dash pattern={on 4.5pt off 4.5pt}]  (118.67,64.33) -- (295.33,250.51) ;
\draw  [dash pattern={on 4.5pt off 4.5pt}]  (118.67,64.33) -- (232.67,251.18) ;
\draw  [dash pattern={on 4.5pt off 4.5pt}]  (118.67,64.33) -- (274,251.18) ;
\draw  [dash pattern={on 4.5pt off 4.5pt}]  (137.87,66.87) -- (315.2,165.53) ;
\draw  [dash pattern={on 4.5pt off 4.5pt}]  (99.61,64.85) -- (102.67,253.85) ;
\draw  [dash pattern={on 4.5pt off 4.5pt}]  (73.33,64.51) -- (81.33,250.51) ;
\draw  [dash pattern={on 4.5pt off 4.5pt}]  (107.27,63.01) -- (132,251.18) ;
\draw  [dash pattern={on 4.5pt off 4.5pt}]  (86.86,64.02) -- (95.33,248.51) ;
\draw  [dash pattern={on 4.5pt off 4.5pt}]  (103.33,64.51) -- (118,250.51) ;
\draw  [dash pattern={on 4.5pt off 4.5pt}]  (343.2,114.87) -- (315.87,120.87) ;
\draw  [dash pattern={on 4.5pt off 4.5pt}]  (395.2,64.87) -- (317.87,77.53) ;
\draw  [dash pattern={on 4.5pt off 4.5pt}]  (365.87,97.53) -- (315.2,109.53) ;
\draw  [dash pattern={on 4.5pt off 4.5pt}]  (405.87,78.87) -- (314.53,98.2) ;
\draw  [dash pattern={on 4.5pt off 4.5pt}]  (347.2,115.53) -- (391.87,250.2) ;
\draw  [dash pattern={on 4.5pt off 4.5pt}]  (323.2,141.53) -- (391.87,250.2) ;
\draw  [dash pattern={on 4.5pt off 4.5pt}]  (323.87,138.2) -- (313.87,141.53) ;
\draw  [dash pattern={on 4.5pt off 4.5pt}]  (333.87,126.2) -- (391.87,250.2) ;
\draw  [dash pattern={on 4.5pt off 4.5pt}]  (405.87,78.87) -- (391.87,250.2) ;
\draw  [dash pattern={on 4.5pt off 4.5pt}]  (390.53,85.53) -- (391.87,250.2) ;
\draw  [dash pattern={on 4.5pt off 4.5pt}]  (375.2,93.53) -- (391.87,250.2) ;
\draw  [dash pattern={on 4.5pt off 4.5pt}]  (360.53,104.2) -- (391.87,250.2) ;
\draw  [dash pattern={on 4.5pt off 4.5pt}]  (421.2,71.53) -- (391.87,250.2) ;
\draw [color={rgb, 255:red, 208; green, 2; blue, 27 }  ,draw opacity=1 ]   (315.87,154.87) .. controls (328.53,117.53) and (373.2,84.87) .. (435.87,64.87) ;
\draw  [dash pattern={on 4.5pt off 4.5pt}]  (333.87,126.2) -- (315.43,130.93) ;
\draw  [dash pattern={on 4.5pt off 4.5pt}]  (452.97,65.85) -- (466.53,248.87) ;
\draw    (452.97,65.85) -- (479.2,250.87) ;
\draw  [dash pattern={on 4.5pt off 4.5pt}]  (452.97,65.85) -- (437.87,248.2) ;
\draw    (452.97,65.85) -- (430.53,250.87) ;
\draw  [dash pattern={on 4.5pt off 4.5pt}]  (460.53,66.2) -- (489.87,248.87) ;
\draw  [dash pattern={on 4.5pt off 4.5pt}]  (445.89,66.38) -- (413.2,250.87) ;
\draw  [dash pattern={on 4.5pt off 4.5pt}]  (452.97,65.85) -- (451.2,250.87) ;
\draw  [dash pattern={on 0.84pt off 2.51pt}]  (315.87,154.87) -- (573.87,154.87) ;
\draw  [dash pattern={on 4.5pt off 4.5pt}]  (186.53,111.87) -- (316.53,187.2) ;
\draw  [dash pattern={on 4.5pt off 4.5pt}]  (276.53,183.87) -- (316.53,207.87) ;
\draw  [dash pattern={on 4.5pt off 4.5pt}]  (151.2,88.53) -- (295.33,250.51) ;
\draw  [dash pattern={on 4.5pt off 4.5pt}]  (307.87,219.87) -- (295.33,250.51) ;
\draw  [dash pattern={on 4.5pt off 4.5pt}]  (192.53,117.87) -- (295.33,250.51) ;
\draw  [dash pattern={on 4.5pt off 4.5pt}]  (279.2,187.87) -- (295.33,250.51) ;
\draw  [dash pattern={on 4.5pt off 4.5pt}]  (225.87,144.53) -- (295.33,250.51) ;
\draw  [dash pattern={on 4.5pt off 4.5pt}]  (237.2,149.87) -- (315.67,196.33) ;
\draw  [dash pattern={on 4.5pt off 4.5pt}]  (295.2,205.2) -- (295.33,250.51) ;
\draw  [dash pattern={on 4.5pt off 4.5pt}]  (255.2,166.53) -- (295.33,250.51) ;
\draw [color={rgb, 255:red, 208; green, 2; blue, 27 }  ,draw opacity=1 ]   (118.67,64.33) .. controls (181.87,111.2) and (273.87,165.2) .. (315.67,231.67) ;
\draw  [dash pattern={on 4.5pt off 4.5pt}]  (295.2,205.2) -- (315,217.67) ;
\draw  [dash pattern={on 0.84pt off 2.51pt}]  (611.87,69.2) -- (315.67,231.67) ;
\draw  [dash pattern={on 0.84pt off 2.51pt}]  (572.53,231.2) -- (315.67,231.67) ;
\draw  [dash pattern={on 4.5pt off 4.5pt}]  (316.53,241.87) -- (306,250.33) ;

\draw (108.67,40.73) node [anchor=north west][inner sep=0.75pt]  [font=\footnotesize]  {$\mathsf{L}$};
\draw (549.33,46.4) node [anchor=north west][inner sep=0.75pt]  [font=\footnotesize]  {$\omega $};
\draw (580,243.73) node [anchor=north west][inner sep=0.75pt]  [font=\footnotesize]  {$x$};
\draw (326,26.73) node [anchor=north west][inner sep=0.75pt]  [font=\footnotesize]  {$t$};
\draw (431.33,40.73) node [anchor=north west][inner sep=0.75pt]  [font=\footnotesize]  {$\mathsf{R}$};
\draw (564,132.4) node [anchor=north west][inner sep=0.75pt]  [font=\scriptsize]  {$\mathbf{\tau }[\mathsf{R} ,\ B,\ f_{r}]$};
\draw (235.33,40.73) node [anchor=north west][inner sep=0.75pt]  [font=\footnotesize]  {$A$};
\draw (356,41.07) node [anchor=north west][inner sep=0.75pt]  [font=\footnotesize]  {$B$};
\draw (141.33,36.73) node [anchor=north west][inner sep=0.75pt]  [font=\footnotesize]  {$\tilde{\mathsf{L}}$};
\draw (611.33,43.07) node [anchor=north west][inner sep=0.75pt]  [font=\scriptsize]  {$\tilde{\mathsf{R}}$};
\draw (566,210.07) node [anchor=north west][inner sep=0.75pt]  [font=\scriptsize]  {$\mathbf{\sigma }[\mathsf{L} ,\ A,\ f_{l}]$};
\draw (299.33,257.73) node [anchor=north west][inner sep=0.75pt]  [font=\footnotesize]  {$\overline{A}$};
\draw (334,258.4) node [anchor=north west][inner sep=0.75pt]  [font=\footnotesize]  {$\overline{B}$};

\end{tikzpicture}

\caption{Case 2. }
\label{case2}
\end{figure}

\begin{figure}[ht]
\centering

\tikzset{every picture/.style={line width=0.75pt}} 

\begin{tikzpicture}[x=0.75pt,y=0.75pt,yscale=-0.7,xscale=0.7]

\draw  [fill={rgb, 255:red, 208; green, 2; blue, 27 }  ,fill opacity=0.45 ] (455.87,69.87) -- (411.87,255.2) -- (335.87,159.87) -- (335.87,91.87) -- cycle ;
\draw [line width=1.5]    (335.59,254.96) -- (335.59,36.2) ;
\draw [shift={(335.59,32.2)}, rotate = 90] [fill={rgb, 255:red, 0; green, 0; blue, 0 }  ][line width=0.08]  [draw opacity=0] (6.97,-3.35) -- (0,0) -- (6.97,3.35) -- cycle    ;
\draw [line width=1.5]    (83.92,255.58) -- (583.26,254.34) ;
\draw [shift={(587.26,254.34)}, rotate = 179.86] [fill={rgb, 255:red, 0; green, 0; blue, 0 }  ][line width=0.08]  [draw opacity=0] (6.97,-3.35) -- (0,0) -- (6.97,3.35) -- cycle    ;
\draw    (88.13,68.46) -- (590.2,69.5) ;
\draw  [dash pattern={on 4.5pt off 4.5pt}]  (324.16,69.29) -- (336.53,71.87) ;
\draw  [dash pattern={on 4.5pt off 4.5pt}]  (369.87,69.2) -- (332.53,73.87) ;
\draw  [dash pattern={on 4.5pt off 4.5pt}]  (335.67,236.67) -- (349.87,254.53) ;
\draw  [dash pattern={on 4.5pt off 4.5pt}]  (335.2,170.53) -- (399.87,254.53) ;
\draw  [dash pattern={on 4.5pt off 4.5pt}]  (337.67,220.67) -- (365.2,255.87) ;
\draw  [dash pattern={on 4.5pt off 4.5pt}]  (335.67,201.33) -- (377.87,255.87) ;
\draw  [dash pattern={on 4.5pt off 4.5pt}]  (335.2,187.2) -- (390.53,254.53) ;
\draw  [dash pattern={on 4.5pt off 4.5pt}]  (363.2,119.87) -- (335.87,125.87) ;
\draw  [dash pattern={on 4.5pt off 4.5pt}]  (415.2,69.87) -- (337.87,82.53) ;
\draw  [dash pattern={on 4.5pt off 4.5pt}]  (385.87,102.53) -- (335.2,114.53) ;
\draw  [dash pattern={on 4.5pt off 4.5pt}]  (425.87,83.87) -- (334.53,103.2) ;
\draw  [dash pattern={on 4.5pt off 4.5pt}]  (367.2,120.53) -- (411.87,255.2) ;
\draw  [dash pattern={on 4.5pt off 4.5pt}]  (343.2,146.53) -- (411.87,255.2) ;
\draw  [dash pattern={on 4.5pt off 4.5pt}]  (343.87,143.2) -- (333.87,146.53) ;
\draw  [dash pattern={on 4.5pt off 4.5pt}]  (353.87,131.2) -- (411.87,255.2) ;
\draw  [dash pattern={on 4.5pt off 4.5pt}]  (425.87,83.87) -- (411.87,255.2) ;
\draw  [dash pattern={on 4.5pt off 4.5pt}]  (410.53,90.53) -- (411.87,255.2) ;
\draw  [dash pattern={on 4.5pt off 4.5pt}]  (395.2,98.53) -- (411.87,255.2) ;
\draw  [dash pattern={on 4.5pt off 4.5pt}]  (380.53,109.2) -- (411.87,255.2) ;
\draw  [dash pattern={on 4.5pt off 4.5pt}]  (441.2,76.53) -- (411.87,255.2) ;
\draw [color={rgb, 255:red, 208; green, 2; blue, 27 }  ,draw opacity=1 ]   (335.87,159.87) .. controls (348.53,122.53) and (393.2,89.87) .. (455.87,69.87) ;
\draw  [dash pattern={on 4.5pt off 4.5pt}]  (353.87,131.2) -- (335.43,135.93) ;
\draw  [dash pattern={on 4.5pt off 4.5pt}]  (472.97,70.85) -- (486.53,253.87) ;
\draw    (472.97,70.85) -- (499.2,253.87) ;
\draw  [dash pattern={on 4.5pt off 4.5pt}]  (472.97,70.85) -- (457.87,253.2) ;
\draw    (472.97,70.85) -- (449.87,253.87) ;
\draw  [dash pattern={on 4.5pt off 4.5pt}]  (480.53,71.2) -- (509.87,253.87) ;
\draw  [dash pattern={on 4.5pt off 4.5pt}]  (465.89,71.38) -- (433.2,255.87) ;
\draw  [dash pattern={on 4.5pt off 4.5pt}]  (472.97,70.85) -- (471.2,253.2) ;
\draw  [dash pattern={on 0.84pt off 2.51pt}]  (335.87,159.87) -- (593.87,159.87) ;
\draw  [dash pattern={on 0.84pt off 2.51pt}]  (578,69.5) -- (334.33,200.35) ;
\draw  [dash pattern={on 0.84pt off 2.51pt}]  (591.2,199.89) -- (334.33,200.35) ;
\draw  [fill={rgb, 255:red, 74; green, 144; blue, 226 }  ,fill opacity=0.64 ] (183.67,68.69) -- (222,256.02) -- (334.33,200.35) -- (335.87,118.53) -- cycle ;
\draw  [dash pattern={on 4.5pt off 4.5pt}]  (217.67,89.02) -- (221.67,255.69) ;
\draw  [dash pattern={on 4.5pt off 4.5pt}]  (197.67,77.02) -- (221.67,255.69) ;
\draw  [dash pattern={on 4.5pt off 4.5pt}]  (278.33,133.69) -- (221.67,255.69) ;
\draw  [dash pattern={on 4.5pt off 4.5pt}]  (183.33,68.35) -- (189,255.02) ;
\draw  [dash pattern={on 4.5pt off 4.5pt}]  (263.67,119.69) -- (221.67,255.69) ;
\draw  [dash pattern={on 4.5pt off 4.5pt}]  (183.33,68.35) -- (208.33,255.02) ;
\draw  [dash pattern={on 4.5pt off 4.5pt}]  (183.33,68.35) -- (171,255.02) ;
\draw  [dash pattern={on 4.5pt off 4.5pt}]  (183.33,68.35) -- (152.53,256.53) ;
\draw  [dash pattern={on 4.5pt off 4.5pt}]  (249.67,108.35) -- (221.67,255.69) ;
\draw  [dash pattern={on 4.5pt off 4.5pt}]  (183.33,68.35) -- (134.53,254.53) ;
\draw  [dash pattern={on 4.5pt off 4.5pt}]  (234.33,97.69) -- (221.67,255.69) ;
\draw    (183.33,68.35) -- (221.67,255.69) ;
\draw    (183.33,68.35) -- (121.67,254.35) ;
\draw    (334.33,200.35) -- (221.67,255.69) ;
\draw [color={rgb, 255:red, 208; green, 2; blue, 27 }  ,draw opacity=1 ]   (183.33,68.35) .. controls (239.67,101.02) and (243,101.69) .. (271,125.02) .. controls (299,148.35) and (317.67,171.02) .. (334.33,200.35) ;
\draw  [dash pattern={on 4.5pt off 4.5pt}]  (276.33,129.02) -- (334,150.69) ;
\draw  [dash pattern={on 4.5pt off 4.5pt}]  (249.67,108.35) -- (331.67,138.35) ;
\draw  [dash pattern={on 4.5pt off 4.5pt}]  (319,176.35) -- (336.33,182.35) ;
\draw  [dash pattern={on 4.5pt off 4.5pt}]  (291.67,143.02) -- (221.67,255.69) ;
\draw  [dash pattern={on 4.5pt off 4.5pt}]  (217.67,89.02) -- (334,127.69) ;
\draw  [dash pattern={on 4.5pt off 4.5pt}]  (309,161.69) -- (335.67,171.69) ;
\draw  [dash pattern={on 4.5pt off 4.5pt}]  (295.67,146.35) -- (333,160.35) ;
\draw  [dash pattern={on 4.5pt off 4.5pt}]  (303.67,157.69) -- (221.67,255.69) ;
\draw  [dash pattern={on 4.5pt off 4.5pt}]  (315,172.35) -- (221.67,255.69) ;
\draw  [dash pattern={on 4.5pt off 4.5pt}]  (325,185.02) -- (221.67,255.69) ;
\draw  [dash pattern={on 4.5pt off 4.5pt}]  (320.53,255.2) -- (336.53,246.87) ;
\draw  [dash pattern={on 4.5pt off 4.5pt}]  (288.67,256) -- (334.33,235.69) ;
\draw  [dash pattern={on 4.5pt off 4.5pt}]  (245.62,254.97) -- (333.67,212.35) ;
\draw  [dash pattern={on 4.5pt off 4.5pt}]  (267.07,255.26) -- (332.33,224.35) ;
\draw  [dash pattern={on 4.5pt off 4.5pt}]  (296.72,68.92) -- (335.2,79.2) ;
\draw  [dash pattern={on 4.5pt off 4.5pt}]  (265.85,68.74) -- (335.2,87.2) ;
\draw  [dash pattern={on 4.5pt off 4.5pt}]  (207.67,68.35) -- (335.2,107.2) ;
\draw  [dash pattern={on 4.5pt off 4.5pt}]  (241.2,69.87) -- (333.67,97.69) ;
\draw  [dash pattern={on 4.5pt off 4.5pt}]  (156.2,68.32) -- (99.2,253.32) ;
\draw  [dash pattern={on 4.5pt off 4.5pt}]  (137.2,68.32) -- (94.53,253.2) ;
\draw  [dash pattern={on 4.5pt off 4.5pt}]  (173.17,70.13) -- (106.2,253.32) ;
\draw  [dash pattern={on 0.84pt off 2.51pt}]  (110.53,72.53) -- (335.87,159.87) ;

\draw (177.33,49.73) node [anchor=north west][inner sep=0.75pt]  [font=\footnotesize]  {$\mathsf{L}$};
\draw (517.33,51.4) node [anchor=north west][inner sep=0.75pt]  [font=\footnotesize]  {$\omega $};
\draw (600,248.73) node [anchor=north west][inner sep=0.75pt]  [font=\footnotesize]  {$x$};
\draw (346,31.73) node [anchor=north west][inner sep=0.75pt]  [font=\footnotesize]  {$t$};
\draw (451.33,45.73) node [anchor=north west][inner sep=0.75pt]  [font=\footnotesize]  {$\mathsf{R}$};
\draw (584,137.4) node [anchor=north west][inner sep=0.75pt]  [font=\scriptsize]  {$\mathbf{\tau }[\mathsf{R} ,\ B,\ f_{r}]$};
\draw (255.33,45.73) node [anchor=north west][inner sep=0.75pt]  [font=\footnotesize]  {$A$};
\draw (376,46.07) node [anchor=north west][inner sep=0.75pt]  [font=\footnotesize]  {$B$};
\draw (100.67,45.07) node [anchor=north west][inner sep=0.75pt]  [font=\footnotesize]  {$\tilde{\mathsf{L}}$};
\draw (572.33,49.07) node [anchor=north west][inner sep=0.75pt]  [font=\scriptsize]  {$\tilde{\mathsf{R}}$};
\draw (587,184.07) node [anchor=north west][inner sep=0.75pt]  [font=\scriptsize]  {$\mathbf{\sigma }[\mathsf{L} ,\ A,\ f_{l}]$};
\draw (285.33,262.73) node [anchor=north west][inner sep=0.75pt]  [font=\footnotesize]  {$\overline{A}$};
\draw (354,263.4) node [anchor=north west][inner sep=0.75pt]  [font=\footnotesize]  {$\overline{B}$};

\end{tikzpicture}

\caption{Case 3. }
\label{case3}
\end{figure}

\begin{figure}
    \centering

\tikzset{every picture/.style={line width=0.75pt}} 

\begin{tikzpicture}[x=0.75pt,y=0.75pt,yscale=-0.7,xscale=0.7]

\draw    (91,262.73) -- (561.86,263.72) ;
\draw [shift={(564.86,263.72)}, rotate = 180.12] [fill={rgb, 255:red, 0; green, 0; blue, 0 }  ][line width=0.08]  [draw opacity=0] (5.36,-2.57) -- (0,0) -- (5.36,2.57) -- (3.56,0) -- cycle    ;
\draw    (364.12,260.74) -- (364.12,43.56) ;
\draw [shift={(364.12,40.56)}, rotate = 90] [fill={rgb, 255:red, 0; green, 0; blue, 0 }  ][line width=0.08]  [draw opacity=0] (5.36,-2.57) -- (0,0) -- (5.36,2.57) -- (3.56,0) -- cycle    ;
\draw    (81.07,83.56) -- (553.67,82.57) ;
\draw    (127.07,84.9) -- (204,261.93) ;
\draw [color={rgb, 255:red, 0; green, 0; blue, 0 }  ,draw opacity=1 ]   (192.77,84.06) -- (364.12,260.74) ;
\draw    (364.12,260.74) -- (515.76,83.56) ;
\draw  [dash pattern={on 4.5pt off 4.5pt}]  (272.38,86.05) -- (365.39,187.88) ;
\draw  [dash pattern={on 4.5pt off 4.5pt}]  (290.07,84.06) -- (366.65,171.95) ;
\draw  [dash pattern={on 4.5pt off 4.5pt}]  (249.64,85.55) -- (362.86,203.8) ;
\draw  [dash pattern={on 4.5pt off 4.5pt}]  (325.45,83.07) -- (365.2,129.57) ;
\draw  [dash pattern={on 4.5pt off 4.5pt}]  (306.5,83.07) -- (362.86,148.06) ;
\draw  [dash pattern={on 4.5pt off 4.5pt}]  (341.88,85.06) -- (365.39,112.23) ;
\draw  [dash pattern={on 4.5pt off 4.5pt}]  (460.66,85.06) -- (362.86,203.8) ;
\draw  [dash pattern={on 4.5pt off 4.5pt}]  (442.47,85.35) -- (365.39,187.88) ;
\draw  [dash pattern={on 4.5pt off 4.5pt}]  (426.04,83.36) -- (366.65,165.98) ;
\draw  [dash pattern={on 4.5pt off 4.5pt}]  (407.09,84.36) -- (362.86,148.06) ;
\draw  [dash pattern={on 4.5pt off 4.5pt}]  (376,85.55) -- (367.91,99.29) ;
\draw  [dash pattern={on 4.5pt off 4.5pt}]  (389.39,86.35) -- (366.65,122.18) ;
\draw  [dash pattern={on 4.5pt off 4.5pt}]  (356.54,83.56) -- (367.91,99.29) ;
\draw  [dash pattern={on 4.5pt off 4.5pt}]  (515.76,83.56) -- (381.81,262.73) ;
\draw  [dash pattern={on 4.5pt off 4.5pt}]  (515.76,83.56) -- (436.15,261.73) ;
\draw    (515.76,83.56) -- (467.74,261.73) ;
\draw  [dash pattern={on 4.5pt off 4.5pt}]  (515.76,83.56) -- (398.24,263.72) ;
\draw  [dash pattern={on 4.5pt off 4.5pt}]  (175.89,84.36) -- (355.87,263.2) ;
\draw  [dash pattern={on 4.5pt off 4.5pt}]  (229.42,85.35) -- (364.12,226.9) ;
\draw  [dash pattern={on 4.5pt off 4.5pt}]  (209.2,83.36) -- (362.86,244.81) ;
\draw  [dash pattern={on 4.5pt off 4.5pt}]  (515.76,83.56) -- (451.31,262.73) ;
\draw  [dash pattern={on 4.5pt off 4.5pt}]  (478.36,86.35) -- (364.12,226.9) ;
\draw  [dash pattern={on 4.5pt off 4.5pt}]  (494.78,84.86) -- (362.86,244.81) ;
\draw  [dash pattern={on 4.5pt off 4.5pt}]  (515.76,83.56) -- (417.19,262.73) ;
\draw    (127.07,84.9) -- (259.2,262.53) ;
\draw  [dash pattern={on 4.5pt off 4.5pt}]  (112.53,83.87) -- (165.2,262.53) ;
\draw  [dash pattern={on 4.5pt off 4.5pt}]  (122.4,84.23) -- (181.87,261.87) ;
\draw  [dash pattern={on 4.5pt off 4.5pt}]  (127.07,84.9) -- (212.53,259.87) ;
\draw  [dash pattern={on 4.5pt off 4.5pt}]  (127.07,84.9) -- (225.87,261.2) ;
\draw  [dash pattern={on 4.5pt off 4.5pt}]  (105.87,84.53) -- (151.2,261.87) ;
\draw  [dash pattern={on 4.5pt off 4.5pt}]  (127.07,84.9) -- (241.87,261.2) ;
\draw  [dash pattern={on 4.5pt off 4.5pt}]  (170.8,85.5) -- (333.87,260.53) ;
\draw  [dash pattern={on 4.5pt off 4.5pt}]  (154.53,87.87) -- (293.87,262.53) ;
\draw  [dash pattern={on 4.5pt off 4.5pt}]  (141.3,84.26) -- (279.2,261.87) ;
\draw  [dash pattern={on 4.5pt off 4.5pt}]  (167.87,87.2) -- (311.2,262.53) ;
\draw  [dash pattern={on 4.5pt off 4.5pt}]  (132.63,84.93) -- (265.2,261.2) ;
\draw  [dash pattern={on 4.5pt off 4.5pt}]  (533.33,82.73) -- (511.87,262.53) ;
\draw  [dash pattern={on 4.5pt off 4.5pt}]  (523.33,83.4) -- (486.53,262.53) ;

\draw (512.79,65.82) node [anchor=north west][inner sep=0.75pt]  [font=\scriptsize]  {$\mathsf{R}$};
\draw (297.97,66.32) node [anchor=north west][inner sep=0.75pt]  [font=\scriptsize]  {$A$};
\draw (411.7,65.33) node [anchor=north west][inner sep=0.75pt]  [font=\scriptsize]  {$B$};
\draw (189.3,67.32) node [anchor=north west][inner sep=0.75pt]  [font=\scriptsize]  {$\mathsf{L}$};
\draw (375.05,29.49) node [anchor=north west][inner sep=0.75pt]  [font=\scriptsize]  {$t$};
\draw (576.73,265.89) node [anchor=north west][inner sep=0.75pt]  [font=\scriptsize]  {$x$};
\draw (89.53,67.66) node [anchor=north west][inner sep=0.75pt]  [font=\scriptsize]  {$\omega $};

\end{tikzpicture}

    \caption{Case 4.}
    \label{fig:case4}
\end{figure}

\begin{remark}
\label{eq-lax-ineq-3}
    Notice that conditions~\eqref{eq:2bl-r1}, \eqref{eq:2b-lr2} imply $\omega(\ms R-)\geq B$.
    On the other hand, if $\ms R<T\cdot f'_r(B)$,
    by virtue of~\eqref{eq:2a}, 
    and because of~\eqref{eq:uv-ba-ineq}, we have 
    $\omega(\ms R+)\leq B$.
    Hence, because of~\eqref{eq:3b-lr2}, it follows that
    the inequality $\omega(\ms R-)\geq \omega(\ms R+)$ is always
    satisfied.
    With similar arguments we  deduce that also 
    the inequality 
    $\omega(\ms L-)\geq \omega(\ms L+)$ is always verified.
\end{remark} 
\begin{remark}\label{rem:threecases}
If 
$\ms R[\omega, f_r]\in\,]0,T\cdot f'_r(B)[$\,,
applying~\eqref{eq:lax-rfs-cond-3}
with $f_r$ in place of $f$ and $\ms R=\ms R[\omega, f_r]$, 
and recalling~\eqref{eq:Rtildedef}, 
we derive
\begin{equation}
\label{eq:lax-rfs-cond-5}
    \frac{\ms R[\omega, f_r]}{f'_r(B)}<T-\bs\tau\big[\ms R[\omega, f_r], B, f_r\big]
    =\frac{\widetilde{\ms L}[\omega,f_l,f_r,A,B]}{f'_l(A)}\,.
\end{equation}
Similarly, if $\ms L[\omega, f_l]\in\,]T\cdot f'_l(A),0[$\,, applying~\eqref{eq:lax-rfs-cond-4}
with $f_l$ in place of $f$ and $\ms L=\ms L[\omega, f_l]$, 
and recalling that $f'_l(A)<0$
we find
\begin{equation}
\label{eq:lax-rfs-cond-6}
    \frac{\ms L[\omega, f_l]}{f'_l(A)}<T-\bs\sigma\big[\ms L[\omega, f_l], A, f_l\big]\,.
\end{equation}
Hence, if $\widetilde 
            {\ms L}[\omega,f_l,f_r,A,B]
            \geq \ms L[\omega, f_l]$, combining~\eqref{eq:lax-rfs-cond-5},
            \eqref{eq:lax-rfs-cond-6}, we deduce
\begin{equation}
    \frac{\ms R[\omega, f_r]}{f'_r(B)}<T-\bs\sigma\big[\ms L[\omega, f_l], A, f_l\big]\,,
\end{equation}            
which, in turn, by~\eqref{eq:Rtildedef} yields
\begin{equation}
    \ms R[\omega, f_r]<\widetilde 
            {\ms R}[\omega,f_l,f_r,A,B]\,.
\end{equation}
With entirely similar arguments one can show that, if $\widetilde 
            {\ms R}[\omega,f_l,f_r,A,B]
            \leq \ms R[\omega, f_r]$, 
            then one has
\begin{equation}
   \ms L[\omega, f_l]>\widetilde 
            {\ms L}[\omega,f_l,f_r,A,B]\,.
\end{equation} 
Therefore, when $\ms L[\omega, f_l]\in\,]T\cdot f'_l(A),0[$\,,\, and $\ms R[\omega, f_r]\in\,]0,T\cdot f'_r(B)[$\,,
we have
    \begin{equation}
    \label{eq:LR-prop26}
         \begin{aligned}
             \widetilde 
             {\ms L}[\omega,f_l,f_r,A,B]
             \geq \ms L[\omega, f_l]
             \quad &\Longrightarrow \quad \widetilde {\ms R}[\omega,f_l,f_r,A,B] > \ms R[\omega, f_r]\,,
             \\
             \noalign{\smallskip}
             \widetilde 
             {\ms R}[\omega,f_l,f_r,A,B]
             \leq \ms R[\omega, f_r]
             \quad &\Longrightarrow \quad \widetilde {\ms L}[\omega,f_l,f_r,A,B] < \ms L[\omega, f_l]\,.
         \end{aligned}
     \end{equation}
These implications, in particular, show that it can never occur the case where
\begin{equation}
    \widetilde 
             {\ms L}[\omega,f_l,f_r,A,B]
             \geq \ms L[\omega, f_l]
             \qquad\text{and}\qquad
             \widetilde {\ms R}[\omega,f_l,f_r,A,B]
             \leq \ms R[\omega, f_r]\,.
\end{equation}
    \end{remark}
    \begin{remark}
    \label{rem:cases-thm-4.1}
Notice that by condition~\eqref{eq:2bl-r1} in Theorem \ref{thm:attprofiles},
and because of~\eqref{eq:Rtildedef},
it follows that
if $\ms L[\omega, f_l]\in\,]T\cdot f'_l(A),0[\,$, and 
$\ms R[\omega, f_r]\leq \widetilde {\ms R}[\omega,f_l,f_r,A,B]$,
then one has 
$\ms R[\omega, f_r]< T\cdot f^{\prime}_r(B)$.
Therefore, we have
\begin{equation}
\label{eq:LRtR-impl1}
 \Big[\ms L[\omega, f_l]\in\,]T\cdot f'_l(A),0[\,\quad\text{and}\quad
 \ms R[\omega, f_r]\geq  T\cdot f^{\prime}_r(B)\Big]
 \quad
 \Longrightarrow\quad
 \ms R[\omega, f_r]> \widetilde {\ms R}[\omega,f_l,f_r,A,B]\,.
\end{equation}
Similarly, one can show that,
by~\eqref{eq:Rtildedef}, \eqref{eq:2bl-r1}, we have
\begin{equation}
\label{eq:LRtR-impl2}
 \Big[\ms R[\omega, f_r]\in\,]0, T\cdot f'_r(B)[\,\quad\text{and}\quad
 \ms L[\omega, f_l]\leq  T\cdot f^{\prime}_l(A)\Big]
 \quad
 \Longrightarrow\quad
 \ms L[\omega, f_l]< \widetilde {\ms L}[\omega,f_l,f_r,A,B]\,.
\end{equation}
Then, relying on~\eqref{eq:LR-prop26},
\eqref{eq:LRtR-impl1}, \eqref{eq:LRtR-impl2}, we deduce 
that, for non critical connections, 
we can distinguish 
six cases 
of pointwise constraints prescribed by condition (ii)
of Theorem~\ref{thm:attprofiles},
which depend on the reciprocal positions of the points
$\ms L=\ms L[\omega, f_l]$,  $\ms R=\ms R[\omega, f_r]$,
and
$\widetilde{\ms L}=\widetilde 
             {\ms L}[\omega,f_l,f_r,A,B]$, 
$\widetilde{\ms R}=\widetilde 
             {\ms R}[\omega,f_l,f_r,A,B]$: 

\noindent
\textsc{Case 1:}
If $\ms L \leq T\cdot f_l^{\prime}(A)<0$,\  $0 <\ms R< T\cdot f^{\prime}_r(B)$\, (Figure \ref{case1}), then    $\widetilde{\ms L} > \ms L$, and 
    it holds true 
    \begin{equation}
    \label{eq:case-i-A}
    \omega(\ms L-)\geq \omega(\ms L+)\,,\qquad
    \omega(x)\leq A 
        \quad\ \forall~x\in\,]{\ms L}, \widetilde{\ms L}\,[\,,
        \qquad\ \ \omega(\,\widetilde{\ms L}-)=A,\qquad \ \ 
     \omega(x)=A\quad\ \forall~x\in\,]\widetilde{\ms L}, 0[\,,        
    \end{equation}
    \begin{equation}    
    \label{eq:case-i-B}
        \omega(x)=B\ \quad\ \forall~x\in\,]0, {\ms R}[\,,
        \qquad \ \ 
  \omega(\ms R+) \leq \bs u[\ms R, B, f_r]\leq B\,;
    \end{equation}
\textsc{Case 2:} If 
    $T\cdot f_l^{\prime}(A)< \ms L<0$,\ $0< \ms R< T\cdot f^{\prime}_r(B)$, and $\widetilde{\ms L} > \ms L$, \ $\widetilde{\ms R }> \ms R$ \ (Figure \ref{case2}), then
    it holds true~\eqref{eq:case-i-B} and
    \begin{equation}
    \label{eq:case-ii-A}
    \omega(\ms L-) \geq \bs v[\ms L, A, f_l]\geq A\,,
    \quad\ \ \omega(x)\leq A 
        \quad \forall~x\in\,]{\ms L}, \widetilde{\ms L}\,[\,,
        \quad\ \ \omega(\,\widetilde{\ms L}-)=A,\quad \ \
     \omega(x)=A\quad \forall~x\in\,]\widetilde{\ms L}, 0[\,;      
    \end{equation}
\\
the symmetric ones:
\smallskip

\noindent\textsc{Case 1b:}
If $T\cdot f_l^{\prime}(A)< \ms L<0$, $0<T\cdot f^{\prime}_r(B)\leq \ms R$, then $\widetilde{\ms R }< \ms R$ and it holds true that
\begin{equation}    
    \label{eq:case-iprime-A}
        \omega(x)=A\qquad \forall~x\in\,]{\ms L}, 0[\,,\qquad\ \omega(\ms L-) \geq \bs v[\ms L, A, f_l]\geq A\,,
    \end{equation}
    \begin{equation}
    \label{eq:case-iprime-B}
    \omega(x)= B 
        \quad\ \forall~x\in\,]0, \widetilde{\ms R}\,[\,,
        \qquad\ \ \omega(\,\widetilde{\ms R}+)=B,\qquad \ \ 
     \omega(x)\geq B\quad\ \forall~x\in\,]\widetilde{\ms R}, \ms R[\,;        
    \end{equation}
\noindent\textsc{Case 2b:}
If    $T\cdot f_l^{\prime}(A)< \ms L<0$,\ $0< \ms R< T\cdot f^{\prime}_r(B)$, and $\widetilde{\ms L}< \ms L$, \ $\widetilde{\ms R }< \ms R$, then it holds true~\eqref{eq:case-iprime-A} and    
    
\begin{equation}
    \label{eq:case-iiprime-B}
    \omega(x)= B 
        \quad \forall~x\in\,]0, \widetilde{\ms R}\,[\,,
        \quad\ \ \omega(\,\widetilde{\ms R}+)=B,\quad \ \
     \omega(x)\geq B\quad \forall~x\in\,]\widetilde{\ms R}, \ms R[\,
     \quad \ \
  \omega(\ms R+) \leq \bs u[\ms R, B, f_r]\leq B\,;
     \end{equation}
\noindent
and the remaining ones:
\smallskip

\noindent\textsc{Case 3:}
If 
    $T\cdot f_l^{\prime}(A)<\ms L <0$, $0<\ms R < T\cdot f_r^{\prime}(B)$, and $\widetilde{\ms L} \leq \ms L$, $\widetilde{\ms R} \geq \ms R$ (Figure \ref{case3}),
    then it holds true~\eqref{eq:case-i-B}, \eqref{eq:case-iprime-A};
\smallskip

\noindent\textsc{Case 4:}
If $\ms L \leq T\cdot f_l^{\prime}(A)<0$ and $\ms R \geq T \cdot f_r^{\prime}(B)>0$
    (Figure~\ref{fig:case4}),
    then it holds true
  \begin{equation}    
    \label{eq:case-iv}
    \begin{aligned}
        &\omega(x)=A\qquad \forall~x\in\,]{\ms L}, 0[\,,
    \qquad\quad
        \omega(\ms L-)\geq \omega(\ms L+)\,,
        \\
        &\omega(x)=B\ \quad\ \forall~x\in\,]0, {\ms R}[\,,
    \qquad\quad
        \omega(\ms R-)\geq \omega(\ms R+)\,.
    \end{aligned}
        \end{equation}

\smallskip

\noindent
The six cases are depicted in Figure~\ref{fig:cases}. One can regard the intervals $]T\!\cdot\! f_l^{\prime}(A), 0[$
and $]0, T\!\cdot\! f_r^{\prime}(B)[$   as ``{\it active zones}" 
for the presence of shocks in an $AB$-entropy solution 
that attains $\omega$
at time $T$: as soon as $\ms L$
belongs to
$]T\!\cdot\!f_l^{\prime}(A), 0[$\, or $\ms R$ 
belongs to  $]0, T\!\cdot\! f_r^{\prime}(B)[$\,, 
{  it is needed
a shock located in $\{x< 0\}$
or in $\{x> 0\}$, respectively,
}
in order to produce the discontinuity occurring in $\omega$ at $\ms L$ or $\ms R$.
\end{remark}

\begin{remark}
\label{rem:a-g}
When the connection  is not critical and
$\ms L\doteq \ms L[\omega, f_l] < 0$, $\ms R\doteq \ms R[\omega, f_r] > 0$,
the analysis of 
attainable profiles
$\omega\in \mc A^{AB}(T)$
pursued in~\cite{adimurthi2020exact}
catches
only the profiles described in Cases 3 and~4 of Remark~\ref{rem:cases-thm-4.1}.
In fact, 
the characterization of 
$\mc A^{AB}(T)$ established in
~\cite[Theorem~6.1]{adimurthi2020exact}
requires that all profiles
$\omega\in \mc A^{AB}(T)$ satisfy the equalities
\begin{equation*}
    \omega(x)=A\qquad\forall~x\in \,]{\ms L}, 0[\,,
    \qquad\qquad
     \omega(x)=B\qquad\forall~x\in \,]0, \ms R[\,.
\end{equation*}
Therefore, such a characterization 
in particular excludes all attainable profiles $\omega$ that either satisfy conditions~\eqref{eq:case-i-A} or \eqref{eq:case-ii-A}, of Cases~1 and~2,
with
\begin{equation*}
    \omega(x){ ~<~} A\qquad\text{{ for some}}\ \ x\in \,]{\ms L}, \widetilde{\ms L}[\,,
\end{equation*}
or satisfy conditions~\eqref{eq:case-iprime-B}, \eqref{eq:case-iiprime-B},
of Cases~1B and~2B,
with
\begin{equation*}
    \omega(x){ ~>~} B\qquad\text{{ for some}}\ \  x\in \,] \widetilde{\ms R}, {\ms R}[\,.
\end{equation*}
%
\end{remark}
\medskip

\begin{figure}
    \centering

\tikzset{every picture/.style={line width=0.75pt}} 

\begin{tikzpicture}[x=0.75pt,y=0.75pt,yscale=-1,xscale=1]

\draw    (199.67,51) -- (199.67,206.67) ;
\draw [shift={(199.67,209.67)}, rotate = 270] [fill={rgb, 255:red, 0; green, 0; blue, 0 }  ][line width=0.08]  [draw opacity=0] (8.93,-4.29) -- (0,0) -- (8.93,4.29) -- cycle    ;
\draw    (199.67,51) -- (366.67,51) ;
\draw[->] [shift={(369.67,51)}, rotate = 180] [fill={rgb, 255:red, 0; green, 0; blue, 0 }  ][line width=0.08] [draw opacity=0] (8.93,-4.29) -- (0,0) -- (8.93,4.29) -- cycle  ;
\draw  [dash pattern={on 0.84pt off 2.51pt}]  (279.67,52) -- (280.33,207.67) ;
\draw  [dash pattern={on 0.84pt off 2.51pt}]  (200.5,131) -- (370.33,131) ;

\draw (262,34.07) node [anchor=north west][inner sep=0.75pt]  [font=\scriptsize]  {$Tf_{r}^{\prime} ( B)$};
\draw (155.33,124.07) node [anchor=north west][inner sep=0.75pt]  [font=\scriptsize]  {$Tf_{l}^{\prime}( A)$};
\draw (205,77.4) node [anchor=north west][inner sep=0.75pt]  
[font=\footnotesize]  {$ \begin{array}{l}
(2) ,\   (2\textsc{B}) \\
\noalign{\smallskip}
\ \ \ \ (3)
\end{array}$};
\draw (316,85.07) node [anchor=north west][inner sep=0.75pt]  
[font=\footnotesize]  {$(1\textsc{B})$};
\draw (234.67,165.4) node [anchor=north west][inner sep=0.75pt]  
[font=\footnotesize]  {$(1)$};
\draw (310.67,162.4) node [anchor=north west][inner sep=0.75pt]  
[font=\footnotesize]  {$(4)$};
\draw (180.67,204.4) node [anchor=north west][inner sep=0.75pt]  [font=\scriptsize]  {$\mathsf{L}$};
\draw (367.33,34.73) node [anchor=north west][inner sep=0.75pt]  [font=\footnotesize]  {$\mathsf{R}$};
\draw (186,35.4) node [anchor=north west][inner sep=0.75pt]  [font=\footnotesize]  {$0$};

\end{tikzpicture}

    \caption{The different cases of Remark \ref{rem:threecases}.}
    \label{fig:cases}
\end{figure}

\begin{remark}
    \label{rem:toildeLR-limitingcase}
    Notice that, if $A=\theta_l$, or $\ms R= 0$,
by definition~\eqref{eq:Ltildedef},
and because of~\eqref{eq:tildeLR-bis}, it follows that
$\widetilde{\ms L}=0$.
Similarly, if $B=\theta_r$, or $\ms L= 0$,
we have $\widetilde{\ms R}=0$.
Thus, in the case of critical connections, or whenever $\ms L=0$
or $\ms R=0$ (for critical and non critical connections),
the characterization 
of the profiles
$\omega\in \mc A^{AB}(T)\cap \msc A^{\ms L, \ms R}$ will not involve the constants~$\widetilde{\ms L}$,~$\widetilde{\ms R}$.
\end{remark}


\begin{thm}\label{thm:attprofilescrit}
    In the same setting of Theorem~\ref{thm:attprofiles}, 
    let $\omega$ be an element of the
    set $ \msc A^{\ms L, \ms R}$  
    in~\eqref{eq:ALR-def1},
    with $\ms L < 0$, $\ms R > 0$,
and assume that $(A, B ) = (\theta_l, B)$ (connection critical from the left). Then, 
$\omega \in \mc A^{[AB]}(T)$ if and only if $B \neq \theta_r$, the limits $\omega(0\pm)$ exist and there hold:

\begin{enumerate}
[leftmargin=20pt]
     \item[(i)]      
    the following Ole\v{\i}nik-type inequalities are satisfied 
    \begin{equation}\label{eq:1acrit}
    \begin{aligned}
    D^+ \omega(x) &\leq \frac{1}{T \cdot f_l^{\second}(\omega(x))} \qquad \forall \; x \in\,]\!-\!\infty,\;  \ms L[\,,
    \\
    \noalign{\smallskip}
       D^+ \omega(x) &\leq \frac{1}{T \cdot f_r^{\second}(\omega(x))} \qquad \forall \; x \in \,]\ms R, +\infty[\,.
        \end{aligned}
    \end{equation}
    Moreover, letting $g$ be the function in~\eqref{eq:ghdef}, then
      one has
    \begin{equation}
    \label{eq:1b1crit}
        D^+\omega(x) \leq g[\omega, f_l, f_r](x) \qquad \forall \; x \in\,]\ms L,\,0\,[\,.
    \end{equation}

\item[(ii)] 
letting $\bs u[\ms R, B, f_r]$, $\bs\tau[\ms R, B, f_r]$,  be 
constants defined as
in~\eqref{eq:urbf-def}, \eqref{eq:tau-rbf-def}, respectively,
the following pointwise state constraints are satisfied
 \begin{equation}\label{eq:critcond-2}
    (f'_l)^{-1}\bigg(\frac{x}{T-\bs\tau[\ms R, B, f_r]}\bigg)\leq \omega(x)<\theta_l,
    \qquad\forall~x\in\,]\ms L, 0[\,,
    \end{equation}
    \smallskip
\begin{equation}\label{eq:critcond}
        \omega(\ms L-)\geq \omega(\ms L+), \qquad\qquad
        \omega(0-) = \theta_l\,,
    \end{equation}
    \smallskip
        \begin{equation}
\label{eq:2bl-r1crit}
\begin{aligned}
\omega(x) = B\quad\ \forall~x\in\,]0, \ms R[\,,
\qquad\quad \ms R \in \,]0, T\cdot f_r^\prime(B)[\,,
\end{aligned}
\end{equation}
\smallskip
\smallskip
\begin{equation}
\label{eq:2acrit}
  \omega(\ms R+) \leq \bs u[\ms R, B, f_r] \leq  \omega(\ms R-)\,.
\qquad
\end{equation}
\smallskip
\end{enumerate}

\noindent
Symmetrycally, assume that $(A, B ) = (A, \theta_r)$ (connection critical from the right). Then, 
$\omega \in \mc A^{[AB]}(T)$ if and only if $A \neq \theta_l$,  the limits $\omega(0\pm)$ exist and there hold:

\begin{enumerate}
[leftmargin=20pt]
     \item[(i)']      
    the following Ole\v{\i}nik-type inequalities are satisfied 
    \begin{equation}\label{eq:1acrit1}
    \begin{aligned}
    D^+ \omega(x) &\leq \frac{1}{T \cdot f_l^{\second}(\omega(x))} \qquad \forall \; x \in\,]\!-\!\infty,\;  \ms L[\,,
    \\
    \noalign{\smallskip}
       D^+ \omega(x) &\leq \frac{1}{T \cdot f_r^{\second}(\omega(x))} \qquad \forall \; x \in \,]\ms R, +\infty[\,.
        \end{aligned}
    \end{equation}
    Moreover, letting $h$ be the function in~\eqref{eq:ghdef},
      then one has
    \begin{equation}
    \label{eq:1b1crit1}
        D^+\omega(x) \leq h[\omega, f_l, f_r](x) \qquad \forall \; x \in\,]0,\ms R[.
    \end{equation}

\item[(ii)'] letting $\bs v[\ms R, B, f_r]$,  $\bs \sigma[\ms L, A, f_l]$, be 
constants defined as
in~\eqref{eq:vlaf-def}, \eqref{eq:sigma-laf-def}, respectively,
the following pointwise state constraints are satisfied
\begin{equation}
\label{eq:2acrit1}
\begin{aligned}
  \omega(\ms L-) \geq \bs v[\ms L, A, f_l] \geq  \omega(\ms L+)\,,
\end{aligned}
\qquad
\end{equation}
\smallskip
\begin{equation}
\label{eq:2bl-r1crit1}
\begin{aligned}
\omega(x) = A\quad\ \forall~x\in\,]\ms L, 0[\,,
\qquad
\ms L \in \,]T\cdot f_l^\prime(A), 0[\,,
\end{aligned}
\end{equation}
\smallskip
\begin{equation}\label{eq:critcond1}
        \omega(0+) = \theta_r\,, \qquad\quad\omega(\ms R-) \geq \omega(\ms R+)\,, 
    \end{equation}
\smallskip
    \begin{equation}\label{eq:critcond1-2}
        \theta_r<\omega(x)\leq (f'_r)^{-1}\bigg(
        \frac{x}{T-\bs \sigma[\ms L, A, f_l]}
        \bigg)\qquad\forall~x\in\,]0, \ms R[\,.
    \end{equation}
\smallskip
\end{enumerate}

\end{thm}

\begin{figure}
    \centering
    
\tikzset{every picture/.style={line width=0.75pt}} 

\begin{tikzpicture}[x=0.75pt,y=0.75pt,yscale=-0.7,xscale=0.7]

\draw  [fill={rgb, 255:red, 208; green, 2; blue, 27 }  ,fill opacity=0.45 ] (441.2,62.87) -- (397.2,248.2) -- (321.2,152.87) -- (321.2,84.87) -- cycle ;
\draw [line width=1.5]    (320.92,247.96) -- (320.92,29.2) ;
\draw [shift={(320.92,25.2)}, rotate = 90] [fill={rgb, 255:red, 0; green, 0; blue, 0 }  ][line width=0.08]  [draw opacity=0] (6.97,-3.35) -- (0,0) -- (6.97,3.35) -- cycle    ;
\draw [line width=1.5]    (76.13,249.87) -- (575.47,248.63) ;
\draw [shift={(579.47,248.62)}, rotate = 179.86] [fill={rgb, 255:red, 0; green, 0; blue, 0 }  ][line width=0.08]  [draw opacity=0] (6.97,-3.35) -- (0,0) -- (6.97,3.35) -- cycle    ;
\draw    (73.47,61.46) -- (575.54,62.5) ;
\draw  [dash pattern={on 4.5pt off 4.5pt}]  (320.5,218.25) -- (364.5,249.25) ;
\draw  [dash pattern={on 4.5pt off 4.5pt}]  (355.2,62.2) -- (322,66.75) ;
\draw  [dash pattern={on 4.5pt off 4.5pt}]  (301.5,63.25) -- (320.1,158.3) ;
\draw  [dash pattern={on 4.5pt off 4.5pt}]  (319,172.75) -- (377.5,247.75) ;
\draw    (321,229.67) -- (346.5,248.75) ;
\draw  [dash pattern={on 4.5pt off 4.5pt}]  (260.5,64.25) -- (319,172.75) ;
\draw  [dash pattern={on 4.5pt off 4.5pt}]  (216,62.75) -- (320.5,192.75) ;
\draw    (124,62.33) -- (152,246.51) ;
\draw  [dash pattern={on 4.5pt off 4.5pt}]  (320.5,192.75) -- (377.5,247.75) ;
\draw  [dash pattern={on 4.5pt off 4.5pt}]  (124,62.33) -- (257.33,247.18) ;
\draw  [dash pattern={on 4.5pt off 4.5pt}]  (124,62.33) -- (182,247.85) ;
\draw  [dash pattern={on 4.5pt off 4.5pt}]  (124,62.33) -- (199.33,247.18) ;
\draw    (124,62.33) -- (321,229.67) ;
\draw  [dash pattern={on 4.5pt off 4.5pt}]  (124,62.33) -- (218.67,247.18) ;
\draw  [dash pattern={on 4.5pt off 4.5pt}]  (124,62.33) -- (168,249.18) ;
\draw  [dash pattern={on 4.5pt off 4.5pt}]  (124,62.33) -- (321.5,242.75) -- (330.5,247.75) ;
\draw  [dash pattern={on 4.5pt off 4.5pt}]  (124,62.33) -- (300.67,248.51) ;
\draw  [dash pattern={on 4.5pt off 4.5pt}]  (124,62.33) -- (238,249.18) ;
\draw  [dash pattern={on 4.5pt off 4.5pt}]  (143.5,65.75) -- (320.5,218.25) ;
\draw  [dash pattern={on 4.5pt off 4.5pt}]  (124,62.33) -- (279.33,249.18) ;
\draw  [dash pattern={on 4.5pt off 4.5pt}]  (182,64.25) -- (321.5,207.75) ;
\draw  [dash pattern={on 4.5pt off 4.5pt}]  (104.94,62.85) -- (108,251.85) ;
\draw  [dash pattern={on 4.5pt off 4.5pt}]  (78.67,62.51) -- (86.67,248.51) ;
\draw  [dash pattern={on 4.5pt off 4.5pt}]  (112.6,61.01) -- (137.33,249.18) ;
\draw  [dash pattern={on 4.5pt off 4.5pt}]  (92.2,62.02) -- (100.67,246.51) ;
\draw  [dash pattern={on 4.5pt off 4.5pt}]  (108.67,62.51) -- (123.33,248.51) ;
\draw  [dash pattern={on 4.5pt off 4.5pt}]  (348.53,112.87) -- (321.2,118.87) ;
\draw  [dash pattern={on 4.5pt off 4.5pt}]  (400.53,62.87) -- (323.2,75.53) ;
\draw  [dash pattern={on 4.5pt off 4.5pt}]  (371.2,95.53) -- (320.53,107.53) ;
\draw  [dash pattern={on 4.5pt off 4.5pt}]  (411.2,76.87) -- (319.87,96.2) ;
\draw  [dash pattern={on 4.5pt off 4.5pt}]  (352.53,113.53) -- (397.2,248.2) ;
\draw  [dash pattern={on 4.5pt off 4.5pt}]  (328.53,139.53) -- (397.2,248.2) ;
\draw  [dash pattern={on 4.5pt off 4.5pt}]  (329.2,136.2) -- (319.2,139.53) ;
\draw  [dash pattern={on 4.5pt off 4.5pt}]  (339.2,124.2) -- (397.2,248.2) ;
\draw  [dash pattern={on 4.5pt off 4.5pt}]  (411.2,76.87) -- (397.2,248.2) ;
\draw  [dash pattern={on 4.5pt off 4.5pt}]  (395.87,83.53) -- (397.2,248.2) ;
\draw  [dash pattern={on 4.5pt off 4.5pt}]  (380.53,91.53) -- (397.2,248.2) ;
\draw  [dash pattern={on 4.5pt off 4.5pt}]  (365.87,102.2) -- (397.2,248.2) ;
\draw  [dash pattern={on 4.5pt off 4.5pt}]  (426.53,69.53) -- (397.2,248.2) ;
\draw [color={rgb, 255:red, 208; green, 2; blue, 27 }  ,draw opacity=1 ]   (321.2,152.87) .. controls (333.87,115.53) and (378.53,82.87) .. (441.2,62.87) ;
\draw  [dash pattern={on 4.5pt off 4.5pt}]  (339.2,124.2) -- (320.77,128.93) ;
\draw  [dash pattern={on 4.5pt off 4.5pt}]  (458.3,63.85) -- (471.87,246.87) ;
\draw    (458.3,63.85) -- (484.53,248.87) ;
\draw  [dash pattern={on 4.5pt off 4.5pt}]  (458.3,63.85) -- (443.2,246.2) ;
\draw    (458.3,63.85) -- (435.87,248.87) ;
\draw  [dash pattern={on 4.5pt off 4.5pt}]  (465.87,64.2) -- (495.2,246.87) ;
\draw  [dash pattern={on 4.5pt off 4.5pt}]  (451.23,64.38) -- (418.53,248.87) ;
\draw  [dash pattern={on 4.5pt off 4.5pt}]  (458.3,63.85) -- (456.53,248.87) ;
\draw  [dash pattern={on 0.84pt off 2.51pt}]  (321.2,152.87) -- (579.2,152.87) ;
\draw  [dash pattern={on 4.5pt off 4.5pt}]  (320.9,161) -- (386.5,248.75) ;
\draw  [dash pattern={on 4.5pt off 4.5pt}]  (321.5,207.75) -- (377.5,247.75) ;

\draw (114,38.73) node [anchor=north west][inner sep=0.75pt]  [font=\footnotesize]  {$\mathsf{L}$};
\draw (554.67,44.4) node [anchor=north west][inner sep=0.75pt]  [font=\footnotesize]  {$\omega $};
\draw (585.33,241.73) node [anchor=north west][inner sep=0.75pt]  [font=\footnotesize]  {$x$};
\draw (331.33,24.73) node [anchor=north west][inner sep=0.75pt]  [font=\footnotesize]  {$t$};
\draw (436.67,38.73) node [anchor=north west][inner sep=0.75pt]  [font=\footnotesize]  {$\mathsf{R}$};
\draw (569.33,130.4) node [anchor=north west][inner sep=0.75pt]  [font=\scriptsize]  {$\mathbf{\tau }[\mathsf{R} ,\ B,\ f_{r}]$};
\draw (301.67,44.73) node [anchor=north west][inner sep=0.75pt]  [font=\footnotesize]  {$\theta _{l}$};
\draw (361.33,39.07) node [anchor=north west][inner sep=0.75pt]  [font=\footnotesize]  {$B$};

\end{tikzpicture}

    \caption{Typical profile of Theorem \ref{thm:attprofilescrit} for connections critical at the left $(\theta_l, B)$.}
    \label{fig:criticalLRneq0}
\end{figure}

\begin{remark}
    \label{rem:cases-thm-4.8}
    For critical connections, 
    whenever $\ms L < 0 < \ms R$ 
we can distinguish 
two cases 
of pointwise constraints prescribed by Theorem~\ref{thm:attprofilescrit}
on an attainable profile $\omega$,
which depend on the side in which
the connection is critical.

\noindent
\textsc{Case 1:}
If $A=\theta_l$, and $\ms L < 0 < \ms R  < T\cdot f'_r(B)$ (Figure \ref{fig:criticalLRneq0}), then it holds true
\begin{equation*}
    (f'_l)^{-1}\bigg(\frac{x}{T-\bs\tau[\ms R, B, f_r]}\bigg)\leq \omega(x)<\theta_l,
    \qquad\forall~x\in\,]\ms L, 0[\,,\qquad \ \ 
     \omega(0-) = \theta_l\,,        
    \end{equation*}
    \begin{equation*}    
        \omega(x)=B\ \quad\ \forall~x\in\,]0, {\ms R}[\,,
        \qquad \ \ 
  \omega(\ms R+) \leq \bs u[\ms R, B, f_r]\leq B\,;
    \end{equation*}
    
\noindent
\textsc{Case 2:}
If $B=\theta_r$, and $T\cdot f_l^\prime(A)<\ms L <  0 < \ms R$, then it holds true
    \begin{equation*}    
        \omega(x)=A\ \quad\ \forall~x\in\,]{\ms L}, 0[\,,
        \qquad \ \ 
  A \geq \bs v[\ms L, A, f_l]\geq \omega(\ms L+)\,;
    \end{equation*}
    \begin{equation*}
 \theta_r<\omega(x)\leq (f'_r)^{-1}\bigg(
        \frac{x}{T-\bs \sigma[\ms L, A, f_l]}
        \bigg)\qquad\forall~x\in\,]0, \ms R[\,,\qquad \ \ 
     \omega(0+) = \theta_r\,.       
    \end{equation*}
    
\noindent
In both cases an $AB$-entropy solution
that attains $\omega$ at time $T$ must contain a shock located in $\{x> 0\}$
(in \textsc{Case 1}),
or in $\{x< 0\}$ (in \textsc{Case 2}), 
in order to produce the discontinuity occurring in $\omega$ at $\ms R$ or 
$\ms L$.
\end{remark}

\begin{thm}\label{thm:attprofiles2}
In the same setting of Theorem~\ref{thm:attprofiles},
    let $\omega$ be an element of the set $\msc A^{\ms L, \ms R}$ in~\eqref{eq:ALR-def1}, 
let $g, h$ be the functions in~\eqref{eq:ghdef},
and let $\bs u[\ms R, B, f_r]$, $\bs v[\ms L, A, f_l]$,  be 
constants defined as
in~\eqref{eq:urbf-def},
\eqref{eq:vlaf-def}.
    Then, if $\ms L < 0$, $\ms R = 0$,
    $\omega \in \mc A^{[AB]}(T)$ if and only if the limits $\omega(0\pm)$ exist, and it
     holds:

\begin{enumerate}
    \item[(i)] 
    the following Ole\v{\i}nik-type inequalities are satisfied 
    \begin{equation}\label{eq:1a2}
    \begin{aligned}
     D^+ \omega(x) &\leq \frac{1}{T \cdot f_l^{\second}(\omega(x))} \qquad \forall \; x \in \,]\!-\!\infty,\;  \ms L[\,,\\
       \noalign{\smallskip}
       D^+ \omega(x) &\leq \frac{1}{T \cdot f_r^{\second}(\omega(x))} \qquad \forall \; x \in \,]0, +\infty[\,,
        \end{aligned}
    \end{equation}
    \medskip
\begin{equation}\label{eq:1b2}
    D^+\omega(x) \leq g[\omega, f_l, f_r](x) \qquad \forall \; x \in \ ]\ms L,0[\,.\quad\
\end{equation}

\item[(ii)] the following pointwise state constraints are satisfied:

\begin{equation}\label{eq:2b2}
    \begin{cases}
        \omega(x) \leq  A\ \ &\text{if}\quad A<\theta_l,
        \\
        \omega(x) <  A\ \ &\text{if}\quad A=\theta_l,
    \end{cases}\qquad \forall \; x \in \,]\ms L, 0[\,,
\end{equation}

\begin{equation}\label{eq:2a2}
    \omega(\ms 0+) \leq \pi^l_{r,-}(\omega(0-))\,,
\end{equation}
and
\begin{align}
\label{eq:2a2b}
\ms L \in \,]T\cdot f_l^{\prime}(A), 0[
  \qquad  &\Longrightarrow \qquad
    \omega(\ms L-) \geq \bs v[\ms L, A, f_l]\geq  \omega(\ms L+)\,,
    \\
    \noalign{\smallskip}
    \label{eq:3a2b}
\ms L~{\leq}~
T\cdot f_l^{\prime}(A)
\qquad  &\Longrightarrow \qquad
    \omega(\ms L-) \geq \omega(\ms L+).
\end{align}
\end{enumerate}
Symmetrically, if $\ms L = 0$, $\ms R > 0$, then $\omega \in \mc A^{[AB]}(T)$ if and only if it holds true:
\begin{enumerate}
    \item[(i)\,$^\prime$] 
     the following Ole\v{\i}nik-type inequalities are satisfied 
    \begin{equation}\label{eq:1a2-prime}
    \begin{aligned}
     D^+ \omega(x) &\leq \frac{1}{T \cdot f_l^{\second}(\omega(x))} \qquad \forall \; x \in \,]\!-\!\infty,\;  0[\,,\\
       \noalign{\smallskip}
       D^+ \omega(x) &\leq \frac{1}{T \cdot f_r^{\second}(\omega(x))} \qquad \forall \; x \in \,]\ms R, +\infty[\,,
        \end{aligned}
    \end{equation}
    \medskip
\begin{equation}\label{eq:1b2-prime}
    D^+\omega(x) \leq h[\omega, f_l, f_r](x) \qquad \forall \; x \in \ ]0, \ms R[\,.\quad\
\end{equation}

\item[(ii)\,$^\prime$] the following pointwise state constraints are satisfied:

\begin{equation}\label{eq:2b2-prime}
    \begin{cases}
        \omega(x) \geq  B\ \ &\text{if}\quad B>\theta_r,
        \\
        \omega(x) > B\ \ &\text{if}\quad B=\theta_r,
    \end{cases}
    \qquad \forall \; x \in \,]0, \ms R[\,,
\end{equation}

\begin{equation}\label{eq:2a2-prime}
    \omega(\ms 0-) \geq \pi^r_{l,+}(\omega(0+))\,,
\end{equation}
and
\begin{align}
\label{eq:2a2b-prime}
\ms R \in \,]0, T \cdot f_r^{\prime}(B)[\,
  \qquad  &\Longrightarrow \qquad 
    \omega(\ms R+) \leq \bs u[\ms R, B, f_r] \leq  \omega(\ms R-)\,,
\\
\noalign{\smallskip}
\label{eq:3a2b-prime}
\ms R ~{\geq}~
T \cdot f_r^{\prime}(B)
\qquad  &\Longrightarrow \qquad 
    \omega(\ms R+) \leq \omega(\ms R-) . 
\end{align}
\end{enumerate}
\end{thm}

\begin{figure}

\tikzset{every picture/.style={line width=0.75pt}} 

\begin{tikzpicture}[x=0.75pt,y=0.75pt,yscale=-0.7,xscale=0.7]
\draw  [fill={rgb, 255:red, 74; green, 144; blue, 226 }  ,fill opacity=0.64 ] (191.33,84.33) -- (229.67,271.67) -- (342,216) -- (341,131.67) -- cycle ;
\draw [line width=1.5]    (342.26,269.96) -- (342.26,51.2) ;
\draw [shift={(342.26,47.2)}, rotate = 90] [fill={rgb, 255:red, 0; green, 0; blue, 0 }  ][line width=0.08]  [draw opacity=0] (6.97,-3.35) -- (0,0) -- (6.97,3.35) -- cycle    ;
\draw [line width=1.5]    (98.13,271.87) -- (597.47,270.63) ;
\draw [shift={(601.47,270.62)}, rotate = 179.86] [fill={rgb, 255:red, 0; green, 0; blue, 0 }  ][line width=0.08]  [draw opacity=0] (6.97,-3.35) -- (0,0) -- (6.97,3.35) -- cycle    ;
\draw    (103.2,83.3) -- (596.87,84.5) ;
\draw  [dash pattern={on 4.5pt off 4.5pt}]  (225.67,105) -- (229.67,271.67) ;
\draw  [dash pattern={on 4.5pt off 4.5pt}]  (205.67,93) -- (229.67,271.67) ;
\draw  [dash pattern={on 4.5pt off 4.5pt}]  (286.33,149.67) -- (229.67,271.67) ;
\draw  [dash pattern={on 4.5pt off 4.5pt}]  (191.33,84.33) -- (197,271) ;
\draw  [dash pattern={on 4.5pt off 4.5pt}]  (271.67,135.67) -- (229.67,271.67) ;
\draw  [dash pattern={on 4.5pt off 4.5pt}]  (191.33,84.33) -- (216.33,271) ;
\draw  [dash pattern={on 4.5pt off 4.5pt}]  (191.33,84.33) -- (179,271) ;
\draw  [dash pattern={on 4.5pt off 4.5pt}]  (191.33,84.33) -- (159,271) ;
\draw  [dash pattern={on 4.5pt off 4.5pt}]  (257.67,124.33) -- (229.67,271.67) ;
\draw  [dash pattern={on 4.5pt off 4.5pt}]  (191.33,84.33) -- (141,271) ;
\draw  [dash pattern={on 4.5pt off 4.5pt}]  (242.33,113.67) -- (229.67,271.67) ;
\draw    (191.33,84.33) -- (229.67,271.67) ;
\draw    (191.33,84.33) -- (129.67,270.33) ;
\draw    (191.33,84.33) -- (341,131.67) ;
\draw    (342.33,216.33) -- (229.67,271.67) ;
\draw [color={rgb, 255:red, 208; green, 2; blue, 27 }  ,draw opacity=1 ]   (191.33,84.33) .. controls (247.67,117) and (251,117.67) .. (279,141) .. controls (307,164.33) and (325.67,187) .. (342.33,216.33) ;
\draw  [dash pattern={on 4.5pt off 4.5pt}]  (284.33,145) -- (342,166.67) ;
\draw  [dash pattern={on 4.5pt off 4.5pt}]  (257.67,124.33) -- (339.67,154.33) ;
\draw  [dash pattern={on 4.5pt off 4.5pt}]  (327,192.33) -- (344.33,198.33) ;
\draw  [dash pattern={on 4.5pt off 4.5pt}]  (299.67,159) -- (229.67,271.67) ;
\draw  [dash pattern={on 4.5pt off 4.5pt}]  (225.67,105) -- (342,143.67) ;
\draw  [dash pattern={on 4.5pt off 4.5pt}]  (317,177.67) -- (343.67,187.67) ;
\draw  [dash pattern={on 4.5pt off 4.5pt}]  (303.67,162.33) -- (341,176.33) ;
\draw  [dash pattern={on 4.5pt off 4.5pt}]  (311.67,173.67) -- (229.67,271.67) ;
\draw  [dash pattern={on 4.5pt off 4.5pt}]  (323,188.33) -- (229.67,271.67) ;
\draw  [dash pattern={on 4.5pt off 4.5pt}]  (333,201) -- (229.67,271.67) ;
\draw  [dash pattern={on 4.5pt off 4.5pt}]  (324.2,271.55) -- (345,261) ;
\draw  [dash pattern={on 4.5pt off 4.5pt}]  (296.67,271.98) -- (342.33,251.67) ;
\draw  [dash pattern={on 4.5pt off 4.5pt}]  (253.62,270.95) -- (341.67,228.33) ;
\draw  [dash pattern={on 4.5pt off 4.5pt}]  (275.07,271.24) -- (340.33,240.33) ;
\draw  [dash pattern={on 4.5pt off 4.5pt}]  (243.85,85.39) -- (342.33,115) ;
\draw  [dash pattern={on 4.5pt off 4.5pt}]  (322.16,85.62) -- (342.33,91.67) ;
\draw  [dash pattern={on 4.5pt off 4.5pt}]  (339.67,154.33) -- (423,270.33) ;
\draw  [dash pattern={on 4.5pt off 4.5pt}]  (342,143.67) -- (433.67,270.33) ;
\draw  [dash pattern={on 4.5pt off 4.5pt}]  (342.33,115) -- (481,267.67) ;
\draw  [dash pattern={on 4.5pt off 4.5pt}]  (341.67,111) -- (503.67,269.67) ;
\draw  [dash pattern={on 4.5pt off 4.5pt}]  (340.33,96.33) -- (525.67,268.33) ;
\draw  [dash pattern={on 4.5pt off 4.5pt}]  (243.85,85.39) -- (341.67,111) ;
\draw  [dash pattern={on 4.5pt off 4.5pt}]  (341.67,127) -- (453,269) ;
\draw  [dash pattern={on 4.5pt off 4.5pt}]  (279.67,85) -- (340.33,96.33) ;
\draw    (341.67,107) -- (518.33,270.33) ;
\draw  [dash pattern={on 4.5pt off 4.5pt}]  (215,83.67) -- (341.67,127) ;
\draw  [dash pattern={on 4.5pt off 4.5pt}]  (342.33,91.67) -- (530.33,267.67) ;
\draw    (243.85,85.39) -- (340.33,119) ;
\draw    (340.33,119) -- (466.33,269.67) ;
\draw    (243.85,85.39) -- (341.67,107) ;
\draw    (341.84,84.65) -- (541.67,270.33) ;
\draw    (341,131.67) -- (442.33,270.33) ;
\draw  [dash pattern={on 4.5pt off 4.5pt}]  (342.33,251.67) -- (356,269.33) ;
\draw  [dash pattern={on 4.5pt off 4.5pt}]  (344.33,198.33) -- (397,271) ;
\draw  [dash pattern={on 4.5pt off 4.5pt}]  (341,176.33) -- (409.67,269.67) ;
\draw  [dash pattern={on 4.5pt off 4.5pt}]  (344.33,235.67) -- (367.67,269.67) ;
\draw  [dash pattern={on 4.5pt off 4.5pt}]  (342.33,216.33) -- (380,270.67) ;
\draw  [dash pattern={on 4.5pt off 4.5pt}]  (457.67,86.33) -- (570,268.33) ;
\draw  [dash pattern={on 4.5pt off 4.5pt}]  (421,86.33) -- (563.33,268.33) ;
\draw  [dash pattern={on 4.5pt off 4.5pt}]  (394.33,85) -- (557,268.33) ;
\draw  [dash pattern={on 4.5pt off 4.5pt}]  (370,84.33) -- (549.67,269) ;
\draw  [dash pattern={on 4.5pt off 4.5pt}]  (355.33,84.67) -- (544.33,268.33) ;
\draw  [dash pattern={on 4.5pt off 4.5pt}]  (494.33,85) -- (574.67,269.67) ;
\draw  [dash pattern={on 4.5pt off 4.5pt}]  (164.2,84.3) -- (107.2,269.3) ;
\draw  [dash pattern={on 4.5pt off 4.5pt}]  (145.2,84.3) -- (98.13,271.87) ;
\draw  [dash pattern={on 4.5pt off 4.5pt}]  (181.17,86.11) -- (114.2,269.3) ;

\draw (187.33,61.4) node [anchor=north west][inner sep=0.75pt]  [font=\footnotesize]  {$\mathsf{L}$};
\draw (576,66.4) node [anchor=north west][inner sep=0.75pt]  [font=\footnotesize]  {$\omega $};
\draw (606.67,263.73) node [anchor=north west][inner sep=0.75pt]  [font=\footnotesize]  {$x$};
\draw (352.67,46.73) node [anchor=north west][inner sep=0.75pt]  [font=\footnotesize]  {$t$};
\draw (278.67,278.71) node [anchor=north west][inner sep=0.75pt]  [font=\footnotesize]  {$\overline{A}$};

\end{tikzpicture}

    \caption{Theorem \ref{thm:attprofiles2} when $\ms L < 0$, $\ms R = 0$ and $\ms L \in \, ]T \cdot f_l^{\prime}(A), 0[$.}
    \label{fig:attprofiles2a}
\end{figure}

\begin{figure}
    \centering

\tikzset{every picture/.style={line width=0.75pt}} 

\begin{tikzpicture}[x=0.75pt,y=0.75pt,yscale=-0.7,xscale=0.7]

\draw [line width=1.5]    (337.59,250.64) -- (337.59,31.89) ;
\draw [shift={(337.59,27.89)}, rotate = 90] [fill={rgb, 255:red, 0; green, 0; blue, 0 }  ][line width=0.08]  [draw opacity=0] (6.97,-3.35) -- (0,0) -- (6.97,3.35) -- cycle    ;
\draw [line width=1.5]    (92.8,252.55) -- (592.13,251.32) ;
\draw [shift={(596.13,251.31)}, rotate = 179.86] [fill={rgb, 255:red, 0; green, 0; blue, 0 }  ][line width=0.08]  [draw opacity=0] (6.97,-3.35) -- (0,0) -- (6.97,3.35) -- cycle    ;
\draw    (90.13,64.15) -- (592.2,65.18) ;
\draw    (140.67,65.02) -- (338.67,128.53) ;
\draw  [dash pattern={on 4.5pt off 4.5pt}]  (239.19,66.07) -- (337.67,95.69) ;
\draw  [dash pattern={on 4.5pt off 4.5pt}]  (317.49,66.31) -- (337.67,72.35) ;
\draw  [dash pattern={on 4.5pt off 4.5pt}]  (337.59,139.26) -- (421.33,249.2) ;
\draw    (338.67,128.53) -- (428,251.87) ;
\draw  [dash pattern={on 4.5pt off 4.5pt}]  (337.67,95.69) -- (476.33,248.35) ;
\draw  [dash pattern={on 4.5pt off 4.5pt}]  (337,91.69) -- (499,250.35) ;
\draw  [dash pattern={on 4.5pt off 4.5pt}]  (335.67,77.02) -- (521,249.02) ;
\draw  [dash pattern={on 4.5pt off 4.5pt}]  (239.19,66.07) -- (337,91.69) ;
\draw  [dash pattern={on 4.5pt off 4.5pt}]  (337,107.69) -- (448.33,249.69) ;
\draw  [dash pattern={on 4.5pt off 4.5pt}]  (275,65.69) -- (335.67,77.02) ;
\draw    (337,87.69) -- (513.67,251.02) ;
\draw  [dash pattern={on 4.5pt off 4.5pt}]  (210.33,64.35) -- (337,107.69) ;
\draw  [dash pattern={on 4.5pt off 4.5pt}]  (337.67,72.35) -- (525.67,248.35) ;
\draw    (239.19,66.07) -- (335.67,99.69) ;
\draw    (335.67,99.69) -- (461.67,250.35) ;
\draw    (239.19,66.07) -- (337,87.69) ;
\draw    (337.17,65.33) -- (537,251.02) ;
\draw  [dash pattern={on 4.5pt off 4.5pt}]  (336.33,115.02) -- (439.33,251.2) ;
\draw  [dash pattern={on 4.5pt off 4.5pt}]  (337.67,232.35) -- (348,249.87) ;
\draw  [dash pattern={on 4.5pt off 4.5pt}]  (336.33,157.02) -- (404.67,250.53) ;
\draw  [dash pattern={on 4.5pt off 4.5pt}]  (339.67,216.35) -- (361.33,251.2) ;
\draw  [dash pattern={on 4.5pt off 4.5pt}]  (337.67,197.02) -- (375.33,251.35) ;
\draw  [dash pattern={on 4.5pt off 4.5pt}]  (453,67.02) -- (565.33,249.02) ;
\draw  [dash pattern={on 4.5pt off 4.5pt}]  (416.33,67.02) -- (558.67,249.02) ;
\draw  [dash pattern={on 4.5pt off 4.5pt}]  (389.67,65.69) -- (552.33,249.02) ;
\draw  [dash pattern={on 4.5pt off 4.5pt}]  (365.33,65.02) -- (545,249.69) ;
\draw  [dash pattern={on 4.5pt off 4.5pt}]  (350.67,65.35) -- (539.67,249.02) ;
\draw  [dash pattern={on 4.5pt off 4.5pt}]  (489.67,65.69) -- (570,250.35) ;
\draw  [dash pattern={on 4.5pt off 4.5pt}]  (190,66.67) -- (333,113.69) ;
\draw  [dash pattern={on 4.5pt off 4.5pt}]  (166.67,65.67) -- (335.33,122.53) ;
\draw    (140.67,65.02) -- (168.67,249.2) ;
\draw  [dash pattern={on 4.5pt off 4.5pt}]  (339.67,179.02) -- (388,250.53) ;
\draw  [dash pattern={on 4.5pt off 4.5pt}]  (140.67,65.02) -- (337.59,139.26) ;
\draw  [dash pattern={on 4.5pt off 4.5pt}]  (140.67,65.02) -- (274,249.87) ;
\draw  [dash pattern={on 4.5pt off 4.5pt}]  (140.67,65.02) -- (198.67,250.53) ;
\draw  [dash pattern={on 4.5pt off 4.5pt}]  (140.67,65.02) -- (216,249.87) ;
\draw  [dash pattern={on 4.5pt off 4.5pt}]  (140.67,65.02) -- (337.67,232.35) ;
\draw  [dash pattern={on 4.5pt off 4.5pt}]  (140.67,65.02) -- (235.33,249.87) ;
\draw  [dash pattern={on 4.5pt off 4.5pt}]  (140.67,65.02) -- (184.67,251.87) ;
\draw  [dash pattern={on 4.5pt off 4.5pt}]  (140.67,65.02) -- (337.59,250.64) ;
\draw  [dash pattern={on 4.5pt off 4.5pt}]  (140.67,65.02) -- (317.33,251.2) ;
\draw  [dash pattern={on 4.5pt off 4.5pt}]  (140.67,65.02) -- (254.67,251.87) ;
\draw  [dash pattern={on 4.5pt off 4.5pt}]  (140.67,65.02) -- (338,174.53) ;
\draw  [dash pattern={on 4.5pt off 4.5pt}]  (140.67,65.02) -- (336.67,212.53) ;
\draw  [dash pattern={on 4.5pt off 4.5pt}]  (140.67,65.02) -- (336,192.53) ;
\draw  [dash pattern={on 4.5pt off 4.5pt}]  (140.67,65.02) -- (296,251.87) ;
\draw  [dash pattern={on 4.5pt off 4.5pt}]  (140.67,65.02) -- (336.33,157.02) ;
\draw  [dash pattern={on 4.5pt off 4.5pt}]  (121.61,65.53) -- (124.67,254.53) ;
\draw  [dash pattern={on 4.5pt off 4.5pt}]  (95.33,65.2) -- (103.33,251.2) ;
\draw  [dash pattern={on 4.5pt off 4.5pt}]  (129.27,63.7) -- (154,251.87) ;
\draw  [dash pattern={on 4.5pt off 4.5pt}]  (108.86,64.7) -- (117.33,249.2) ;
\draw  [dash pattern={on 4.5pt off 4.5pt}]  (125.33,65.2) -- (140,251.2) ;

\draw (137.33,42.09) node [anchor=north west][inner sep=0.75pt]  [font=\footnotesize]  {$\mathsf{L}$};
\draw (571.33,47.09) node [anchor=north west][inner sep=0.75pt]  [font=\footnotesize]  {$\omega $};
\draw (602,244.42) node [anchor=north west][inner sep=0.75pt]  [font=\footnotesize]  {$x$};
\draw (348,27.42) node [anchor=north west][inner sep=0.75pt]  [font=\footnotesize]  {$t$};

\end{tikzpicture}

    \caption{Theorem \ref{thm:attprofiles2} when $\ms L < 0$, $\ms R = 0$ and $\ms L \leq T \cdot f_l^{\prime}(A)$.}
    \label{fig:attprofiles2b}
\end{figure}

\begin{remark}
\label{eq-lax-ineq-4}
Notice that
the implications~\eqref{eq:2a2b}-\eqref{eq:3a2b},
\eqref{eq:2a2b-prime}-\eqref{eq:3a2b-prime} can be extended \linebreak to $\ms L =T\cdot f'_l(A)$ and to $\ms R=T\cdot f'_r(B)$, respectively.
In fact,
by definition~\eqref{eq:LR-def} of
$\ms L=\ms L[\omega, f_l]$, one has $f'_l(\omega(\ms L-))\geq \ms L/T$.
Hence, if $\ms L =T\cdot f'_l(A)$ it follows that 
$f'_l(\omega(\ms L-))\geq f'_l(A)$ which yields
$\omega(\ms L-)\geq A$
by the monotonicity of $f'_l$.
Thus, recalling that by~\eqref{eq:u-v-cont-ext}
we have $\bs v[T\cdot f^{\prime}(A), A, f] =
     A$, we derive
     \begin{equation}
     \label{eq:est-tfprimaea-1}
         \omega(T\cdot f'_l(A)-) \geq \bs v[T\cdot f'_l(A), A, f_l]\,.
     \end{equation}
On the other hand, 
since~\eqref{eq:2b2} implies $\omega(T\cdot f'_l(A+)\leq A$,
we deduce from~\eqref{eq:est-tfprimaea-1} that
\begin{equation}
     \label{eq:est-tfprimaea-2}
         \omega(T\cdot f'_l(A)-) \geq \omega(T\cdot f'_l(A)+)\,.
     \end{equation}
With entirely similar arguments one can show that we have
\begin{align}
\label{eq:est-tfprimaea-3}
     \omega(T\cdot f'_l(B)+) &\leq \bs u[T\cdot f'_l(B), B, f_r]\,,
    \\
    \noalign{\smallskip}
    \label{eq:est-tfprimaea-4}
       \omega(T\cdot f'_r(B)+)&\leq \omega(T\cdot f'_r(B)-)\,.
\end{align}
Hence, relying on~\eqref{eq:2b2}, \eqref{eq:2a2b}, \eqref{eq:3a2b},
\eqref{eq:2b2-prime}, 
\eqref{eq:2a2b-prime}, 
\eqref{eq:3a2b-prime}, 
and on~\eqref{eq:est-tfprimaea-2}, \eqref{eq:est-tfprimaea-4}, 
with the same arguments of Remark~\ref{eq-lax-ineq-3} we deduce that the 
inequalities $\omega(\ms L-)\geq \omega(\ms L+)$,
$\omega(\ms R-)\geq \omega(\ms R+)$
are always satisfied.
\end{remark}

\begin{remark}\label{rem:threecases1}
Relying on Remark~\ref{rem:toildeLR-limitingcase}, we can view 
the conditions that characterize the pointwise constraints 
of attainable profiles in Theorem~\ref{thm:attprofiles2} as limiting cases of 
 the conditions of Theorems~\ref{thm:attprofiles}, \ref{thm:attprofilescrit},
 classified in Remarks~\ref{rem:cases-thm-4.1}, \ref{rem:cases-thm-4.8}. Namely:
\begin{itemize}
[leftmargin=20pt]
    \item[-] For non critical connections, the case  $\ms L \in \,]T \cdot f_l^{\prime}(A), 0[\,$, $\ms R = 0$ (Figure~\ref{fig:attprofiles2a}),  is the limiting situation 
   as $\ms R\to 0$ of \textsc{Case~2} in Remark~\ref{rem:cases-thm-4.1}.
   For critical connections  with $A<\theta_l$, $B=\theta_r$, if the constraint~\eqref{eq:2b2}
   is satisfied with the equality, 
    the case  $\ms L \in \,]T \cdot f_l^{\prime}(A), 0[\,$, $\ms R = 0$, is the limiting situation 
   as $\ms R\to 0$ 
    of \textsc{Case~2} in Remark~\ref{rem:cases-thm-4.8}.
    %
    \item[-] For non critical connections, the case  $\ms L \leq T \cdot f_l^{\prime}(A)$, $\ms R = 0$ (Figure~\ref{fig:attprofiles2b}),  is the limiting situation as $\ms R\to 0$ of \textsc{Case~1} 
    in Remark~\ref{rem:cases-thm-4.1}.
   For critical connections with $A=\theta_l$, $B>\theta_r$, 
   the case  $\ms L \leq T \cdot f_l^{\prime}(A)$, $\ms R = 0$,  is the limiting situation as $\ms R\to 0$ of 
    \textsc{Case~1} in Remark~\ref{rem:cases-thm-4.8}
\end{itemize}
Symmetrically, we have:
    \begin{itemize}
[leftmargin=20pt]
    \item[-] For non critical connections, the case  $\ms L = 0$, $\ms R \in \,]0, T \cdot f_r^{\prime}(B)[\,$,    is the limiting situation 
   as $\ms L\to 0$ of \textsc{Case~2B} in Remark~\ref{rem:cases-thm-4.1}.
   For critical connections  with $A=\theta_l$, $B>\theta_r$,
   if the constraint~\eqref{eq:2b2-prime}
   is satisfied with the equality,
   the case  $\ms L = 0$, $\ms R \in \,]0, T \cdot f_r^{\prime}(B)[\,$ is the limiting situation 
   as $\ms L\to 0$ of \textsc{Case~1} in Remark~\ref{rem:cases-thm-4.8}.
    \item[-] For non critical connections, the case  $\ms L = 0$, $\ms R\geq 
    T \cdot f_r^{\prime}(B)$, is the limiting situation 
   as $\ms L\to 0$ of \textsc{Case~1B} in Remark~\ref{rem:cases-thm-4.1}.
    For critical connections with $A<\theta_l$, $B=\theta_r$, 
    the case $\ms R\geq 
    T \cdot f_r^{\prime}(B)$ is the limiting situation 
   as $\ms L\to 0$ of \textsc{Case~2} in Remark~\ref{rem:cases-thm-4.8}.
\end{itemize}
Notice that, for non critical connections, no limiting situation of \textsc{Case~3}
or of~\textsc{Case~4} in Remark~\ref{rem:cases-thm-4.1} arises as 
characterizing the pointwise constraints 
of attainable profiles in Theorem~\ref{thm:attprofiles2}.\\
The same type 
of conditions discussed in Remark~\ref{rem:cases-thm-4.1} 
require the presence
of shocks in an $AB$-entropy solution 
that attains at time $T$ a profile 
satisfying the conditions of
Theorem~\ref{thm:attprofiles2}.
In fact, for such profiles
it is needed a shock located in $\{x<0\}$
(in $\{x>0\}$) to produce the 
discontinuity in $\omega$ at $x=\ms L$
(at $x=\ms R$)
if and only if 
$\ms L \in \,]T\cdot f_l^\prime(A), 0[$, and $\ms R = 0$
($\ms L = 0$
and $\ms R \in \,]0, T\cdot f_r^\prime(B)[\,$). 
\end{remark}

\begin{thm}\label{thm:attprofiles3}
    In the same setting of Theorem~\ref{thm:attprofiles},
    let $\omega$ be an element of the set $\msc A^{\ms L, \ms R}$ in~\eqref{eq:ALR-def1},   with $\ms L = 0$, $\ms R = 0$.
Then $\omega \in \mc A^{[AB]}(T)$ if and only if the limits $\omega(0\pm)$ exist, and it holds true:

\begin{enumerate}
    \item[(i)] 
     the following Ole\v{\i}nik-type inequalities are satisfied 
    \begin{equation}\label{eq:1a3}
    \begin{aligned}
    D^+ \omega(x) &\leq \frac{1}{T \cdot f_l^{\second}(\omega(x))} \qquad \forall \; x \in \,]-\!\infty,\;  0[\,,\\
       \noalign{\smallskip}
       D^+ \omega(x) &\leq \frac{1}{T \cdot f_r^{\second}(\omega(x))} \qquad \forall \; x \in \,]0, +\infty[\,.
        \end{aligned}
    \end{equation}

\item[(ii)] the following pointwise state constraints are satisfied:
\begin{equation}\label{eq:2a3}
    \omega(\ms 0-) \geq \overline A, \qquad \omega(\ms 0+) ~{\leq}~ \overline B\,, 
\end{equation}
%
\end{enumerate}
\end{thm}
\begin{figure}
    \centering

\tikzset{every picture/.style={line width=0.75pt}} 

\begin{tikzpicture}[x=0.75pt,y=0.75pt,yscale=-0.7,xscale=0.7]

\draw [line width=1.5]    (355.59,266.62) -- (355.59,47.87) ;
\draw [shift={(355.59,43.87)}, rotate = 90] [fill={rgb, 255:red, 0; green, 0; blue, 0 }  ][line width=0.08]  [draw opacity=0] (6.97,-3.35) -- (0,0) -- (6.97,3.35) -- cycle    ;
\draw [line width=1.5]    (110.8,268.53) -- (610.13,267.3) ;
\draw [shift={(614.13,267.29)}, rotate = 179.86] [fill={rgb, 255:red, 0; green, 0; blue, 0 }  ][line width=0.08]  [draw opacity=0] (6.97,-3.35) -- (0,0) -- (6.97,3.35) -- cycle    ;
\draw    (108.13,80.13) -- (610.2,81.17) ;
\draw    (183.87,81.2) -- (183.2,267.2) ;
\draw  [dash pattern={on 4.5pt off 4.5pt}]  (183.87,81.2) -- (223.2,267.87) ;
\draw  [dash pattern={on 4.5pt off 4.5pt}]  (183.87,81.2) -- (237.2,265.87) ;
\draw    (183.87,81.2) -- (253.33,265.85) ;
\draw  [dash pattern={on 4.5pt off 4.5pt}]  (183.87,81.2) -- (211.87,266.53) ;
\draw  [dash pattern={on 4.5pt off 4.5pt}]  (139.61,81.51) -- (142.67,270.51) ;
\draw  [dash pattern={on 4.5pt off 4.5pt}]  (116.67,81.18) -- (124.67,267.18) ;
\draw  [dash pattern={on 4.5pt off 4.5pt}]  (147.27,79.68) -- (172,267.85) ;
\draw  [dash pattern={on 4.5pt off 4.5pt}]  (126.86,80.68) -- (135.33,265.18) ;
\draw  [dash pattern={on 4.5pt off 4.5pt}]  (143.33,81.18) -- (158,267.18) ;
\draw  [dash pattern={on 4.5pt off 4.5pt}]  (492.97,82.52) -- (506.53,265.53) ;
\draw    (492.97,82.52) -- (519.2,267.53) ;
\draw  [dash pattern={on 4.5pt off 4.5pt}]  (492.97,82.52) -- (477.87,264.87) ;
\draw    (492.97,82.52) -- (470.53,267.53) ;
\draw  [dash pattern={on 4.5pt off 4.5pt}]  (500.53,82.87) -- (529.87,265.53) ;
\draw  [dash pattern={on 4.5pt off 4.5pt}]  (415.87,81.2) -- (461.2,265.87) ;
\draw  [dash pattern={on 4.5pt off 4.5pt}]  (492.97,82.52) -- (491.2,267.53) ;
\draw  [dash pattern={on 4.5pt off 4.5pt}]  (391.2,267.2) -- (356.08,146.55) ;
\draw  [dash pattern={on 4.5pt off 4.5pt}]  (355.87,79.87) -- (293.2,267.87) ;
\draw    (435.2,265.3) -- (355.87,79.87) ;
\draw  [dash pattern={on 4.5pt off 4.5pt}]  (356.53,237.87) -- (346.12,268) ;
\draw  [dash pattern={on 4.5pt off 4.5pt}]  (354.53,203.2) -- (335.65,267.1) ;
\draw    (355.87,79.87) -- (279.2,265.3) ;
\draw  [dash pattern={on 4.5pt off 4.5pt}]  (353.87,168.53) -- (327.87,265.2) ;
\draw  [dash pattern={on 4.5pt off 4.5pt}]  (407.2,268.53) -- (357.26,123.02) ;
\draw  [dash pattern={on 4.5pt off 4.5pt}]  (377.97,266.27) -- (355.87,188.53) ;
\draw  [dash pattern={on 4.5pt off 4.5pt}]  (356.53,97.2) -- (307.2,267.87) ;
\draw  [dash pattern={on 4.5pt off 4.5pt}]  (355.2,140.53) -- (317.26,267.58) ;
\draw  [dash pattern={on 4.5pt off 4.5pt}]  (420.53,265.87) -- (356.53,97.2) ;
\draw  [dash pattern={on 4.5pt off 4.5pt}]  (356.53,228.53) -- (367.2,267.87) ;
\draw    (211.75,79.8) -- (257.87,267.87) ;
\draw  [dash pattern={on 4.5pt off 4.5pt}]  (166.03,81.41) -- (177.87,265.87) ;
\draw  [dash pattern={on 4.5pt off 4.5pt}]  (183.87,81.2) -- (195.87,266.53) ;
\draw  [dash pattern={on 4.5pt off 4.5pt}]  (262.39,81.11) -- (257.87,267.87) ;
\draw  [dash pattern={on 4.5pt off 4.5pt}]  (244.26,80.98) -- (257.87,267.87) ;
\draw  [dash pattern={on 4.5pt off 4.5pt}]  (317.01,82.4) -- (269.2,268.53) ;
\draw    (282.11,80.83) -- (257.87,267.87) ;
\draw  [dash pattern={on 4.5pt off 4.5pt}]  (225.71,80.42) -- (257.87,267.87) ;
\draw  [dash pattern={on 4.5pt off 4.5pt}]  (339.57,81.63) -- (272.53,267.2) ;
\draw  [dash pattern={on 4.5pt off 4.5pt}]  (300.53,81.2) -- (267.2,265.2) ;
\draw  [dash pattern={on 4.5pt off 4.5pt}]  (397.33,81) -- (453.2,267.53) ;
\draw  [dash pattern={on 4.5pt off 4.5pt}]  (466,81) -- (470.53,267.53) ;
\draw  [dash pattern={on 4.5pt off 4.5pt}]  (382.67,82) -- (449.2,265.2) ;
\draw  [dash pattern={on 4.5pt off 4.5pt}]  (476.67,81.33) -- (470.53,267.53) ;
\draw  [dash pattern={on 4.5pt off 4.5pt}]  (434.67,81.33) -- (465.87,266.53) ;
\draw  [dash pattern={on 4.5pt off 4.5pt}]  (366.67,80.67) -- (443.87,265.87) ;
\draw  [dash pattern={on 4.5pt off 4.5pt}]  (451.33,80.33) -- (465.87,266.53) ;

\draw (589.33,63.07) node [anchor=north west][inner sep=0.75pt]  [font=\footnotesize]  {$\omega $};
\draw (620,260.4) node [anchor=north west][inner sep=0.75pt]  [font=\footnotesize]  {$x$};
\draw (366,43.4) node [anchor=north west][inner sep=0.75pt]  [font=\footnotesize]  {$t$};

\end{tikzpicture}

    \caption{Structure of profiles described by Theorem \ref{thm:attprofiles3}.}
    \label{fig:att3}
\end{figure}

\begin{remark}\label{rem:LR0}
\label{rem:constr-LR=0}
Recalling that by~\eqref{eq:u-v-cont-ext} we have $\bs v[0, A, f_l] = \overline{A}$, $\bs u[0, B, f_r] = \overline{B}$, we can rephrase the constraint~\eqref{eq:2a3} as
\begin{equation}
    \label{eq:2a3-bis}
    \omega(\ms 0-) \geq~\bs v[0, A, f_l] , \qquad \omega(\ms 0+) \leq ~\bs u[0, B, f_r]\,. 
\end{equation}
{ 
Any profile $\omega$ 
satisfying the conditions of Theorem~\ref{thm:attprofiles3} is attainable by $AB$-entropy solutions that
don't contain shocks in $\{x<0\}$
or in $\{x>0\}.$

Since  by Lemma \ref{lemma:dualshocks} we have
$$
\lim_{\ms R \to 0^+} \bs u[\ms R, f_r, B] = \overline B, \qquad \lim_{\ms R \to 0^-} \bs v[\ms L, f_l, A] = \overline A,
$$
and because of Remark 4.8,
we can recover the conditions that characterize the pointwise constraints 
of attainable profiles in Theorem~\ref{thm:attprofiles3} 
 as limiting cases of the conditions of Theorems~\ref{thm:attprofiles}, \ref{thm:attprofilescrit}, classified in Remarks~\ref{rem:cases-thm-4.1}, \ref{rem:cases-thm-4.8}. Namely:
\begin{enumerate}
[leftmargin=20pt]
    \item[-] For a non critical connection, 
the condition~\eqref{eq:2a3} is the limit situation as $\ms L, \ms R \to 0$ of
    the CASE 2 of Remark \ref{rem:cases-thm-4.1}.
    \item[-] For a critical connection  with $A=\theta_l$, $B>\theta_r$, 
    the second condition of \eqref{eq:2a3} is 
    is the limiting situation 
   as $\ms R\to 0$ 
    of \textsc{Case~1} in Remark~\ref{rem:cases-thm-4.8}.
    The first condition of \eqref{eq:2a3}
    is trivially satisfied, because $\overline A= \theta_l$, and since $\ms L = 0$ 
by definition~\eqref{eq:LR-def}
implies
$\omega(0-) \geq \theta_l$.   
The case of a critical connection with $A<\theta_l$, $B=\theta_r$ is symmetric,
and can be recovered as 
limiting situation 
   as $\ms L\to 0$ 
    of \textsc{Case~2} in Remark~\ref{rem:cases-thm-4.8}.
\end{enumerate}
}
\end{remark}

\medskip

\begin{remark}
\label{rem:denseconds}
    By Remarks~\ref{rem:threecases1}, \ref{rem:LR0}, the conditions that characterize the pointwise constraints of attainable profiles provided by Theorems~\ref{thm:attprofiles}, \ref{thm:attprofilescrit} are essentially ``dense" in the set of all conditions 
    characterizing the pointwise constraints of any profile $\omega \in \mc A^{[AB]}(T)$
    (in the sense that the further conditions 
    provided by Theorems~\ref{thm:attprofiles2}, \ref{thm:attprofiles3}
    can be recovered via a limiting procedure as
    the parameters $\ms L, \ms R \to 0$).
\end{remark}

Combining Theorems \ref{thm:attprofiles}, \ref{thm:attprofilescrit}, \ref{thm:attprofiles2}, \ref{thm:attprofiles3}, with Theorem \ref{thm:backfordiscflux}, we obtain:
   \begin{thm}\label{thm:backfordiscfluxcycle}
   In the same setting of Theorem \ref{thm:backfordiscflux}, let
   $(A,B)$ be
 a connection. Then,
for every $T>0$,
 and for any $\omega\in  {\bf L^\infty}(\R)$,
 the following conditions are equivalent. 
\begin{enumerate}
    \item $\omega \in \mc A^{AB}(T)$.
    \smallskip
        \item $\mc \sabpT \circ \sabmT \omega = \omega\,$.
        \smallskip
    \item $\omega$ is an element of 
    the set $\msc A^{\ms L, \ms R}$ in~\eqref{eq:ALR-def1},   with $\ms L \leq 0$, $\ms R \geq 0$, that
    satisfies the conditions of Theorem \ref{thm:attprofiles}, 
    \ref{thm:attprofilescrit}, \ref{thm:attprofiles2}, or \ref{thm:attprofiles3}.
\end{enumerate}
Moreover, if $(A,B)$ is a non critical connection, i.e. if $A\neq \theta_l, B\neq \theta_r$, then the conditions (2) and~(3) are equivalent to
\begin{enumerate}
    \item[(1)']
    $\omega \in \mc A_{bv}^{[AB]}(T)$, where
    \begin{equation}
\label{eq:attset-bv}
    \mc A_{bv}^{[AB]}(T) \doteq  \big\{\mc \sabpT u_0 \; : \; u _0 \in BV_{loc}(\mathbb R)\big\}\,,
\end{equation}
\end{enumerate}
and it holds true
\begin{equation}
    \label{eq:atteq-bv-inf}
    \mc A^{[AB]}(T)=
    \mc A_{bv}^{[AB]}(T)\,.
\end{equation}
\end{thm} 

\begin{remark}
[Comparison with previous results]
\label{rem:otherresults}
Theorems~\ref{thm:attprofiles}, \ref{thm:attprofilescrit}, \ref{thm:attprofiles2}, \ref{thm:attprofiles3} yield the first  {\it complete 
characterization of the attainable set at time $T > 0$
in terms of Ole\v{\i}nik-type inequalities and unilateral constraints},
for critical and non critical connections. 
Partial results in this direction have been
recently obtained for \textit{strict subsets} of $\mc A^{AB}(T)$.
In particular, we refer to: 
\begin{itemize}
[leftmargin=14pt]
    \item the work \cite{anconachiri}, where 
    it is characterized
    only the  subset ${\mc A}_{\rm L}^{AB}(T) \subset \mc A^{AB}(T)$ given by
    $$\quad \ \ 
   {\mc A}_{\rm L}^{[AB]}(T) = \{\omega \in \mc A^{[AB]}(T) \; | \; \exists \; \text{$AB$-entropy solution} \; u \in \mr{Lip}_{\mr{loc}}((0,T) \times \mathbb R \setminus\{0\}) \;  : \; u(T, x) = \omega \}.
    $$
   In particular, all the profiles $\omega$ for which $\ms L \in \,]T \cdot f_l^\prime(A), 0[\,$ or $\ms R \in \,]0, T \cdot f_r^\prime(B)[\,$ are missing in the characterization provided in~\cite{anconachiri}.
    In fact, as observed in Remarks~\ref{rem:cases-thm-4.1},
    \ref{rem:cases-thm-4.8},
    \ref{rem:threecases1},
    an $AB$-entropy solutions leading to such profiles at time $T$
    must contain
a shock located in $\{x< 0\}$
or in $\{x> 0\}$, respectively,
in order to produce the discontinuity occurring in $\omega$ at $\ms L$ or $\ms R$.
    
    \item the work \cite{adimurthi2020exact}, 
    in which, whenever either $\ms L=0$, or $\ms R =0$,
    the set
    $\mc A^{[AB]}(T)$ is fully characterized  
    in terms 
    of triples (a monotone function and a pair of points)
    related to the
    Lax-Oleinik representation formula of solutions (obtained in~\cite{MR2028700} via
the Hamiton-Jacobi dual formulation).
    Instead, in the case of critical connections, all attainable profiles 
    with $\ms L < 0$ and $\ms R > 0$
    described by Theorem~\ref{thm:attprofilescrit}
    are missing in~\cite{adimurthi2020exact}.
    On the other hand, 
    when $\ms L < 0$, $\ms R > 0$
    and $(A, B)$,
    is a non critical connection, 
   only the profiles of 
    \textsc{Cases 3, 4}, discussed in Remark~\ref{rem:threecases}, are characterized in~\cite{adimurthi2020exact}, while  the ones of \textsc{Cases 1, 2, 1b, 2b} are missing. In fact, the profiles constructed in~\cite{adimurthi2020exact} with $\ms L< 0$, $\ms R> 0$ for non critical connections,  satisfy always the condition
    $\omega(x) = A$ for all $x \in (\ms L, 0)$, and $\omega(x) = B$ for all $x \in (0, \ms R)$,
    which is in general not fulfilled by profiles  of \textsc{Cases 1, 2, 1b, 2b}
    (cfr. Remark~\ref{rem:a-g}).
\end{itemize}
We point out that, as a byproduct of
the characterization
of $\mc A^{AB}(T)$ via Ole\v{\i}nik-type estimates, one can establish
uniform BV bounds on solutions to~\eqref{conslaw}, \eqref{discflux}
in the case of non critical connections,  
and on the flux of solutions to~\eqref{conslaw}, \eqref{discflux}
for general connections (see. Proposition~\ref{BVbound} in Appendix~\ref{app:stabconn}).
In turn such bounds yield the 
${\mathbf L^1_{loc}}$-Lipschitz continuity 
in time of $AB$-entropy solutions
(see the proof of Theorem \ref{theoremsemigroup}-$(v)$)
in Appendix~\ref{app:stabconn}).
\end{remark}


\section{Proof of Theorem~\ref{thm:backfordiscfluxcycle}}
\label{sec:proof-main-thm}

\subsection{Proof roadmap}
\label{roadmap}

Observe that if $(A,B)$ is a non critical connection, then recalling Definition~\ref{def:backop},
and relying on Proposition~\ref{BVbound}
in Appendix~\ref{app:stabconn},
we deduce that $\sabmT \omega\in BV_{loc}(\R)$
for all $\omega\in {\bf L^\infty}(\R)$.
Hence setting $u_0\doteq \sabmT \omega$,
we deduce immediately the implication 
$(2) \Rightarrow (1)'$.
On the other hand, since
$\mc A_{bv}^{[AB]}(T)\subset \mc A^{[AB]}(T)$, from the implication 
$(1) \Rightarrow (3)$, one deduces that
$(1)' \Rightarrow (3)$ holds as well.

Therefore, 
in order to establish Theorem \ref{thm:backfordiscfluxcycle}
it will be sufficient to prove the equivalence of the conditions (1), (2), (3).
We provide here a road map of the proof of $(1) \Rightarrow (2)\Rightarrow  (3) \Rightarrow (1)$. 
There are three main parts, which are somewhat independent one from the other.
\begin{itemize}
[leftmargin=43pt]
    \item[{\bf Part 1.}] {\bf The case of a non critical connection ${\bf (1)}\Rightarrow{\bf (3)}$.} 
    In Sections~\ref{sec:(1)-(3)}-\ref{subsec:reductionBV} we prove the implication $(1) \Rightarrow (3)$ of Theorem~\ref{thm:backfordiscfluxcycle} 
    when $(A, B)$ is  a \textit{non critical} connection. 
    The proof has a bootstrap-like structure,
    and it is divided in two steps.
    We first prove that ${\bf (1)}\Rightarrow{\bf (3)}$ under the regularity assumption \eqref{eq:Hhyp}  formulated below, and next we show that this regularity property always holds true.
    \begin{itemize}
    [leftmargin=10pt]
    \item[$\bullet$]{\it \nameref{sec:(1)-(3)}.} 
This is the \uline{first fundamental block} of our proof.
We prove in \S~\ref{sec:(1)-(3)} the implication $(1) \Rightarrow (3)$ for profiles $\omega\in \mc A^{[AB]}(T)$ that satisfy the BV condition:
     \begin{equation}\label{eq:Hhyp}
     \tag{H}
     \quad\exists~u_0\in {\bf L}^{\infty}(\R)\ : \ 
     \omega = \sabpT u_0\,, \quad\ \text{and}\quad\ \sabpt u_0\in BV_{\mr{loc}}(\R)\quad\forall~t>0\,.
    \end{equation}
    The derivation of
    the conditions of Theorem \ref{thm:attprofiles}, \ref{thm:attprofiles2}, and \ref{thm:attprofiles3}
     is obtained exploiting 
     as in~\cite{anconachiri} the non crossing property of genuine characteristics 
     in the domains $\{x>0, \, t>0\}$,
     $\{x<0, \, t>0\}$,
    together with the non existence of
    rarefactions emanating from the interface (cfr.~Appendix~\ref{app:no-rarefaction} and~\cite{adimurthi2020exact}).
    Two key novel points of the analysis here are:\\
    - a blowup argument, possible thanks to assumption~\eqref{eq:Hhyp},  to 
    %
    derive the Ole\v{\i}nik-type inequalities
    satisfied by $\omega$  in regions comprising points 
    with characteristics reflected by the interface $x=0$, and points with characteristics refracted by 
    $x=0$.  \\
    - a comparison argument
    (based on the duality of forward and backward shocks of~\S~\ref{sec:dual-prop}, and on the property of the states  $\bs u[\ms R, B, f_r]$, $\bs v[\ms L, A, f_l]$, defined in \S~\ref{defi:ur}, \ref{defi:vl})
    to establish the unilateral inequalities satisfied by $\omega$
    at points of discontinuity generated by shocks that {\it isolate} the interface
    $\{x=0\}$ 
    from the semiaxes $\{x<0\}$, $\{x>0\}$ (cfr. Remark~\ref{rem:caratt-u}).
    \item[$\bullet$]{\it \nameref{subsec:reductionBV}.}
    We prove in~\S~\ref{subsec:reductionBV}
    the implication $(1) \Rightarrow (3)$ 
    for every $\omega \in \mc A^{AB}(T)$
    by showing that 
    every \linebreak $\omega \in \mc A^{AB}(T)$ actually
    satisfies condition (H), and then the conclusion follows 
    by Part 1.a. 
    This is achieved: considering a sequence of functions $u_{n,0}\in BV(\R)$ that \linebreak $\bf L^1_{\mr{loc}}$-converge to $u_0\in {\bf L^\infty}(\R)$;
    observing that $\sabpt u_{n,0}\in BV(\R)$ (see~\cite{MR2743877,Garavellodiscflux});
    deriving uniform BV bounds on $\sabpT u_{n,0}$ based on the Ole\v{\i}nik-type inequalities enjoyed by $\sabpT u_{n,0}$ because of Part 1.a;
    relying on the $\bf L^1_{\mr{ loc}}$-stability 
    of the semigroup map $u_0\mapsto \sabpT u_0$ (see Theorem \ref{theoremsemigroup}-(iii)) and on the lower semicontinuity 
    of the total variation with respect to ${\bf L^1}$-convergence. 
    \end{itemize}
    \smallskip
    \item[{\bf Part 2.}] {\bf The case of a non critical connection ${\bf (3)} \Rightarrow {\bf (2)}\Rightarrow {\bf (1)}$.} 
    The implication 
    $(2) \Rightarrow (1)$\linebreak 
    of Theorem~\ref{thm:backfordiscfluxcycle}
    immediately follows observing that,
    by virtue of (2), one has
    $\omega = \sabpT u_0\in \mc A^{[AB]}(T)$, with $u_0\doteq \sabmT \omega$.
    Hence, in Sections~\ref{sec:(3)-(2)}-\ref{subsec:part3b} we  prove 
    only the implication $(3) \Rightarrow (2)$ of Theorem~\ref{thm:backfordiscfluxcycle}, in the case of a {non critical} connection $(A,B)$.
    This is the \uline{second fundamental block} of our proof,
    which consists in first showing that $(3) \Rightarrow (1)$, and next in proving that $(3) \Rightarrow (2)$.
    \begin{itemize}
    [leftmargin=10pt]
    \item[$\bullet$]{\it \nameref{sec:(3)-(2)}.}
    Given $\omega\in \msc A^{\ms L, \ms R}$
    satisfying the condition of Theorem~\ref{thm:attprofiles},
    we construct explicitly 
    in \S~\ref{sec:(3)-(2)}
    an $AB$-entropy admissible solution $u(x,t)\doteq \mc S_t^{[AB]+}u_0 (x)$, $u_0\in {\bf L^\infty}(\R)$, such that $u( \cdot,T) = \omega$. 
    The case where $\omega\in \msc A^{\ms L, \ms R}$
    satisfies the condition of Theorem~\ref{thm:attprofiles2}, or \ref{thm:attprofiles3} is entirely similar or simpler. The construction of $u_0$ and $u$ follows  a  by now  standard procedure (see \cite{anconachiri}, \cite{anconamars})
    in regions of $\{x<0\}$ or of $\{x>0\}$
    that are not influenced by waves reflected or refracted by the interface $x=0$.
    Namely, in these regions, one
     construct the solution $u$ along two type of lines that correspond to its characteristics:
     genuine characteristics $\vartheta_y$ 
     ending at points $(y,T)$, where 
     $u=\omega(y)$, in the case 
     $\omega$ is continuous at $y$; 
     compression fronts $\eta_{y,z}$
     connecting points $(z,0)$ and $(y,T)$,
     where $u=(f'_l)^{-1}(\frac{y-x}{T})$,
     if $y<0$, and $u=(f'_r)^{-1}(\frac{y-x}{T})$,
     if $y>0$,
     in the case 
     $\omega$ is discontinuous at $y$.
     A key novel point of the analysis here is the construction of $u$ in two polygonal regions around the interface $x=0$, 
     which relies on the properties of
     the {\it shock-rarefaction/rarefaction-shock wave patterns} established in \S~\ref{def:rsr-block}-\ref{def:lrs-block},
     which in turn are based on the 
     duality properties of forward/backward
     shocks derived in~\ref{sec:dual-prop}.
     Thanks to this construction, one can in
     particular explicitly produce $AB$-entropy solutions that attain at time $T$
     the profiles of~\textsc{Cases 1, 2, 1b, 2b} discussed in Remark~\ref{rem:threecases}, that are not present in~\cite{adimurthi2020exact} 
     (cfr. Remark~\ref{rem:otherresults}).
    \item[$\bullet$]{\it \nameref{subsec:part3b}.}
    Given $\omega\in \msc A^{\ms L, \ms R}$
    satisfying the conditions of Theorem~\ref{thm:attprofiles}, we show in \S~\ref{subsec:part3b} that $\omega$
    is a fixed point of the 
     backward-forward operator $\sabpT \circ \sabmT$.
    The case where $\omega\in \msc A^{\ms L, \ms R}$
    satisfies the condition of Theorem~\ref{thm:attprofiles2}, or \ref{thm:attprofiles3} is entirely similar.
    Building on the analysis
    pursued in the previous part, 
    in order to prove that 
    $\omega = \mc \sabpT \circ \sabmT \omega$
    it is sufficient to show that, if $u_0$ is the initial datum of the $AB$-entropy solution $u(x,t)$ constructed in Part 2.a,
    then  one has $u_0=\mc S^{[AB]-}_T \omega$.
    This is again achieved 
    exploiting the duality properties of forward/backward shocks derived in~\ref{sec:dual-prop},
    and the structural properties of
     the {\it shock-rarefaction/rarefaction-shock wave patterns} established in \S~\ref{def:rsr-block}-\ref{def:lrs-block}.
    \end{itemize}
    \smallskip
    \item[{\bf Part 3.}] {\bf The case of a critical connection
    ${\bf (1) \Leftrightarrow (2)\Leftrightarrow  (3)}$.} 
    In Sections~\ref{sec:criticalcases-a}, \ref{sec:criticalcases-b}, \ref{sec:criticalcases-c} we recover the equivalence of the conditions (1), (2), (3) of Theorem~\ref{thm:backfordiscfluxcycle}  in the case of critical
    connections, invoking the validity of this equivalence for non critical connections established in Parts~1-2.
    The proof is divided in three steps.
    \begin{itemize}
    [leftmargin=10pt]
    \item[$\bullet$]{\it \nameref{sec:criticalcases-a}.}
    In \S~\ref{sec:criticalcases-a}
    we prove the implication $(1)\Rightarrow (2)$, 
    relying on the $\bf L^1_{\mr{loc}}$-stability 
    of the maps $(A,B, u_0)\mapsto \sabpT u_0$, $(A,B, u_0)\mapsto \sabmT u_0$ (see~Theorem \ref{theoremsemigroup}-(iv)
    and Definition~\ref{def:backop}).
    The reverse implication  is immediate
as observed in Part 2.
    \item[$\bullet$]{\it \nameref{sec:criticalcases-b}.}
    In \S~\ref{sec:criticalcases-b}
    we prove the implication $(1)\Rightarrow (3)$, \linebreak
    relying on the ${\bf L^1}$-weak stability 
    of the maps 
    $(A,B)\mapsto f_l(u_l)$,  
    $(A,B)\mapsto f_r(u_r)$,
    where $u_l, u_r$ denote, respectively the left and right states of $u(x,t)\doteq \sabpt u_0(x)$ at $x=0$ (see Corollary~\ref{cor:fluxtraces-stab}), and on the lower/upper ${\bf L^1}$-semicontinuity property of solutions to conservation laws with uniformly convex flux (see Lemma~\ref{lemma:chara} in Appendix~\ref{app:uplwsmicsolns}). 
    \item[$\bullet$]{\it \nameref{sec:criticalcases-c}.}
    In \S~\ref{sec:criticalcases-c}
    we prove the implication 
    $(3)\Rightarrow (1)$\linebreak 
    exploiting again the ${\bf L^1}_{\ms loc}$-stability of the semigroup map of Theorem~\ref{theoremsemigroup}-(iv), and using a perturbation argument.
    Namely, given  $\omega\in \msc A^{\ms L, \ms R}$
    satisfying the conditions of Theorem~\ref{thm:attprofilescrit}, \ref{thm:attprofiles2}, or \ref{thm:attprofiles3}, we construct a sequence $\{\omega_n\}_{n}$
of perturbations
of $\omega$
with the property that $\omega_n\stackrel{{\bf L^1}}{\to}\omega$,
and 
$\omega_n \in \mc A^{[A_nB_n]}(T)$, for 
    a sequence of non critical connections $\{(A_n,B_n)\}_n$.
    This is another key point of our analysis, since it provides a general explicit procedure to approximate an attainable profile for a critical connection by attainable profiles for non critical connections.
    \end{itemize}
\end{itemize}
{ \begin{remark}
%
In the case of critical connections,
one may provide a direct proof of the
implications $(2) \Rightarrow (1)$,
$(3) \Rightarrow (1)$, $(3) \Rightarrow (2)$ of Theorem \ref{thm:backfordiscfluxcycle} with similar arguments as the ones used in the case of non critical connections. Only the implication 
$(1) \Rightarrow (3)$ in the case of
critical connections cannot be 
directly established
with the same line of proof followed in 
\S~\ref{sec:(1)-(3)} for non critical connection. The reason is twofold. 
On one hand 
we
cannot rely on the 
property of non existence of
    rarefactions emanating from the interface, since we establish in 
    Appendix~\ref{app:no-rarefaction} 
    this property only in the case of non critical connections.
    On the other hand 
we cannot exploit
the uniform $BV_{\mr{loc}}$ bounds to
perform the blowup argument of \S~\ref{sec:1b12}, since they are
enjoyed by $AB$-entropy solutions only when the connection is non critical
(see \S~\ref{sec:BVboundsABsol}).
An alternative, direct proof of
$(1) \Rightarrow (3)$
can be obtained
relying  on the property of preclusion of rarefactions emanating from the interface derived in~\cite{adimurthi2020exact}
for general connections.
Using this property, it seems 
reasonable that one may then establish 
the Ole\v{\i}nik-type estimates 
that characterize the attainable profiles 
for critical connections performing a longer, technical analysis of
the structure of characteristics that avoids the blow up argument of~\S~\ref{sec:1b12}.
\end{remark}
}

\medskip

\subsection{Part 1.a -
$(1) \Rightarrow (3)$ for non critical connections assuming~\eqref{eq:Hhyp} }
\label{sec:(1)-(3)}

In this Subsection,
given an element $\omega$
of the set $\mc A^{[AB]}(T)$
for a non critical connection
$(A, B)$, assuming 
that 
$\omega$ satisfies~\eqref{eq:Hhyp},
we will show that $\omega$ fulfills condition (3) of Theorem~\ref{thm:backfordiscfluxcycle}.
Recalling~\eqref{eq:att-set-decomp}, 
this is equivalent to show that,
letting
\begin{equation}
\label{eq:LR-def-23}
\ms L \doteq \ms L[\omega,f_l]\,,
     \qquad\quad  
     \ms R \doteq  \ms R[\omega, f_r],
\end{equation}
be quantities defined as in~\eqref{eq:LR-def},
it holds true that:
\begin{itemize}
    \item[2a-i)] If $\ms L < 0$, $\ms R > 0$, and if $\omega$ satisfies \eqref{eq:Hhyp}, then $\omega$ satisfies the conditions of Theorem \ref{thm:attprofiles};
    \item[2a-ii)] If $\ms L = 0$, $\ms R > 0$ or viceversa, and if $\omega$ satisfies \eqref{eq:Hhyp}, then $\omega$ satisfies the conditions\linebreak  of Theorem~\ref{thm:attprofiles2};
    \item[2a-iii)] If $\ms L = 0$, $\ms R = 0$, then $\omega$ satisfies the conditions of Theorem~\ref{thm:attprofiles3}.
\end{itemize}
We will  prove 2a-i) in \S~\ref{sec:3b-lr2}-\ref{sec:1b12},
while 2a-ii) is proven in~\S~\ref{sec:othercases},
and 2a-iii) is discussed in~\S~\ref{sec:othercases-2}.
The further assumption that $\omega$ satisfies \eqref{eq:Hhyp} is needed only to 
ensure the existence of the one-sided limits 
$\omega(0\pm)$, and to show that $\omega$ satisfies~\eqref{eq:1b1}-\eqref{eq:1b12} in case 2a-i), and 
\eqref{eq:1b2}, \eqref{eq:1b2-prime} in case 2a-ii).
\smallskip

Throughout the subsection we will assume that 
\begin{equation}
\label{eq:att-prof-23}
    \omega = \mc S_T^{[AB]+}u_0, \qquad u_0 \in \mathbf L^{\infty}(\mathbb R)\,,
\end{equation}
and we set $u(x,t)\doteq S_t^{[AB]+}u_0(x)$, $x\in\R$, $t\geq 0$.
Under assumption \eqref{eq:Hhyp} 
there exist the limits  $u(0\pm, t)$, for all $t>0$.
We let $u_l(t), u_r(t)$ denote the left and right traces 
at~$x=0$ of $u(x,t)$, $t>0$.


\subsubsection{
\rm 
($\ms L<0$, $\ms R>0$, proof of \eqref{eq:3b-lr2})}
\label{sec:3b-lr2} 
The inequalities $\omega(\ms L-)\geq \omega(\ms L+)$, 
$\omega(\ms R-)\geq \omega(\ms R+)$
are the Lax conditions 
which are satisfied since $u$ is an entropy admissible solution of
 the conservation law 
$u_t+f_l(u)_x=0$, on $x<0$,
and of
$u_t+f_r(u)_x=0$, on $x>0$,
and the fluxes $f_l, f_r$
are convex.

\subsubsection{\rm 
($\ms L<0$, $\ms R>0$, proof of \eqref{eq:1a})}
\label{sec:1a} 
By definition~\eqref{eq:LR-def}, \eqref{eq:LR-def-23} of $\ms L$,
$\ms R$,
it follows that backward characteristics for $u$
starting at $(x,T)$,
with $x  \in ]-\infty, 0[\, \cup\, ]\ms R, +\infty[$,
never crosses the interface $x=0$.
Thus, we  recover the Ole\v{\i}nik estimates~\eqref{eq:1a} 
as a classical property
of solutions 
to conservation laws with strictly convex flux, which 
follows from the fact that genuine characteristics never intersect in the interior of the domain
(e.g. see~\cite[Lemma 3.2]{anconachiri}).


\subsubsection{\rm 
($\ms L<0$, $\ms R>0$, first part of the proof of \eqref{eq:2a})}
\label{sec:2a} Letting $\bs u[\ms R, B, f_r]$  be the
constant defined as
in~\eqref{eq:urbf-def} with $f=f_r$,
we will prove 
the implication
\begin{equation}
\label{eq:2aproof}
    \ms R \in ]0, T\cdot f^{\prime}(B)[ \quad \Longrightarrow \quad \omega(\ms R+) \leq \bs u[\ms R, B, f_r],
\end{equation}
assuming  
\begin{equation}\label{eq:contraip}
   \ms R \in ]0, T\cdot f^{\prime}(B)[\, ,\qquad\quad  \omega(\ms R+) > \bs u[\ms R, B, f_r]\,,
\end{equation}
and showing that~\eqref{eq:contraip} leads to a contradiction. To complete the proof of~\eqref{eq:2a} we will show 
in \S~\ref{sec:2bl-r1}
that
\begin{equation}
\label{eq:2aproof-2}
    \ms R \in ]0, T\cdot f^{\prime}(B)[ \quad \Longrightarrow \quad \bs u[\ms R, B, f_r]\leq \omega(\ms R-).
\end{equation}
The proof of the first implication in~\eqref{eq:2a} is obtained in entirely similar way.

We divide the proof of~\eqref{eq:2aproof} in two steps.
In the first step we construct 
the leftmost characteristic curve $\xi_{\strut {\ms R}}$ that
starts on the interface $x=0$
and reaches the point
$(\ms R,T)$, remaining
in the region $\{x>0\}$,
with the property
that all maximal backward characteristics 
starting on $\xi_{\strut {\ms R}}$ don't cross the interface~$x=0$.
In the second step,
we show that
$\xi_{\strut {\ms R}}$ is located on the left 
of the shock curve $\bs x$ constructed as in \S~\ref{def:rsr-block}
that emanates from the interface $x=0$ and reaches the point 
$(\ms R,T)$. 
{ Thanks to the assumption~\eqref{eq:contraip}
this leads to a contradiction} in accordance
with the characterizing 
property of  $\bs u[\ms R, B, f_r]$ discussed in Remark~\ref{rem:caratt-u}.
\vspace{0.5cm}

\noindent
\textbf{Step 1} 
Consider the map 
$\xi_{\strut {\ms R}}: [\tau_{\strut {\ms R}},T]\to [0,+\infty[$
defined by setting
\begin{equation}\label{eq:hr}
\begin{aligned}
    \xi_{\strut {\ms R}}(t) &\doteq  \inf \big\{ R > 0 \; : \; x-t \cdot 
    f_r^{\prime}(u{ (x,t)})
   \geq 0  \quad \forall \; x \geq R\big\}, 
    \quad t\geq 0\,,
\\
\noalign{\medskip}
    \tau_{\strut {\ms R}}&\doteq\inf\big\{ t \in [0, T] \; : \; \xi_{\strut {\ms R}}(s) > 0 \quad \forall \; s \in [t, T]\big\}\,.
\end{aligned}
\end{equation}
Notice that by definition~\eqref{eq:hr} 
we have
\begin{equation}
\label{eq:bc-cond-5}
    \xi_{\strut {\ms R}}(\tau_{\strut {\ms R}}) = 0,\qquad \xi_{\strut {\ms R}}(T) = \ms R,\qquad
    \xi_{\strut {\ms R}}(t)>0
    \quad\ \forall~t\in\, ]\tau_{\strut {\ms R}} , T],  
\end{equation}
and that $\xi_{\strut {\ms R}}$ is a backward characteristic for $u$ starting at $(\ms R,T)$, so that
it holds true
(e.g. see~\cite{dafermosgenchar})
\begin{equation}\label{eq:hRH}
\xi'_{\strut {\ms R}}(t)=
\begin{cases}
    f'_r(u(\xi_{\strut {\ms R}}(t)\pm,t)
    \qquad &\text{if}\quad u(\xi_{\strut {\ms R}}(t)-,t)=u(\xi_{\strut {\ms R}}(t)+,t),
    \\
    \noalign{\smallskip}
    \lambda_r\big(u(\xi_{\strut {\ms R}}(t)-,t),\, u(\xi_{\strut {\ms R}}(t)+,t))\big)
    \qquad &\text{if}\quad u(\xi_{\strut {\ms R}}(t)-,t)\neq u(\xi_{\strut {\ms R}}(t)+,t)\,,
\end{cases}
\end{equation}
where
\begin{equation}
\label{eq:lambdar-def}
    \lambda_r(u,v)\doteq 
    \frac{f_r(v)-f_r(u)}{v-u},\qquad u,v\in\R, \ u\neq v\,.
\end{equation}
%
We shall provide now a lower bound on the slope of $\xi_{\strut {\ms R}}$.
Let $t_0 \in \,]\tau_{\strut {\ms R}}, T]$, and observe that 
by  definition~\eqref{eq:hr}
it follows that
the minimal backward characteristic starting at $(\xi_{\strut {\ms R}}(t_0), t_0)$ must cross the interface $x=0$
at some non-negative time. 
Since such a characteristic is genuine and
has slope $f'_r(u(\xi_{\strut {\ms R}}(t_0)-,t_0))\geq 0$,
and because of the 
$AB$-entropy condition~\eqref{ABtraces},
it follows that $f_r(u(\xi_{\strut {\ms R}}(t_0)-, t_0))\geq f_r(B)$
and $u(\xi_{\strut {\ms R}}(t_0)-, t_0) \geq\theta_r$.
Hence, it holds true
\begin{equation}
\label{eq:geqBineq}
    u(\xi_{\strut {\ms R}}(t_0)-,  t_0) \geq  B\,.
\end{equation}
On the other hand, by definition~\eqref{eq:hr} we have
\begin{equation}
\label{eq:leqineq-uxiR}
    f'_r(u(\xi_{\strut {\ms R}}(t_0)+,t_0))\leq \xi_{\strut {\ms R}}(t_0)/t_0\,.
\end{equation}
Thus, letting $\vartheta_{\xi_{{\ms R}}(t_0),+}$ denote the maximal backward characteristic 
starting at
$(\xi_{\strut {\ms R}}(t_0),  t_0)$,
because of~\eqref{eq:leqineq-uxiR} it holds true
\begin{equation}
\label{eq:char1-bound1}
    \vartheta_{\xi_{{\ms R}}(t_0),+}(0)=
    \xi_{\strut {\ms R}}(t_0) -  t_0 \cdot  f_r^{\prime}(u(\xi_{\strut {\ms R}}(t_0)+, t_0))\geq 0\,,
\end{equation}
and~\eqref{eq:bc-cond-5} implies 
\begin{equation}
\begin{aligned}
    \vartheta_{\xi_{{\ms R}}(t_0),+}(t)&>0\qquad
    \forall~t\in\,]0, t_0]\,.
\end{aligned}
\end{equation}
Moreover, 
observe that by
the properties of backward characteristics,
and by definition~\eqref{eq:hr},
the maximal backward characteristics $\vartheta_{\ms R,+}$
starting at $(\ms R, T)$
satisfies
\begin{equation*}
    \xi_{\strut {\ms R}}(t)\leq 
    \vartheta_{\ms R, +}(t)\qquad \forall~t\in [\tau_{\strut {\ms R}}, T]\,,
\end{equation*}
and, in particular, one has
\begin{equation}
    \label{eq:char2-bound2}
    \xi_{\strut {\ms R}}(t_0)\leq 
    \vartheta_{\ms R, +}(t_0)\,.
\end{equation}
Since maximal 
{  \st{genuine?}}
backward characteristics 
cannot intersect in the interior of the domain, 
it follows from~\eqref{eq:char2-bound2} that
%
\begin{equation}
\label{eq:char2-bound3}
        \xi_{\strut {\ms R}}(t_0) -  t_0 \cdot  f_r^{\prime}(u(\xi_{\strut {\ms R}}(t_0)+, t_0))=\vartheta_{\xi_{{\ms R}}(t_0),+}(0) \leq \vartheta_{\ms R, +}(0) =\ms R- T \cdot f_r^{\prime}
        (\omega(\ms R+))\,.
\end{equation}
In turn, \eqref{eq:char2-bound3} 
yields
\begin{equation}
\label{eq:xietar-est-3}
    \xi_{\strut {\ms R}}(t)-\ms R + T \cdot f_r^{\prime}(\omega(\ms R+))\leq 
     t_0 \cdot  f_r^{\prime}(u(\xi_{\strut {\ms R}}(t_0)+, t_0))\,.
\end{equation}
Moreover, 
one has
\begin{equation}
\label{eq:xietar-est-1}
    \frac{\xi_{\strut {\ms R}}(t_0)-
    \vartheta_{\ms R, +}(0)}{t_0}\leq
    \vartheta'_{\ms R, +}=f'_r(\omega(\ms R+))\,.
\end{equation}
Since the definition~\eqref{eq:LR-def} of $\ms R$ and ~\eqref{eq:contraip} imply
$f'_r(\omega(\ms R+))
    \leq{\ms R}/{T}<f'_r(B)$,
    we deduce from~\eqref{eq:xietar-est-1} that
    \begin{equation}
    \label{eq:xietar-est-2}
        \frac{\xi_{\strut {\ms R}}(t_0)-\ms R + T \cdot f_r^{\prime}(\omega(\ms R+))}{t_0}<
    f'_r(B).
    \end{equation}
By the monotonicity of $f_r^{\prime}$, in turn
the estimates~\eqref{eq:xietar-est-3}, \eqref{eq:xietar-est-2}
yield
\begin{equation}
\label{eq:xiR-ineq2}
\begin{aligned}
(f_r^{\prime})^{-1} \left( \frac{\xi_{\strut {\ms R}}(t_0)-\ms R + T \cdot f_r^{\prime}(\omega(\ms R+))}{t_0}   \right)
&\leq u(\xi_{\strut {\ms R}}(t_0)+,  t_0),
\\
\noalign{\smallskip}
(f_r^{\prime})^{-1} \left( \frac{\xi_{\strut {\ms R}}(t_0)-\ms R + T \cdot f_r^{\prime}(\omega(\ms R+))}{t_0}   \right)&<B.
\end{aligned}
\end{equation}
Therefore, recalling~\eqref{eq:hRH}, \eqref{eq:lambdar-def}, 
and because of the convexity of $f_r$,
we derive
from~\eqref{eq:contraip},\eqref{eq:geqBineq},
\eqref{eq:xiR-ineq2}, that
\begin{equation}
\label{eq:fullspeedineqh}
    \xi_{\strut {\ms R}}'(t_0)>
    \lambda_r \!\left(\!     (f_r^{\prime})^{-1} \!\!\left( \frac{\xi_{\strut {\ms R}}(t_0)-\ms R + T \cdot f_r^{\prime}(\bs u[\ms R, B, f_r])}{t_0}   \right)\!, B \!\right)
    \qquad\forall~t_0\in\,]\tau_{\strut {\ms R}}, T]\,.
\end{equation}


\vspace{0.5cm}

\begin{figure}
    \centering
   
\tikzset{every picture/.style={line width=0.75pt}} 

\begin{tikzpicture}[x=0.75pt,y=0.75pt,yscale=-0.9,xscale=0.9]

\draw [line width=1.5]    (228.26,248.96) -- (228.26,30.2) ;
\draw [shift={(228.26,26.2)}, rotate = 90] [fill={rgb, 255:red, 0; green, 0; blue, 0 }  ][line width=0.08]  [draw opacity=0] (6.97,-3.35) -- (0,0) -- (6.97,3.35) -- cycle    ;
\draw [line width=1.5]    (130.13,250.2) -- (492.53,249.21) ;
\draw [shift={(496.53,249.2)}, rotate = 179.84] [fill={rgb, 255:red, 0; green, 0; blue, 0 }  ][line width=0.08]  [draw opacity=0] (6.97,-3.35) -- (0,0) -- (6.97,3.35) -- cycle    ;
\draw    (127.47,61.8) -- (490.53,61.2) ;
\draw [color={rgb, 255:red, 74; green, 144; blue, 226 }  ,draw opacity=1 ]   (228.67,189.07) .. controls (244.53,165.87) and (262.51,142.3) .. (282.53,121.2) .. controls (302.56,100.1) and (323.87,79.87) .. (348.67,62.07) ;
\draw    (348.67,62.07) -- (228.53,99.2) ;
\draw    (348.67,62.07) -- (329.2,247.87) ;
\draw  [dash pattern={on 4.5pt off 4.5pt}]  (227.87,155.87) .. controls (233.87,150.53) and (233.81,140.21) .. (241.2,133.2) .. controls (248.59,126.19) and (258.14,119.97) .. (267.87,113.87) .. controls (277.59,107.76) and (282.67,95.28) .. (291.87,87.87) .. controls (301.07,80.45) and (335.67,74.74) .. (348.67,62.07) ;
\draw    (227.87,155.87) -- (452.53,247.2) ;
\draw [color={rgb, 255:red, 74; green, 144; blue, 226 }  ,draw opacity=1 ] [dash pattern={on 4.5pt off 4.5pt}]  (228.67,189.07) -- (379.87,248.53) ;
\draw    (262.53,115.2) -- (278.77,41.15) ;
\draw [shift={(279.2,39.2)}, rotate = 102.37] [color={rgb, 255:red, 0; green, 0; blue, 0 }  ][line width=0.75]    (6.56,-1.97) .. controls (4.17,-0.84) and (1.99,-0.18) .. (0,0) .. controls (1.99,0.18) and (4.17,0.84) .. (6.56,1.97)   ;

\draw (463.33,43.4) node [anchor=north west][inner sep=0.75pt]  [font=\footnotesize]  {$\omega $};
\draw (498.53,252.6) node [anchor=north west][inner sep=0.75pt]  [font=\footnotesize]  {$x$};
\draw (242.67,17.73) node [anchor=north west][inner sep=0.75pt]  [font=\footnotesize]  {$t$};
\draw (296,254.8) node [anchor=north west][inner sep=0.75pt]  [font=\footnotesize]  {$\zeta _{\mathsf{R}}{}_{,+}( 0)$};
\draw (434.67,255.47) node [anchor=north west][inner sep=0.75pt]  [font=\footnotesize]  {$\zeta _{0}{}_{,}{}_{+}( 0)$};
\draw (205.33,147.8) node [anchor=north west][inner sep=0.75pt]  [font=\footnotesize]  {$\tau _{\mathsf{R}}$};
\draw (267.33,18.47) node [anchor=north west][inner sep=0.75pt]  [font=\footnotesize]  {$\xi _{\mathsf{R}}( t)$};
\draw (209.33,180.47) node [anchor=north west][inner sep=0.75pt]  [font=\footnotesize]  {$\mathbf{s}$};
\draw (344,42.13) node [anchor=north west][inner sep=0.75pt]  [font=\footnotesize]  {$\mathsf{R}$};
\draw (280.53,125.27) node [anchor=north west][inner sep=0.75pt]  [font=\footnotesize,color={rgb, 255:red, 74; green, 144; blue, 226 }  ,opacity=1 ]  {$\mathbf{x}$};
\draw (373.33,256.47) node [anchor=north west][inner sep=0.75pt]  [font=\footnotesize,color={rgb, 255:red, 74; green, 144; blue, 226 }  ,opacity=1 ]  {$-\mathsf{L}$};

\end{tikzpicture}

    \caption{Illustration of the proof in \S~\ref{sec:2a}. The black lines are characteristics of the solution $u$, that cross inside the domain and therefore lead to a contradiction. The blue lines are the comparison curves.}
    \label{fig:sec2a}
\end{figure}

\noindent
\textbf{Step 2} (Comparison with an extremal shock). 
Let $\bs y[\ms R, B, f_r](\cdot)$ be the function defined in
\S~\ref{defi:ur}
with $f=f_r$, 
set
\begin{equation}
\label{eq:def_l_32}
    \ms L \doteq \bs y[\ms R, B, f_r](T),
\end{equation}
and consider the function
\begin{equation}
    \bs x[\ms L, \overline B, f_r](t),\qquad\quad 
    t \in \big[ \bs s[\ms L, \overline B, f_r] , T\big]\,,
\end{equation}
defined as in \S~\ref{defi:vl},
with $A = \overline B$ ($\overline B$ as in~\eqref{eq:bar-AB-def}),
and $f=f_r$.
By definition \eqref{eq:xL},
and applying Lemma \ref{lemma:dualshocks},
it holds true
\begin{equation}
\label{eq:fwsh-cond23}
 \bs x[\ms L, \overline B, f_r](\bs s[\ms L, \overline B, f_r])=0,
 \qquad\quad
\bs x[\ms L, \overline B, f_r](T) = \ms R,
\end{equation}
and 
\begin{equation}
\label{eq:fwsh-cond24}
     \frac{d}{dt} {\bs x}[\ms L, \overline B, f_r](t) = \lambda_r \!\left(\!(f_r^{\prime})^{-1}\bigg(\frac{\bs x[\ms L, \overline B, f_r](t)+\ms L}{t}\bigg), B\,\right), \qquad t \in \big[ \bs s[\ms L, \bar B, f] , T\big].
\end{equation}
Moreover,  because of~\eqref{eq:urbf-def}, \eqref{eq:def_l_32}, we have
\begin{equation}
    \ms L= T \cdot f_r^{\prime}(\bs u[\ms R, B, f_r])-\ms R\,.
\end{equation}
Recall that by~\eqref{eq:bc-cond-5}, \eqref{eq:fwsh-cond23}, it holds 
$$
\xi_{\ms R}(T) = \ms R  = {\bs x}[\ms L, \overline B, f_r](T).
$$
Then, by virtue of~\eqref{eq:fullspeedineqh}, \eqref{eq:fwsh-cond24}, a  comparison argument yields
\begin{equation}
\label{eq:comparison}
    \xi_{\strut {\ms R}}(t) < \bs x[\ms L, \overline B, f_r](t), \qquad \forall \; t \in \big[\max\big\{\!\tau_{\strut {\ms R}},\, \bs s[\ms L, \overline B, f_r]\big\},\, T\big[\,.
\end{equation}
Notice that, if $\bs s[\ms L, \overline B, f_r] \geq \tau_{\strut {\ms R}}$,
then because of~\eqref{eq:fwsh-cond23},
\eqref{eq:comparison}, 
and since $\xi_{\strut {\ms R}}(t)\geq 0$, 
{ for all $t\in [\tau_{\strut {\ms R}} , T]$},
we find the contradiction $0\leq \xi_{\strut {\ms R}}(\bs s[\ms L, \overline B, f_r])<0$.
Hence it must be 
\begin{equation}
\label{eq:stau-est-2}
    \bs s[\ms L, \overline B, f_r] < \tau_{\strut {\ms R}}.
\end{equation}
Next, observe that by definition~\eqref{eq:hr}
and because of~\eqref{eq:bc-cond-5}, we have  
$u(0 +, \tau_{\strut {\ms R}}) \leq \theta_r$. Thus, by virtue of the 
$AB$-entropy condition~\eqref{ABtraces},
it follows that
 $u(0 +, \tau_{\strut {\ms R}}) \leq \overline B$. 
 Then, letting \linebreak $\zeta_{0,+}: [0, \tau_{\strut {\ms R}}]\to [0,+\infty[$
 denote
 the maximal backward characteristic 
 starting at $(0,\tau_{\strut {\ms R}})$, one has 
\begin{equation}\label{eq:secondcontraineq}
    \zeta_{0,+}(0) = - \tau_{\strut {\ms R}} \cdot f_r^{\prime}(u(0+,\tau_{\strut {\ms R}}))  \geq - \tau_{\strut {\ms R}} \cdot f_r^{\prime}(\overline B).
\end{equation}
On the other hand, 
by virtue of~\eqref{eq:contraip},
\eqref{eq:def_l_32},
\eqref{eq:stau-est-2}, 
and recalling the definitions~\eqref{eq:urbf-def}, \eqref{eq:slaf-def} of $\bs u[\ms R, B, f_r]$, $\bs s[\ms L, \overline B, f_r]$, we find that 
the maximal backward characteristic 
$\vartheta_{\ms R,+}
: [0,T]\to [0,+\infty[$
from $(\ms R, T)$
satisfies
\begin{equation}
\label{eq:secondcontraineq-2}
\begin{aligned}
    \vartheta_{\ms R,+}(0) = \ms R-T \cdot f_r^{\prime}(\omega(\ms R+))  & <
    \ms R-T \cdot f_r^{\prime}(\bs u[\ms R, B, f_r]) 
    \\
    &=-\bs y[\ms R, B, f_r](T)
    \\
    &=  -\bs s[\ms L, \overline B, f_r] \cdot f_r^{\prime}(\overline B)\
    \\
    & <
    -\tau_{\strut {\ms R}} \cdot f_r^{\prime}(\overline B) \,. 
    \end{aligned}
\end{equation}
Thus, we deduce from~\eqref{eq:secondcontraineq}-\eqref{eq:secondcontraineq-2} that 

\begin{equation}
\label{eq:char1-ineq23}
    \vartheta_{\ms R,+}(0)< \zeta_{0,+}(0),
\end{equation}
while~\eqref{eq:char2-bound2} yield 
\begin{equation}
\label{eq:char2-ineq23}
    \vartheta_{\ms R,+}(\tau_{\strut {\ms R}})>0 =\zeta_{0,+}(\tau_{\strut {\ms R}}).
\end{equation}
The inequalities~\eqref{eq:char1-ineq23}-\eqref{eq:char2-ineq23}
imply that the genuine characteristics 
$\zeta_{0,+}, \vartheta_{\ms R,+}$
 intersect each other in the interior of the domain, which gives a contradiction
 and thus completes the proof of the implication~\eqref{eq:2aproof}. 

\subsubsection{\rm 
($\ms L<0$, $\ms R>0$, proof of~
\eqref{eq:2b-lr2}-\eqref{eq:2b-lr2-2})
}
\label{sec:2bl-r1r2}
We will prove only the implication~\eqref{eq:2b-lr2-2}, the
proof of~\eqref{eq:2b-lr2} being entirely similar. 
Let $\widetilde {\ms L}\doteq \widetilde{\ms L}[\omega,f_l,f_r,A,B]$
    be the constant
    in~
    \eqref{eq:Ltildedef},
and assume that 
\begin{equation}
\label{eq:assumpt-534}
    \ms R \in \,]0, T\cdot f_r^{\prime}(B)[\,,\qquad\quad \ms L < \widetilde{\ms L}\,.
\end{equation}
\smallskip

\noindent
\textbf{Step 1.} \big(proof of:\, $\omega(x) \leq A$ in $]\ms L, 0[\,$\big).\\
By definition~\eqref{eq:LR-def}, \eqref{eq:LR-def-23} of $\ms L$,
it follows that backward genuine characteristics starting at points $(x,T)$,
with $x \in \,]\ms L, 0[$ of continuity for $\omega$, must cross the interface $x=0$ at some non-negative time. 
Since such characteristics 
have slope $f'_l(\omega(x))\leq 0$,
and because of the 
$AB$-entropy condition~\eqref{ABtraces},
it follows that $f_l(\omega(x))\geq f_l(A)$
and $\omega(x) \leq\theta_l$
at any point $x \in \,]\ms L, 0[$ of continuity for $\omega$.
Hence, we have
$\omega(x\pm)\leq A$
for all $x\in \,]\ms L, 0[\,$.
\smallskip

\noindent
\textbf{Step 2.} \big(proof of:\, $\omega(x) = A$ in $]\widetilde {\ms L}, 0[\,$\big).\\
In a similar way to \eqref{eq:hr}, consider  
the map
$\xi_{\strut {\ms L}}: [\tau_{\strut {\ms L}},T]\to ]-\infty, 0[$
defined symmetrically by setting
\begin{equation}\label{eq:hl}
\begin{aligned}
    \xi_{\strut {\ms L}}(t) &\doteq  \sup \big\{ L < 0 \; : \; x-t \cdot f_l^{\prime}(u{ (x,t)})
    \leq 0  \quad \forall \; x \leq L\big\}, 
    \quad t\geq 0\,,
\\
\noalign{\medskip}
    \tau_{\strut {\ms L}}&\doteq\inf\big\{ t \in [0, T] \; : \; \xi_{\strut {\ms L}}(s) < 0 \quad \forall \; s \in [t, T]\big\}\,.
\end{aligned}
\end{equation}
Notice that by definition~\eqref{eq:hl}
we have
$$\xi_{\strut {\ms L}}(\tau_{\strut {\ms L}}) = 0, \qquad \xi_{\strut {\ms L}}(T) = \ms L, \qquad \xi_{\strut {\ms L}}(t) < 0 \ \quad \forall \; t \in \,]\tau_{\strut {\ms L}},\,T].$$
We claim that
\begin{equation}\label{eq:timeclaim}
    \begin{aligned}
     \tau_{\strut {\ms L}}\leq 
    \tau_{\strut {\ms R}}
        \quad\ \Longrightarrow \ \quad \tau_{\strut {\ms R}} \leq \bs \tau[\ms R, B, f_r],
    \end{aligned}
\end{equation}
where $\bs\tau[\ms R, B, f_r]$ is the constant defined as in~\eqref{eq:tau-rbf-def}, with $f=f_r$.
We will prove the implication~\eqref{eq:timeclaim}
with similar arguments to the proof of~\eqref{eq:2a} in \S~\ref{sec:2a},
assuming
\begin{equation}
    \label{eq:tau-ineq-29}
    \tau_{\strut {\ms L}}\leq 
    \tau_{\strut {\ms R}}\,,\qquad\quad
    \tau_{\strut {\ms R}}>\bs \tau[\ms R, B, f_r]
    \,,
\end{equation}
and showing that~\eqref{eq:tau-ineq-29}
lead to a contradiction.
\smallskip

Since $\tau_{\strut {\ms L}}\leq 
 \tau_{\strut {\ms R}}$, by definitions~\eqref{eq:hr}, \eqref{eq:hl},
    and by virtue of the 
$AB$-entropy condition~\eqref{ABtraces},
    it follows 
\begin{equation}\label{eq:tracesABtr}
    (u_l(t), u_r(t)) = (A, B)\qquad \forall \; t \in \,]\tau_{\strut {\ms R}}, T]\,,
\end{equation}
which in turn implies 
\begin{equation}
\label{eq:Beq29}
    u(\xi_{\strut {\ms R}}(t)-,  t) = B\qquad \forall \; t \in \,]\tau_{\strut {\ms R}}, T]\,.
\end{equation}
Let
$\zeta_{0,+}$,
$\vartheta_{\xi_{{\ms R}}(t),+}$\,,  be
 the maximal backward characteristic
 starting at $(0,\tau_{\strut {\ms R}})$, and  at
$(\xi_{\strut {\ms R}}(t),  t)$,
$t\in \,]\tau_{\strut {\ms R}}, T]$,
respectively. Relying on~\eqref{eq:char1-bound1}, \eqref{eq:secondcontraineq}, 
and since maximal backward characteristics cannot intersect in the interior of the domain, we find
\begin{equation}
\label{eq:xietar-est-7}
    - \tau_{\strut {\ms R}} \cdot f_r^{\prime}(\overline B)\leq \zeta_{0,+}(0) \leq \vartheta_{\xi_{{\ms R}}(t),+}(0)=\xi_{\strut {\ms R}}(t) -  t \cdot  f_r^{\prime}(u(\xi_{\strut {\ms R}}(t)+, t))\,.
\end{equation}
In turn, \eqref{eq:xietar-est-7} 
together with~\eqref{eq:tau-ineq-29}, yields
\begin{equation}
    \label{eq:xietar-est-9}
    t \cdot  f_r^{\prime}(u(\xi_{\strut {\ms R}}(t)+, t))\leq \xi_{\strut {\ms R}}(t)+
    \bs\tau[\ms R, B, f_r]
    \cdot f_r^{\prime}(\overline B)
    \qquad\forall~t\in \,]\tau_{\strut {\ms R}}, T],
\end{equation}
since $f_r^{\prime}(\overline B)<0$.
By the monotonicity of $f'_r$
we deduce from~\eqref{eq:xietar-est-9} that
\begin{equation}
\label{eq:xiR-ineq5}
    u(\xi_{\strut {\ms R}}(t)+, t)\leq 
    (f_r^{\prime})^{-1} \left( \frac{\xi_{\strut {\ms R}}(t)
   +\bs\tau[\ms R, B, f_r]
    \cdot f_r^{\prime}(\overline B)
    }{t}   \right).
\end{equation}
Therefore, recalling~\eqref{eq:hRH}, \eqref{eq:lambdar-def}, 
and because of the convexity of $f_r$,
we derive
from~\eqref{eq:Beq29}, \eqref{eq:xiR-ineq5}
that 
\begin{equation}
\label{eq:fullspeedineqh-2}
    \xi_{\strut {\ms R}}'(t)\leq 
    \lambda_r \!\left(\!     (f_r^{\prime})^{-1} \!\!\left( \frac{\xi_{\strut {\ms R}}(t)+\bs\tau[\ms R, B, f_r]
    \cdot f_r^{\prime}(\overline B)}{t}   \right)\!, B \!\right)
    \qquad\forall~t\in\,]\tau_{\strut {\ms R}}, T]\,.
\end{equation}
On the other hand, letting $\bs x[\ms L, \overline B, f_r](\cdot)$
be the function defined in \S~\ref{defi:vl},
with $\ms L$ as in~\eqref{eq:def_l_32},
$A = \overline B$,
and $f=f_r$, 
we have~\eqref{eq:fwsh-cond23},
\eqref{eq:fwsh-cond24}. 
Moreover, because of~\eqref{eq:tau-rbf-def}, \eqref{eq:def_l_32}, it holds true
\begin{equation}
\label{eq:Ltau-eq-2}
    \ms L= \bs\tau[\ms R, B, f_r]
    \cdot f_r^{\prime}(\overline B),\qquad
    \bs s[\ms L, \overline B, f_r]=\bs\tau[\ms R, B, f_r]
    \,.
\end{equation}
Then, by virtue of~\eqref{eq:bc-cond-5}, \eqref{eq:fullspeedineqh-2}, 
and because of~\eqref{eq:fwsh-cond23}, \eqref{eq:fwsh-cond24},
\eqref{eq:tau-ineq-29}, 
\eqref{eq:Ltau-eq-2}, 
with a comparison argument we deduce
\begin{equation}
\label{eq:comparison-2}
    \xi_{\strut {\ms R}}(t)\geq 
    \bs x[\ms L, \overline B, f_r](t)
    \qquad\forall~t\in [\tau_{\strut {\ms R}}, T]\,.
\end{equation}
But~\eqref{eq:comparison-2},
together with~\eqref{eq:bc-cond-5},
\eqref{eq:tau-ineq-29}, \eqref{eq:Ltau-eq-2},
and recalling~\eqref{eq:xL},
implies
\begin{equation}
    0=\xi_{\strut {\ms R}}(\tau_{\strut {\ms R}})
    \geq \bs x[\ms L, \overline B, f_r](\tau_{\strut {\ms R}})>
    x[\ms L, \overline B, f_r](\bs\tau[\ms R, B, f_r])=
    0\,,
\end{equation}
which gives a contradiction, 
proving the claim~\eqref{eq:timeclaim}. 
\smallskip

Relying on the implication~\eqref{eq:timeclaim}, we  show now that $\omega(x) = A$ in $]\widetilde {\ms L}, 0[\,$,
considering two cases:
\smallskip

\noindent
\textsc{Case 1:}
$\tau_{\strut {\ms R}}<\tau_{\strut {\ms L}}$.
Then, by definitions~\eqref{eq:hr}, \eqref{eq:hl},
    and by virtue of the 
$AB$-entropy condition~\eqref{ABtraces},
    it follows 
\begin{equation}\label{eq:tracesABtr-2}
    (u_l(t), u_r(t)) = (A, B)\qquad \forall \; t \in \,]\tau_{\strut {\ms L}}, T]\,.
\end{equation}
Observe that
the maximal backward characteristic 
$\vartheta_{\ms L,+}$ 
starting at $(\ms L, T)$
crosses the interface $x=0$
at time $T-\ms L/f'_l(\omega(\ms L+)$.
Since $ \xi_{\strut {\ms L}}$ is a
backward characteristic 
starting at the same point~$(\ms L, T)$ and
crossing the interface $x=0$
at time $\tau_{\strut {\ms L}}$,
one has $\tau_{\strut {\ms L}}\leq T-\ms L/f'_l(\omega(\ms L+)$.
This implies that the backward genuine characteristics from 
points $(x,T)$, $x\in\,]\ms L, 0[$, impact the interface $x=0$
at times $t_x \geq T-\ms L/f'_l(\omega(\ms L+)\geq \tau_{\strut {\ms L}}$. Since the value of the solution $u$  is constant along genuine characteristics, we deduce from~\eqref{eq:tracesABtr-2}
that  $\omega(x) = A$ for all $x \in \,]\ms L, 0[$.
Hence, by~\eqref{eq:assumpt-534} in particular 
it follows that 
$\omega(x) = A$
for all
$x\in\, ]\widetilde {\ms L}, 0[\,$. 
\medskip

\noindent
\textsc{Case 2:}
$\tau_{\strut {\ms L}}\leq \tau_{\strut {\ms R}}$.
Then, because of~\eqref{eq:timeclaim} we have
$\tau_{\strut {\ms R}} \leq \bs \tau[\ms R, B, f_r]$.
Observe that by Step 1 we have $\omega(x)\leq A$ for all $x \in \,]\ms L, 0[$. 
Relying on  the monotonicity of $f'_l$,
this implies that  the backward genuine 
characteristics starting 
from 
points $(x,T)$, $x\in\,]\ms L, 0[$, impacts the interface $x=0$
at times
\begin{equation}
\label{eq:tx-ineq-23}
    \tau(x)\doteq  T-\frac{x}{f_l^{\prime}(\omega(x))} \geq T-\frac{x}{f_l^{\prime}(A)} \,.
\end{equation}
On the other hand, 
recalling definitions~\eqref{eq:Ltildedef},
\eqref{eq:LR-def-23}, we have
\begin{equation}
\label{eq:tx-ineq-24}
   T-\frac{x}{f_l^{\prime}(A)}\geq 
   T-\frac{ (T-\bs \tau[\ms R, B, f_r]) \cdot f_l^{\prime}(A)}{f_l^{\prime}(A)}  = \bs \tau[\ms R, B, f_r]\geq \tau_{\strut {\ms R}}
   \,,
\end{equation}
for all
$x\in\, ]\widetilde {\ms L}, 0[\,$. 
Combining~\eqref{eq:tx-ineq-23}, \eqref{eq:tx-ineq-24}, we deduce that the backward genuine 
characteristics starting 
from 
points $(x,T)$, $x\in\,]\widetilde {\ms L}, 0[$, cross the interface $x=0$
at times $\tau(x)\geq \tau_{\strut {\ms R}}$.
Hence, relying again on the property that
the solution $u$  is constant along genuine characteristics, we 
infer from~\eqref{eq:tracesABtr} that $\omega(x) = A$
for all
$x\in\, ]\widetilde {\ms L}, 0[\,$
also in this case, thus completing the proof of Step 2.

\smallskip

\noindent
\textbf{Step 3.} \big(proof of:\, $\omega(\widetilde{\ms L}-) = A$\big).\\
We know by Step 1 and Step 2 that 
$\omega(\widetilde{\ms L}-)\leq A$
and 
$\omega(\widetilde{\ms L}+)= A$.
On the other hand the Lax entropy condition (see \S~\ref{sec:3b-lr2}) implies $\omega(\widetilde{\ms L}-)\geq \omega(\widetilde{\ms L}+) = A$. 
Therefore one has $A\geq \omega(\widetilde{\ms L}-)\geq \omega(\widetilde{\ms L}+) = A$
which yields $\omega(\tilde{\ms L}-) =A$.
This concludes the proof of \eqref{eq:2b-lr2-2}.

\subsubsection{\rm 
($\ms L<0$, $\ms R>0$, proof of \eqref{eq:2bl-r1} and completion of the proof of \eqref{eq:2a})}\label{sec:2bl-r1} We will prove only the 
second implication in~\eqref{eq:2bl-r1}, the proof of the first one being entirely symmetric. 
Assume that 
\begin{equation}
\label{eq:assumptio413}
    \Big[\ms R \in \,]0, T\cdot f'_r(B)[\,\;  \ \mathrm{and} \;  \  \widetilde{\ms L}\leq \ms L  \,
\Big]
\quad \mathrm{or} \quad\ \ms R \geq T\cdot f'_r(B),
\end{equation}
and let $\tau_{\strut {\ms L}}, \tau_{\strut {\ms R}}$
be the constants defined in~\eqref{eq:hr}, \eqref{eq:hl},
in connection with the characteristics $\xi_{\strut {\ms L}}, \xi_{\strut {\ms R}}$. 
As observed in Step 2 of \S~\ref{sec:2bl-r1r2}, 
the fact that 
$\xi_{\strut {\ms L}}$ is a
backward characteristic 
starting at~$(\ms L, T)$
and crossing the interface $x=0$ at time 
$\tau_{\strut {\ms L}}$
implies
\begin{equation}
\label{eq:taul-ineq23}
    \tau_{\strut {\ms L}}\leq\,
    \tau_+(\ms L)\doteq
    T-\frac{\ms L}{f'_l(\omega(\ms L+)}\,.
\end{equation}
We claim that~\eqref{eq:assumptio413} implies
\begin{equation}
\label{eq:taul-ineq24}
    \tau_{\strut {\ms R}}\leq\,
    \tau_+(\ms L)\,.
\end{equation}
Since~\eqref{eq:taul-ineq23} clearly implies~\eqref{eq:taul-ineq24}
when $\tau_{\strut {\ms R}}\leq \tau_{\strut {\ms L}}$, it will be sufficient to prove the claim under the assumption  $\tau_{\strut {\ms L}}<\tau_{\strut {\ms R}}$.
Let's consider first the case that
\begin{equation}
\label{eq:assumptio413-1}
    \ms R \in \,]0, T\cdot f'_r(B)[\,\;  \ \ \mathrm{and} \ \;  \  \widetilde{\ms L}\leq \ms L\,.
\end{equation}
Observe that, 
because of~\eqref{eq:timeclaim},
$\tau_{\strut {\ms L}}<\tau_{\strut {\ms R}}$ implies 
\begin{equation}
\label{eq:taul-ineq25}
    \tau_{\strut {\ms R}} \leq \bs \tau[\ms R, B, f_r].
\end{equation}
Moreover, by Step 1 of \S~\ref{sec:2bl-r1r2}, one has $\omega(\ms L+)\leq A$.
Therefore,
recalling the definition~\eqref{eq:Ltildedef},
and because of the monotonicity
of $f'_l$, we deduce from
$\widetilde{\ms L}\leq \ms L$
that
\begin{equation}
    (T-\bs \tau[\ms R, B, f_r])\cdot f'_l(\omega(\ms L+)
    \leq \ms L\,,
\end{equation}
which, together with~\eqref{eq:taul-ineq25}, yields~\eqref{eq:taul-ineq24}, under the assumption~\eqref{eq:assumptio413-1}.
Next, consider the case that
\begin{equation}
\label{eq:assumptio413-2}
    \ms R \geq T\cdot f'_r(B)\,.
\end{equation}
Observe that by the analogous argument of Step 1 of \S~\ref{sec:2bl-r1r2} for~\eqref{eq:2b-lr2}, one has $\omega(\ms R-)\geq B$. Moreover, if $\omega(\ms R-)=B$,
by definition~\eqref{eq:LR-def} of $\ms R$ it follows that
$f'_r(B)\geq \ms R/T$, which together with~\eqref{eq:assumptio413-2}, implies 
$f'_r(B)=\ms R/T$.
In turn, $f'_r(B)=\ms R/T$ implies that the minimal characteristic starting at $(\ms R,T)$ reaches the interface $x=0$ at time 
$t=0$, and by definition~\eqref{eq:hr},
it coincides with $\xi_{\strut {\ms R}}$.
Therefore, one has $\tau_{\strut {\ms R}}=0$,
which proves~\eqref{eq:taul-ineq24}.
Hence, it remains to consider the case~\eqref{eq:assumptio413-2}
when $\omega(\ms R-)> B$.
Notice that, if 
\begin{equation}
    \frac{\ms L}{f'_l(\omega(\ms L+))}>\frac{\ms R}{f'_r(\omega(\ms R-))}\,,
\end{equation}
it follows that the minimal backward characteristic $\vartheta_{\ms R,-}$
from $(\ms R,T)$ crosses the interface $x=0$ at a time 
\begin{equation}
\label{eq:tr-def-23}
\tau_-(\ms R)\doteq
 T-\frac{\ms R}{f'_r(\omega(\ms R-))}
\end{equation}
strictly greater than the time 
$\tau_+(\ms L)$
at which the maximal 
backward characteristic $\vartheta_{\ms L,+}$
from $(\ms L,T)$ crosses the interface $x=0$.
On the other hand, since $\vartheta_{\ms R,-}$ is a genuine characteristic, it follows that
$u_r(\tau_-(\ms R))
=\omega(\ms R-)>B$.
Because of  the $AB$-entropy condition~\eqref{ABtraces}
this implies that $u_l(\tau_-(\ms R))
>\theta_l$. 
Thus we can trace the minimal backward characteristic starting at $(0,\,\tau_-(\ms R))$
and lying in $\{x<0\}$, which 
has slope $f'_l(u_l(\tau_-(\ms R))
>0$, and hence it
will intersect  the characteristic $\vartheta_{\ms L,+}$ at a positive time $t^* \geq \tau_+(\ms L)$,
giving a contradiction.
Therefore, $\omega(\ms R-)> B$ implies
\begin{equation}
\label{eq:taul-ineq28}
    \frac{\ms L}{f'_l(\omega(\ms L+))}\leq\frac{\ms R}{f'_r(\omega(\ms R-))}\,.
\end{equation}
On the other hand,
since $\xi_{\strut {\ms R}}$ is a
backward characteristic 
starting at~$(\ms R, T)$
and crossing the interface $x=0$ at time 
$\tau_{\strut {\ms R}}$,
it holds true
\begin{equation}
\label{eq:taul-ineq27}
    \tau_{\strut {\ms R}}\leq
    \tau_-(\ms R)\,.
\end{equation}
Hence, \eqref{eq:tr-def-23}, \eqref{eq:taul-ineq28}, \eqref{eq:taul-ineq27} together yield~\eqref{eq:taul-ineq24}. This completes the proof of the Claim that \eqref{eq:assumptio413} implies~\eqref{eq:taul-ineq24}. 
Then, by definitions~\eqref{eq:hr}, \eqref{eq:hl},
relying on~\eqref{eq:taul-ineq23}, \eqref{eq:taul-ineq24},  
    and by virtue of the 
$AB$-entropy condition~\eqref{ABtraces},
    we find that  
\begin{equation}\label{eq:tracesABtr-4}
    (u_l(t), u_r(t)) = (A, B)\qquad \forall \; t \in \,]\tau_+(\ms L)
        ,\, T]\,.
\end{equation}
Since backward genuine characteristics starting 
from 
points $(x,T)$, $x\in\,]{\ms L}, 0[$, cross the interface $x=0$
at times $t_x\geq \tau_+(\ms L)$,
we 
infer from~\eqref{eq:tracesABtr-4} that $\omega(x) = A$
for all
$x\in\, ]{\ms L}, 0[\,$. This concludes the proof of the 
second implication in~\eqref{eq:2bl-r1}.

Concerning 
~\eqref{eq:2a}, 
we prove now the implication~\eqref{eq:2aproof-2}.
To this end observe that,  
because of~\eqref{eq:2bl-r1} and \eqref{eq:2b-lr2} (established in \S~\ref{sec:2bl-r1r2}), we have
\begin{equation*}
\ms R \in ]0, T\cdot f^{\prime}(B)[ \quad \Longrightarrow \quad
    \omega(x)\geq B\qquad\forall~x\in\,]0, \ms R[\,,
\end{equation*}
and hence 
\begin{equation}
\label{eq:2aproof-3}
\ms R \in ]0, T\cdot f^{\prime}(B)[ \quad \Longrightarrow \quad
    \omega(\ms R -)\geq B\,.
\end{equation}
Thus, relying on~\eqref{eq:uB-ineq}
with $f=f_r$,
we deduce~\eqref{eq:2aproof-2} from~\eqref{eq:2aproof-3}, which completes the proof of the 
second implication in
~\eqref{eq:2a}.

\subsubsection{\rm 
($\ms L<0$, $\ms R>0$, proof of \eqref{eq:1b1}-\eqref{eq:1b12})}\label{sec:1b12}
We will prove only~\eqref{eq:1b1}, the proof of~\eqref{eq:1b12} being entirely similar. 
Then, assume that~\eqref{eq:assumpt-534} holds as in \S~\ref{sec:2bl-r1r2}.
\smallskip

\noindent
\textbf{Step 1.} 
For every point $x\in \,]\ms L, \widetilde {\ms L}[$\, where $\omega$
 is continuous, consider the map
\begin{equation}
\label{eq:phi3def}
\vartheta_{x}(t) \doteq  \begin{cases}
            x- (T-t) \cdot f_l^{\prime}(\omega(x)), & \text{if \ $ \tau(x) \leq t \leq T$},\\
            \noalign{\medskip}
            (t-\tau(x)) \cdot f_r^{\prime}\circ \pi_{r,-}^l(\omega(x)),
            & \text{if\  $0 \leq t < \tau(x)$},
        \end{cases}
\end{equation}
with
\begin{equation}
\label{eq:tau3def}
        \tau(x)\doteq  T-\frac{x}{f_l^{\prime}(\omega(x))}\,,
\end{equation}
and set
\begin{equation}
\label{eq:phi4def}
    \phi(x)\doteq \vartheta_{x}(0)=-\tau(x)\cdot f'_r\circ  \pi_{r,-}^l(\omega(x))\,.
\end{equation}
Observe that 
\begin{equation}
\label{eq:char-rifracted-1}
    \begin{aligned}
        &\vartheta_{x}|_{\,]\tau(x), T]}\quad \text{is a genuine characteristic for u in the halfplane}\ \ \{x<0\}\,,
        \\
        \noalign{\smallskip}
        &\vartheta_{x}|_{\,]0,\tau(x)[\,}\quad \text{is a genuine characteristic for u in the halfplane}\ \ \{x>0\}\ \ \text{if}\ \ u_r(\tau(x))\leq \overline B\,,
    \end{aligned}
\end{equation}
and thus $\vartheta_x$ is a genuine characteristic for $u$ as $AB$-solution (see Remark~\ref{rem:char-ABsol}) only in the case where $u_r(\tau(x))\leq \overline B$.
Note also
that $\tau(x)$ is the  impact time of 
$\vartheta_{x}$ with the interface $x=0$,
and that the function $\tau$ has at most countably many discontinuity points as $\omega$.
Since genuine characteristics cannot intersect in the interior of the domain,
it follows that the right continuous extension of~$\tau$ is a nondecreasing map.
On the other hand,
because 
we are assuming that $\omega$ satisfies \eqref{eq:Hhyp} and that $(A,B)$ is a non critical connection, we know by Proposition~\ref{prop:norare} in Appendix~\ref{app:no-rarefaction} that
no pair of genuine characteristics can meet together on the interface $x=0$. Hence, we deduce that the right continuous extension of
the map $\tau$ is actually increasing
on
$]\ms L, \widetilde {\ms L}[$\,.

We will next show that the right continuous extension of the map $\phi$ 
is nondecreasing on~$]\ms L, \widetilde {\ms L}[$\,.
\smallskip

\noindent
\textbf{Step 2.} 
Consider two points 
$\ms L< x_1<x_2<\widetilde {\ms L}$ of continuity for $\omega$. By Step 1 we know that $\tau(x_1)<\tau(x_2)$. 
Moreover, by~\eqref{eq:2b-lr2-2} (established in \S~\ref{sec:2bl-r1r2}) we have $\omega(x)\leq A$ for all $x\in \,]\ms L, \widetilde {\ms L}[\,$.
Then, we shall provide a proof of
\begin{equation}
\label{eq:phi-decr}
    \phi(x_1)\leq \phi(x_2)
\end{equation}
considering different cases 
according to the fact that $\omega(x_i)=A$
or $\omega(x_i)<A$, $i=1,2$.

\textsc{Case 1:} $\omega(x_i)<A$, $i=1,2$.
Since $u$ is constant along genuine characteristics, and because of the $AB$-entropy condition~\eqref{ABtraces},
it follows that $u_r(\tau(x_i))=\pi_{r,-}^l(\omega(x_i))
    < \overline B$, $i=1,2$.
Therefore, by~\eqref{eq:char-rifracted-1} $\vartheta_{x_i}|_{\,]0,\tau(x_i)[\,}$, $i=1,2$,
are genuine characteristics
in the half plane $\{x>0\}$
starting at $(0, \tau(x_i))$,
which cannot intersect at positive times.
This implies $\phi(x_1)=\vartheta_{x_1}(0)\leq\vartheta_{x_2}(0)=\phi(x_2)$.

\textsc{Case 2:} $\omega(x_i)=A$, $i=1,2$.
By
definition~\eqref{eq:phi3def}
we know that $\vartheta_{x_i}|_{\,]0,\tau(x_i)[\,}$, $i=1,2$, are parallel lines 
(possibly not characteristics for $u$)
with slope $f'_r\big( \pi_{r,-}^l(A) \big)=f'_r(\overline B)$, 
starting at $(0, \tau(x_i))$.
Hence, $\tau(x_1)<\tau(x_2)$ implies 
$\phi(x_1)=\vartheta_{x_1}(0)<\vartheta_{x_2}(0)=\phi(x_2)$.
\smallskip

\textsc{Case 3:} $\omega(x_1) = A$, $\omega(x_2) < A$. 
Notice that,
by the monotonicity of $f'_l, f'_r$, the map
$$\big]-\infty, (f_l^{\prime})^{-1}(x_1/T)\big] \ni u \mapsto -\left(T-\frac{x_1}{f_l^{\prime}(u)}\right)\cdot f_r^{\prime}\circ \pi^l_{r,-}(u)$$
is decreasing. Then we have
\begin{equation}
    \begin{aligned}
\phi(x_1) & \leq -\left(T-\frac{x_1}{f_l^{\prime}(\omega(x_2))}\right) \cdot f_r^{\prime}\circ   \pi^l_{r,-}(\omega(x_2))\\
& \leq -\left(T-\frac{x_2}{f_l^{\prime}(\omega(x_2))}\right)\cdot f_r^{\prime}\circ  \pi^l_{r,-}(\omega(x_2)) = \phi(x_2)\,.
    \end{aligned}
\end{equation}
\smallskip

\textsc{Case 4:} $\omega(x_2) = A$, $\omega(x_1) <A$. Since $\omega(x_2) = A$,
it follows with the same arguments as above that $u_l(\tau(x_2)) = A$  and that either $u_r(\tau(x_2)) = \overline B$
or $u_r(\tau(x_2)) = B$. In the first case, 
because of~\eqref{eq:char-rifracted-1}
one can proceed as in Case 1 to deduce that
$\phi(x_1)\leq \phi(x_2)$. Then, assume $u_r(\tau(x_2)) = B$,
and set (see Figure \ref{x_1<x_2})
\begin{equation}
\label{eq:bar-t-def-22}
    \begin{aligned}
        &\overline t \doteq \inf \Big\{t\leq \tau(x_2) \; \big| \; (u_l(s),u_r(s)) = (A,B) \quad \forall s \in [t, \tau(x_2)]\Big\}.
    \end{aligned}
\end{equation}
Notice that since $\tau({x_1})< \tau({x_2})$
and because $u_l(\tau(x_1))<A$ implies $u_r(\tau(x_1))<\overline B$, it follows that $\overline  t \in \,]\tau({x_1}),\tau(x_2)]$. We claim that  it must hold 
\begin{equation}
\label{eq:urc4}
    u_r(\,\overline t\,) = \overline B\,.
\end{equation}
Towards a proof of~\eqref{eq:urc4}, notice first that, since $\omega(x)\leq A$ for all $x\in \,]\ms L, \widetilde {\ms L}[\,$,
it follows that $u_l(t)\leq A$ for all
$t\in [\tau(x_1), \tau(x_2)]$.
{ 
Because of  the $AB$-entropy condition~\eqref{ABtraces}
and by definition of $\overline t$,
this implies that there exists a sequence of times $t_n \uparrow \overline t$ such that $u_r( t_n) \leq \overline B$. 
Then, since $(A,B)$ is a non critical connection, we trace the 
backward characteristics from points $(0, t_n)$, with slope $f'_r(u_r(t_n))\leq f_r^\prime(\overline B)$. Using the stability of characteristics with respect to uniform convergence (see for example the proof of Lemma \ref{lemma:chara}), we thus find that
there is a backward characteristic with slope $\leq f_r^\prime(\overline B)$ starting from $(0, \overline t)$.  This immediately implies that 
}
\begin{equation}
\label{eq:urc5}
    u_r(\,\overline t\,) \leq  \overline B\,.
\end{equation}
Then, consider the blow ups 
\begin{equation}
\label{eq:bup-def}
    u_\rho(x,t) \doteq u(\rho x, \overline t+\rho(t-\overline t))
    \qquad\quad x\in\R, \ t\geq 0\,,
\end{equation} 
of $u$ at the point $(0, \overline t\,)$, as in the proof of Proposition~\ref{prop:norare}. When $\rho \downarrow 0$, the blow-ups $u_\rho(\cdot, t)$ converge in ${\bf L^1}_{\ms loc}$, up to a subsequence,  to a limiting $AB$-entropy solution $v(\cdot, t)$,
for all $t>0$.
Moreover, we have
\begin{equation}
\label{eq:indata-blup}
    v(x, \overline t\,)= 
\begin{cases}
    u_l(\,\overline t\,), & \text{if $x < 0$},\\
    u_r(\,\overline t\,), & \text{if $x > 0$}.
\end{cases} 
\end{equation}
By definitions~\eqref{eq:bar-t-def-22}, \eqref{eq:bup-def}, 
it holds true
\begin{equation}
\label{eq:tracesit-3}
    (u_{\rho,l}(t), u_{\rho,r}(t))=(A,B)\quad\qquad\forall~t\in\Big]\overline t, \, \overline t + \frac{\tau(x_2)-\overline t}{\rho}\Big[\,,
\end{equation}
where 
$u_{\rho,l}(t), u_{\rho,r}(t)$ denote the 
left and right traces of $u_\rho(\cdot, t)$ at $x=0$.
Taking the limit as $\rho \downarrow 0$
in~\eqref{eq:tracesit-3}, 
{  and invoking Corollary~\ref{cor:fluxtraces-stab}
(with $(A_n, B_n) = (A,B)$ for all $n$),
}
we deduce that 
{ 
\begin{equation}
\label{eq:tracesit}
 v(0-,t) \in \{A, \overline A\,\} \qquad\quad  v(0+, t) \in \{B, \overline B\,\}, \qquad \forall \; t >\overline t,
\end{equation}
}
while~\eqref{eq:urc5}, \eqref{eq:indata-blup} imply
\begin{equation}
\label{eq:tracesit-2}
 v(x,\overline t\,) 
 = u_r(\overline t\,) 
 \leq \overline B, \qquad \forall~x > 0\,.
\end{equation}
By a direct inspection
we find that, if 
an $AB$-entropy solution
of a Riemann problem for~\eqref{conslaw} with initial datum~\eqref{eq:indata-blup}
at time $\overline t$, enjoys the properties
~\eqref{eq:tracesit}-\eqref{eq:tracesit-2}, 
then the initial datum 
 on $\{(x,\overline t\,),\ x>0\}$ 
 must be 
$v(x,\overline t\,)=u_r(\overline t\,) 
 = \overline B$, thus proving~\eqref{eq:urc4}.

\begin{figure}[ht]
\centering
\begin{tikzpicture}[scale = 1.2]
\footnotesize{
\draw[->, thick] (-3.5,0)--(5,0)node[right]{$x$};
\draw[->, thick] (0,0)--(0,5)node[above]{$t$};
\draw[very thin] (-3.5,4)--(1.5,4)node[above]{$\omega$}--(5,4);

\draw[dotted] (-3.5,1.5)node[left]{$\overline  t$}--(0,1.5);
\draw[dotted] (-3.5,0.7)node[left]{$\tau(x_1)$}--(0,0.7);
\draw[dotted] (-3.5,2.5)node[left]{$\tau(x_2)$}--(0,2.5);
\draw[dashed]  (-3,4)node[above]{$x_1$}--(0,0.7);
\draw[dashed] (-1,4)node[above]{$x_2$}--(0,2.5);

\draw[smooth, tension = 0.8, thick] plot coordinates{(0,1.5)(0.1,2)(0.5,2.5)(0.6,3)};
\draw[dashed] (0,1.5)--(2.6,0)node[below]{$\textcolor{black}{-\overline t \cdot f_r^{\prime}(\overline B)}$};
\draw[dashed] (0,0.7)--(1.7,0);
\draw (1.4,0)node[below]{$\phi(x_1)$};

\draw [dashed] (0.1,3)--(0,2.9);
\draw  [dashed](0.2,3)--(0,2.8);
\draw [dashed](0.3,3)--(0,2.7);
\draw [dashed](0.4,3)--(0,2.6);
\draw  [dashed](0.5,3)--(0,2.5);
\draw [dashed](0.6,3)--(0,2.4);
\draw [dashed](0.55,2.85)--(0,2.3);
\draw [dashed](0.5,2.7)--(0,2.2);
\draw [dashed](0.45,2.55)--(0,2.1);
\draw [dashed](0.4,2.4)--(0,2);
\draw [dashed](0.1,2)--(0,1.9);

\draw[dashed, red] (0, 2.5)--(4,0)node[below]{\textcolor{black}{$\phi(x_2)$}};

}
\end{tikzpicture}
\caption{The situation described in  Case 4.}
\label{x_1<x_2}
\end{figure}

\noindent
Relying on~\eqref{eq:urc4} we can now complete
the proof of~\eqref{eq:phi-decr}. 
Since $u_r(\tau(x_1))<\overline B$, we know by~\eqref{eq:char-rifracted-1} that $\vartheta_{x_1}$ is a genuine characteristic 
in the halfplane $\{x>0\}$ starting at $(0, \tau(x_1))$.
On the other hand, because of~\eqref{eq:urc4} and since $(A,B)$ is a non critical connection, we can trace the  maximal backward characteristic from $(0, \overline t\,)$ in $\{x>0\}$, 
which has slope $f_r^{\prime}(\overline B)$
and reaches the $x$-axis at the point $-\overline t\cdot f'_r(\overline B)$.
Such a (genuine) characteristic cannot intersect at a positive time
the genuine characteristic $\vartheta_{x_1}$.
Therefore, one has
\begin{equation}
\label{eq:phi-est-45}
    \phi(x_1) = \vartheta_{x_1}(0)\leq -\overline t\cdot f_r^{\prime}(\overline B).
\end{equation}
Moreover, since $\overline t \leq  \tau(x_2)$,
and because $\pi_{r,-}^l(\omega(x_2))=\pi_{r,-}^l(A)=\overline B$,
we deduce 
$$-\overline t\cdot f_r^{\prime}(\overline B) \leq -\tau(x_2)\cdot f_r^{\prime}(\overline B) = \phi(x_2),$$
which together with~\eqref{eq:phi-est-45}, yields~\eqref{eq:phi-decr}.
This concludes the proof of the nondecreasing monotonicity of $\phi$ on
$]\ms L, \widetilde {\ms L}[$\,.
Invoking Lemma 4.4 in~\cite{anconachiri}, 
this is equivalent to the inequality  
$$
D^+\omega(x) \leq g[\omega, f_l, f_r](x), \qquad \forall \; x \in \,]\ms L, \tilde{\ms L}[\,,
$$
where $g$ is the function in~\eqref{eq:ghdef}.
This concludes the proof of~\eqref{eq:1b1}, and thus 
the proof that $\omega$ satisfies conditions (i)-(ii) of Theorem \ref{thm:attprofiles}
is completed.

\subsubsection{\rm 
($\ms L < 0$, $\ms R = 0$ or viceversa, proof of conditions (i)-(ii), or (i)$^\prime$-(ii)$^\prime$, of Theorem~\ref{thm:attprofiles2})}
\label{sec:othercases}
We consider only the case $\ms L < 0$, $\ms R = 0$, the other case $\ms L = 0$, $\ms R > 0$ being symmetrical. The proofs of~\eqref{eq:1a2},
\eqref{eq:1b2}, \eqref{eq:2b2}, \eqref{eq:2a2b}, \eqref{eq:3a2b} in this case, are   
entirely similar to the proofs of~\eqref{eq:1a}, \eqref{eq:1b1}, \eqref{eq:2b-lr2-2}, \eqref{eq:2a}, \eqref{eq:3b-lr2}, respectively,
in the case $\ms L<0$, $\ms R>0$.
We provide here only the proof 
of~\eqref{eq:2a2}, which is
the only new constraint arising 
in the case $\ms L < 0$, $\ms R = 0$,
that was not present
in the case $\ms L<0$, $\ms R>0$.
Notice first that
by~\eqref{eq:2b-lr2-2} (established in \S~\ref{sec:2bl-r1r2}) we know that $\omega(x)\leq A$ for all $x\in \,]\ms L, 0[\,$.
Hence, since the connection $(A,B)$
is non critical, tracing the backward characteristics (with negative slope) 
in the half plane $\{x<0\}$
from any sequence of points 
$(x_n,T)$, $x_n\in \,]\ms L, 0[\,$,
$x_n\uparrow 0$,
we deduce that there exists the one-sided limit $u_l(T-)$ 
and it holds true 
\begin{equation}
\label{eq:ul-omega-A}
    u_l(T-)=\omega(0-)\leq A\,.
\end{equation}
Then, we will distinguish two cases.
\smallskip

\textsc{Case 1:} Assume that $u_r(t)\geq B$ for all $t\in \,]\tau, T[$\,,
for some $\tau<T$. Then, by the $AB$-entropy condition~\eqref{ABtraces}, and because of~\eqref{eq:ul-omega-A},
we deduce that $\omega(0-)=u_l(T-)=A$. On the other hand,
since $(A,B)$ is a non critical connection,
by definition~\eqref{eq:LR-def} it follows that 
$\ms R=0$ implies $f'_r(\omega(0+))<0$. Therefore we have
$\omega(0+)\leq\overline B = \pi^l_{r,-}(A)=\pi^l_{r,-}(\omega(0-))$, proving~\eqref{eq:2a2}.
\smallskip

\textsc{Case 2:} 
Assume that there exists a sequence of times $t_n \uparrow T$ such that $u_r(t_n)\leq \overline B$ for all $n$, and such that 
$\lim_n u_r(t_n) = u^*$, for some $u^*\leq \overline B$.
By the $AB$-entropy condition~\eqref{ABtraces}
we may also assume
that $u_r(t_n)=
\pi^l_{r,-}(u_l(t_n))$ for all $n$.
Therefore, relying on~\eqref{eq:ul-omega-A}, we find
$u^*=\lim_n \pi^l_{r,-}(u_l(t_n))
=\pi^l_{r,-}(\omega(0-))$.
On the other hand we have 
$\omega(0+)\leq u^*$, since
otherwise 
backward genuine characteristics
issuing from points $(x_n,T)$,
$x_n \downarrow 0$, would eventually cross backward genuine characteristics 
in the half plane $\{x>0\}$
starting from
points $(0, t_n)$.
In fact, if $\omega(0+)> u^*$
then 
we can find points $(x_n,T)$,
$x_n>0$ ($x_n$ point of continuity for $\omega$), and
$(0,t_n)$, $t_n<T$ ($t_n$ point of continuity for $u_r$), such that $\omega(x_n)>u_r(t_n)$,
which would imply that the backward characteristic 
starting from
$(x_n,T)$ with negative slope
$f'_r(\omega(x_n,T))$
intersect the backward characteristic 
starting from $(0,t_n)$
with slope $f'_r(u_r(t_n))<f'_r(\omega(x_n,T))$.
Therefore it must be 
$\omega(0+)\leq u^*$, which
together with $u^*=\pi^l_{r,-}(\omega(0-))$, yields~\eqref{eq:2a2}.

This concludes the proof of~\eqref{eq:2a2}, and thus 
the proof that $\omega$ satisfies conditions (i)-(ii) 
\linebreak 
(or (i)$^\prime$-(ii)$^\prime$) of Theorem~\ref{thm:attprofiles2}
is completed.

\subsubsection{\rm 
($\ms L = 0$, $\ms R = 0$, proof of conditions (i)-(ii) of Theorem~\ref{thm:attprofiles3})}
\label{sec:othercases-2}
The proofs of~\eqref{eq:1a3},  
\eqref{eq:2a3-bis}, are entirely similar to the proofs of~\eqref{eq:1a}, 
\eqref{eq:2a}, 
in the case $\ms L<0$, $\ms R>0$,
and of
\eqref{eq:2a2}, 
in the case $\ms L=\ms R=0$,
respectively.
Further, \eqref{eq:2a3}
can be established with the same arguments of
the proof of~\eqref{eq:2a} in the case $\ms L<0$, $\ms R>0$,
recalling 
Remark~\ref{rem:constr-LR=0}.
This completes the proof that $\omega$ satisfies conditions (i)-(ii) of Theorem~\ref{thm:attprofiles3}.

 \subsection{Part 1.b - $(1) \Rightarrow (3)$ for non critical connections  without assuming~\eqref{eq:Hhyp}}
 \label{subsec:reductionBV}
In this Subsection,
given an element $\omega$
of the set $\mc A^{[AB]}(T)$
for a non critical connection
$(A, B)$, 
we will show that
$\omega$ satisfies \eqref{eq:Hhyp}.
In view of the analysis in \S~\ref{sec:(1)-(3)},
this will imply that
$\omega$ fulfills condition~(3) of Theorem~\ref{thm:backfordiscfluxcycle},
thus completing the proof of the implication $(1) \Rightarrow (3)$
of Theorem~\ref{thm:backfordiscfluxcycle}.

Then, given $\omega\in \mc A^{[AB]}(T)$
with 
\begin{equation}
\label{eq:att-prof-25}
    \omega = \mc S_T^{[AB]+}u_0, \qquad u_0 \in \mathbf L^{\infty}(\mathbb R)\,,
\end{equation}
set $u(x,t)\doteq S_t^{[AB]+}u_0(x)$, $x\in\R$, $t\geq 0$.
%
Next, let $\{u_{n,0}\}_n$ be a sequence of
functions in $BV(\R)$ 
such that
$$
u_{n,0} \ \ \to \ \ u_0 \quad \text{in}\quad  \mathbf L^1_{\ms loc}(\R)\,,
$$
and define $u_n(x,t) \doteq \mc S_t^{[AB]+} u_{n,0}(x)$.
Then, by Theorem~\ref{theoremsemigroup}-(iii) 
it follows
\begin{equation}
\label{eq:untconv}
    \begin{aligned}
    u_n(\cdot, t) \  &\rightarrow\ \ u(\cdot, t)\qquad\text{in}\quad
    \bf L^1_{\mr{loc}}(\R)\qquad \forall~t\geq 0\,.
    %
\end{aligned}
\end{equation}
Since $(A, B)$ is a non critical connection
and because the initial data $u_{n,0}$
are in $BV$, invoking the BV bounds on $AB$-entropy solutions 
provided in~\cite[Lemma 8]{Garavellodiscflux}
(see also~
\cite[Theorem 2.13-(iii)]{MR2743877}),
we deduce that $u_n(\cdot, t)\in BV(\R)$
for all $t>0$, and for all $n$.
Therefore, 
\begin{equation*}
    u_n(\cdot, t)\in \mc A^{[AB]}(t)\,,\quad \text{and \, satisfies~~\eqref{eq:Hhyp}}\quad\forall~t>0\,,\ \ \forall~n\,.
\end{equation*}
Hence, relying on the analysis in \S~\ref{sec:(1)-(3)}, 
and recalling~\eqref{eq:att-set-decomp},
we know that,
setting
\begin{equation}
\label{eq:LR-n-def-33}
    \ms L_n(t) \doteq \ms L[u_n(\cdot, t),f_l]\,,
      \quad\qquad \ms R_n(t) \doteq  \ms R[u_n(\cdot, t), f_r]\,,
\end{equation}
each 
$u_n(\cdot, t)$ satisfies the
conditions stated in:
\begin{itemize}
    \item[-] Theorem~\ref{thm:attprofiles} \ if \ $\ms L_n(t)<0, \ \ms R_n(t)>0$;
    \item[-] Theorem~\ref{thm:attprofiles2} \ if \ $\ms L_n(t)<0, \ \ms R_n(t)= 0$ \ or viceversa;
    \item[-] Theorem~\ref{thm:attprofiles3} \ if \ $\ms L_n(t)=0$, \ $R_n(t)= 0$.
    \end{itemize}
Thus, in particular, $u_n(\cdot, t)$
satisfies the Ole\v{\i}nik-type inequalities 
\begin{equation}
\label{eq:ol-est-un-1}
    \begin{aligned}
     D^+ u_n(x, t) &\leq \frac{1}{t\cdot f''_l(u_n(x,t))}  \qquad\quad
     \text{in}\quad  \,]\!-\!\infty, \ms L_n(t)[\,,
     \\
       \noalign{\smallskip}
        D^+ u_n(x, t) &\leq g[u_n(\cdot, t), f_l, f_r](x)  \quad \ \,
        \text{in} \quad \,]\ms L_n(t), 0[\,,
        \quad\text{if}\quad \ms L_n(t)<0\,,
        \\
       \noalign{\medskip}
D^+ u_n(x, t) &\leq h[u_n(\cdot, t), f_l, f_r](x)   \quad\ \, 
\text{in}\quad \,]0,\ms R_n(t)[\,,
\quad\text{if}\quad \ms R_n(t)>
0\,,
\\
   \noalign{\smallskip}
    D^+ u_n(x, t) &\leq \frac{1}{t\cdot f''_r(u_n(x,t))}  \qquad \quad  \,\text{in}\quad  \,]\ms R_n(t), +\infty[\,,
    \end{aligned}
\end{equation}
and the constraints
\begin{equation}
\label{eq:constr-n-2}
\begin{aligned}
      u_n(x, t)&\leq A\qquad \forall~x\in \,]\ms L_n(t), 0[\,,
      \\
      \noalign{\medskip}
      u_n(x, t)&\geq B\qquad \forall~x\in \,]0, \ms R_n(t)[\,,
\end{aligned}
\end{equation}
for all $t>0$.
Since~\eqref{eq:constr-n-2} implies $f'_r(u_n(x,t))\geq f'_r(B)$
for all $x\in\,]0,\ms R_n(t)[$\,,
by the monotonicity of $f'_r$,
we find
\begin{equation}
\label{eq:lwhn-1}
\begin{aligned}
     t\cdot f'_r(u_n(x,t))-x
\geq  \frac{t\cdot f'_r(B)}{2}\qquad\  \forall~x\in \left[0,\, \min\bigg\{\ms R_n(t),\, \frac{t\cdot f'_r(B)}{2}
    \bigg\}\right[\,.
\end{aligned}
\end{equation}
Therefore,
recalling definition~\eqref{eq:ghdef},
setting 
$\overline \Lambda\doteq \sup_{|z|\leq M}  \max \{|f'_l(z)|, |f'_r(z)|\}$, with
$M$ being a uniform ${\bf L^\infty}$ bound
for $u_n$,
and letting $a$ be the lower bound on $f''_l, f''_r$
given in~\eqref{eq:flux-assumption-1},
we deduce from~\eqref{eq:lwhn-1},
that, 
if 
\begin{equation*}
\ms R_n(t)\leq \frac{t\cdot f'_r(B)}{2}\,,
\end{equation*}
then for all $n$ it holds true
\begin{equation}
   \label{eq:ol-est-un-212}
    \qquad  h[u_n(\cdot, t), f_l, f_r](x)\leq 
     \dfrac{{\overline \Lambda}^{\,2}}{a\, f'_r(B)\big(t\cdot f'_r(u_n(x,t))-x\big)}
     \leq 
     \dfrac{2}{a\, t}\cdot \bigg(\dfrac{\overline \Lambda }{f'_r(B)}\bigg)^{\!2}\qquad \forall~x\in [0, \ms R_n(t)[\,,
\end{equation}
while if 
\begin{equation*}
\ms R_n(t)> \frac{t\cdot f'_r(B)}{2}\,,
\end{equation*}
then for all $n$ it holds true
\begin{equation}
    \label{eq:ol-est-un-21}
     h[u_n(\cdot, t), f_l, f_r](x)\leq 
       \begin{cases}
     \dfrac{{\overline \Lambda}^{\,2}}{a\, f'_r(B)\big(t\cdot f'_r(u_n(x,t))-x\big)}
     \leq 
     \dfrac{2}{a\, t}\cdot \bigg(\dfrac{\overline \Lambda}{f'_r(B)}\bigg)^{\!2}\qquad \forall~x\in \left[0,\, \frac{t\cdot f'_r(B)}{2}\right]\,,
    \\
    \noalign{\bigskip}
     \dfrac{\overline \Lambda}{x\, a}~\leq~  
     \dfrac{2\, \overline \Lambda}{a\, t\!\cdot\! f'_r(B)}
     \qquad \qquad 
    \forall~x\in \left[\frac{t\cdot f'_r(B)}{2},\, \ms R_n(t)\right[\,.
    \end{cases}
    \end{equation} 
Hence, 
we derive from~\eqref{eq:ol-est-un-1}, \eqref{eq:ol-est-un-212}, \eqref{eq:ol-est-un-21}, the uniform bounds
\begin{equation}
    \label{eq:ol-est-un-32}
    \begin{aligned}
D^+ u_n(x, t) &\leq 
\dfrac{2\, \overline \Lambda}{a\, t\!\cdot\! f'_r(B)}\cdot\max\bigg\{
1,\, \dfrac{\overline \Lambda}{f'_r(B)}
\bigg\}
\quad
\text{in}\quad \,]0,\ms R_n(t)[\,,
\qquad\text{if}\quad \ms R_n(t)>
0\,,
\\
   \noalign{\smallskip}
    D^+ u_n(x, t) &\leq \frac{1}{t\cdot a}  \quad   \text{in}\quad  \,]\ms R_n(t), +\infty[\,,
    \end{aligned}
\end{equation}
for all n. 
Since $(A,B)$ is a non critical connection, for every fixed $\delta>0$,
the one-sided uniform upper bounds 
provided by~\eqref{eq:ol-est-un-32} 
yield  uniform bounds 
on the total increasing variation (and hence on the total variation as well) of $u_n(t)$\,, $t\geq \delta$, on bounded subsets 
of $[0, +\infty[$. 
Thus, by the lower-semicontinuity of the total variation with respect to  the $\bf L^1_{\blu{\mr{loc}}}$ convergence,
and because of~\eqref{eq:untconv},
we find that
\begin{equation}
\label{eq:bound-tv-nc-2}
u(\cdot, t)\in BV_{\mr{loc}}([0, +\infty[\,),\qquad \blu{\forall~t\geq\delta}\,.
\end{equation}
With the same type of arguments, relying on~\eqref{eq:ol-est-un-1}, \eqref{eq:constr-n-2}, we can show that
\begin{equation}
\label{eq:bound-tv-nc-3}
u(\cdot, t)\in BV_{\mr{loc}}(\,]\!-\infty, 0]),\qquad \blu{\forall~t\geq\delta}\,.
\end{equation}
Therefore, we deduce from~
\eqref{eq:bound-tv-nc-2},
\eqref{eq:bound-tv-nc-3}, that
\begin{equation}
\label{eq:bound-tv-nc-4}
u(\cdot, t)\in BV_{\mr{loc}}(\R)\qquad\forall~t>0\,,
\end{equation}
which shows that the function $\omega$ in~\eqref{eq:att-prof-25}
satisfies condition~\eqref{eq:Hhyp}, thus completing the proof of 
the implication $(1) \Rightarrow (3)$  of Theorem~\ref{thm:backfordiscfluxcycle} in the case of a non critical connection.

\subsection{Part 2.a - $(3) \Rightarrow (1)$ for non critical connections.}\label{sec:(3)-(2)}
In this Subsection, given 
\begin{equation}
\label{eq:LR-def-24}
    \omega\in \msc A^{\ms L, \ms R},\qquad\quad 
    \ms L \doteq \ms L[\omega, f_l]< 0, \quad \ms R \doteq \ms R[\omega, f_r] > 0\,,
\end{equation}
($\msc A^{\ms L, \ms R}$ being the set in~\eqref{eq:ALR-def1}), 
assuming that 
\begin{equation}
\label{eq:hyp-3a2}
    \omega \quad  \text{satisfies conditions (i)-(ii) of Theorem \ref{thm:attprofiles}}, 
\end{equation}
we will
show that $\omega \in \mc A^{AB}(T)$ by explicitly constructing an $AB$-entropy solution attaining $\omega$ at time $T$. 
With entirely similar arguments one can show that
the same conclusion hold assuming that  $\omega\in \msc A^{\ms L, \ms R}$:
\begin{itemize}
    \item[-] satisfies the conditions of Theorem \ref{thm:attprofiles2}, if $\ms L = 0$, $\ms R > 0$ or viceversa;
    \item[-] satisfies the conditions of Theorem \ref{thm:attprofiles3}, if $\ms L = 0$, $\ms R = 0$.
\end{itemize}
%
\medskip

Then, consider $\omega$ satisfying~\eqref{eq:LR-def-24},
\eqref{eq:hyp-3a2}.
By Remark \ref{rem:threecases}
we can distinguish six cases 
of pointwise constraints on $\omega$, prescribed by condition (ii)
of Theorem~\ref{thm:attprofiles},
which depend on the reciprocal positions of the points
$\ms L$,  $\ms R$,
and
$\widetilde{\ms L}$, 
$\widetilde{\ms R}$, defined in~\eqref{eq:LR-def}-\eqref{eq:Ltildedef}. 
We shall consider here only the \textsc{Cases 1} and \textsc{2} discussed in Remark~\ref{rem:threecases}. The \textsc{Cases 1b, 2b} are symmetrical to \textsc{Cases 1, 2}, up to a change of variables $x \mapsto -x$,
while the \textsc{Cases 3, 4} are entirely similar or simpler.

Notice that, by Remark \ref{rem:threecases},
in \textsc{Case 1} it holds true~\eqref{eq:case-i-A}, \eqref{eq:case-i-B},
and in particular we shall assume that
\begin{equation}
\label{eq:ur-constr}
    \omega(\ms R+) < \bs u[\ms R, B, f_r]\,,
\end{equation}
while in \textsc{Case 2} it holds true~\eqref{eq:case-i-B}, \eqref{eq:case-ii-A},
and we shall assume that~\eqref{eq:ur-constr} is
verified together with
\begin{equation}
\label{eq:vl-constr}
    \omega(\ms L-) > \bs v[\ms L, A, f_l]\,.
\end{equation}
The cases in which $\omega(\ms R+) = \bs u[\ms R, B, f_r]$ or $\omega(\ms L-) = \bs v[\ms L, A, f_l]$ can be treated with entirely similar or simpler arguments. 
Moreover, in both \textsc{Cases 1} and \textsc{2} we have
\begin{equation}
\label{eq:hyp-cases-1-2}
    \widetilde {\ms L} > \ms L\,,
    \qquad\quad \omega(\,\widetilde{\ms L} -)=\omega(\,\widetilde{\ms L} +)\,.
\end{equation}
The construction of the initial datum $u_0$
so that the corresponding $AB$-entropy solution solution $u(x,t)\doteq \mc S_t^{[AB]+}u_0 (x)$ attains the value $\omega$
at time $T$ follows a  by now  standard procedure (see \cite{anconachiri}, \cite{anconamars}), that we describe in \S~\ref{subsec:const-ABsol}-\ref{subsec:const-inda} below. To this end we first introduce some technical notations
in~\S~\ref{subsec:charact}-\ref{subsec:partition}.

\subsubsection{Characteristics of compression waves}
\label{subsec:charact}
We introduce a class of
curves connecting two points
 $(z,0)$, $(y,T)$,
that will be treated as characteristics
of compression waves
generating a shock
at the point $(y,T)$.
In particular, in the case $y< 0 < z$,
such curves will be
characteristics of a compression wave
that starts at time $t=0$ on the half plane $\{z\geq 0\}$,
and generates a shock
at time $t=T$
after being refracted at
the discontinuity interface.
Given any  $y<0$,
consider the 
continuous function
$$
]-\infty, (f_l^{\prime})^{-1}(y/T)] \ni u \mapsto h_y(u) \doteq  -\Big(T-\frac{y}{f_l^{\prime}(u)}\Big)\cdot f_r^{\prime}\circ \pi^{l}_{r,-}(u).
$$
Notice that, by definition~\eqref{pimap-def} and since $f'_l, f'_r$ are increasing functions, 
it follows that \linebreak $u\mapsto -(T-y/f'_l(u))$, $u\mapsto f_r^{\prime}\circ \pi^{l}_{r,-}(u)$ are decreasing maps, and hence the map $h_y$ is decreasing as well. On the other hand we have
 $\displaystyle{\lim_{u\to-\infty}}h_y(u) = +\infty$, $h_y((f_l^{\prime})^{-1}(y/T)-) = 0$.
Therefore by a continuity and monotonicity argument, it follows that,
for every $z>0$, there exists a unique state $u_{y,z}\leq (f_l^{\prime})^{-1}(y/T)$, such that 
\begin{equation}
\label{eq:hyz-def}
h_y(u_{y,z}) = z\,.
\end{equation}
Moreover, the map $z\mapsto u_{y,z}$, $z>0$
is continuous.
Then, for every pair
$y< 0 < z$,
we denote by $\eta_{y,z}:[0,T] \mapsto \mathbb R$ the polygonal line given by 
\begin{equation}
\label{eq:poly-line-def1}
        \eta_{y,z}(t) \doteq \begin{cases}
            y-(T-t)\cdot f_l^{\prime}(u_{y,z}), & \  \ \text{if \ $\tau(y,z) < t\leq  T$},\\
            \noalign{\smallskip}
            (t-\tau({y,z}))\cdot f_r^{\prime}\circ \pi^l_{r,-}(u_{y,z}), & \ \ \text{if \ $0 \leq t\leq  \tau(y,z)$},
        \end{cases}
\end{equation}
where 
\begin{equation}
\label{eq:tau-def-23}
    \tau(y,z)\doteq T-\frac{y}{f_l^{\prime}(u_{y,z})}.
\end{equation}
\blu{Next}, for every pair
$y,z < 0$, or $y,z > 0$, we 
denote by $\eta_{y,z}:[0,T] \to \mathbb R$ the segment 
\begin{equation}
\label{eq:poly-line-def2}
    \eta_{y,z}(t) \doteq  y-(T-t)\cdot \dfrac {(y-z)}{T}\qquad \ \forall~0\leq t\leq T\,.
\end{equation}
Notice that, 
in the case $y< 0 < z$,
if we consider a function $u(x,t)$ that 
assumes the values 
\begin{equation*}
    \begin{aligned}
        u= u_{y,z}\qquad &\text{on the segment}\quad 
        \eta_{y,z}(t), \ \tau(y,z) < t\leq  T,
        \\
        u=\pi^l_{r,-}(u_{y,z})
        \qquad &\text{on the segment}\quad 
        \eta_{y,z}(t), \ 0\leq t\leq \tau(y,z),
    \end{aligned}
\end{equation*}
then the states $u_l=u_{y,z}$, $u_r=\pi^l_{r,-}(u_{y,z})$ satisfy the interface entropy condition~\eqref{ABtraces}
at time $t=\tau(y,z)$, and $\eta_{y,z}$
enjoys the properties of a (genuine) characteristic for $u$ as
an $AB$-entropy solution
(see Remark~\ref{rem:char-ABsol}).
Similar observations hold for $\eta_{y,z}$
in the case $y,z < 0$, considering a function
$u(x,t)$ that 
assumes the value $(f'_l)^{-1}((y-z)/T)
= u_{y,z}$
along the segment $\eta_{y,z}$,
and in the case $y,z > 0$, considering a function
$u(x,t)$ that 
assumes the value $(f'_r)^{-1}((y-z)/T)= u_{y,z}$
along the segment $\eta_{y,z}$.


\subsubsection{Maximal/minimal backward characteristics}
\label{subsec:minmaxbackcharact}
We introduce a class of
curves with end point~$(y,T)$
that will be treated as 
maximal and minimal 
backward characteristics
starting at $(y,T)$.
For every $y \in \,]-\infty, \widetilde{\ms L}\,]\,\cup \,]\ms R, +\infty[$\,, we denote by  $\vartheta_{y, \pm}:[0,T] \to \mathbb R$ the segments or polygonal
lines 
\begin{equation}
\label{eq:pollines}
     \vartheta_{y, \pm}(t) \doteq \!\begin{cases}
y-(T-t) \cdot f_l^{\prime}\big(\omega(y\pm)\big), & \text{ if \ $y < \ms L, \quad 0\leq t \leq T$},\\
\noalign{\smallskip}
  y-(T-t) \cdot f_l^{\prime}\big(\omega(y\pm)\big), & \text{ if \ $\ms L \leq y \leq \widetilde {\ms L}, \quad \tau_{\pm}(y)
  \leq t \leq  T$},\\
  \noalign{\smallskip}
 \big(t-\tau_{\pm}(y)\big)
 \cdot f_r^{\prime}\circ\pi^l_{r,-}(\omega(y\pm)), & \text{ if \ $\ms L \leq y \leq \widetilde {\ms L}, \quad 0 \leq t < \tau_{\pm}(y),
 $}\\
 \noalign{\smallskip}
    y-(T-t) \cdot f_r^{\prime}\big(\omega(y\pm)\big), & \text{ if \ $y> \ms R, \quad 0\leq t \leq T$},
        \end{cases}
\end{equation}\
where 
\begin{equation}
\label{eq:tau-def-24}
    \tau_{\pm}(y)\doteq T-\frac{y}{f_l^{\prime}(\omega(y\pm))}.
\end{equation}
We will write $\vartheta_{y}(t)\doteq \vartheta_{y, \pm}(t)$ for all $t\in [0,T]$, whenever $\omega(y-)=\omega(y+)$.
In particular, because o\blu{f}~\eqref{eq:hyp-cases-1-2}, we have $\vartheta_{\widetilde{\ms L}}(t)\doteq \vartheta_{\widetilde{\ms L}, \pm}(t)$.
\blu{Further}, for $y=\ms R$, we denote 
 by  $\vartheta_{\ms R, +}:[0,T] \to \mathbb R$ the segment
\begin{equation}
    \label{eq:pollines-R}
     \vartheta_{\ms R, +}(t) \doteq 
     \ms R -(T-t) \cdot f_r^{\prime}\big(\omega(\ms R+)\big)
     \qquad\forall~0\leq t \leq T\,.
\end{equation}
Notice that, because of definition~\eqref{eq:LR-def}, \eqref{eq:LR-def-24}, whenever $y\in \,]-\infty, {\ms L}\,[\,\cup \,]\ms R, +\infty[$, the curves $\vartheta_{y, \pm}$ are segments that never cross the interface
$\{x=0\}$, 
instead for all $y\in  \,]\ms L,\, \widetilde {\ms L}]$, $\vartheta_{y, \pm}$ 
are polygonal lines that are refracted 
at $\{x=0\}$\,.
Moreover, at every point of discontinity $y\in \,]-\infty, {\ms L}\,[\,\cup \,]\ms R, +\infty[$ of $\omega$, conditions~\eqref{eq:1a}, \eqref{eq:1b1}
imply the Lax condition $\omega(y-)>\omega(y+)$, which in turn,
by the monotonicity of $f'_l, f'_r$, implies
\begin{equation}
\label{eq:lax-start-char}
    \vartheta_{y, -}(0)< 
    \vartheta_{y, +}(0)\qquad\quad
    \forall~y\in \,]-\infty, {\ms L}\,[\,\cup \,]\ms R, +\infty[\,.
\end{equation}

As in \S~\ref{subsec:charact},
observe that in the case $\ms L < y \leq \widetilde {\ms L}$,
if we consider a function $u(x,t)$ that 
assumes the values 
\begin{equation*}
    \begin{aligned}
        u= \omega(y\pm)\qquad &\text{on the segment}\quad 
        \vartheta_{y, \pm}(t), \ \tau_\pm(y) < t\leq  T,
        \\
        u=\pi^l_{r,-}(\omega(y\pm))
        \qquad &\text{on the segment}\quad 
        \vartheta_{y, \pm}(t), \ 0\leq t\leq \tau_\pm(y),
    \end{aligned}
\end{equation*}
than $\vartheta_{y, \pm}$
enjoys the properties of a maximal/minimal backward characteristic for $u$ as
an $AB$-entropy solution that attains the value $\omega$ at time $T$.
Similar observations hold for $\vartheta_{y, \pm}$
in the case $y<\ms L$ or $y\geq \ms R$, 
considering a function $u(x,t)$ that 
assumes the value $\omega(y\pm)$ along $\vartheta_{y, \pm}$.


\subsubsection{Partition of $\R$}
\label{subsec:partition}
The initial datum will be defined in a different way on different intervals of
 the following partition of $\mathbb R$ (see Figure~\ref{partition}):
\begin{equation}
\label{eq:partition}
    \begin{aligned}
        \mc I_{{\ms L}} &\doteq  \Big\{ x \in \mathbb R \; \big| \;  \vartheta_{\ms L,-}(0) < x < \vartheta_{\ms L,+}(0)\Big\}, 
        \\
        \mc I_{{\ms R}} &\doteq  \Big\{x \in \mathbb R \; \big| \;  -\bs y[\ms R, B, f_r](T)< x < \ms R-T \cdot f_r^{\prime}(\omega(\ms R+))\Big\},
        \\
        \mc I_{\mc C} &\doteq  \Big\{x \in \mathbb R \setminus  (\mc I_{\ms L} \cup \mc I_{\ms R} )\; \big| \; \nexists \; y \in \mathbb R   \; : \; \vartheta_{y,+}(0) = x \; \text{or}\;\vartheta_{y,-}(0) = x \Big\},
        \\
        \mc I_{\mc Ra} &\doteq  \Big\{x \in \mathbb R \setminus  (\mc I_{\ms L} \cup \mc I_{\ms R} )\; \big| \; \exists \; y< z   \; : \; \vartheta_{y,+}(0) = \vartheta_{z,-} (0) = x\Big\},
        \\
        \mc I_{\mc W} &\doteq  \Big\{x \in \mathbb R \setminus  (\mc I_{\ms L} \cup \mc I_{\ms R} )\; \big| \; \exists ! \; y \in \mathbb R   \; : \; \vartheta_{y,+}(0) = x \; \text{or}\;\vartheta_{y,-}(0) = x \Big\},
    \end{aligned}
\end{equation}
where $\bs y[\ms R, B, f_r](T)$
is defined as in~\S~\ref{defi:ur}
with $f=f_r$. Notice that the set $\mc I_{\ms R}$ is non empty because 
the increasing monotonicity of $f'_r$,
together with~\eqref{eq:urbf-def}, \eqref{eq:ur-constr},
implies 
\begin{equation*}
    f'_r(\omega(\ms R+))<f'_r\big(\bs u[\ms R, B, f_r]\big)=\frac{\ms R+\bs y[\ms R, B, f_r](T)}{T}\,.
\end{equation*}

\begin{figure}[ht]
\centering
\footnotesize{
\begin{tikzpicture}[scale = 0.8]

\draw (0,1.5)--(1.5,0);
\draw[->] (-6.5,0)--(6.5,0)node[above]{$x$};
\draw[->] (0,0)--(0,4.5);
\draw[ultra thin] (-6.5,4)--(6.5,4);

\draw (-3.5,4)node[above]{$\ms L$}--(0,1.5)--(1.5,0);
\draw (-3.5,4)--(-4,0);
\draw  (2,4)--(0,3.5);
\draw (2,4)node[above]{$\ms R$}--(4,0);

\draw (-0.3,4)node[above]{$A$};
\draw (0.5,4)node[above]{$B$};
\draw (-2,0)node[below]{$\mc I_{\ms L}$};
\draw (3,0)node[below]{$\mc I_{\ms R}$};

\draw[dashed] (-3.5,4)--(0,0.25);
\draw[dashed] (-3.5,4)--(0,0);
\draw[dashed] (-3.5,4)--(-0.2,0);
\draw[dashed] (-3.5,4)--(-0.4,0);
\draw[dashed] (-3.5,4)--(-0.6,0);
\draw[dashed] (-3.5,4)--(-0.8,0);
\draw[dashed] (-3.5,4)--(-1,0);
\draw[dashed] (-3.5,4)--(-1.2,0);
\draw[dashed] (-3.5,4)--(-1.4,0);
\draw[dashed] (-3.5,4)--(-1.6,0);
\draw[dashed] (-3.5,4)--(-1.8,0);
\draw[dashed] (-3.5,4)--(-2,0);
\draw[dashed] (-3.5,4)--(-2.2,0);
\draw[dashed] (-3.5,4)--(-2.4,0);
\draw[dashed] (-3.5,4)--(-2.6,0);
\draw[dashed] (-3.5,4)--(-2.8,0);
\draw[dashed] (-3.5,4)--(-3,0);
\draw[dashed] (-3.5,4)--(-3.2,0);
\draw[dashed] (-3.5,4)--(-3.4,0);
\draw[dashed] (-3.5,4)--(-3.6,0);
\draw[dashed] (-3.5,4)--(-3.8,0);

\draw[dashed] (-3.5,4)--(0,0.5)--(0.2,0);
\draw[dashed] (-3.5,4)--(0,0.7)--(0.5,0);
\draw[dashed] (-3.5,4)--(0,0.9)--(0.8,0);
\draw[dashed] (-3.5,4)--(0,1.1)-- (1.1,0);
\draw[dashed] (-3.5,4)--(0,1.3)-- (1.3,0);

\draw[dashed] (2.2,4)--(4.2,0);
\draw[dashed] (2.4,4)--(4.4,0);
\draw (2.6,4)--(4.6,0);
\draw[dashed] (2.6,4)--(4.8,0);
\draw[dashed] (2.6,4)--(5,0);
\draw[dashed] (2.6,4)--(5.2,0);
\draw[dashed] (2.6,4)--(5.4,0);
\draw (2.6,4)--(5.6,0);
\draw (5.2,0) node[below]{$\mc I_{\mc C}$};

\draw[dashed] (3, 4)--(5.8, 0);
\draw[dashed] (3.4, 4)--(6, 0);
\draw[dashed] (3.8, 4)--(6.2, 0);

\draw[dashed] (-3.7, 4)--(-4.2, 0);
\draw[dashed] (-3.9, 4)--(-4.4, 0);
\draw[dashed] (-4.1, 4)--(-4.6, 0);
\draw[dashed] (-4.3, 4)--(-4.85, 0);
\draw[dashed] (-4.5, 4)--(-5.1, 0);

\draw (-4.6, 0) node[below]{$\mc I_{\mc W}$};

\draw (-1.5,4)node[above]{$\tilde{\ms L}$}--(0,2.5)--(2,0);
\draw[dashed] (-2,4)--(0,2.3)--(1.9,0);
\draw [dashed](-2.5,4)--(0,2.1)-- (1.8,0);
\draw[dashed] (-3,4)--(0,1.9)-- (1.7,0);
\draw[dashed] (-3.5,4)--(0,1.7)-- (1.6,0);

\draw[dashed] (-1.25,4)--(0,2.75);
\draw[dashed] (-1,4)--(0,3);
\draw[dashed] (-0.75,4)--(0,3.25);
\draw[dashed] (-0.5,4)--(0,3.5);
\draw[dashed] (-0.25,4)--(0,3.75);
\draw[dashed] (1,4)--(0,3.75);

\draw[smooth, thick] plot coordinates{(0,2.5)(0.2,3)(1,3.6)(2,4)};

\draw (2,4)--(2,0);

\draw[dashed] (0,2.75)--(0.1,2.85);
\draw[dashed] (0,3)--(0.4,3.15);
\draw[dashed] (0,3.25)--(0.8,3.55);

\draw[dashed] (0.2,3)--(2,0);
\draw[dashed] (1.5,3.8)--(2,0);
\draw[dashed] (0.5,3.3)--(2,0);

\draw[dashed] (1,3.6)--(2,0);

\draw[dashed] (2,4)--(2.25,0);
\draw[dashed] (2,4)--(2.5,0);
\draw[dashed] (2,4)--(2.75,0);
\draw[dashed](2,4)--(3,0);
\draw[dashed](2,4)--(3.25,0);
\draw[dashed] (2,4)--(3.5,0);
\draw[dashed] (2,4)--(3.75,0);

\draw (6,4) node[above]{$\omega$};
\end{tikzpicture}
}
\caption{Partition of $\R$ in  Case  1. \blu{The picture displays some connected components in ${\mc I}_{\ms L} \cup {\mc I}_{\ms R} \cup {\mc I}_{\mc C} \cup {\mc I}_{\mc W}$.}}
\label{partition1}
\end{figure}

\begin{figure}[ht]
\centering
\footnotesize{
\begin{tikzpicture}[scale = 0.8]

\draw[->] (-6.5,0)--(6.5,0)node[right]{$x$};
\draw[->] (0,0)--(0,4.5);
\draw[ultra thin] (-6.5,4)--(6.5,4);

\draw (-3.5,4)node[above]{$\ms L$}--(0,1.5)--(1.5,0);
\draw (-3.5,4)--(-4,0);
\draw  (2,4)--(0,3.5);
\draw (2,4)node[above]{$\ms R$}--(4,0);

\draw (-0.3,4)node[above]{$A$};
\draw (0.5,4)node[above]{$B$};
\draw (-1,0)node[below]{$\mc I_{\ms L}$};
\draw (3,0)node[below]{$\mc I_{\ms R}$};

\draw (-3.5,4)--(0,1.1)-- (1.1,0);
\draw[dashed] (-3.5,4)--(0,1.3)-- (1.3,0);

\draw[smooth,  thick]  plot coordinates{(0, 0.5) (-2, 2.6) (-3.5,4)};

\draw (-3.5, 4)--(-0.5, 0)--(0,0.5);
\draw[dashed] (-0.5, 0)--(-3, 3.5);
\draw[dashed] (-0.5, 0)--(-2, 2.6)--(0,0.8)--(0.75,0);
\draw[dashed] (0, 0.4)--(0. 4, 0);
\draw[dashed] (-0.5, 0)--(-1, 1.5);
\draw[dashed] (-0.5, 0)--(-0.5, 1);
\draw[dashed] (-0.5, 0)--(-0.2, 0.7);

\draw[dashed] (-0.2,0)--(0,0.2);

\draw[dashed] (-3.5,4)--(-0.8,0);
\draw[dashed] (-3.5,4)--(-1,0);
\draw[dashed] (-3.5,4)--(-1.2,0);
\draw[dashed] (-3.5,4)--(-1.4,0);
\draw[dashed] (-3.5,4)--(-1.6,0);
\draw[dashed] (-3.5,4)--(-1.8,0);
\draw[dashed] (-3.5,4)--(-2,0);
\draw[dashed] (-3.5,4)--(-2.2,0);
\draw[dashed] (-3.5,4)--(-2.4,0);
\draw[dashed] (-3.5,4)--(-2.6,0);
\draw[dashed] (-3.5,4)--(-2.8,0);
\draw[dashed] (-3.5,4)--(-3,0);
\draw[dashed] (-3.5,4)--(-3.2,0);
\draw[dashed] (-3.5,4)--(-3.4,0);
\draw[dashed] (-3.5,4)--(-3.6,0);
\draw[dashed] (-3.5,4)--(-3.8,0);

\draw (-1.5,4)node[above]{$\tilde{\ms L}$}--(0,2.5)--(2,0);
\draw[dashed] (-2,4)--(0,2.3)--(1.9,0);
\draw [dashed](-2.5,4)--(0,2.1)-- (1.8,0);
\draw[dashed] (-3,4)--(0,1.9)-- (1.7,0);
\draw[dashed] (-3.5,4)--(0,1.7)-- (1.6,0);

\draw[dashed] (-1.25,4)--(0,2.75);
\draw[dashed] (-1,4)--(0,3);
\draw[dashed] (-0.75,4)--(0,3.25);
\draw[dashed] (-0.5,4)--(0,3.5);
\draw[dashed] (-0.25,4)--(0,3.75);
\draw[dashed] (1,4)--(0,3.75);

\draw[smooth,  thick] plot coordinates{(0,2.5)(0.2,3)(1,3.6)(2,4)};

\draw (2,4)--(2,0);

\draw[dashed] (0,2.75)--(0.1,2.85);
\draw[dashed] (0,3)--(0.4,3.15);
\draw[dashed] (0,3.25)--(0.8,3.55);

\draw[dashed] (0.2,3)--(2,0);
\draw[dashed] (1.5,3.8)--(2,0);
\draw[dashed] (0.5,3.3)--(2,0);

\draw[dashed] (1,3.6)--(2,0);

\draw[dashed] (2,4)--(2.25,0);
\draw[dashed] (2,4)--(2.5,0);
\draw[dashed] (2,4)--(2.75,0);
\draw[dashed](2,4)--(3,0);
\draw[dashed](2,4)--(3.25,0);
\draw[dashed] (2,4)--(3.5,0);
\draw[dashed] (2,4)--(3.75,0);

\draw[dashed] (6,4) node[above]{$\omega$};

\draw[dashed] (-3.7, 4)--(-4.2, 0);
\draw[dashed] (-3.9, 4)--(-4.4, 0);
\draw[dashed] (-4.1, 4)--(-4.6, 0);
\draw[dashed] (-4.3, 4)--(-4.85, 0);
\draw[dashed] (-4.5, 4)--(-5.1, 0);

\draw (-4.6, 0) node[below]{$\mc I_{\mc W}$};

\draw[dashed] (2.2,4)--(4.2,0);
\draw[dashed] (2.4,4)--(4.4,0);
\draw (2.6,4)--(4.6,0);
\draw[dashed] (2.6,4)--(4.8,0);
\draw[dashed] (2.6,4)--(5,0);
\draw[dashed] (2.6,4)--(5.2,0);
\draw[dashed] (2.6,4)--(5.4,0);
\draw (2.6,4)--(5.6,0);
\draw (5.2,0) node[below]{$\mc I_{\mc C}$};

\end{tikzpicture}
}
\caption{Partition of $\R$ in  Case  2. \blu{The picture displays some connected components in ${\mc I}_{\ms L} \cup {\mc I}_{\ms R} \cup {\mc I}_{\mc C} \cup {\mc I}_{\mc W}$.}}
\label{partition}
\end{figure}

The elements of this partition 
enjoy the following properties.
\begin{itemize}
\item[-] In the \textsc{Case 1}, the set $\mc I_{\ms L}$ consists of the starting points of a compression wave 
that is partly refracted by the interface, and that generates a shock at the point $(\ms L, T)$.
In the \textsc{Case 2}, only the subsets of $\mc  I_{\ms L}$
given by $\,]\vartheta_{\ms L,-}(0) ,\, -\bs \sigma[\ms L, A, f_l]\cdot f_l^{\prime}(\overline A)[$, \, and\,  $](\ms L/f'_l(A)-T)\cdot f'_r(\overline B),\,\vartheta_{\ms L,+}(0)[$,
consist of the starting points of compression waves with center at the point $(\ms L, T)$.
In the complementary sets of $\mc  I_{\ms L}$:
$]-\bs \sigma[\ms L, A, f_l]\cdot f_l^{\prime}(\overline A),\, 0[$ \, and\, $]0,\, (\ms L/f'_l(A)-T)\cdot f'_r(\overline B)[$,
the initial datum will assume the constant values $\overline A$ and $\overline B$, respectively.
Here $\bs\sigma[\ms L, A, f_l]$ is the
constant defined as in \S~\ref{def:lrs-block},
with $f=f_l$.

\item[-] The set $\mc I_{\ms R}$ consists of the starting points of a compression wave that generates a shock at the point $(\ms R, T)$.

\item[-] The set $\mathcal{I}_{\mc C}$ consists of the 
starting points of compression
waves that generate a shock at points $(y,T)$,
$y\in \,]-\infty, \ms L[\, \cup\,]\ms L, \widetilde {\ms L}[\, \cup\,]\ms R, +\infty[\,$.
The set $\mathcal{I}_{\mc C}$ is a disjoint union of at most countably many open intervals 
of the form 
\vspace{-5pt}
\begin{equation}
\begin{aligned}
\mathcal{I}^n_{\ms L}&=\,]x_n^-,x_n^+[\,, \quad x_n^\pm=\vartheta_{y_n,\pm}(0),
\qquad y_n\in \,]-\infty, \ms L[\,,
\\
\mathcal{I}^n_{\ms R}&=\,]x_n^-,x_n^+[\,, \quad x_n^\pm=\vartheta_{y_n,\pm}(0),
\qquad y_n\in  \,]\ms R,+\infty[\,,
\\
\widetilde{\mathcal{I}}^n_{\ms L}
&=\,]x_n^-,x_n^+[\,, \quad x_n^\pm=\vartheta_{y_n,\pm}(0),
\qquad y_n\in \,]\ms L, \tilde{\ms L}[\,,
\end{aligned}
\end{equation}
which are non empty because of~\eqref{eq:lax-start-char}.

\item[-] The set $\mathcal{I}_{\mc Ra}$ consists of at most countably many points that are the
centers of rarefaction waves originated at time $t=0$. 

\item[-] The set $\mathcal{I}_{\mc W}$ 
consists of the 
starting points of all genuine characteristics reaching points $(y,T)$,
$y\in \,]-\infty, \ms L[\, \cup\,]\ms L, \widetilde {\ms L}[\, \cup\,]\ms R, +\infty[\,$.
\end{itemize}

\subsubsection{Construction of  $AB$-entropy solution on two regions with  vertexes
at $(\ms L, T)$ and at $({\ms R}, T)$}
\label{subsec:const-ABsol}

Consider the two polygonal regions
\begin{equation}
\label{eq:delta-L- gamma-R-def}
\begin{aligned}
    \Delta_{\ms L}&\doteq 
    \Big\{ (x, t) \in\R \times [0, T] \; : \; \vartheta_{\ms L,-}(t) < x < \vartheta_{\ms L,+}(t)\Big\},
    \\
    \noalign{\smallskip}
    \Gamma_{\ms R}&\doteq 
    \Big\{ (x, t) \in\R \times [0, T] \; : \; \vartheta_{\, \widetilde{\ms L}}(t) < x < \vartheta_{\ms R,+}(t)\Big\}.
    \end{aligned}
\end{equation}
In the \textsc{Case 2} 
(see Figure \ref{case2}), 
letting $\Delta[\ms L, A, f_l]$ be the region defined as in \S~\ref{def:lrs-block}, with  $f=f_l$, 
we can express $\Delta_{\ms L}$ as
\begin{equation}
\label{eq:delta-L-def}
\Delta_{\ms L} = 
\Delta[\ms L, A, f_l]\,\cup\, \bigcup_{i=1}^4 \Delta_{\ms L, i},
\end{equation}
where
\begin{equation}
\label{eq:delta-i-def}
    \begin{aligned}
        \Delta_{\ms L, 1}&\doteq 
        \Big\{ (x, t) \in\,]-\infty, 0[\,\times [0,T] \; : \; \vartheta_{\ms L,-}(t) < x \leq  \ms L -(T-t) \cdot f_l^{\prime}(\bs v[\ms L, A, f_l])\Big\},
        \\
        \noalign{\smallskip}
        \Delta_{\ms L, 2}&\doteq  
        \Big\{ (x, t) \in\,]-\infty, 0]\,\times \, [0,T] \; : \;  x\geq  (t-\bs \sigma[\ms L, A, f_l])\cdot f'_l(\,\overline{A}\,)\Big\},
        \\
        \noalign{\smallskip}
        \Delta_{\ms L, 3}&\doteq 
        \Big\{ (x, t) \in\,]0, +\infty[\,\times \,[0,T] \; : \; x< \eta_{\ms L,\,  x({A,\overline{B}})}(t)\Big\},
        \\
        \noalign{\smallskip}
        \Delta_{\ms L, 4}&\doteq
        \Big\{ (x, t) \in\R\, \times\, [0,T] \; : \; 
        \eta_{\ms L,\,  x({A,\overline{B}})}(t)\leq x < \vartheta_{\ms L,+}(t)\Big\},  
    \end{aligned}
\end{equation}
with $\bs v[\ms L, A, f_l]$ as in~\eqref{eq:vlaf-def} taking $f=f_l$, 
and
\begin{equation}
    \label{eq:def-}
    x({A,\overline{B}})\doteq \big(\ms L/f'_l(A)-T\big)\cdot f'_r(\overline{B})>0\,.
\end{equation}
Similarly, in both  \textsc{Cases 1, 2} 
(see Figures \ref{case1}-\ref{case2}), 
letting $\Gamma[\ms R, B, f_r]$,  be the region defined as in \S~\ref{def:rsr-block}, with $f=f_r$, 
we can express $\Gamma_{\ms R}$ as
\begin{equation}
\label{eq:gamma-r-def}
    \Gamma_{\ms R}=\Gamma[\ms R, B, f_r]\,\cup\, \bigcup_{i=1}^3 \Gamma_{\ms R, i},
\end{equation}
where
\begin{equation}
\label{eq:gamma-i-def}
    \begin{aligned}
        \Gamma_{\ms R, 1}&\doteq 
        \Big\{ (x, t) \in\,]-\infty, 0]\,\times\,[0,T] \; : \; \vartheta_{\, \widetilde{\ms L}}(t) < x\Big\},
        \\
        \noalign{\smallskip}
        \Gamma_{\ms R, 2}&\doteq  \Big\{ (x, t) \in\,]0, +\infty[\,\times\,[0,T] \; : \; x\leq \ms R - (T-t)\cdot f'_r(B)\Big\},
        \\
        \noalign{\smallskip}
        \Gamma_{\ms R, 3}&\doteq \Big\{ (x, t) \in\,]0, +\infty[\,\times\,[0,T]\, \; : \; 
        \ms R -(T-t) \cdot f_r^{\prime}(\bs u[\ms R, B, f_r])\leq x<\vartheta_{\ms R,+}(t)\Big\},
        \end{aligned}
        \end{equation}
with $\bs u[\ms R, B, f_l]$ as in~\eqref{eq:urbf-def}, taking $f=f_r$.
\smallskip

Now,  consider the function $ u_{\ms L} : \Delta_{\ms L} \to \R$ defined by
setting for every $(x,t)\in \Delta_{\ms L}$\,:\\
in  \textsc{Case 1}:
\begin{equation}
\label{eq:uL-sol-def-1}
     u_{\ms L}(x, t)  = \begin{cases}
            (f_l^{\prime})^{-1}\big(\frac{\ms L-x}{T-t}\big), & \text{if $x \leq  0$}, \\
            \noalign{\smallskip}
            \pi^l_{r,-}(u_{\ms L,z}),
            & \text{if $x = \eta_{\ms L, z}(t)$, \
             for some \ $z>0$},
        \end{cases}
\end{equation}
where $u_{\ms L,z}$ is defined as in \S~\ref{subsec:charact}, with $y =\ms L$;\\

\noindent
in  \textsc{Case 2}:
\begin{equation}
\label{eq:uL-sol-def-2}
        u_{\ms L}(x, t)  = \begin{cases}            (f_l^{\prime})^{-1}\big(\frac{\ms L-x}{T-t}\big), & \text{if 
        $(x,t)\in \Delta_{\ms L, 1}\cup \Delta_{\ms L, 4},\ \ x\leq 0,$}
            \\
            \noalign{\smallskip}
            \pi^l_{r,-}(u_{\ms L,z}), 
            & \text{if $(x,t)\in \Delta_{\ms L, 4}$, \ $x = \eta_{\ms L, z}(t)$, \
             for some \ $z>0$ }\\
             \noalign{\smallskip}
            \ms v[\ms L, A, f_l](x,t), & \text{if $(x,t) \in \Delta[\ms L, A, f_l],$}\\
            \overline A, & \text{if $(x,t)\in \Delta_{\ms L, 2}$,}\\
            \overline B, & \text{if $(x,t)\in \Delta_{\ms L, 3}$.}\\
        \end{cases}
\end{equation}
where $\ms v[\ms L, A, f_l]$ denotes the function defined in~\eqref{eq:vlaff-def}, with $f=f_l$.
\smallskip

\noindent
By construction, because of~\eqref{eq:flux-assumption-1}, and relying on the analysis in~\ref{def:lrs-block}, it follows that 
in both  \textsc{Cases~1,~2}, the function $u_{\ms L}(x, t)$:
%
\begin{itemize}
[leftmargin=20pt]
\item[-] is locally Lipschitz continuous on
$\big(\Delta_{\ms L} \setminus
\overline{\Delta[\ms L, A, f_l]}\,\big)\cap \big((\R\setminus\{0\})\,\times\,]0,T[\,\big)$, and it is continuous on the boundary $\partial\Delta[\ms L, A, f_l]\setminus \big(\{0\}\,\times\,]0,T[\big)$;
\item[-] is a classical solution of $u_t+f_l(u)_x=0$ on
$\big(\Delta_{\ms L} \setminus
\overline{\Delta[\ms L, A, f_l]}\,\big)\cap \big(\,]-\infty, 0[\,\times\,]0,T[\,\big)$, and of $u_t+f_r(u)_x=0$ on
$\Delta_{\ms L} \cap \big(\,]0, +\infty[\,\times\,]0,T[\,\big)$;
\item[-] is an entropy weak solution of $u_t+f_l(u)_x=0$ on
$\Delta[\ms L, A, f_l]$;
\item[-] satisfies 
the interface entropy condition~\eqref{ABtraces} 
at any point $(0,t)$, $t\leq \tau_+(\ms L)$.
\end{itemize}
Therefore, by Definition~\ref{defiAB}, 
we deduce that $u_{\ms L}$ is an $AB$-entropy solution of~\eqref{conslaw} on $\Delta_{\ms L}$. 

Next, consider (for both  \textsc{Cases~1,~2}) the function $ u_{\ms R} : \Gamma_{\ms R} \to \R$ defined by
setting for every $(x,t)\in \Gamma_{\ms R}$\,:\\
\begin{equation}
\label{eq:uR-sol-def}
        u_{\ms R}(x, t)  = \begin{cases}
            A, & \text{if $(x,t)\in \Gamma_{\ms R,1},$} \\
            \noalign{\smallskip}
            B, & \text{if $(x,t)\in \Gamma_{\ms R,2},$} \\
            \noalign{\smallskip}
            \ms u[\ms R, A, f_r](x, t), & \text{if $(x,t) \in \Gamma[\ms R, B, f_r]$},
            \\
            \noalign{\smallskip}
            (f_r^{\prime})^{-1}\big(\frac{\ms R-x}{T-t}\big), & \text{if $(x,t)\in \Gamma_{\ms R,3},$}
        \end{cases}
\end{equation}
where $\ms u[\ms R, A, f_r]$ denotes the function defined in~\eqref{eq:urbff-def}, with $f=f_r$.
By construction and relying on the analysis in
\S~\ref{def:rsr-block},
we deduce as above that $u_{\ms R}$
provides an $AB$-entropy solution of~\eqref{conslaw}
on~$\Gamma_{\ms R}$.
Moreover, because of~\eqref{eq:case-i-A}, \eqref{eq:case-i-B}, \eqref{eq:case-ii-A},
we have
\begin{equation}
\label{eq:term-cond-23}
    u_{\ms R}(x,T)=\omega(x)
    \qquad \forall~x\in \,]\widetilde {\ms L}, \ms R[\,.
\end{equation}

\subsubsection{Construction of 
$AB$-entropy solution on whole $\R\times [0,T]$}
\label{subsec:const-inda}

Observing that, because of~\eqref{eq:partition}, \eqref{eq:delta-L- gamma-R-def}, we have  $\Delta_{\ms L} \cap \{x=0\}= \mc I_{\ms L}$, $\Gamma_{\ms R} \cap \{x=0\}= \mc I_{\ms R}$,
we define the initial datum on $\mc I_{\ms L}\cup \mc I_{\ms R}$ as
\begin{equation}
    \label{eq:idatum-22}
    u_0(x)=\begin{cases}
        u_{\ms L}(x, 0)\ \ &\text{if}\quad x\in 
        \mc I_{\ms L},
        \\
        \noalign{\smallskip}
        u_{\ms R}(x, 0)\ \ &\text{if}\quad x\in 
        \mc I_{\ms R}.
    \end{cases}
\end{equation}
where, in  \textsc{Case 1},
\begin{equation}
\label{eq:init-datum-L-1}
    u_{\ms L}(x, 0)= \begin{cases}
            (f_l^{\prime})^{-1}\big(\frac{\ms L-x}{T}\big), & \text{if $x\in 
        \mc I_{\ms L},$ \ $x \leq  0$}, \\
            \noalign{\smallskip}
            \pi^l_{r,-}(u_{\ms L,z}),
            & \text{if $x\in 
        \mc I_{\ms L}$, \ $x = \eta_{\ms L, z}(0)$, \
             for some \ $z>0$},
        \end{cases}
\end{equation}
while, in  \textsc{Case 2},
\begin{equation}
        u_{\ms L}(x, 0)  = \begin{cases}            (f_l^{\prime})^{-1}\big(\frac{\ms L-x}{T}\big), & \text{if \
        $x\in \,]\vartheta_{\ms L,-}(0),\, \ms L -T \cdot f_l^{\prime}(\bs v[\ms L, A, f_l])[$,}
            \\
             \noalign{\smallskip}
            \overline A, & \text{if \ $x\in 
            \,]\ms L -T \cdot f_l^{\prime}(\bs v[\ms L, A, f_l]),\, 0[\,$,}\\
             \noalign{\smallskip}
            \overline B, & \text{if \ $x\in\,]0,\, \eta_{\ms L,\,  x({A,\overline{B}})}(0)[$,}
             \\
            \noalign{\smallskip}
            \pi^l_{r,-}(u_{\ms L,z}), 
            & \text{if \ $x\in [\eta_{\ms L,\,  x({A,\overline{B}})}(0),\, \vartheta_{\ms L,+}(0)[\,$, \ $x = \eta_{\ms L, z}(t)$, \
             for some \ $z>0$,}
        \end{cases}
\end{equation}
and, in both  \textsc{Cases 1, 2}, 
\begin{equation}
\label{eq:init-datum-R}
        u_{\ms R}(x, 0)  = 
            (f_r^{\prime})^{-1}\big(
            \textstyle{\frac{\ms R-x}{T}}\big), \quad \text{if \ $x\in \mc I_{\ms R}.$}
\end{equation}
In view of the observations in \S~\ref{subsec:charact}-\ref{subsec:minmaxbackcharact},
the construction of the \blu{$AB$-entropy solution} on 
$\big(\R\times [0,T]\big)\setminus \big(\Delta_{\ms L}\cup \Gamma_{\ms R}\big)$, and the corresponding definition of the initial datum on $\R\setminus\big(\mc I_{\ms L}\cup \mc I_{\ms R}\big)$, proceed as follows:
\begin{itemize}
 [leftmargin=20pt]
\item[-] For any $y\in \,]-\infty, \ms L[\, \cup\,]\ms L, \widetilde {\ms L}[\, \cup\,]\ms R, +\infty[\,$, we trace the lines $\vartheta_{y, \pm}$ starting at $(y,T)$ until they reach the $x$-axis at the point $\phi_{\pm}(y)\doteq  \vartheta_{y, \pm}(0)$. 
Since conditions~\eqref{eq:1a}, \eqref{eq:1b1} of Theorem~\ref{thm:attprofiles} is equivalent to the monotonicity of the map $\phi(y)\doteq \vartheta_{y}(0)$ (see~\cite[Lemma 4.4]{anconachiri}), it follows that $\vartheta_{y, \pm}$ never intersect each other in the region 
$\R\times \,]0, T[\,$.
Then, if $y\in \,]-\infty, \ms L[\,  \cup\,]\ms R, +\infty[\,$, we define a function $u(x,t)$ that is
equal to $\omega(y\pm)$ along the segment $\vartheta_{y, \pm}$. Instead if
$y\in \,]\ms L, \widetilde {\ms L}[\,$
we define $u$ to be equal to $\omega(y\pm)$
along the segment $\vartheta_{y, \pm}(t)$, $\tau_\pm(y)\leq t\leq T$, and to be equal
to $\pi^l_{r,-}(\omega(y\pm))$ along the segment $\vartheta_{y, \pm}(t)$, $0\leq t\leq \tau_\pm(y)$.
\item[-] For any $z\in \mathcal{I}^n_{\ms L}\,\cup\, \mathcal{I}^n_{\ms R}$, we trace the line $\eta_{y_n,z}$, $y_n\in \,]-\infty, \ms L[\,
\cup\,]\ms R, +\infty[\,$. 
By construction the lines $\eta_{y_n,z}$ never cross each other in the region 
$\R\times \,]0, T[\,$.
Then, if $y_n\in \,]-\infty, \ms L[$\,, we define $u(x,t)$ to be equal to
$(f'_l)^{-1}((y-z)/T)
= u_{y,z}$ along the segment $\eta_{y_n,z}$, instead  if $y_n\in \,]\ms R, +\infty[$\,, we define $u(x,t)$ to be equal to
$(f'_r)^{-1}((y-z)/T)
= u_{y,z}$ along the segment $\eta_{y_n,z}$. 
\item[-]  For any $z\in \widetilde{\mathcal{I}}^n_{\ms L}$, we trace the polygonal line $\eta_{y_n,z}$, $y_n\in\,]\ms L, \widetilde {\ms L}[\,$. By construction the lines $\eta_{y_n,z}$ never cross each other in the region 
$\R\times \,]0, T[\,$.
Then, we define $u(x,t)$ to be equal to
$(f'_l)^{-1}((y_n-x)/(T-t))
= u_{y,z}$ along the segment $\eta_{y_n,z}$,
$\tau(y,z)< t \leq T$, and to be equal to 
$\pi^l_{r,-}(u_{y,z})$ along the segment
$\eta_{y_n,z}$,
$0\leq t\leq \tau(y,z)$.
\end{itemize}

Therefore, we define the 
 the function 
\begin{equation}
\label{ab-entr-sol-def-21}
u(x, t)  \doteq \begin{cases}
\omega(y\pm), & \text{if \ $x=\vartheta_{y,\pm}(t)$ \ for some \ $y \in \,]-\infty, \ms L[\, \cup \,]\ms R, +\infty[\,$},\\
\noalign{\smallskip}
\omega(y\pm), & \text{if \ 
$x=\vartheta_{y,\pm}(t)<0$ \ for some \ $y \in \,]\ms L, \widetilde{\ms L}[ \,$},\\
\noalign{\smallskip}
\pi^l_{r,-}(\omega(y\pm)), & \text{if \ 
$x=\vartheta_{y,\pm}(t)>0$ \ for some \ $y \in \,]\ms L, \widetilde{\ms L}[ \,$},\\
\noalign{\smallskip}
(f_l^{\prime})^{-1}\big(\frac{y_n-z}{T}\big), & \text{if \ $ x = \eta_{y_n, z}(t)$ \ for some \ $z \in \mc I^n_{\ms L}$},\\
\noalign{\smallskip}
(f_r^{\prime})^{-1}\big(\frac{y_n-z}{T}\big), & \text{if \ $ x = \eta_{y_n, z}(t)$ \ for some \ $z \in \mc I^n_{\ms R}$},\\
\noalign{\smallskip}
(f_l^{\prime})^{-1}\big(\frac{\ms L-x}{T-t}\big), & \text{if \ $ x = \eta_{y_n, z}(t)<0$ \ for some \ $z \in \widetilde{\mathcal{I}}^n_{\ms L}$},\\
\noalign{\smallskip}
\pi^l_{r,-}(u_{y_n,z}), & \text{if \ $ x = \eta_{y_n, z}(t)>0$ \ for some \ $z \in \widetilde{\mathcal{I}}^n_{\ms L}$},\\
\noalign{\smallskip}
u_{\ms L}(x, t), & \text{if \ $(t, x) \in \Delta_{\ms L}$},\\
\noalign{\smallskip}
u_{\ms R}(x, t), & \text{if \ $(t, x) \in \Gamma_{\ms R}$},
\end{cases}
\end{equation}

\noindent
and the initial datum
\smallskip
\begin{equation}
\label{eq:ab-initial-datum}
u_0(x)  \doteq \begin{cases}
\omega(y\pm), & \text{if \ $x \in \mc I_{\mc W}$, $x = \theta_{y,\pm}(0)$ \ for some \ $y \in\,]-\infty,\ms L[\, \cup \,]\ms R, +\infty[\,$},\\
\noalign{\smallskip}
\pi^l_{r,-}(\omega(y\pm)), & \text{if \ $x \in \mc I_{\mc W}$, $x = \theta_{y,\pm}(0)$, $y \in \,]\ms L, \tilde{ \ms L}[\,$},\\
\noalign{\smallskip}
(f_l^{\prime})^{-1}\big(\frac{y_n-x}{T}\big), & \text{if \ $x \in \mc I^n_{\ms L}$},
\\
\noalign{\smallskip}
(f_r^{\prime})^{-1}\big(\frac{y_n-x}{T}\big), & \text{if \ $x \in \mc I^n_{\ms R}$},
\\
\noalign{\smallskip}
\pi^l_{r,-}(u_{y_n,x}), & \text{if \ $x \in \widetilde{\mathcal{I}}^n_{\ms L}$},
\\
\noalign{\smallskip}
u_{\ms L}(x,0), & \text{if \ $x \in \mc I_{{\ms L}}$},
\\
\noalign{\smallskip}
u_{\ms R}(x,0), & \text{if \ $x \in \mc I_{\ms R}$}.
\end{cases}
\end{equation}
Notice that $u_0$ is not defined on the countable set $\mathcal{I}_{\mc Ra}$ which is  of measure zero, and clearly $u_0\in {\bf L}^\infty(\R)$. 
By construction, the function $u(x,t)$: 
\begin{itemize}
[leftmargin=20pt]
\item[-] is locally Lipschitz continuous on
$\big(\R\times \,]0,T[\,\big)\setminus 
\big(\,\overline{\Delta_{\ms L}\cup \Gamma_{\ms R}}
\,\cup\, (\{0\}\,\times\,]0,T[\,)\big)$,
and it is continuous on the boundary $\partial \big(\Delta_{\ms L}\cup \Gamma_{\ms R}\big) \setminus \big(\{0\}\,\times\,]0,T[\big)$;
\item[-]
is a classical solution of $u_t+f_l(u)_x=0$ on
$\big(\,]-\infty, 0[\,\times\,]0,T[\,\big)\setminus \,\overline{\Delta_{\ms L}\cup \Gamma_{\ms R}}$\,;
\item[-]
is a classical solution of $u_t+f_r(u)_x=0$ on
$\big(\,]0, +\infty[\,\times\,]0,T[\,\big)\setminus \,\overline{\Delta_{\ms L}\cup \Gamma_{\ms R}}$\,;
\item[-] is an $AB$-entropy solution of~\eqref{conslaw} on $\Delta_{\ms L}\cup\Gamma_{\ms R}$;
\item[-]satisfies 
the interface entropy condition~\eqref{ABtraces} 
at any point $(0,t)$, $t\in \,]0,T[\,$.
\end{itemize}
Thus, by Definition~\ref{defiAB}, 
it follows that the function $u(x,t)$ in~\eqref{ab-entr-sol-def-21}
provides an
 $AB$-entropy solution to \eqref{conslaw}
on $\R\times [0,T]$. Moreover, 
because of~\eqref{eq:term-cond-23},  \eqref{ab-entr-sol-def-21}, \eqref{eq:ab-initial-datum},
we have 
\begin{equation}
\label{eq:indatum-finalprpfile}
    u(x,0) = u_0(x),\qquad\quad u(x,T)=\omega(x)\qquad\ \text{for a.e.}\quad x\in\R\,.
\end{equation}
This proves that 
\begin{equation}
\label{eq:att-prof-sol}
  \omega= \mc S_T^{[AB]+}u_0,  
\end{equation}
and thus
$\omega \in \mc A^{AB}(T)$, which completes
the proof of 
the implication $(3) \Rightarrow (1)$  of Theorem~\ref{thm:backfordiscfluxcycle} in the case of a non critical connection.

\subsection{Part 2.b - $(3) \Rightarrow (2)$ for non critical connections.}\label{subsec:part3b}
As a byproduct of the construction described in~\S~\ref{sec:(3)-(2)}, 
we show in this Subsection  that, 
if $\omega$ satisfies~\eqref{eq:LR-def-24}, \eqref{eq:hyp-3a2}, 
then $\omega$ verifies condition (2) of 
Theorem \ref{thm:backfordiscfluxcycle}, 
i.e. $\omega$ is a fixed point 
of the map $\omega\mapsto \mc \sabpT \circ \sabmT \omega$.
We shall assume that $\omega$ satisfies the pointwise constraints of~\textsc{Case 1} discussed in Remark~\ref{rem:threecases}, the other cases being simmetric, or entirely similar,
or simpler.

In order to verify that $\mc \sabpT \circ \sabmT \omega = \omega$, 
because of~\eqref{eq:att-prof-sol} it is sufficient to prove that,
letting $u_0$ be the function defined
by~\eqref{eq:init-datum-L-1}, \eqref{eq:init-datum-R}, \eqref{eq:ab-initial-datum}, \blu{it holds true}
\begin{equation}
\label{eq:indatum-back-1}
    u_0= \sabmT \omega\,.
\end{equation}
In turn, recalling the definition ~\eqref{eq:backw-conn-sol-def} of $AB$ backward solution operator, the equality~\eqref{eq:indatum-back-1}
is equivalent to the equality
\begin{equation}
  \label{eq:indatum-back-2}  
  u_0(-x)=\sobapT
    \big(\omega(-\ \cdot\,)\big)
     (x)\qquad x\in\R, 
\end{equation}
where 
\begin{equation}
(x,t)\mapsto \sobapt 
    \big(\omega(-\ \cdot\,)\big)
     (x)
\end{equation}
denotes
the unique $\overline B\,\overline A$-entropy solution
of
\begin{equation}
\label{eq:invertedproblem-2}
    \begin{cases}
        v_t+\overline{f}(x,v)_x = 0 & x \in \mathbb{R}, \quad t \geq 0, \\
            \noalign{\smallskip}
        v(x,0) = \omega(-x) &  x \in \mathbb R,
    \end{cases}
\end{equation}
$\overline{f}(x,v)$ being the symmetric flux in~\eqref{eq:symm-flux}.

Towards a proof of~\eqref{eq:indatum-back-2}, we will determine the solution of~\eqref{eq:invertedproblem-2} on $\R\times [0,T]$ relying on the construction in~\S~\ref{sec:(3)-(2)} and on the properties of the
left forward rarefaction-shock wave pattern
derived in \S~\ref{def:lrs-block}.
Observe that  the function $u(x,t)$ defined by~\eqref{ab-entr-sol-def-21}
for ~\textsc{Case 1},
with $u_{\ms L}$, $u_{\ms R}$ defined by~\eqref{eq:uL-sol-def-1}, \eqref{eq:uR-sol-def}, respectively, is:
\begin{itemize}
[leftmargin=20pt]
\item[-] locally Lipschitz continuous in the region
\begin{equation*}
    \msc L\doteq \big(\,\R\times \,]0,T[\,\big) \setminus \big( (\{0\}\,\times\,]0,T[\,) \cup \Gamma[\ms R, B, f_r] \big)
\end{equation*}
where $\Gamma[\ms R, B, f_r]$,  is defined as in \S~\ref{def:rsr-block}, with $f=f_r$
(the region $\msc L$ is the complement of the pink region and of the axis $\{x = 0\}$ in Figure \ref{case1});
\item[-]
a classical solution of $u_t+f_l(u)_x=0$ on
$\,]-\infty, 0[\,\times\,]0,T[\,$;
\item[-]
a classical solution of $u_t+f_r(u)_x=0$ on
$\big(\,]0, +\infty[\,\times\,]0,T[\,\big)\setminus \overline{\Gamma[\ms R, B, f_r]}$\,;
\item[-]satisfies 
the interface entropy condition~\eqref{ABtraces} 
at any point $(0,t)$, $t\in \,]0,T[\,$.
\end{itemize}
\smallskip

\noindent
Therefore, if we define the transformation $(x, t) \mapsto \alpha( x,t)\doteq  (-x, T-t)$,  the function
\begin{equation}
\label{eq:ualpha-def}
    v(x,t) \doteq u(-x, T-t),
    \qquad (x,t) \in \alpha(\,\overline{\msc L})\setminus (\{0\}\,\times\,]0,T[\,),
\end{equation}
is: 
\begin{itemize}
[leftmargin=20pt]
\item[-] an entropy weak solution of $v_t + f_r(v)_x = 0$ in the open set $\alpha(\msc L) \cap \{x < 0\}$;
\item[-] an entropy weak solution of
$v_t + f_l(v)_x = 0$ in the open set  $\alpha(\msc L) \cap \{x > 0\}$.
\end{itemize}
\smallskip

\noindent
On the other hand, 
letting $\Delta[\bs y[\ms R, B, f_r], \overline B, f_r]$ denote the region defined in~\eqref{eq:deltadef} with $\ms L =y[\ms R, B, f_r]$, $A=\overline B$, and $f=f_r$.
one can directly verify that
\begin{equation}
\label{eq:imgamma}
    \alpha\big(\Gamma[\ms R, B,f_r]\big)=\Delta[\bs y[\ms R, B, f_r], \overline B, f_r]
    \subset \,]-\inf, 0[\times\,]0,T[\,.
\end{equation} 
Notice that $(\R\times [0,T])\setminus (\{0\}\,\times\,]0,T[\,)$ is the disjoint union of
$\alpha(\,\overline{\msc L})\setminus (\{0\}\,\times\,]0,T[\,)$ and of $\alpha\big(\Gamma[\ms R, B,f_r]\big)$.
Then, letting $\ms v [\bs y[\ms R, B, f_r], \overline B, f_r](x,t)$ denote the function defined in~\eqref{eq:vlaff-def}, with $\ms L =y[\ms R, B, f_r]$, $A=\overline B$, and $f=f_r$,
consider the function $v:\R\times [0,T]\to\R$ defined by setting
\begin{equation}
    \label{eq:sol-invertedproblem-2}
    v(x,t)\doteq
    \begin{cases}
            u(-x, T-t), & \text{if \ \ $(x,t) \in \alpha(\,\overline{\msc L})\setminus (\{0\}\,\times\,]0,T[\,)$},\\
            \noalign{\smallskip}
            \ms v[\bs y[\ms R, B, f_r], \overline B, f_r](x,t), & \text{if \ \ $(x,t) \in \Delta[\bs y[\ms R, B, f_r], \overline B, f_r]$}\,.
        \end{cases}
\end{equation}
By construction and because of the analysis in
\S~\ref{def:lrs-block}, the function $v(x,t)$ :
\begin{itemize}
[leftmargin=20pt]
\item[-] is locally Lipschitz continuous on 
$\big(\R\times \,]0,T[\,\big)\setminus 
\big(\,\overline{\Delta[\bs y[\ms R, B, f_r], \overline B, f_r]}\,
\cup \, (\{0\}\,\times\,]0,T[\,)\big)$,
and it is continuous on the boundary $\partial \big(\Delta[\bs y[\ms R, B, f_r], \overline B, f_r]\big) \setminus \big(\{0\}\,\times\,]0,T[\big)$;
\smallskip
\item[-]
is a classical solution of $v_t+f_r(v)_x=0$ on
$\big(\,]-\infty, 0[\,\times\,]0,T[\,\big)\setminus \,\overline{\Delta[\bs y[\ms R, B, f_r], \overline B, f_r]}$\,;
\smallskip
\item[-]is an entropy weak solution of $v_t+f_r(v)_x=0$ on $\Delta[\bs y[\ms R, B, f_r], \overline B, f_r]$;
\smallskip
\item[-]
is a classical solution of $v_t+f_l(v)_x=0$ on
$\,]0, +\infty[\,\times\,]0,T[\,$\,;
\smallskip
\item[-] satisfies 
the $\overline{B} \overline{A}$ interface entropy condition,
namely, setting $v_l(t)\doteq v(0-,t)$, $v_r(t)\doteq v(0+,t)$, and considering the function
\begin{equation*}
\label{IBAdef}
I^{\overline B \overline A}(v_l,v_r)  \stackrel{\cdot}{=} \mathrm{sgn}(v_r-\overline A\,)\left(f_l(v_r)-f_l(\,\overline A\,)\right) \\- \mathrm{sgn}(v_l-\overline B\,)\left(f_r(v_l)-f_r(\,\overline B\,)\right),
\end{equation*}
\blu{it holds true}
\begin{equation}
    f_r(v_l(t))=f_l(v_r(t)),\qquad\quad 
    I^{\overline B \overline A}\big(v_l(t), v_r(t)\big)\leq 0,\qquad 
    \text{for a.e. $t\in\,]0,T[$\,.}
\end{equation}
\end{itemize}
Notice also that,
%
because of~\eqref{eq:indatum-finalprpfile}, \eqref{eq:imgamma}, \eqref{eq:sol-invertedproblem-2}, it follows
\begin{equation}
    v(x,0)=u(-x,T)=\omega(-x)
    \qquad\ \text{for a.e.}\quad x\in\R\,.
\end{equation}
Therefore, by Definition~\ref{defiAB}, 
we deduce that the function $v(x,t)$ in~\eqref{eq:sol-invertedproblem-2}
provides the $\overline B\,\overline A$-entropy solution to \eqref{eq:invertedproblem-2}
on $\R\times [0,T]$, and hence
we have
\begin{equation}
    v(x,t)=\sobapt 
    \big(\omega(-\ \cdot\,)\big)
     (x)\qquad x\in\R,\ t\in [0,T]\,.
\end{equation}
Moreover, by~\eqref{eq:indatum-finalprpfile}, \eqref{eq:sol-invertedproblem-2}, it holds true
\begin{equation*}
    \sobapT 
    \big(\omega(-\ \cdot\,)\big)
     (x)
         = u(-x,0)=u_0(-x)\qquad x\in\R\,,
\end{equation*}
which proves~\eqref{eq:indatum-back-2}, and thus concludes the proof of the implication $(3) \Rightarrow (2)$  of Theorem~\ref{thm:backfordiscfluxcycle} in the case of a non critical connection.

\medskip

 \subsection{Part 3.a - $(1) \Leftrightarrow  (2)$ for critical connections}\label{sec:criticalcases-a} 
In this Subsection we \blu{rely on the fact}
that 
the equivalence of conditions (1), (2) of Theorem~\ref{thm:backfordiscfluxcycle} holds for connections which are non critical
\blu{(by the proofs in \S~\ref{sec:(1)-(3)}, \ref{subsec:reductionBV}, \ref{sec:(3)-(2)},
\ref{subsec:part3b})},
and we will show that it remains true also for critical connections.
To fix the ideas, throughout this and the following subsections we shall assume that 
the connection $(A, B)$ is critical at the left, i.e. that 
\begin{equation}
\label{eq:critical-ip}
    A = \theta_l\,,
\end{equation}
the case where one assumes that
$B = \theta_r$ being symmetric.
Notice that the assumption  $A = \theta_l$ does not prevent the connection to be critical also at the right, i.e. $B = \theta_r$: it might or might not happen.
\blu{Notice that there exists } a
sequence 
$\{A_n,B_n\}_n$  of non critical connections that satisfy
\begin{equation}
\label{eq:AnBnconv}
    \lim_n (A_n,B_n) = (A,B)\,.
\end{equation}

%
%
%
We will show only that 
$(1)\Rightarrow (2)$, since
the reverse implication  is clear
(see~\S~\ref{roadmap}).
Then, given 
$\omega\in \mc A^{[AB]}(T)$
with
\begin{equation}
\label{eq:att-prof-21}
    \omega = \mc S_T^{[AB]+}u_0, \qquad u_0 \in \mathbf L^{\infty}(\mathbb R)\,,
\end{equation}
set 
\begin{equation}
\label{eq:att-prof--approx}
    \omega_n \doteq \mc S^{[A_nB_n]+}_T u_0\,.
\end{equation}
Hence, since $\omega_n\in  \mc A^{[A_nB_n]}(T)$,
by the validity of Theorem~\ref{thm:backfordiscfluxcycle} in the non critical case it holds
\begin{equation}\label{eq:noncrifixedn}
    \omega_n=  \mc S^{[A_nB_n]+}_T \circ \mc S^{[A_nB_n]-}_T \omega_n \qquad \forall \; n\,.
\end{equation}
Notice that by definition~\eqref{eq:backw-conn-sol-def} 
it follows that the 
$\mathbf L^1_{\blu{\mr{loc}}}$ stability property (iv) of Theorem~\ref{theoremsemigroup} holds also for the 
$AB$-backward solution operator $\mc S^{[AB]-}_T$, so that we have
\begin{equation}
\label{eq:backelleunostab}
    \mc S^{[A_nB_n]-}_T \omega_n
    \ \ \rightarrow\ \
    \mc S^{[AB]-}_T\omega
    \qquad\text{in}\quad
    \mathbf L^1_{\blu{\mr{loc}}}(\R)\,.
\end{equation}
Hence, we deduce that
$$
\omega \stackrel{[\text{Thm \ref{theoremsemigroup}-(iv)]}}{=} \lim_n \omega_n \stackrel{[\text{\eqref{eq:noncrifixedn}]}}{=} \lim_n  \mc S^{[A_nB_n]+}_T \circ \mc S^{[A_nB_n]-}_T \omega_n \stackrel{[\text{Thm \ref{theoremsemigroup}-(iv) and~\eqref{eq:backelleunostab}]}}{=}  \mc S^{[AB]+}_T \circ \mc S^{[AB]-}_T \omega\,,
$$
which proves $(1)\Rightarrow (2)$. 
\qed

 \subsection{Part 3.b - $(1) \Rightarrow (3)$ for critical connections}\label{sec:criticalcases-b} 
In this Subsection \blu{we rely on the fact} that 
the implication $(1)\Rightarrow (3)$ of Theorem~\ref{thm:backfordiscfluxcycle} holds for connections which are non critical, and in particular \blu{we know (by \S~\ref{sec:(1)-(3)}, \ref{subsec:reductionBV}, \ref{sec:(3)-(2)})} that
Theorems~\ref{thm:attprofiles}, \ref{thm:attprofiles2}, \ref{thm:attprofiles3}, 
are verified 
for non critical connections.
We will prove that, for a critical connection
$(A,B)$, 
any element $\omega\in \mc A^{AB}(T)$
satisfies the conditions of~\blu{Theorem~\ref{thm:attprofilescrit},} or of
Theorem~\ref{thm:attprofiles2}, or of Theorem~\ref{thm:attprofiles3}. 
We divide the proof 
in nine steps.
\smallskip

\noindent
\textbf{Step 1.}
Let $\{A_n,B_n\}_n$  be a sequence of non critical connections as in Part 3.a.
Given $\omega\in \mc A^{AB}(T)$
as in~\eqref{eq:att-prof-21}, and
$\omega_n$,
as in~\eqref{eq:att-prof--approx},
set
\begin{equation}
\label{eq:u-def-34}
    u(x,t) \doteq \mc S^{[AB]+}_t u_0(x), \qquad t \geq 0, \quad x \in \mathbb R\,,
\end{equation}
and consider the sequence $u_n$ of $A_nB_n$-entropy weak solutions  defined by 
\begin{equation}
\label{eq:un-def-35}
    u_n(x,t) \doteq \mc S^{[A_nB_n]+}_t u_0(x), \qquad t \geq 0, \quad x \in \mathbb R\,.
\end{equation}
Let $u_{n,l}, u_{n,r}$ denote, respectively, the left and right traces of $u_n$ at $x=0$ defined as in~\eqref{traces}, and let $u_l, u_r$ be the left and right traces of $u$ at $x=0$
\blu{(whose existence is derived in Steps 5, 8)}.
Then, by Theorem \ref{theoremsemigroup} 
and Corollary~\ref{cor:fluxtraces-stab},
and because of~\eqref{eq:AnBnconv},
it follows
\begin{align}
\label{eq:eq:un-41}
    &u_n(\cdot, t) \ \ \rightarrow\ \ u(\cdot, t)\qquad\text{in}\quad
    \mathbf L^1_{\blu{\ms loc}}(\R)\qquad \forall~t\in [0,T]\,,
    \\
    \noalign{\smallskip}
    \label{eq:eq:flun-41}
    &f_l(u_{n,l}) 
     \ \rightharpoonup\ \ f_l(u_l)\qquad\ \text{in}\quad
    \mathbf L^1([0,T])\,\,,
    \\
    \noalign{\smallskip}
    \label{eq:eq:flrun-41}
    &f_r(u_{n,r}) 
     \ \rightharpoonup\ \ f_r(u_r)\quad\ \ \, \text{in}\quad
    \mathbf L^1([0,T])\,\,,
\end{align}
and hence, in particular, we have
\begin{equation}
\label{eq:omegan-conv-35}
    \omega_n
\ \ \rightarrow\ \
\omega
\qquad\text{in}\quad
    \mathbf L^1_{\blu{\mr{loc}}}(\R)\,.
\end{equation}
In order to prove that $\omega$ satisfies condition (3) of Theorem~\ref{thm:backfordiscfluxcycle},
letting
\begin{equation}
\label{eq:LR-def-21}
\ms L \doteq \ms L[\omega,f_l]\,,
     \qquad\quad  
     \ms R \doteq  \ms R[\omega, f_r],
\end{equation}
be quantities defined as in~\eqref{eq:LR-def}, 
we need to show that:
\begin{itemize}
    \item[-] If $\ms L = 0$, $\ms R > 0$ or viceversa, then $\omega$ satisfies the conditions of Theorem \ref{thm:attprofiles2};
    \item[-] If $\ms L = 0$, $\ms R = 0$, then $\omega$ satisfies the conditions of Theorem \ref{thm:attprofiles3}.
    \blu{ \item[-] If $\ms L < 0$, $\ms R > 0$, then $\omega$ satisfies the conditions of Theorem \ref{thm:attprofilescrit}.}
\end{itemize}
\blu{We shall first address  the two cases $\ms L =0, \ms R > 0$, and $\ms L = 0$, $\ms R = 0$ (the analysis of the case $\ms L < \ms R=0$ being entirely similar to the one of $\ms L =0< \ms R$). 
Next, we shall analyze the third case $\ms L < 0 <\ms R$.
}
Throughout the subsection we let  $\ms L_n$,. $\ms R_n$, denote the objects defined as in~\eqref{eq:LR-def-21} for $\omega_n$:
\begin{equation}
\label{eq:LR-n-def-21}
    \ms L_n \doteq \ms L[\omega_n,f_l]\,,
      \quad\qquad \ms R_n \doteq  \ms R[\omega_n, f_r]\,.
\end{equation}
%
Observe that by Remark~\ref{rem:linf-bound-ABsol}, and because of~\eqref{eq:omegan-conv-35}, the functions $\omega_n$
have a uniform bound 
\begin{equation}
\label{linf-bound-omega-n}
    \|\omega_n\|_{{\bf L}^\infty}\leq \overline C
    \qquad\quad\forall~n,
\end{equation}
for constant~$\overline C>0$.
Hence, by definition ~\eqref{eq:LR-def}, 
the constant $|\ms L_n|$
are bounded by $T\cdot \sup_{|u|\leq \overline C} |f'_l(u)|$, and 
the constant~$\ms R_n$
are bounded by $T\cdot \sup_{|u|\leq \overline C} |f'_r(u)|$. 
Thus, up to a subsequence, we can 
define the limits
\begin{equation}
\label{eq:RLn-lim-24}
     \widehat {\ms L}\doteq 
    \lim_{n \to \infty} \ms L_n\,,
    \qquad\quad
     \widehat{\ms R}\doteq 
    \lim_{n \to \infty} \ms R_n\,.
\end{equation}
We claim that 
\begin{equation}\label{eq:Rsemic}
    \widehat {\ms L}\leq {\ms L},
    \qquad\quad
     \widehat {\ms R}\geq \ms R\,.
\end{equation}
 By definition~\eqref{eq:LR-def}, \eqref{eq:LR-def-21} of $\ms R$, 
 in order to prove the second inequality in~\eqref{eq:Rsemic},
 it is sufficient to show that 
 \begin{equation}
 \label{eq:Rsemic2}
     \widehat {\ms R}
     - T\cdot f'_r(\omega(\widehat {\ms R}+))\geq 0\,.
 \end{equation}
Observe that by Definition~\ref{defiAB}
$u_n$ and $u$ are entropy weak solutions of
\begin{equation}
  \label{eq:conlaw-n-qp}
u_t+f_r(u)_x=0 
\quad\ x >0, \quad t \in [0,T],
  \end{equation}
that, because of~\eqref{eq:eq:un-41},
\eqref{eq:eq:flrun-41}, satisfy the
assumptions~\eqref{eq:conlaw-qp-conv1}, \eqref{eq:conlaw-qp-conv2}
of Lemma~\ref{lemma:chara} in Appendix~\ref{app:uplwsmicsolns}. 
Thus, applying~\eqref{eq:liminf},
and recalling~\eqref{eq:att-prof-21}, \eqref{eq:att-prof--approx}, we find
\begin{equation}
\label{eq:liminf-omega-27}
     \omega(\widehat {\ms R}+) \leq \liminf_{\substack{n \to \infty \\ y \to \widehat {\ms R},\, \blu{y >0}}} \; \omega_n(y+)\,.
\end{equation}
Since~\eqref{eq:RLn-lim-24} and the liminf property imply
\begin{equation}
    \liminf_{\substack{n \to \infty \\ y \to \widehat {\ms R},\, \blu{y >0}}} \; \omega_n(y+)\leq
    \liminf_n \omega_n(\ms R_n+)\,,
\end{equation}
we derive from~\eqref{eq:liminf-omega-27} that
\begin{equation}
\label{lsemic-omega-2}
\omega(\widehat {\ms R}+) \leq \liminf_n \omega_n(\ms R_n+)\,.
\end{equation}
On the other hand, by definition~\eqref{eq:LR-def}, 
\eqref{eq:LR-n-def-21}
of $\ms R_n$, it holds
\begin{equation}
\label{eq:omegan-ineq-23}
    \omega_n(\ms R_n+) \leq (f_r^{\prime})^{-1}\left(\frac{\ms R_n}{T}\right)\,.
\end{equation}
Hence, from~\eqref{lsemic-omega-2}, \eqref{eq:omegan-ineq-23} and~\eqref{eq:RLn-lim-24}
we deduce 
\begin{equation}
    \omega(\widehat {\ms R}+)\leq 
    \lim_{n\to\infty}
    (f_r^{\prime})^{-1}\left(\frac{\ms R_n}{T}\right)
    =(f_r^{\prime})^{-1}\left(\frac{\widehat{\ms R}}{T}\right)\,,
\end{equation}
which yields~\eqref{eq:Rsemic2}.
This completes the proof of the second inequality in~\eqref{eq:Rsemic}, while the proof of
the first one is entirely similar.

\smallskip
Relying on~\eqref{eq:Rsemic}, we will  show in {Steps~2-7}
\blu{the existence of~$\omega(0\pm)$, and}
that 
$\omega$ satisfies 
the
conditions (i)', (ii)' of Theorem~\ref{thm:attprofiles2} 
in the case
$\ms L = 0$, $\ms R > 0$.
Namely, 
in { Step 2} we prove~\eqref{eq:2b2-prime},
 in { Step~3} we prove~\eqref{eq:2a2b-prime},
in { Step~4} we prove~\eqref{eq:3a2b-prime},
in { Step~5} we prove~\eqref{eq:1b2-prime}
\blu{and the existence of $\omega(0\pm)$},
 while in Step~6 we prove~\eqref{eq:2a2-prime}.
Finally, in { Step 7} we 
prove~\eqref{eq:1a2-prime},
concluding the proof of
conditions (i)', (ii)' of Theorem~\ref{thm:attprofiles2}.
The proof \blu{of the existence of~$\omega(0\pm)$, and} that 
$\omega$ satisfies  
conditions~(i),~(ii) of Theorem~\ref{thm:attprofiles2} in the case
$\ms L < 0$, $\ms R = 0$ is entirely similar
\blu{to the case $\ms L = 0$, $\ms R > 0$, although the symmetry is broken
(because of assumption~\eqref{eq:critical-ip}),
and it is briefly discussed in Step 8. }
Next, in {Step~9} we will show
that $\omega$ satisfies
conditions~(i),~(ii) of Theorem~\ref{thm:attprofiles3}
 in 
the case $\ms L = 0$, $\ms R = 0$.
\blu{Finally, in Steps~10-13 we will show
that $\omega$ satisfies
conditions~(i),~(ii) of Theorem~\ref{thm:attprofilescrit}}.

\medskip

\noindent
\textbf{Step 2.} \big($\ms L = 0$, $\ms R > 0$, proof of~\eqref{eq:2b2-prime}:\, $\omega(x) \geq B$ in $]0, \ms R[\,$\big). \\
Applying~\eqref{eq:2bl-r1}, \eqref{eq:2b-lr2} of 
Theorem~\ref{thm:attprofiles}-(ii) or~\eqref{eq:2b2-prime} of
Theorem~\ref{thm:attprofiles2}-(ii)' for $\omega_n$, 
in the case of the non critical connections $(A_n, B_n)$, we deduce
that 
\begin{equation}
\label{eq:eq:sl-max-bckw-3}
    \omega_n(x) \geq B_n\qquad\quad  \forall~
    x \in \,]0, \ms R^n[\,,\quad\forall~n\,.
\end{equation}
%
On the other hand, by virtue of~\eqref{eq:omegan-conv-35}, \eqref{eq:Rsemic},
we can extract a subsequence of $\{\omega_n\}$ that converges to $\omega$ for almost every $x\in \,]0,\ms R[$. Then, taking the limit in~\eqref{eq:eq:sl-max-bckw-3},  relying on~\eqref{eq:AnBnconv}, 
\blu{and because of the normalization of
$\omega$ as a right continuous function (see  Remark~\ref{rem:BVclass}),} we derive~\eqref{eq:2b2-prime}.
\medskip

\noindent
\textbf{Step 3.} \big($\ms L = 0$, $\ms R > 0$, 
proof of~\eqref{eq:2a2b-prime}:\, $\ms R \in \,]0, T \cdot f_r^{\prime}(B)[\ \ \Rightarrow \ \omega(\ms R+) \leq \bs u[\ms R, B, f_r]\leq \omega(\ms R-)$\big).\\
Observe first that, in the case
$\widehat {\ms R} = \ms R$, 
by the continuity of the function
$\bs u[\ms R, B, f_r]$ with respect to $\ms R$, $B$ (see \S~\ref{defi:ur}), and because of~\eqref{eq:AnBnconv}, we find
\begin{equation}
\label{eq:lim-u-23}
    \lim_{n\to \infty} \bs u[\ms R_n, B_n, f_r] 
    =\bs u[\ms R, B, f_r]\,. 
\end{equation}
On the other hand, 
if $\ms R=\widehat {\ms R} \in \,]0, T \cdot f_r^{\prime}(B)[\,$,
we may assume that $\ms R_n \in  \,]0, T \cdot f_r^{\prime}(B_n)[\,$, 
for $n$ sufficiently large.
Hence,  
applying 
either~\eqref{eq:2a} of
Theorem~\ref{thm:attprofiles}-(ii), or~\eqref{eq:2a2b-prime} 
of Theorem~\ref{thm:attprofiles2}-(ii)' for the corresponding $\omega_n$
in the case of the non critical connections $(A_n, B_n)$, we deduce
\begin{equation}
\label{eq:lim-omegan-23}
    \liminf_{n \to \infty} \omega_n(\ms R_n+) \leq \lim_{n\to \infty} \bs u[\ms R_n, B_n, f_r] \,.
\end{equation}
Then, combining~\eqref{eq:lim-u-23}, \eqref{eq:lim-omegan-23},
with~\eqref{lsemic-omega-2},
and recalling~\eqref{eq:uB-ineq}
with $f=f_r$,
we derive
\begin{equation}
\label{eq:eq:2a2b-prime-1}
    \omega(\ms R+) \leq 
    \bs u[\ms R, B, f_r]<B\,,
\end{equation}
which, together with~\eqref{eq:2b2-prime}
(established in Step 2),
proves~\eqref{eq:2a2b-prime} in the case 
$\widehat {\ms R} = \ms R$.

Thus, because of~\eqref{eq:Rsemic},
it remains to analyze the case
$\widehat {\ms R} > \ms R$.
Let
$\vartheta_n^-$
denote the
minimal 
backward characteristic
for $u_n$ starting from
$(\ms R_n,T)$
and lying in the domain $x>0$.
Recalling the definition
\eqref{eq:LR-def}, \eqref{eq:LR-n-def-21}
of $\ms R_n$
this is a map
$\vartheta_n^- :\ ]\tau_n, T]\to \ ]0,+\infty[\,,$ $\tau_n\geq 0$,
with the property that 
$\lim_{t\to\tau_n} \vartheta_n^-(t)=0$.
By possibly taking a subsequence, we may assume that $\{\tau_n\}_n$ converges to some $\overline\tau\geq 0$.
Observe that $\vartheta_n^-$
are genuine characteristics which, up to a subsequence, converge to a genuine characteristic $\vartheta^-:\ ]\overline \tau, T]\to \ ]0,+\infty[\,$
for $u$, starting from
$(\vartheta^-(T),T)=(\widehat {\ms R},T)$
(see proof of Lemma~\ref{lemma:chara} in Appendix~\ref{app:uplwsmicsolns}). 
The trajectory of $\vartheta^-$ is a segment
with slope $f'_r(u(\widehat {\ms R}-,T))=f'_r(\omega(\widehat {\ms R}-))$.
Therefore, if $\overline\tau > 0$,
it follows that
$f'_r(\omega(\widehat {\ms R}-))>\widehat {\ms R}/T$, which 
by definition of $\ms R$ implies $\ms R \geq \widehat{\ms R}$,
contradicting the assumption $\widehat {\ms R} > \ms R$.
Hence, it must be $\overline \tau=0$,
$\lim_{t\to 0} \vartheta^-(t)=0$,
and the trajectory of $\vartheta^-$ is a segment joining the point $(\widehat {\ms R},T)$
with the origin $(0,0)$.
Since $\ms R<\widehat {\ms R}$ and because backward characteristics
cannot intersect in the domain $x>0, t>0$, 
this in turn implies that
the slope $f'_r(\omega(\ms R+))$
of the maximal 
backward characteristic for $u$ starting at
$(\ms R,T)$ must be
greater or equal than $\ms R/T$. On the other hand,
by definition~\eqref{eq:LR-def}, \eqref{eq:LR-def-21}
of $\ms R$, we have
$f'_r(\omega(\ms R+))\leq \ms R/T$, and hence it follows
that 
\begin{equation}
\label{eq:sl-max-bckw-2}
    f'_r(\omega(\ms R+))={\ms R}/{T}.
\end{equation}
Observe now that, 
applying 
Theorem~\ref{thm:attprofiles}-(ii) or
Theorem~\ref{thm:attprofiles2}-(ii)' for $\omega_n$, 
in the case of the non critical connections $(A_n, B_n)$, we 
know that~\eqref{eq:eq:sl-max-bckw-3} is verified. 
Moreover, because of $\ms R<\widehat {\ms R}=\lim_n \ms R_n$
we may assume that $\ms R<\ms R_n$
for $n$ sufficiently large. Hence, 
by virtue of~\eqref{eq:omegan-conv-35}, 
we can extract a subsequnce of $\{\omega_n\}$ that converges to $\omega$ for almost every $x\in \,]0,\ms R[$,
and thus we derive from~\eqref{eq:AnBnconv}, \eqref{eq:eq:sl-max-bckw-3} that $\omega (\ms R+)\geq B$.
This inequality, together with~\eqref{eq:sl-max-bckw-2}, yields
\begin{equation}
\label{eq:rfprimer-ineq-2}
    \ms R \geq T\cdot f_r^{\prime}(B)\,,
\end{equation}
proving the implication~\eqref{eq:2a2b-prime} also in the case $\widehat {\ms R} > \ms R$.
%
%
\medskip

\noindent
\textbf{Step 4.} 
\big($\ms L = 0$, $\ms R > 0$, proof of~\eqref{eq:3a2b-prime}: $\ms R >  T \cdot f_r^{\prime}(B)\ \ \Rightarrow \ \omega(\ms R+) \leq \omega(\ms R-)$\big).\\
By virtue of~\eqref{eq:AnBnconv}, \eqref{eq:Rsemic}, we may assume that 
\begin{equation}
\label{eq:st3-b}
 \ms R_n >  T \cdot f_r^{\prime}(B_n)\qquad \forall~n\,.
\end{equation}
Then, applying~\eqref{eq:3b-lr2}
of Theorem~\ref{thm:attprofiles}-(ii) or~\eqref{eq:3a2b-prime} of Theorem~\ref{thm:attprofiles2}-(ii)' for $\omega_n$ and the non critical connections $(A_n, B_n)$, we derive
\begin{equation}
\label{eq:R-cond-34}
    \omega_n(\ms R_n-)\geq \omega_n(\ms R_n+)\qquad\forall~n\,.
\end{equation}
On the other hand, if $\widehat {\ms R}=\ms R$, 
invoking
~\eqref{eq:liminf},
~\eqref{eq:limsup}
of Lemma~\ref{lemma:chara} in Appendix~\ref{app:uplwsmicsolns},
we deduce as in Step 3  that
\begin{equation}
\label{lsemic-omega-23}
\omega({\ms R}-) \geq \limsup_n \omega_n(\ms R_n-)\,,
\qquad\quad
\omega({\ms R}+) \leq \liminf_n \omega_n(\ms R_n+)\,.
\end{equation}
Then, \eqref{eq:R-cond-34}-\eqref{lsemic-omega-23}
together
yield $\omega(\ms R-)\geq \omega(\ms R+)$, proving~\eqref{eq:3a2b-prime}
in the case $\widehat {\ms R}=\ms R$.
Instead, if $\widehat {\ms R}> \ms R$, we can assume that
${\ms R}_n> \ms R$ for all $n$ sufficiently large. 
Then, observe that 
applying~\eqref{eq:1b12}, \eqref{eq:2bl-r1}, \eqref{eq:2b-lr2}, 
of Theorem~\ref{thm:attprofiles}, or~\eqref{eq:1b2-prime} of Theorem~\ref{thm:attprofiles2}, for $\omega_n$ and the non critical connections $(A_n, B_n)$, we deduce
\begin{equation}
\label{eq:R-cond-37}
    \omega_n(\ms R-)\geq \omega_n(\ms R+)\qquad\forall~n\,.
\end{equation}
Hence, with the same arguments of above we find that 
\begin{equation}
\label{lsemic-omega-27}
\omega({\ms R}-) \geq \limsup_n \omega_n(\ms R-)\,,
\qquad\quad
\omega({\ms R}+) \leq \liminf_n \omega_n(\ms R+)\,,
\end{equation}
which, together with~\eqref{eq:R-cond-37}, yields $\omega(\ms R-)\geq \omega(\ms R+)$, completing the proof of~\eqref{eq:3a2b-prime}.
\medskip

\noindent
\textbf{Step 5.} \big($\ms L = 0$, $\ms R > 0$, proof of~\eqref{eq:1b2-prime}:\, $D^+\omega(x) \leq h[\omega, f_l, f_r](x)$ in $\,]0, \ms R[\,$, \blu{and of the existence of~$\omega(0\pm)$
}\big). \\
Applying 
Theorem~\ref{thm:attprofiles}-(i) or
Theorem~\ref{thm:attprofiles2}-(i)' for $\omega_n$ in the case of the non critical connections $(A_n, B_n)$,
we know that
\begin{equation}
\label{eq:ol-n-ineq-3}
    D^+\omega_n(x) \leq h[\omega_n, f_l, f_r](x)
    \qquad \forall~x\in \,]0, \ms R_n[\,.
\end{equation}
As shown in~\cite[Lemma 4.4]{anconachiri},
the inequality~\eqref{eq:ol-n-ineq-3} is equivalent to 
the fact that the maps
\begin{equation}
\label{eq:phi-psi-n-def}
\phi_n(x)\doteq -\tau_n(x) \cdot f'_l\circ
\pi_{l,+}^r(\omega_n(x)),
\qquad\quad \tau_n(x)\doteq T-\frac{x}{f_r^{\prime}(\omega_n(x))}\,,
\qquad x\in\,]0, \ms R_n[\,,
\end{equation}
are, respectively, nondecreasing 
and decreasing.
Since by~\eqref{eq:Rsemic} it holds $\lim_n \ms R_n \geq \ms R$, relying on~\eqref{eq:omegan-conv-35} we deduce that, up to a subsequence, $\{\omega_n\}_n$
converges to $\omega$
for almost every $x\in \,]0,\ms R[$\,. In turn, this implies that
the sequences $\{\phi_n\}_n$, $\{\tau_n\}_n$, converges
for almost every $x\in \,]0,\ms R[$\, to the maps
\begin{equation}
\label{eq;phi-psi-def}
\phi(x)\doteq -\tau(x) \cdot f'_l\circ
\pi_{l,+}^r(\omega(x)),
\qquad\quad \tau(x)\doteq T-\frac{x}{f_r^{\prime}(\omega(x))}\,,
\qquad x\in\,]0, \ms R[\,.
\end{equation}
Then, the monotonicity of each
map $\phi_n(x)$ and $\tau_n(x)$,
imply the same  monotonicity of the maps $\phi(x), \tau(x)$
defined in~\eqref{eq;phi-psi-def}. 
Namely, $\phi$ is a nondecreasing map and $\tau$
is a decreasing map.
But this is equivalent to the inequality~\eqref{eq:1b2-prime}, relying again on~\cite[Lemma 4.4]{anconachiri}.
\blu{
Next, we observe that the monotonicity of the maps $x \mapsto \phi(x)$,  $x \mapsto \tau(x)$, readily implies the existence of the one-sided limit $\omega(0+)$. In fact, since $\phi$ 
and $\tau$ are monotone, it follows that the limits $\phi(0+)$, $\tau(0+)$ do exist.
On the other hand, observing that  
the map $\omega \mapsto f'_l\circ
\pi_{l,+}^r(\omega)$, $\omega\geq B$, is invertible, by \eqref{eq;phi-psi-def} we can write
$$
\omega(x) = \big[f'_l\circ
\pi_{l,+}^r\big]^{-1} \bigg(-\frac{\phi(x)}{\tau(x)}\bigg)
\qquad \forall~x\in \,]0, \ms R[\,.
$$
Therefore, since the limit for $x \to 0+$ of the right hand side exists, 
it follows that the limit $\omega(0+)$ exists as well.
Finally, concerning the existence of $\omega(0-)$,
given any sequence $\{x_n\}_n\subset \,]-\infty, 0[\,$
of points of continuity for $\omega$ such that $x_n\to 0$, consider the backward genuine characteristics $\vartheta_n$ for $u$ starting at $(x_n,T)$. Because of the assumption $\ms L = 0$, and by definition~\eqref{eq:LR-def}, \eqref{eq:LR-def-21} of $\ms L$, it follows that 
$\vartheta_n$ never cross the interface
$x=0$.
Observe that $\{\vartheta_n\}_n$ is a sequence of
Lipschitz continuous functions with a uniform Lipschitz constant, 
defined on  $[0,T]$ and lying in the semiplane $\{x<0\}$.
Hence, by Ascoli-Arzel\`a Theorem,
 we can assume that, up to a subsequence,  $\{\vartheta_n\}_n$ converges uniformly to some Lipschitz continuous function $\vartheta:\ [0, T]\to \ \,]-\infty, 0[\,.$
Therefore, with the same arguments of 
the proof of Lemma~\ref{lemma:chara} in Appendix~\ref{app:uplwsmicsolns},
since uniform limit of genuine characteristics is a genuine characteristic,
and because genuine characteristics
cannot intersect in $\{x<0\}$,
we deduce that
$\vartheta$ is the minimal 
backward  characteristic for u
in $\{x\leq 0\}$
starting at~$(0,T)$.
Moreover, $\vartheta$ has slope $\vartheta'=\lim_n \vartheta'_n=\lim_n f'_l(\omega(x_n))$.
In turn, this implies that $\lim_n\omega(x_n)=(f'_l)^{-1}(\theta')$.
Since this limit is independent on the choice of $x_n$ we deduce that the one-sided limit $\omega(0-)$
exists and $\omega(0-)=(f'_l)^{-1}(\theta')$.
}
\medskip

\noindent
\textbf{Step 6.} \big($\ms L = 0$, $\ms R > 0$, proof of~\eqref{eq:2a2-prime}:\, $\omega(0-) \geq \pi^r_{l,+}(\omega(0+)$\big). \\
Let $x\in\,]0,\ms R[\,$ be a point of continuity for $\omega$,
and consider the backward genuine characteristics for~$u$
starting at $(x,T)$, 
defined by
\begin{equation}
    \label{eq:bcw-char-21}
    \vartheta_x(t)\doteq x-(T-t)\cdot f'_r((\omega(x))\qquad t\in\,]\tau (x), T]\,,
\end{equation}
with
\begin{equation}
\label{eq-def-tau-2}
    \tau (x)\doteq T-\frac{x}{f_r^{\prime}(\omega(x))}\,,
\end{equation}
so that one has $\lim_{t\to\tau (x)} \vartheta_x(t)=0$.
Observe that the inequality~\eqref{eq:1b2-prime}
(established at {\bf Step 3})
implies
that the function $\tau(x)$ is decreasing. On the other hand,
because of~\eqref{linf-bound-omega-n}, the slopes of $\vartheta_x$ are uniformly bounded by $\sup_{|u|\leq \overline C} |f'_r(u)|$.
Therefore,
letting $\{x_n\}_n\subset \,]0,\ms R[\,$ be a sequence of points of continuity for $\omega$, such that $x_n\to 0$,
it follows that 
\begin{equation}
\label{eq:lim-taun}
    \lim_n\tau(x_n)=T\,.
\end{equation}
Notice that, since
$\vartheta_{x_n}$ are genuine
characteristics, we have
\begin{equation}
\label{eq:omegaurn}
    \omega(x_n)=u_r(\tau(x_n))
    \qquad\forall~n\,.
\end{equation}
Since, by definition~\eqref{pimap-def}
one has $\pi^r_{l,+}(u)\geq \theta_l$ for any $u$, we may assume that, up to a subsequence,
either 
\begin{equation}
\label{eq:hyp1-limurn}
    \pi^r_{l,+}(\omega(x_n))=\theta_l\qquad\forall~n\,,
\end{equation}
or
\begin{equation}
\label{eq:hyp2-limurn}
    \pi^r_{l,+}(\omega(x_n))>\theta_l\qquad\forall~n\,.
\end{equation}
In the first case~\eqref{eq:hyp1-limurn} we deduce 
that
\begin{equation}
\label{eq:lim-ur-n-3}
    \pi^r_{l,+}(\omega(0+))=\lim_n \pi^r_{l,+}(\omega(x_n))=\theta_l\,,
\end{equation}
which yields~\eqref{eq:2a2-prime}
observing that, by definition~\eqref{eq:LR-def}, \eqref{eq:LR-def-21},
$\ms L=0$ implies $f'_l(\omega(0-))\geq 0$,
which is equivalent to
$\omega(0-)\geq\theta_l$.
In the second case~\eqref{eq:hyp2-limurn}
observe that, since the map
$\tau$ in~\eqref{eq-def-tau-2} is decreasing, and because
$x_n$ are points of continuity
for $\omega$, then 
$\tau(x_n)$ are points of continuity for $u_r$.
Hence, by the interface entropy condition~\eqref{ABtraces} it follows that $\tau(x_n)$ are points of continuity also for $u_l$.
Then,
we can trace the backward genuine characteristics for $u$ starting
at $(0,\tau(x_n))$,
that, because of~\eqref{eq:omegaurn},
are defined by
\begin{equation}
    \vartheta_n(t)\doteq (t-\tau(x_n))\cdot f'_l\circ \pi^r_{l,+}(\omega(x_n)),
    \qquad t\in [0,\tau(x_n)]\,.
\end{equation}
\blu{
Notice that $\{\vartheta_n\}_n$ is a sequence of
Lipschitz continuous functions with a uniform Lipschitz constant, 
defined on uniformly bounded intervals $[0,\tau(x_n)]$.
Hence, by Ascoli-Arzel\`a Theorem,
and because of~\eqref{eq:lim-taun}, we can assume that, up to a subsequence,  $\{\vartheta_n\}_n$ converges uniformly to some Lipschitz continuous function $\vartheta:\ [0, T]\to \ [0,+\infty[\,.$
Therefore, with the same arguments of 
the proof of Lemma~\ref{lemma:chara} in Appendix~\ref{app:uplwsmicsolns},
since uniform limit of genuine characteristics is a genuine characteristic
we deduce that
$\vartheta$ is a 
backward genuine characteristic 
starting at~$(0,T)$,
that has slope $f'_l\circ \pi^r_{l,+}(\omega(0+))$.
}
On the other hand
the minimal backward characteristic starting at $(0,T)$ has slope $f'_l(\omega(0-))$. 
Since the slope of the minimal backward characteristic is larger than the slope of any other backward characteristic passing through the same point, it follows that $f'_l(\omega(0-))\geq f'_l\circ \pi^r_{l,+}(\omega(0+))$,
which implies~\eqref{eq:2a2-prime}. 
This concludes the proof of
this step.
%
%
\smallskip

\noindent
\textbf{Step 7.} \big($\ms L = 0$, $\ms R > 0$, proof of~\eqref{eq:1a2-prime}:\, $D^+\omega(x) \leq \frac{1}{T\cdot f''_l(\omega(x))}$
in $\,]\!-\!\infty, 0[$, and $D^+\omega(x) \leq \frac{1}{T \cdot f''_r(\omega(x))}$
in $\,]\ms R, +\infty[\,$\big). \\
Observe that, by 
definition~\eqref{eq:LR-def}, \eqref{eq:LR-def-21}
of $\ms L, \ms R$,
and since $\ms L=0$,
backward characteristics starting at points $(x,T)$,
with $x  \in ]-\infty, 0[\, \cup\, ]\ms R, +\infty[$ do no intersect the interface $x = 0$.
Hence, we recover the Ole\v{\i}nik estimates~\eqref{eq:1a2-prime} 
as a classical property
of solutions 
to conservation laws with strictly convex flux, which 
follows from the fact that genuine characteristics never intersect at positive times. 
This completes the proof of the existence of~$\omega(0\pm)$ and that $\omega$ satisfies conditions  (i)'-(ii)'
of Theorem \ref{thm:attprofiles2}. 

\smallskip

\noindent
\blu{\textbf{Step 8.} 
\big($\ms L < 0$, $\ms R = 0$, proof of the existence of~$\omega(0\pm)$, and of conditions (i)-(ii)
of Theorem~\ref{thm:attprofiles2}\big).\\
The Ole\v{\i}nik-type inequalities~\eqref{eq:1a2}, \eqref{eq:1b2}, and the existence of~$\omega(0\pm)$ can be established with the same arguments of Steps 5, 7.
The proofs of~\eqref{eq:2b2}, \eqref{eq:2a2}, 
are entirely similar to the proofs of~\eqref{eq:2b2-prime}, 
~\eqref{eq:2a2-prime}, in Steps 2 and 6, respectively. Since $A=\theta_l$,
and hence $f'_l(A)=0$, the implication~\eqref{eq:2a2b} is trivially verified. Finally, with the same arguments of the proof of~\eqref{eq:3a2b-prime} in Step~4
one can recover the inequality $\omega(\ms L-) \geq \omega(\ms L+)$, thus proving
the implication
\eqref{eq:3a2b}.
Therefore the proof of the existence of $\omega(0\pm)$ and that $\omega$ satisfies the conditions (i)-(ii) of
Theorem~\ref{thm:attprofiles2} is completed.
}
\smallskip

\noindent
\textbf{Step 9.} \big($\ms L = 0$, $\ms R = 0$, proof of~\eqref{eq:1a3}: $D^+\omega(x) \leq \frac{1}{T\cdot  f''_l(\omega(x))}$
in $\,]\!-\!\infty, 0[$, and $D^+\omega(x) \leq \frac{1}{T \cdot f''_r(\omega(x))}$
in $\,]0, +\infty[$\,,
\blu{of the existence of $\omega(0\pm)$}: $\omega(0-) \geq \pi^r_{l,+}(\omega(0+)$,
and of~\eqref{eq:2a3}: $\omega(0-) \geq \overline A$, $\omega(0+) \leq \overline B$\big).\\ 
Since $\ms L = 0$, $\ms R = 0$,
by 
definition~\eqref{eq:LR-def} it follows that 
backward characteristics starting at $(x,T)$, $x\in\,]-\infty, 0[\,\cup \,]0, +\infty[$,
never intersect the interface $x=0$. 
Thus, as observed in Step 7, the Ole\v{\i}nik estimates in~\eqref{eq:1a3}
are a classical property
of solutions. 
\blu{Moreover, with the same arguments of Step 5 one deduces the existence of $\omega(0\pm)$.}
\blu{Furhter}, the inequality 
$\omega(0-) \geq \pi^r_{l,+}(\omega(0+)$ can be established with the same proof of~\eqref{eq:2a2-prime}
in Step 6.
Finally, the inequality  
$\omega(0+) \leq \overline B$
is obtained with the same arguments of the proof of~\eqref{eq:2a2b-prime}
in Step 3, observing that
by Remark~\ref{rem:monotonicityxy} we have
$\bs u[0, B, f_r] = \bs u[0+, B, f_r] = \overline B$.
The other inequality
$\omega(0-) \geq \overline A$
can derived in entirely similar way.
Therefore the proof of the existence of $\omega(0\pm)$ and
that $\omega$ satisfies the
conditions (i)-(ii)
of Theorem~\ref{thm:attprofiles3} is completed.
\smallskip

\noindent
\blu{\textbf{Step 10.} \big($\ms L< 0 < \ms R$, proof of  condition (i) of  Theorem~\ref{thm:attprofilescrit}
and of the existence of $\omega(0-)$\big). \\
The Ole\v{\i}nik inequalities~\eqref{eq:1acrit} 
are a classical property
of solutions 
to conservation laws with strictly convex flux
as observed in Step 7.
The proof of the Ole\v{\i}nik type inequality~\eqref{eq:1b1crit} can be recovered with the same limiting procedure of Step 5, passing to the limit the monotonicity of the maps 
\begin{equation}
\label{eq:phi-psi-n-def-L}
\phi_n(x)\doteq -\tau_n(x) \cdot f'_r\circ
\pi_{r,-}^l(\omega_n(x)),
\qquad\quad \tau_n(x)\doteq T-\frac{x}{f'_l(\omega_n(x))}\,,
\qquad x\in\,]\ms L_n, 0[\,,
\end{equation}
ensured by the Ole\v{\i}nik type inequalities satisfied by $\omega_n\in  \mc A^{[A_nB_n]}(T)$,
and
relying on~\eqref{eq:omegan-conv-35}, \eqref{eq:Rsemic}.
This also shows that the one-sided limit $\omega(0-)$ exists, using  the monotonicity of the limiting maps $\phi(x)=\lim_n \phi_n(x)$, $\tau(x)=\lim_n \tau_n(x)$, $x\in\,]\ms L, 0[\,$, as in Step 5.
}
\smallskip

\noindent
\blu{\textbf{Step 11.} \big($\ms L< 0 < \ms R$, proof  
of~\eqref{eq:critcond}: $ \omega(\ms L-)\geq \omega(\ms L+)$, \ $\omega(0-)=\theta_l$\big). \\
The proof of the first constraint
in~\eqref{eq:critcond} can be obtained with the same procedure of Step 4
(with $\ms L$ in place of $\ms R$).
Concerning the second constraint in~\eqref{eq:critcond},
notice first that we have $\omega(0-) \leq \theta_l$, since otherwise we may consider 
a sequence $\{x_n\}_n$  of continuity points for $\omega$, such that $x_n \uparrow 0$, and for $n$ sufficiently large 
the backward characteristics for $u$ from $(x_n,T)$
would intersect in $\{x < 0\}$ the maximal backward  characteristic for $u$ from the point $(\ms L,T)$, which gives a contradiction. Next, assume 
that the strict inequality $\omega(0-) < \theta_l$ holds. Then, let $x < 0$ be a continuity point of $\omega$ 
sufficiently close to $0$ 
so that $\omega(x) < \theta_l$,
and consider the time 
$$
\tau(x) = T-\frac{x}{f_l^\prime(\omega(x))}
$$
at which the backward (genuine) characteristic for $u$ from $(x,T)$ impacts the interface $x=0$. 
Since $x$ is a continuity point for $\omega$, by the strict monotonicity of the map $\tau$ on $\,]\ms L, 0[\,$ (derived in Step~9
as in Step 5)
it follows
that $u_l(\tau(x)) = \omega(x) < \theta_l$.
Using also the $AB$-entropy conditions \eqref{ABtraces} we then deduce
 that $u_r(\tau(x)) = \pi_{r, -}^l (\omega(x)) < \theta_r$. But this implies that the maximal backward characteristic for $u$ from $(0, \tau(x))$ intersects in $\{x > 0\}$ the minimal backward characteristic for $u$ from $(\ms R, T)$, which again gives a contradiction, and thus completes the proof of~\eqref{eq:critcond}.
}
\smallskip

\noindent
\blu{\textbf{Step 12.} \big($\ms L< 0 < \ms R$, proof  
of~\eqref{eq:2bl-r1crit}: $\omega(x) = B$ in $]0, \ms R[\,$, \  $\ms R \in ]0, T\cdot f_r^\prime(B)[$\,, 
of $B\neq \theta_r$,
and of~\eqref{eq:2acrit}:
$\omega(\ms R+) \leq \bs u[\ms R, B, f_r]\leq \omega(\ms R-)$\big). \\
Towards a proof of~\eqref{eq:2bl-r1crit},  first notice that the maximal backward characteristic 
for $u$ from $(\ms L,T)$ must intersect the
interface $x=0$ at a time 
\begin{equation}
    \label{eq:tauL-def-est}
    \tau_{_{\ms L}}
\doteq T-\frac{\ms L}{f'_l(\ms L+)} < T-\frac{\ms R}{f'_r(\ms R-)}.
\end{equation} 
In fact otherwise,  we could consider a point 
$x>\ms L$ of continuity for $\omega$
sufficiently close to $\ms L$, so that $\tau(x)>
T-\ms R/f'_r(\ms R-)$.
But then, by the analysis in Step 11,
the maximal backward characteristic 
for~$u$ from $(0,\tau(x))$ would intersect in $\{x > 0\}$ the minimal backward characteristic for $u$ from $(\ms R, T)$, thus giving a contradiction.
Next, observe that at any point $x\in\,]0, \ms R[\,$
of continuity for $\omega$ we have 
$\omega(x) \geq \theta_r$, since otherwise
the backward characteristic for $u$ from $(x,T)$
would intersect in $\{x > 0\}$ the minimal backward  characteristic for $u$ from the point $(\ms R,T)$, which gives a contradiction.
By the  $AB$-entropy conditions, and because of the strict monotonicity of the map 
$$
\tau(x) = T-\frac{x}{f'_r(\omega(x))},
\qquad x\in \,]0, \ms R[\,,
$$
(that can be established with the same limiting procedure of Steps~5, 10), it then follows that $\omega(x) \geq B$ for all $x\in\,]0, \ms R[\,$. Assume now that $\omega(x) > B$ for some $x\in\,]0, \ms R[\,$ continuity point for~$\omega$. 
Observe that,
because of the non crossing property of characteristics in $\{x>0\}$,
and by~\eqref{eq:tauL-def-est},
the backward characteristic for $u$ from $(x,T)$ impacts the interface $x=0$
at time 
\begin{equation}
\label{eq:tau-def-11}
    \tau(x)\geq T - \frac{\ms R}{f'_r(\ms R-)}
    > \tau_{_{\ms L}}.
    %
\end{equation}
Then, relying on the strict monotonicity of the map $\tau$, 
and using the $AB$-entropy conditions \eqref{ABtraces}, we deduce that 
$u_r(\tau(x))=\omega(x)>B$, $u_l(\tau(x))=\pi^r_{l,+}(\omega(x))>\theta_l$.
But, because of~\eqref{eq:tau-def-11}, this implies that the minimal backward characteristic for $u$ from $(0, \tau(x))$ intersects in $\{x < 0\}$ the maximal backward characteristic for $u$ from $(\ms L, T)$
which again gives a contradiction. Hence we have shown that 
$\omega(x) = B$ in $]0, \ms R[\,$.
By definition~\eqref{eq:LR-def} of $\ms R$, 
and because of~\eqref{eq:tau-def-11}, 
this implies
that 
$\ms R \in\,]0, T\cdot  f_r^\prime(B)[\,,$
and thus completes the proof of~\eqref{eq:2bl-r1crit}.
This also shows that 
we must have $B \neq \theta_r$.
Furthermore, we can derive \eqref{eq:2acrit}
with exactly the same arguments contained in the proof of~\eqref{eq:2a2b-prime} 
in Step 3.
}

\smallskip

\noindent
\blu{\textbf{Step 13.} \big($\ms L< 0 < \ms R$, proof  of~\eqref{eq:critcond-2}: 
$(f'_l)^{-1}(\frac{x}{T-\bs\tau[\ms R, B, f_r]})
\leq \omega(x)<\theta_l$ in $\,]\ms L, 0[\,$
        \big). \\
Let $\ms L_n, \ms R_n$ be the constants defined at~\eqref{eq:LR-n-def-21} as in~\eqref{eq:LR-def} and, according with~\eqref{eq:Ltildedef}, define
\begin{equation}
\label{eq:tau-n-def-11}
    {\widetilde {\ms L}}_n\doteq 
    \big(T-\bs\tau[\ms R_n, B_n, f_r]\big)\cdot f'_l(A_n).
\end{equation}
Since \eqref{eq:critical-ip}, \eqref{eq:AnBnconv} imply $\lim_n f'_l(A_n)=0$, 
and recalling that $\ms R_n$ are bounded
(see Step 1),
it follows that $\lim_n
{\widetilde {\ms L}}_n=0$. 
Thus, recalling also that the 
limit $\widehat{\ms R}$ of a subsequence of $\{\ms R_n\}_n$ satisfies~\eqref{eq:Rsemic},
we may assume that $\ms L_n < {\widetilde {\ms L}}_n<0<\ms R_n$ for $n$ sufficiently large.
Then, applying~\eqref{eq:2bl-r1} or \eqref{eq:2b-lr2-2} of Theorem~\ref{thm:attprofiles} for $\omega_n$ in the case of the non critical connections $(A_n, B_n)$, we deduce
that 
\begin{equation}
    \omega_n(x)\leq A_n\qquad \forall~x\in\,]{\widetilde {\ms L}}_n, \ms L_n[\,,
    \qquad\quad 
    \omega_n({\widetilde {\ms L}}_n\pm)=A_n\,.
\end{equation}
Therefore, by~\eqref{eq:tau-n-def-11} the backward characteristic for $u_n$
starting at $({\widetilde {\ms L}}_n,T)$ reaches
the interface $x=0$ at time $\bs\tau_n\doteq \bs\tau[\ms R_n, B_n, f_r]$. In turn, this implies that,
for every  
$x\in\, ]\ms L_n, \widetilde L_n[$
point of continuity for~$\omega_n$, 
the backward characteristic starting at $(x,T)$ must cross the interface $x=0$ at a time smaller or equal than $\bs\tau_n\doteq \bs\tau[\ms R_n, B_n, f_r]$, since otherwise it would intersect the
backward characteristic for $u_n$
starting at $({\widetilde {\ms L}}_n,T)$
in the domain $\{x<0\}$, which gives a contradiction. Thus we have
\begin{equation}
\label{eq:critcond-2-lim1}
    T-\frac{x}
{f'_l(\omega_n(x))}\leq \bs\tau_n\,,
\end{equation}
for every  
$x\in\, ]\ms L_n, \widetilde L_n[$
point of continuity for $\omega_n$.
On the other hand, recall that by Lemma~\ref{lemma:dualshocks}
the map $\ms R \mapsto \bs y[\ms R, B, f_r](T)<0$ is strictly increasing, and hence by~\eqref{eq:tau-rbf-def}
the map $\ms R \to \bs \tau[\ms R, B, f_r]$ is strictly decreasing.
Therefore, 
since $\bs\tau[\ms R, B, f_r]$
depends continuously on the parameters $\ms R, B$
(see~\S~\ref{def:rsr-block}), and because of~\eqref{eq:Rsemic},
we deduce that
\begin{equation}
    \lim_n \bs \tau_n=\bs \tau[\,\widehat{\ms R}, B, f_r]\leq \bs \tau[\ms R, B, f_r]\,.
\end{equation}
Hence, taking the limit in~\eqref{eq:critcond-2-lim1} as $n\to\infty$, and relying
again on~\eqref{eq:omegan-conv-35}, \eqref{eq:Rsemic}, 
we derive
\begin{equation}
\label{eq:critcond-2-lim22}
     T-\frac{x}
{f'_l(\omega(x))}\leq \bs\tau[\ms R, B, f_r]\qquad\text{for \ a.e.}~x\in \, ]{\ms L}, 0[\,.
\end{equation}
In turn, the inequality~\eqref{eq:critcond-2-lim22}
yields the first inequality in~\eqref{eq:critcond-2}.
On the other hand, if $\omega(x)\geq \theta_l$
for some $x\in\,]\ms L, 0[\,$,
with the same arguments of Step 11
one deduces that the backward characteristic
for $u$ starting from $(x,T)$
must intersect in $\{x<0\}$
the maximal backward  characteristic for $u$ from the point $(\ms L,T)$, which gives a contradiction. This shows that also the
second inequality in~\eqref{eq:critcond-2}
is satisfied.
Therefore, the proof of the existence of $\omega(0\pm)$ and
that $\omega$ satisfies the
conditions (i)-(ii)
of Theorem~\ref{thm:attprofilescrit} is completed.
This concludes the proof
of the implication $(1) \Rightarrow (3)$ for critical connections.
}
\subsection{Part 3.c - $(3) \Rightarrow (1)$ for critical connections} 
\label{sec:criticalcases-c}
In this Subsection we \blu{rely on the fact}
that 
Theorem~\ref{thm:backfordiscfluxcycle} holds for connections which are non critical, and in particular  \blu{we know (by \S~\ref{sec:(1)-(3)}, \ref{subsec:reductionBV}, \ref{sec:(3)-(2)})} that
Theorems~\ref{thm:attprofiles}, \ref{thm:attprofiles2}, \ref{thm:attprofiles3}, 
are verified 
for non critical connections. 
We will prove that 
if $\omega\in \msc A^{\ms L, \ms R}$
satisfies the conditions of
\blu{Theorem~\ref{thm:attprofilescrit}}, \ref{thm:attprofiles2}, or of Theorem~\ref{thm:attprofiles3}, 
then $\omega \in \mc A^{[AB]}(T)$
also for a critical connection
$(A,B)$ satisfying~\eqref{eq:critical-ip}.

\smallskip

\noindent
\textbf{Step 1.}
Given an element $\omega$
of the set $\msc A^{\ms L, \ms R}$   
    in~\eqref{eq:ALR-def1}, assuming that:
\begin{itemize}
\blu{\item[-] if $\ms L < 0$, $\ms R > 0$ or viceversa, $\omega$ satisfies the conditions of Theorem \ref{thm:attprofilescrit};}
    \item[-] if $\ms L = 0$, $\ms R > 0$ or viceversa, $\omega$ satisfies the conditions of Theorem \ref{thm:attprofiles2};
    \item[-] if $\ms L = 0$, $\ms R = 0$, $\omega$ satisfies the conditions of Theorem \ref{thm:attprofiles3};
\end{itemize}
we will construct a sequence $\{\omega_n\}_{n}$
of suitable 
perturbations
of $\omega$
with the property that:
\begin{equation}\label{eq:familyomegan}
\omega_n \in \mc A^{[A_nB_n]}(T)\quad\forall~n, \qquad\quad \omega_n \to \omega \quad\ \text{in \ $\bf L^1_{\mr{loc}}(\R)$},
\end{equation}
for a sequence of non critical connections $\{(A_n,B_n)\}_n$ 
satisfying~\eqref{eq:AnBnconv}
and
\begin{equation}
    \label{eq:AnBncond2}
    A_n<A,\qquad\qquad B_n>B\qquad\quad\forall~n\,.
\end{equation}
The conditions in~\eqref{eq:familyomegan} in turn will imply that $\omega \in \mc A^{[AB]}(T)$.
In fact, by the validity of Theorem~\ref{thm:backfordiscfluxcycle} in the \blu{non critical} case,
and because of~\eqref{eq:familyomegan}, it holds
\begin{equation}\label{eq:noncrifixedn2}
    \omega_n=  \mc S^{[A_nB_n]+}_T \circ \mc S^{[A_nB_n]-}_T \omega_n \qquad \forall \; n\,.
\end{equation}
On the other hand, relying on the stability property (iv) of Theorem \ref{theoremsemigroup},
and thanks to ~\eqref{eq:familyomegan}, 
one finds as in \S~\ref{sec:criticalcases-a} that
\begin{equation}
\label{eq:backelleunostab-3}
    S^{[A_nB_n]+}_T \circ
    \mc S^{[A_nB_n]-}_T \omega_n
    \ \ \rightarrow\ \
    \mc S^{[AB]+}_T \circ
    \mc S^{[AB]-}_T\omega
    \qquad\text{in}\quad
    \mathbf L^1_{\mr{loc}}(\R)\,.
\end{equation}
Hence, combining together~\eqref{eq:familyomegan}, \eqref{eq:noncrifixedn2}, \eqref{eq:backelleunostab-3},
we derive
\begin{equation*}
    \omega=\mc S^{[AB]+}_T \circ \mc S^{[AB]-}_T \omega\,,
\end{equation*}
which clearly yields $\omega \in \mc A^{[AB]}(T)$.
Therefore, to establish the implication $(3) \Rightarrow (1)$ of Theorem~\ref{thm:backfordiscfluxcycle} for non critical connections, it remains to produce a family $\{\omega_n\}_{n}$ that satisfies \eqref{eq:familyomegan}.
We shall construct such perturbations of $\omega \in \msc A^{\ms L, \ms R}$ as suitable 
\textit{``$(A_n, B_n)$ admissible envelopes}" 
of~$\omega$. 
\smallskip

We will first consider in Steps 2-8 the case 
$\ms L = 0$, $\ms R \geq 0$, while the \blu{symmetric} case $\ms L <0$, $\ms R = 0$ is entirely similar. \blu{Next, we will consider separately the case $\ms L < 0, \ms R > 0$, in Step 9.}

\noindent
\textbf{Step 2.}
We shall assume throughout Steps 2-8 that
\begin{equation}
\label{eq:LR-cond45}
    \ms L=\ms L[\omega, f_l] = 0, \quad\qquad \ms R=\ms R[\omega, f_r]\geq 0\,,
\end{equation}
and that: $\omega$ satisfies the conditions of Theorem~\ref{thm:attprofiles2} if $\ms R>0$;
$\omega$ satisfies the conditions of Theorem~\ref{thm:attprofiles3} if $\ms R=0$.

\noindent
We will perturb $\omega$
to obtain an attainable profile $\omega_n$ for the
$(A_n,B_n)$ connection by:
\begin{itemize}
    \item[-]shifting $\omega$
    on the right of $x=0$ by a size $\delta_{1,n}$, and
    on the left of $x=0$ by a size $\delta_{2,n}$;
    \item[-] choosing $\delta_{1,n}$ so to satisfy the admissibility condition~\eqref{eq:2a2b-prime} at $x=\ms R+\delta_{1,n}$ (if $\ms R+\delta_{1,n}>0$); 
    \item[-]lifting $\omega$ to the value $B_n$ of the connection when it is below, in the interval
    $]0, \ms R+\delta_{1,n}[$,
    so to satisfy the admissibility condition~\eqref{eq:2b2-prime}
    (if $\ms R+\delta_{1,n}>0$);
    \item[-]inserting a profile of a rarefaction in the interval $]-\delta_{2,n}, 0[$, so to satisfy the Lax-type admissibility condition~\eqref{eq:2a2-prime} 
    at $x=0$.
\end{itemize}
Namely,
consider the function
\begin{equation}
\label{eq:omega-n-def}
    \omega_n(x)\doteq  
\begin{cases}
\omega(x-\delta_{1,n}), & \text{if \ \ $x \geq \ms R+ \delta_{1,n}$},
\\
\noalign{\smallskip}
\max\{\omega(\ms R+)
, B_n\}, & \text{if \ \ $x \in \,] \ms R, \ms R+\delta_{1,n}[$},\\
     \noalign{\smallskip}
 \max \{\omega(x), B_n\}, & \text{if \ \ $ x \in \,]0, \ms R[$},\\
     \noalign{\smallskip}
     (f_l^{\prime})^{-1} \bigg(\dfrac{x+T \cdot
     f_l^{\prime}\big(\max\{\omega(0-), \; \overline A_n\}\big)}
       {T}\bigg), \quad & \text{if \ \ $x \in \,]-\delta_{2,n},0[$},\\
    \omega(x+\delta_{2,n}), & \text{if \ \ $x \leq  -\delta_{2,n}$}.
\end{cases}
\end{equation}
with
\begin{equation}
    \label{eq:delta12n-def}
    \begin{aligned}
       \delta_{1,n} &\doteq
    \inf \Big\{ \delta \geq  0 \ \ : \ \   \text{either} \quad  \ms R +\delta \geq T\cdot f_r^{\prime}(B_n), \\ 
    &\hspace{1,2in} \text{or}\quad
    \ms R +\delta < T\cdot f_r^{\prime}(B_n) \quad \text{and}  \ \ 
    \omega(\ms R+)\leq 
    \bs u[\ms R+\delta, B_n, f_r]  \Big\}\,,
    %
    \\
    \noalign{\smallskip}
    \delta_{2,n}&\doteq T \cdot f_l^{\prime}\big(\max\{\omega(0-), \; \overline A_n\}\big)-
    T \cdot f_l^{\prime}(\omega(0-)),
    \end{aligned}
\end{equation}
where $\overline A_n$ is defined as in~\eqref{eq:bar-AB-def}.
Recalling the definitions~\eqref{eq:LR-def},
and because of~\eqref{eq:LR-cond45},
we deduce
that
\begin{equation}
\label{eq:LR-end-def}
    \ms L_n\doteq 
    \ms L[\omega_n, f_l] = 0, \quad\qquad \ms R_n\doteq \ms R[\omega_n, f_r] = \ms R + \delta_{1,n}\qquad\quad \forall~n\,.
\end{equation}

\begin{figure}
    \centering

\tikzset{every picture/.style={line width=0.75pt}} 

\begin{tikzpicture}[x=0.75pt,y=0.75pt,yscale=-1,xscale=1]

\draw    (220.33,80.33) -- (449.67,80.33) ;
\draw    (221.67,170.33) -- (470,169.67) ;
\draw    (330.33,60.33) -- (330.67,170.33) ;
\draw    (367.67,80.33) -- (397,171) ;
\draw    (367.67,80.33) -- (330.5,115.33) ;
\draw    (330,79.67) -- (262.33,170.33) ;
\draw  [dash pattern={on 3.75pt off 3.75pt}]  (414.67,80.67) -- (446.33,169.67) ;
\draw    (319.33,73.33) -- (305.33,73.61) ;
\draw [shift={(302.33,73.67)}, rotate = 358.88] [fill={rgb, 255:red, 0; green, 0; blue, 0 }  ][line width=0.08]  [draw opacity=0] (5.36,-2.57) -- (0,0) -- (5.36,2.57) -- cycle    ;
\draw    (380,73.33) -- (395.33,73.61) ;
\draw [shift={(398.33,73.67)}, rotate = 181.04] [fill={rgb, 255:red, 0; green, 0; blue, 0 }  ][line width=0.08]  [draw opacity=0] (5.36,-2.57) -- (0,0) -- (5.36,2.57) -- cycle    ;
\draw  [dash pattern={on 4.5pt off 4.5pt}]  (289,81.67) -- (240.33,171) ;
\draw  [dash pattern={on 4.5pt off 4.5pt}]  (330,79.67) -- (240.33,171) ;
\draw  [dash pattern={on 4.5pt off 4.5pt}]  (414.67,80.67) -- (331.67,156.33) ;

\draw (361.33,66.4) node [anchor=north west][inner sep=0.75pt]  [font=\scriptsize]  {$\mathsf{R}$};
\draw (404,66.73) node [anchor=north west][inner sep=0.75pt]  [font=\scriptsize]  {$\mathsf{R} +\delta _{1}^{n}$};
\draw (278.67,64.07) node [anchor=north west][inner sep=0.75pt]  [font=\scriptsize]  {$-\delta _{2}^{n}$};
\draw (232,61.07) node [anchor=north west][inner sep=0.75pt]  [font=\scriptsize]  {$\omega ,\ \omega ^{n}$};

\end{tikzpicture}

    \caption{The ``candidate" characteristics of $\omega$ and of its admissible $A^nB^n$ envelopes $\omega^n$ (dashed).}
    \label{fig:envelopes}
\end{figure}

\noindent
Notice that the assumption that
$\omega$ satisfies conditions (ii)' of 
of Theorem \ref{thm:attprofiles2} or
conditions (ii) of Theorem \ref{thm:attprofiles3}, together with~\eqref{eq:AnBnconv},
imply that
\begin{equation}
\label{eq:lim-delta-n}
    \lim_{n\to\infty} \delta_{1,n}=
      \lim_{n\to\infty} \delta_{2,n}=0\,.
\end{equation}
In fact, if $\ms R \leq  T\cdot f_r^{\prime}(B)$, 
relying on conditions~\eqref{eq:2a2b-prime}, \eqref{eq:est-tfprimaea-3} of Theorem~\ref{thm:attprofiles2} or on condition~\eqref{eq:2a3-bis} of Theorem~\ref{thm:attprofiles3}
we deduce that $ \omega(\ms R+)\leq \bs u[\ms R, B , f_r]$.
Moreover, we know by Remark \ref{rem:monotonicityxy}
that $\bs u[\cdot, \cdot , f_r]$ is continuous in the first two entries.
Therefore, because of~\eqref{eq:AnBnconv},  we derive from definition~\eqref{eq:delta12n-def}
that, if $\ms R \leq T\cdot f_r^{\prime}(B)$,
then $\lim_n \delta_{1,n}=0$.
On the other hand, if $\ms R > T\cdot f_r^{\prime}(B)$, then it follows 
from definition~\eqref{eq:delta12n-def}
and~\eqref{eq:AnBnconv},
that $\delta_{1,n}=0$ for sufficiently large $n$.
Next, observe that, since $\ms L=0$,
by definition~\eqref{eq:LR-def}
one has $\omega(0-) \geq \theta_l$. On the other hand
by assumptions~\eqref{eq:critical-ip}, \eqref{eq:AnBnconv}
it follows that $\lim_n\overline A_n=\theta_l$.
Therefore, by definition~\eqref{eq:delta12n-def} we deduce that
$\lim_n \delta_{2,n}=0$.


\smallskip

 Because of~\eqref{eq:AnBnconv},
 and relying on conditions (ii) of  Theorem~\ref{thm:attprofiles2} or of Theorem~\ref{thm:attprofiles3},
 we deduce that the limit 
 \eqref{eq:lim-delta-n} implies
 the $\mathbf L^1_{\mr{loc}}$ convergence
 of $\omega_n$ to $\omega$
     as $n\to \infty$. Hence, in order to show that $\omega_n$ satisfy~\eqref{eq:familyomegan}
  it remains to prove that $\omega_n\in \mc A^{[A_nB_n]}(T)$
  for all $n$.
Since we are assuming in particular the validity of the implication $(2) \Rightarrow (1)$ of Theorem~\ref{thm:backfordiscfluxcycle} for non critical connections, 
in order to establish
$\omega_n\in \mc A^{[A_nB_n]}(T)$ it will be sufficient to show that:
if~${\ms R}_n=0$ then 
$\omega_n$ satisfies conditions (i)-(ii) of Theorem ~\ref{thm:attprofiles3};
if ${\ms R}_n>0$ then $\omega_n$ satisfies conditions (i)'-(ii)' of Theorem \ref{thm:attprofiles2}.
This is established in the steps below distinguishing the cases where ${\ms R}=0$ or
${\ms R}>0$ and ${\ms R_n}=0$ or ${\ms R_n}>0$.
Notice that, by definition~\eqref{eq:omega-n-def},
we always have
\begin{equation}
\label{eq:R-Rn-lax-ineq}
    \omega_n({\ms R}_n-)\geq \omega_n({\ms R}_n+)\,.
\end{equation}
\medskip

\noindent
\textbf{Step 3.}
\big($\ms R_n>0$, $\ms R=0$,
proof that $\omega_n$ satisfies condition (ii)' of Theorem \ref{thm:attprofiles2}\big).\\
 By definition~\eqref{eq:omega-n-def}
one has 
\begin{equation}\label{eq:omegan-cond33}
    \omega_n(0-)=\max\{\omega(0-), \overline A_n\}\,,
    \qquad\
    \omega_n(0+)=\max\{\omega(0+),  B_n\},
    \qquad\
    \omega_n(\ms R_n+)=\omega(\ms 0+)\,,
\end{equation}
and
\begin{equation}
\label{eq:omegan-cond34}
\omega_n(x)\geq B_n\qquad \
\ \forall x\in\,]0, \ms R_n[\,,
\end{equation}
while definition~\eqref{eq:delta12n-def}, together with~\eqref{eq:LR-end-def}, \eqref{eq:omegan-cond33}, 
\eqref{eq:omegan-cond34},
and~\eqref{eq:uB-ineq} with $f=f_r$, yield
\begin{equation}
\label{eq:omegan-cond35}
    {\ms R}_n\in\,]0, T\cdot f'_r(B_n)[\, \qquad \Longrightarrow \qquad \omega_n(\ms R_n+)\leq \bs u[{\ms R}_n, B_n, f_r]\leq \omega_n(\ms R_n-) \,.
\end{equation}
Since $\pi^r_{l,+}(B_n)=\overline A_n$,
from~\eqref{eq:omegan-cond33} 
we deduce 
\begin{equation}
\label{eq:omegan-cond36}
    \pi^r_{l,+}(\omega_n(0+))=\max\{\pi^r_{l,+}(\omega(0+)), \overline A_n\}.
    \end{equation}
    Hence~\eqref{eq:omegan-cond33}, \eqref{eq:omegan-cond36}, 
imply
\begin{equation}
\label{eq:omegan-cond39}
    \omega_n(0-)\geq \pi^r_{l,+}(\omega_n(0+))\,.
\end{equation}
Therefore, if $\ms R_n>0$ and $\ms R=0$, then
conditions~\eqref{eq:R-Rn-lax-ineq}, \eqref{eq:omegan-cond34}, \eqref{eq:omegan-cond35}, \eqref{eq:omegan-cond39}
show that $\omega_n$ satisfies condition (ii)' of Theorem \ref{thm:attprofiles2}.

\noindent
\textbf{Step 4.}
\big($\ms R_n>0$, $\ms R>0$,
proof that $\omega_n$ satisfies condition (ii)' of Theorem \ref{thm:attprofiles2}\big).\\
By definition~\eqref{eq:omega-n-def}
one has 
\begin{equation}
\label{eq:omegan-cond40}
    \omega_n(0-)=\max\{\omega(0-), \overline A_n\}\,,
    \qquad\
    \omega_n(0+)=\max\{\omega(0+),  B_n\},
    \qquad\
    \omega_n(\ms R_n+)=\omega(\ms R+)\,,
\end{equation}
and
\begin{equation}
\label{eq:omegan-cond41}
\omega_n(x)\geq B_n\qquad \
\ \forall x\in\,]0, \ms R_n[\,,
\end{equation}
while definition~\eqref{eq:delta12n-def}, together with~\eqref{eq:LR-end-def}, \eqref{eq:omegan-cond40}, \eqref{eq:omegan-cond41},
and~\eqref{eq:uB-ineq} with $f=f_r$,
yield the implication~\eqref{eq:omegan-cond35}.
Since we are assuming that $\omega$ satisfies condition~\eqref{eq:2a2-prime} of Theorem~\ref{thm:attprofiles2},
relying on~\eqref{eq:omegan-cond40}
we deduce as in Step 4
that~\eqref{eq:omegan-cond36},
\eqref{eq:omegan-cond39}
hold.
Therefore, if $\ms R_n>0$ and $\ms R>0$, then
\eqref{eq:R-Rn-lax-ineq}, \eqref{eq:omegan-cond35}, \eqref{eq:omegan-cond39},
\eqref{eq:omegan-cond41}
show that $\omega_n$ satisfies condition (ii)' of Theorem \ref{thm:attprofiles2}.
\smallskip

\noindent
\textbf{Step 5.}
\big($\ms R_n=\ms R=0$, proof that $\omega_n$ satisfies condition (ii)
of Theorem \ref{thm:attprofiles3}\big).\\
By definition~\eqref{eq:omega-n-def}, 
we have
\begin{equation}
\label{eq:omegan-cond32}
\omega_n(0-)=\max\{\omega(0-), \; \overline A_n\}\geq \overline A_n,
\qquad\quad 
\omega_n(0+)=\omega(0+)\,,
\end{equation}
while
definition~\eqref{eq:delta12n-def}
yields
$\omega(0+)\leq \bs u[0, B_n, f_r]$.
Since  by Remark~\ref{rem:monotonicityxy} 
we have
$\bs u[0, B_n, f_r]  = \overline B_n$
($\overline B_n$ as in~\eqref{eq:bar-AB-def}),
it follows that 
\begin{equation}
\label{eq:omegan-cond32b}
\omega_n(0+)\leq \overline B_n.
\end{equation}
Moreover, 
by virtue of~\eqref{eq:omegan-cond32} 
we deduce~\eqref{eq:omegan-cond39}.
Hence, if $\ms R_n=0$ and $\ms R=0$, then conditions~\eqref{eq:omegan-cond39}, \eqref{eq:omegan-cond32}, \eqref{eq:omegan-cond32b} show that  
$\omega_n$ satisfies condition (ii)
of Theorem \ref{thm:attprofiles3}.
\smallskip

\noindent
\textbf{Step 6.}
\big($\ms R_n>0$, $\ms R\geq 0$,
proof that $\omega_n$ satisfies~\eqref{eq:1a2-prime}
of Theorem~\ref{thm:attprofiles2}\big).\\
Since we are assuming that $\omega$ satisfies either
the Ole\v{\i}nik estimates~\eqref{eq:1a3}
of Theorem~\ref{thm:attprofiles3} (in case $\ms R=0$), or
the Ole\v{\i}nik estimates
~\eqref{eq:1a2-prime}
of Theorem~\ref{thm:attprofiles2} (in case $\ms R>0$),
computing the Dini derivative 
of $\omega_n$ in~\eqref{eq:omega-n-def}, we find

\begin{equation}
\label{eq:oleinik-23}
    \begin{aligned}
    D^+\omega_n(x) &= D^+\omega(x-\delta_{1,n})\leq 
    \frac{1}{T \cdot f_r^{\second}(\omega(x-\delta_{1,n}))}
    =\frac{1}{T \cdot f_r^{\second}(\omega_n(x))}
    \qquad\forall~x> {\ms R}_n\,,
    \\
    \noalign{\smallskip}
    D^+\omega_n(x) &=
    D^+\omega(x+\delta_{2,n})\leq
    \frac{1}{T \cdot f_l^{\second}(\omega(x+\delta_{2,n}))}
    =
    \frac{1}{T \cdot f_l^{\second}(\omega_n(x))}
    \qquad\forall~x<-\delta_{2,n}\,,
    \\
    \noalign{\smallskip}
    D^+\omega_n(x) &=
    \frac{1}{T \cdot f_l^{\second}(\omega_n(x))}
    \qquad\quad \forall~x\in [-\delta_{2,n}, 0[\,.
    \end{aligned}
\end{equation}
Observing that $\omega_n$ is continuous at $x=-\delta_{2,n}$,
we deduce from~\eqref{eq:oleinik-23} 
that $\omega_n$
satisfies the Ole\v{\i}nik estimates~\eqref{eq:1a2-prime}
of Theorem~\ref{thm:attprofiles2}.

\smallskip

\noindent
\textbf{Step 7.}
\big($\ms R_n>0$, $\ms R\geq 0$, proof that $\omega_n$ satisfies~\eqref{eq:1b2-prime}
of Theorem~\ref{thm:attprofiles2}\big).\\
Observe that by definition~\eqref{eq:omega-n-def} $\omega_n$ is constant in $\,]\ms R, \ms R_n[$.
Therefore, since it holds~\eqref{eq:R-Rn-lax-ineq},
in order to show that $\omega_n$ satisfies the estimate~\eqref{eq:1b2-prime} on $]0, \ms R_n[$ it will be sufficient to show that ~\eqref{eq:1b2-prime} 
is verified on $]0, \ms R[$\,,
assuming that $\ms R>0$.

As observed in Step 5 of \S~\ref{sec:criticalcases-b},
the assumption
that $\omega$ satisfies condition~\eqref{eq:1b2-prime}
of Theorem~\ref{thm:attprofiles2} is equivalent to the 
fact that the maps
\begin{equation}
\label{eq;phi-psi-def-2}
\phi(x)\doteq -\tau(x) \cdot f'_l\circ
\pi_{l,+}^r(\omega(x)),
\qquad\quad \tau(x)\doteq T-\frac{x}{f_r^{\prime}(\omega(x))}\,,
\qquad x\in\,]0, \ms R[\,.
\end{equation}
are, respectively, nondecreasing 
and decreasing.
Then consider the corresponding maps for $\omega_n$
\begin{equation}
\phi_n(x)\doteq -\tau_n(x) \cdot f'_l\circ
\pi_{l,+}^r(\omega_n(x)),
\qquad\quad \tau_n(x)\doteq T-\frac{x}{f_r^{\prime}(\omega_n(x))}\,,
\qquad x\in\,]0, \ms R[\,,
\end{equation}
and compare their values 
in two points $0<x_1<x_2<\ms R$, 
of continuity for $\omega$ and $\omega_n$: \\
- if $\omega_n(x_i)=\omega(x_i)$ for $i=1,2$, then one clearly has that
$\phi_n(x_1)=\phi (x_1)\leq \phi (x_2)=\phi_n(x_2)$, 
$\tau_n(x_1)=\tau(x_1)>\tau(x_2)=\tau_n(x_2)$;\\
- if $\omega_n(x_i)\neq \omega(x_i)$ for $i=1,2$, then by definition~\eqref{eq:omega-n-def} we have $\omega_n(x_i)=B_n$ for $i=1,2$ and therefore 
one has $\phi_n(x_1)<\phi_n(x_2)$, $\tau_n(x_1)>
\tau_n(x_2)$;\\
- if $\omega_n(x_1)=\omega(x_1)$
and $\omega_n(x_2)\neq \omega(x_2)$, then by definition~\eqref{eq:omega-n-def} we have  
$\omega_n(x_1)\geq B_n$, 
\linebreak $\omega_n(x_2)=B_n$, which implies
$f'_r(\omega_n(x_2))\leq f'_r(\omega_n(x_1))$.
Moreover, since $\omega(x)\geq \theta_r$, $\omega_n(x)\geq \theta_r$,
it follows that 
$f'_l\circ \pi_{l,+}^r(\omega_n(x_1))
\geq f'_l\circ \pi_{l,+}^r(\omega_n(x_2))$.
Hence,  we derive that 
$\phi_n(x_1)\leq \phi_n(x_2)$
$\tau_n(x_1)>\tau_n(x_2)$;\\
- if $\omega_n(x_1)\neq \omega(x_1)$
and $\omega_n(x_2)=\omega(x_2)$, then by definition~\eqref{eq:omega-n-def} we have 
$\omega_n(x_1)=B_n > \omega(x_1)$, which implies
$f'_r(\omega_n(x_1))> f'_r(\omega(x_1))$.
Notice that by definition~\eqref{eq:LR-def}
of $\ms R$ it follows that
$\omega(x_1)\geq \theta_r$.
Since also $\omega_n(x_1)\geq \theta_r$,
we deduce that 
$f'_l\circ \pi_{l,+}^r(\omega_n(x_1))>
f'_l\circ \pi_{l,+}^r(\omega(x_1))$.
Thus, it follows that
$\phi_n(x_1)<\phi (x_1)\leq \phi (x_2)=\phi_n(x_2)$, 
$\tau_n(x_1)>\tau(x_1)>\tau(x_2)=\tau_n(x_2)$.\\
Hence, extending the above 
estimates to the right
limits of $\omega_n$ in its points of discontinuity,
we have shown that
 it holds true
\begin{equation}
\label{eq:phi-tau-n-monoton-2}
    \phi_n(x_1)\leq \phi_n(x_2),\qquad\quad
    \tau_n(x_1)>\tau_n(x_2)
    \qquad\ \forall~0<x_1<x_2<\ms R\,.
\end{equation}
In turn, the monotonicity~\eqref{eq:phi-tau-n-monoton-2} of
$\phi_n$, $\tau_n$
is equivalent to the fact that $\omega_n$ satisfies~\eqref{eq:1b2-prime}, by the same arguments of  Step 5 of \S~\ref{sec:criticalcases-b}.
\smallskip

\noindent
\textbf{Step 8.}
\big($\ms R_n=\ms R= 0$,
proof that $\omega_n$ satisfies~\eqref{eq:1a3}
of Theorem~\ref{thm:attprofiles3}\big).\\
The proof is entirely similar to the one of Step 6,
under the assumption that
$\omega$ satisfies 
the Ole\v{\i}nik estimates~\eqref{eq:1a3}
of Theorem~\ref{thm:attprofiles3}.

{ 
\vspace{0.3cm}
\noindent \textbf{Step 9.}
Finally, let us assume
\begin{equation}
\label{eq:LR-cond45'}
    \ms L=\ms L[\omega, f_l] < 0, \quad\qquad \ms R=\ms R[\omega, f_r]> 0\,,
\end{equation}
and (because of~\eqref{eq:critical-ip})
that $\omega$ satisfies the conditions (i), (ii) of Theorem~\ref{thm:attprofilescrit}.
Hence, by virtue of~\eqref{eq:AnBnconv} we may assume also that, for $n$ sufficiently large there hold
\begin{equation}
\label{eq:LR-cond45n}
   \ms L< T\cdot f'_l(A_n),\qquad\qquad  
   \ms R< T\cdot f'_r(B_n)\,.
\end{equation}

\noindent
In a similar way to what is done in Step 2, we will perturb $\omega$
to obtain an attainable profile $\omega_n$ for the
$(A_n,B_n)$ connection by:
\begin{itemize}
    \item[-]shifting $\omega$
    on the right of $x=0$ by a size $\delta_{1,n}$
    \item[-] choosing $\delta_{1,n}$ so to satisfy the admissibility condition~\eqref{eq:2a} at $x=\ms R+\delta_{1,n}$; 
        \item[-]dropping $\omega$ to the value $A_n$ of the connection when it is above, in the interval
    $]\ms L,0[$, so to satisfy the admissibility condition~\eqref{eq:2b-lr2-2}.
\end{itemize}
Namely,
consider the function
\begin{equation}
\label{eq:omega-n-def-2}
    \omega_n(x)\doteq  
\begin{cases}
\omega(x-\delta_{1,n}), & \text{if \ \ $x \geq \ms R+ \delta_{1,n}$},
\\
\noalign{\smallskip}
B_n & \text{if \ \ $ x \in \,]0, \ms R+\delta_{1,n}[$},\\
     \min\{A_n, \omega(x)\} & \text{if \ \ $x \in ]\ms L, 0$[}, \\
     \omega(x) & \text{if \ \ $x \leq \ms L$},
\end{cases}
\end{equation}
with 
\begin{equation}
     \label{eq:delta12n-def-2}
     \delta_{1,n} \doteq
    \inf \Big\{ \delta \in\R \ \ : \ \  
    \bs\tau[\ms R + \delta, B_n, f_r]=
    \bs\tau[\ms R , B, f_r]\Big\}\,.
     \end{equation}
    %
Notice that the definition~\eqref{eq:delta12n-def-2} is meaningful since 
 the map $\ms R \mapsto \bs \tau[\ms R, B_n, f_r]$ is strictly monotone and continuous, and because the image 
 of the maps 
 $$
 \begin{aligned}
     \bs\tau[\,\cdot\,, B, f_r]\, &: \, 
  \,]0, T\cdot f'(B)[\, \ \to \ \,]0,+\infty[\,,
  \\
  \bs\tau[\,\cdot\,, B_n, f_r]\, &: \, 
  \,]0, T\cdot f'(B_n)[\, \ \to \ \,]0,+\infty[\,,
 \end{aligned}
 $$
 is the set $]0, T[$
(see~\S~\ref{def:rsr-block}).
Then, recalling the definitions~\eqref{eq:LR-def},
\eqref{eq:Ltildedef}, and because of~
\eqref{eq:LR-cond45n},
we deduce
that
\begin{equation}
\label{eq:LR-end-def-2}
    \ms L_n\doteq 
    \ms L[\omega_n, f_l] = \ms L, \quad\qquad \ms R_n\doteq \ms R[\omega_n, f_r] = \ms R + \delta_{1,n},
\end{equation}
and
\begin{equation}
\label{eq:tildeLR-end-def-2}
    \widetilde{\ms L}_n\doteq 
    \widetilde{\ms L}[\omega_n, f_l] = (T-\bs \tau_n) \cdot f_l^\prime(A_n) = (T-\bs \tau) \cdot f_l^\prime(A_n),
\end{equation} 
where
\begin{equation}
    \label{eq:tau-taun-def}
    \bs \tau_n\doteq 
    \bs\tau[\ms R + \delta_{1,n}, B_n, f_r],
    \qquad\quad
    \bs \tau\doteq 
    \bs\tau[\ms R, B, f_r].
\end{equation}
Relying on~\eqref{eq:AnBnconv} and 
since $\bs \tau[\ms R, B, f_r]$
depends continuously on the parameters $\ms R, B$
(see~\S~\ref{def:rsr-block}),
one deduces that~$\lim_n\delta_{1,n}=0$, that $\lim_n \widetilde {\ms L}_n = 0$,
and that
$\omega_n$ converges to $\omega$
  in $\mathbf L^1_{\mr{loc}}$   as $n\to \infty$. Hence, as in Step 2 above we conclude that, in order to show the validity of~\eqref{eq:familyomegan}
  it remains to prove that 
$\omega_n$ satisfies conditions (i)-(ii) of Theorem~\ref{thm:attprofiles}.
 
 Assuming that $\widetilde{\ms L}_n\in\,]\ms L, 0[$ for $n$ sufficiently large,
 in order to show that $\omega_n$ satisfies~\eqref{eq:2b-lr2-2} of Theorem~\ref{thm:attprofiles} it will be sufficient to prove that
 \begin{equation}
 \label{eq:2b-lr2-2-n}
     A_n\leq \omega(x)\qquad\forall~x\in\,]\widetilde{\ms L}_n,0[\,,
     \qquad\quad A_n\leq \omega(\,\widetilde{\ms L}_n-)\,.
 \end{equation}
 To this end observe that, by definition~\eqref{eq:tildeLR-end-def-2},
 we have $(f'_l)^{-1}(\frac{x}{T-\bs\tau[\ms R, B, f_r]})>A_n$ for all $x \in \,]\widetilde  {\ms L}_n, 0[\,$.
 Thus, the first inequality in~\eqref{eq:critcond-2} satisfied by $\omega$
 implies that
 $\omega(x) > A_n$ for all $x \in \,]\widetilde  {\ms L}_n, 0[\,$,
 which proves 
 the first condition in~\eqref{eq:2b-lr2-2-n}.
Next observe that, since
$\widetilde{\ms L}_n\in\,]\ms L, 0[$,
from the first inequality in~\eqref{eq:critcond-2} 
and by definition~\eqref{eq:tildeLR-end-def-2} it follows 
\begin{equation*}
    f'_l(\omega(\,\widetilde {\ms L}_n-))\geq\frac{\widetilde {\ms L}_n}{T-\bs\tau}\geq f'_l(A_n),
\end{equation*}
which 
 implies $\omega(\,\widetilde {\ms L}_n-) \geq A_n$. This completes the proof of~\eqref{eq:2b-lr2-2-n} and thus that $\omega_n$ satisfies~\eqref{eq:2b-lr2-2}.
The verification that $\omega_n$ satisfies the remaining conditions 
 in (i)-(ii) of Theorem~\ref{thm:attprofiles}
is entirely similar to the one performed in Steps 4, 6, 7 above, and is accordingly omitted.}
This concludes the proof
of the implication $(3) \Rightarrow (1)$ for critical connections.
\blu{\begin{remark}
Whenever $\mr{Tot.Var.}(\omega)<+\infty$,
the  perturbed profiles $\omega_n$
approximating  $\omega$
constructed in Step 2 and in Step 9 of \S~\ref{sec:criticalcases-c}
may possibly have larger total variation than the one of~$\omega$.
However, $\omega_n$
have always local bounded variation,
even in the case where $\mr{Tot.Var.}(\omega) = +\infty$.
In fact, assuming~\eqref{eq:LR-cond45}
and that $\omega$ satisfies the conditions of Theorem \ref{thm:attprofiles3}, suppose that $\omega$ has unbounded total variation on a right neighborhood of $x=0$.
Then, letting $\{(A_n,B_n)\}_n$ 
be a sequence of non critical connections satisfying~\eqref{eq:AnBnconv}, \eqref{eq:AnBncond2}, 
there should exist a sequence of positive values $\rho_n \downarrow 0$,
so that 
\begin{equation}
\label{eq:bound-approx-profile}
    \omega(x)\leq B_n,\qquad \forall~x\in\,]0, \rho_n]\,,
\end{equation}
for all $n$ sufficiently large. If this is not the case, then there should exist $\overline\rho>0$ and $\overline n$ so that $\omega(x)\geq B_{\overline n}>\theta_r$ for all $x\in \,]0, \overline \rho]$.
But this in turn would yield 
uniform upper bounds on $D^+\omega$
(and hence on the total variation of $\omega$ as well) 
on bounded subsets $K$ of $[0, +\infty[$,
with the same type of analysis of~\S~\ref{subsec:reductionBV}. 
Therefore, because of~\eqref{eq:bound-approx-profile}, by definition~\eqref{eq:omega-n-def} we have
\begin{equation}
\label{eq:bound-approx-profile-2}
    \omega_n(x)= B_n,\qquad \forall~x\in\,]0, \rho_n]\,,
\end{equation}
for all $n$ sufficiently large.
The property~\eqref{eq:bound-approx-profile-2}
has precisely  the effect to cut the 
possible large oscillation of $\omega$
occurring in a right neighborhood of $x=0$,
and hence to ensure that $\mr{Tot.Var.}(\omega_n,\, K)<+\infty$
for all $n$ large.
Clearly, we will have that $\lim_n \mr{Tot.Var.}(\omega_n,\, K)=+\infty$.
With entirely similar 
arguments one can show that,
if $\omega$ satisfies the conditions (i), (ii) of Theorem~\ref{thm:attprofilescrit}, then the profile $\omega_n$ defined by~\eqref{eq:omega-n-def-2} has always local bounded variation.
\end{remark}}

\section{$BV$ bounds for $AB$-entropy solutions}
\label{sec:BVboundsABsol}
We collect in this section the $BV$ bounds for solutions, and for the flux of the solutions, that arise as a corollary of our analysis. 

\begin{prop}\label{BVbound}
In the same setting of Theorem~\ref{theoremsemigroup}, 
for every $u_0 \in {\bf L}^{\infty}(\R)$,
and for any bounded set $K\subset\R$, 
the following properties are verified.
\begin{enumerate}
[leftmargin=25pt]
\item[(i)] 
For any non critical connection $(A,B)$, there exists a constant $C_1=C_1(A,B,\|u_0\|_{{\bf L^\infty}}, K)>0$ such that
it holds true
\begin{equation}
\label{eq:bv-bound-2}
    \mr{Tot.Var.}\big(\sabpt u_0,\, K\big)\leq \frac{C_1}{t}
    \qquad\forall~t>0\,.
\end{equation}
{  In particular, 
any attainable profile
$\omega\doteq \sabpT u_0$, $u_0 \in {\bf L}^{\infty}(\R)$, $T>0$,
enjoy the
 property~\eqref{eq:Hhyp} stated in \S~\ref{roadmap}-{\bf Part 1}.
}
\item[(ii)] 
There exists a constant $C_2=C_2(\|u_0\|_{{\bf L^\infty}}, K)>0$ such that,
for any connection  $(A,B)$, 
it holds true
\begin{equation}
\label{eq:bv-bound-3}
\begin{aligned}
    \mr{Tot.Var.}\big(f_l\circ \sabpt u_0,\, K\cap\,]-\infty, 0]\big)&\leq \frac{C_2}{t},
    \\
    \mr{Tot.Var.}\big(f_r\circ \sabpt u_0,\, K\cap [0, +\infty[\big)&\leq \frac{C_2}{t},
\end{aligned}
    \qquad\forall~t>0\,,
\end{equation}
where the inequalities are understood to be verified  whenever $K\,\cap\,]-\infty, 0]\neq \emptyset$, or \linebreak $K\,\cap [0, +\infty[\, \neq \emptyset$, respectively. 
\end{enumerate}
\end{prop}
\begin{proof}
Since $\mc S^{[AB]+}_t u_0\in \mc A^{[AB]}(t)$ and
thanks to the implication $(1) \Rightarrow (3)$ of Theorem~\ref{thm:backfordiscfluxcycle}, we know that 
$\mc S^{[AB]+}_t u_0$
satisfies the conditions stated in 
Theorem \ref{thm:attprofiles}, \ref{thm:attprofilescrit}, \ref{thm:attprofiles2}, or \ref{thm:attprofiles3}, {  that cover all possible cases.}
We divide the proof in four steps.\\

\noindent
\noindent
\textbf{Step 1.} {\it (proof of (i)).}\\
    In the case of a non critical connection $(A,B)$, 
    it is well known that for initial data $u_0\in BV(\R)$, one has $\mc S^{[AB]+}_t u_0\in BV(\R)$ for all $t>0$
(see~\cite[Lemma 8]{Garavellodiscflux}
and
\cite[Theorem 2.13-(iii)]{MR2743877}).
\blu{On the other hand}, for initial data $u _0 \in {\bf L}^{\infty}(\R)$,
we know that 
$\mc S^{[AB]+}_t u_0$
satisfies the Ole\v{\i}nik-type inequalities stated in 
Theorem \ref{thm:attprofiles}, \ref{thm:attprofiles2}, or \ref{thm:attprofiles3}.
Thus, since $(A,B)$ is a non critical connection, 
by the analysis in~\S~\ref{subsec:reductionBV}
we deduce that $D^+(\mc S^{[AB]+}_t u_0)$
satisfies
one-sided uniform upper bounds 
as the ones provided by~\eqref{eq:ol-est-un-32}. 
In turn, such bounds
yield the existence of  uniform bounds on the total increasing variation (and hence on the total variation as well) of $S^{[AB]+}_t u_0$ on bounded subsets $K$ of $[0, +\infty[$,
which depend on 
the connection $(A,B)$, on the set $K$, and on $\|u_0\|_{{\bf L^\infty}}$.
By similar arguments we derive bounds on the 
total variation  of $S^{[AB]+}_t u_0$ on bounded subsets of $]-\infty, 0]$, which yields~\eqref{eq:bv-bound-2},
completing the proof of~(i).
%
\noindent

\smallskip
\textbf{Step 2.} 
{\it (proof of (ii) when $S^{[AB]+}_t u_0$ satisfies the conditions of
Theorem~\ref{thm:attprofiles3}).}\\
Since $S^{[AB]+}_t u_0$ satisfies the Ole\v{\i}nik-type inequalities~\eqref{eq:1a3} of
Theorem~\ref{thm:attprofiles3},
we immediately deduce 
a uniform bound 
on the total increasing variation
of $S^{[AB]+}_t u_0$
on bounded sets, 
which does not depend on the
values $f'_l(A), f'_r(B)$.
In turn, such bounds
yield the existence of  uniform bounds on the total increasing variation (and hence on the total variation as well) of $S^{[AB]+}_t u_0$ on bounded subsets $K$ of $[0, +\infty[$,
which depend on 
the set $K$ and on~$\|u_0\|_{{\bf L^\infty}}$.
By similar arguments we derive bounds on the 
total variation  of $S^{[AB]+}_t u_0$ on bounded subsets of $]-\infty, 0]$, which yields~\eqref{eq:bv-bound-2},
with a constant $C_1$ that
depends only on 
the set $K$ and on~$\|u_0\|_{{\bf L^\infty}}$.
In turn, ~\eqref{eq:bv-bound-2} yields~\eqref{eq:bv-bound-3}
relying on the Lipschitzianity of $f_l, f_r$ on the set $[-M, M]$,
with $M\doteq \|u_0\|_{{\bf L^\infty}}$. 
This completes the proof of (ii) in 
the case where $S^{[AB]+}_t u_0$ satisfies
the conditions
 stated in Theorem~\ref{thm:attprofiles3}.
 \smallskip
 
\noindent
\textbf{Step 3.} 
{\it (proof of (ii) when $S^{[AB]+}_t u_0$ satisfies the conditions of
Theorem~\ref{thm:attprofiles2}).}\\
To fix the ideas, we assume that 
 $\omega\doteq S^{[AB]+}_t u_0$ satisfies
 the inequalities (i)'
 and the pointwise constraints-(ii)'
 stated in Theorem~\ref{thm:attprofiles2}.
 Notice that, by the same arguments of above, ~\eqref{eq:1a2-prime} yields  
 the estimate~\eqref{eq:bv-bound-2}
 (and hence also~\eqref{eq:bv-bound-3}) 
 for bounded set $K\subset \,]-\infty, 0]$ or $K\subset [\ms R,+\infty[$\,.
 Then consider a  set $K\subset [0,\ms R]$, with $\ms R= \ms R[\omega, f_r]$
 defined as in~\eqref{eq:LR-def},
 and assume that the inequalities~\eqref{eq:1b2-prime}, \eqref{eq:2b2-prime}, are satisfied.
Observe that, 
by the uniform convexity~\eqref{eq:flux-assumption-1}  of $f_l, f_r$,
we have
\begin{equation}
\label{eq:esth-2}
\frac{f''_l\circ 
     \pi_{l,+}^{r}
     (u)}{
    \big[f'_l\circ 
     \pi_{l,+}^{r}
     (u)\big]^{2}}\geq c_1,
     \qquad\
     f''_r(u)\geq c_1,
     \qquad\forall~|u|\leq \|\omega\|_{{\bf L}^\infty}\,,
\end{equation}
for some constant $c_1>0$
depending on $ \|\omega\|_{{\bf L}^\infty}$.
Moreover,
 by definition~\eqref{eq:LR-def}
 of $\ms R$
it holds true
 \begin{equation}
 \label{eq:esth-3}
     t\cdot f'_r(\omega(x)) > x
     \qquad \forall~x\in[0, \ms R[\,,
     \quad t>0\,.
 \end{equation}
 Hence, recalling the definition~\eqref{eq:ghdef} of the function $h$,
 and relying on~\eqref{eq:1b2-prime}, 
 \eqref{eq:2b2-prime},
 \eqref{eq:esth-2}, \eqref{eq:esth-3}, 
 we derive
 \begin{equation}
 \label{eq:esth-4}
     \begin{aligned}
     D^+(f_r\circ\omega)(x)&=
     f'_r(\omega(x)) \,D^+\omega(x)
     \\
     &\leq f'_r(\omega(x))\, 
     h[\omega,f_l,f_r](x)
     \\
     &\leq \frac{[f'_r(\omega(x))]^2}{c_1\,[f'_r(\omega(x))]^2\big(t\cdot f'_r(\omega(x)) - x\big)+c_1 \,x}
     \qquad \forall~x\in[0, \ms R[\,.
     \quad t>0\,.
     \end{aligned}
 \end{equation}
 Towards an estimate of~\eqref{eq:esth-4}, consider the map
 \begin{equation}
 \label{eq:phi-def}
     \Phi(x,t,u)\doteq 
     \begin{cases}
         \dfrac{[f'_r(u)]^2}{[f'_r(u)]^2\big(t\cdot f'_r(u) - x\big)+x},\quad&\text{if}\quad u>\theta_r,
         \\
         \noalign{\medskip}
         \  0,\quad&\text{if}\quad u=\theta_r,
     \end{cases}
     \qquad\quad x\in[0, \ms R[\,,\quad t>0\,.
 \end{equation}
 By direct computations one finds
  \begin{equation*}
      \Phi_u(x,t,u)=\frac{f'_r(u)\,f''_r(u)\big(2x-t\,[f'_r(u)]^3\big)}{\Big([f'_r(u)]^2\big(t\cdot f'_r(u) - x\big)+x\Big)^2\ }.
  \end{equation*}
  Hence, since $f'_r(u)\geq 0$ for all $u\geq \theta_r$, and because $f''_r(u)>0$ for 
  all $u$, we deduce that,
  setting
\begin{equation}
\label{eq:min-uxt}
     u_{x,t}\doteq (f'_r)^{-1}
     \bigg(\sqrt[3]{\frac{2x}{t}\,}\, 
     \bigg),
\end{equation}
  for all $x, t >0$ it holds true
  \begin{equation}
  \label{eq:der-phi}
  u_{x,t}>\theta_r,\qquad\quad 
      \Phi_u(x,t,u)
      \begin{cases}
      \, \geq 0\ \ &\text{if}\quad u\in [\theta_r, u_{x,t}],
       \\
       \noalign{\smallskip}
       \, \leq 0 
       \ \  &\text{if}\quad u\geq 
       u_{x,t}\,.
       \end{cases}
  \end{equation}
  In turn, \eqref{eq:min-uxt}, \eqref{eq:der-phi} imply
  that $u_{x,t}$ is a point of global maximum for the map $u\mapsto  \Phi(x,t,u)$, $u\geq \theta_r$. On the other hand, 
  because of~
  \eqref{eq:der-phi}
  we have
  \begin{equation}
  \label{eq:esth-6}
      t\cdot f'_r(u_{x,t}) > x
      \quad \Longrightarrow\quad x< \sqrt{2} \, t\,.
  \end{equation}
  Thus we find
  \begin{equation}
  \label{eq:max-phi-1}
      \Phi(x,t,u)\leq \Phi(x,t,u_{x,t})
      =\frac{1}{\sqrt[3]{x}\Big(3\big(\frac{t}{2}\big)^{\frac{2}{3}}-x^{\frac{2}{3}}\Big)}
      <\frac{1}{\sqrt[3]{x}}\bigg(\frac{2}{t}\bigg)^{\!\!\frac{2}{3}}\,,
      \qquad \forall~
      x< \sqrt{2} \, t,\ \ u\geq \theta_r\,,
  \end{equation}
  and
  \begin{equation}
      \label{eq:max-phi-2}
      \Phi(x,t,u)\leq \frac{[f'_r(u)]^2}{x}\leq 
      \frac{c_2}{t}\,,
      \qquad \forall~
      x\geq  \sqrt{2} \, t,\ \ \theta_r\leq u\leq 
       \|\omega\|_{{\bf L}^\infty}\,,
  \end{equation}
  for some constant $c_2$ depending on $ \|\omega\|_{{\bf L}^\infty}$.
Then, relying on~\eqref{eq:esth-3}, \eqref{eq:esth-4}, \eqref{eq:phi-def}, \eqref{eq:esth-6}, \eqref{eq:max-phi-1}, \eqref{eq:max-phi-2},
we derive
\begin{equation}
\label{eq:max-phi-4}
D^+(f_r\circ\omega)(x)\leq 
\begin{cases}
        \dfrac{c_3}{\sqrt[3]{x}\ t^\frac{2}{3}}
        \ \ &\text{if} \quad x< \sqrt{2}\, t, \quad  x\in[0, \ms R[\,,
        \\
        \noalign{\bigskip}
        \dfrac{c_3}{t}
        \ \  &\text{if} \quad x\geq \sqrt{2}\, t,\quad  x\in[0, \ms R[\,,
    \end{cases}
\end{equation}
for some other constant $c_3$ depending on $ \|\omega\|_{{\bf L}^\infty}$.
%
Hence, recalling Remark~\ref{rem:linf-bound-ABsol}, we deduce
that, given 
a bounded set $K\subset [0,\ms R]$,
we have
\begin{equation*}
    \int_K D^+\big(f_r\circ \sabpt u_0\big)(x)~dx 
    \leq \frac{C}{t},
\end{equation*}
for some constant $C$ depending only on 
$\|u_0\|_{{\bf L^\infty}}, K$, which
yields~\eqref{eq:bv-bound-3}.
This completes the proof of (ii) in 
the case where $S^{[AB]+}_t u_0$ satisfies
the conditions
 stated in Theorem~\ref{thm:attprofiles2}.
\smallskip

 \smallskip
 
\noindent
\textbf{Step 4.} 
{\it (proof of (ii) when $S^{[AB]+}_t u_0$ satisfies the conditions of
Theorem~\ref{thm:attprofiles}
\blu{or of Theorem~\ref{thm:attprofilescrit}}).}\\
Since $S^{[AB]+}_t u_0$ satisfies the Ole\v{\i}nik-type inequalities~\eqref{eq:1a} of
Theorem~\ref{thm:attprofiles}
\blu{(or~\eqref{eq:1acrit}, \eqref{eq:1acrit1} of Theorem~\ref{thm:attprofilescrit})},
with the same analysis of Step~2
we deduce the uniform bound in~\eqref{eq:bv-bound-3} for bounded subset $K$ of  
$]-\infty, \ms L]$ or of $[\ms R,+\infty[$.
\blu{Next}, for sets $K\subset [0,\ms R]$
or $K\subset [\ms L, 0]$, relying on the 
Ole\v{\i}nik-type inequalities~\eqref{eq:1b1}, \eqref{eq:1b12}, of
Theorem~\ref{thm:attprofiles}
\blu{(or \eqref{eq:1b1crit}, \eqref{eq:1b1crit1} of Theorem~\ref{thm:attprofilescrit})},
we recover the  bound in~\eqref{eq:bv-bound-3} performing \blu{the} same analysis of Step~3. 
This completes the proof of (ii) in 
the case where $S^{[AB]+}_t u_0$ satisfies
the conditions
 stated in Theorem~\ref{thm:attprofiles}
 \blu{or in Theorem~\ref{thm:attprofilescrit}},
 and concludes the proof of the proposition.
\end{proof}

\medskip
\medskip

\appendix
\section{Stability of solutions with respect to connections and BV bounds}\label{app:stabconn}

We provide here a proof of Properties (iv)-(v) of Theorem~\ref{theoremsemigroup},
which seems to be absent in the literature. To this end
we first recall a by now classical technical lemma, useful for the analysis
of stability of discontinuous conservation laws  (e.g. see~\cite{audusse2005uniqueness}, \cite[Proposition 1]{Garavellodiscflux}).  
For sake of completeness we provide  a proof below.

\begin{lemma}\label{alphalemma}
Fix a connection $(A,B)$
and let $I^{AB}$ be the map in~\eqref{IABdef}. Then,
for any couple of  pairs  $(u_l,u_r)$, $(v_l,v_r)\in \R^2$  that verify
\begin{equation}\label{ineqABentropylemma}
    I^{AB}(u_l,u_r) \leq 0, \qquad I^{AB}(v_l,v_r) \leq 0\,,
\end{equation}
and 
\begin{equation}\label{lemmaRH}
    f_l(u_l) = f_r(u_r), \qquad f_l(v_l)=f_r(v_r)\,,
\end{equation}
setting
\begin{equation}\label{alpha}
    \alpha(u_l,u_r,v_l,v_r)  \doteq \mathrm{sgn}(u_r-v_r) \cdot (f_r(u_r) - f_r(v_r)) - \mathrm{sgn}(u_l-v_l) \cdot (f_l(u_l) - f_l(v_l))\,,
\end{equation}
it holds true
\begin{equation}\label{alphaneg}
    \alpha(u_l,u_r,v_l,v_r) \leq 0\,.
\end{equation}
\end{lemma}

\begin{proof}
Observe that, if $u_r = v_r$ or $u_l = v_l$, then the left hand side of \eqref{alphaneg} is zero and~\eqref{alphaneg}
is verified. 
Hence, 
without loss of generality, we may assume that $u_r > v_r$
and $u_l\neq v_l$. If $u_l >v_l$, the left hand side of \eqref{alphaneg} is again zero, because of~\eqref{lemmaRH}. Thus, assuming that $u_l<v_l$, we have
\begin{equation}\label{ulleqvlcase}
\alpha(u_l, u_r, v_l, v_r) = 2(f_r(u_r)-f_r(v_r))\,.
\end{equation}
If we suppose, by contradiction, 
that \eqref{alphaneg} is not verified,
it would follow by \eqref{ulleqvlcase}, that 
$f_r(u_r)>f_r(v_r)$. 
Moreover, because of assumptions~\eqref{ineqABentropylemma}-\eqref{lemmaRH}, and applying Lemma \ref{lem:traceseq2},
we know that $f_r(u_r), f_r(v_r)\geq f_r(B)$.
Since $u_r >v_r$,  these inequalities together imply that $u_r \geq B$. 
On the other hand, by \eqref{lemmaRH}, it also holds  $f_l(u_l) >f_l(v_l)$. 
Relying again on
~\eqref{ineqABentropylemma}-\eqref{lemmaRH} and Lemma-\ref{lem:traceseq2},
we deduce that
$f_l(u_l), f_l(v_l)\geq f_l(A)$, which,
coupled with $u_l<v_l$,
$f_l(u_l) >f_l(v_l)$, implies
 $u_l \leq A$. Hence, by Lemma~\ref{lem:traceseq2} it follows that $u_r = B$ and $u_l = A$, and then we would have\begin{equation}
\begin{aligned}
    \alpha(u_l, u_r, v_l, v_r) &= \alpha(A, B, v_l, v_r) \\
   & =  \mathrm{sgn}(B-v_r) \cdot (f_r(B) - f_r(v_r)) - \mathrm{sgn}(A-v_l) \cdot (f_l(A) - f_l(v_l)) \\ 
    & = I^{AB}(v_l, v_r) \leq 0
    \end{aligned}
\end{equation}
\blu{which} is a contradiction. Therefore \eqref{alphaneg} is satisfied, and the proof is concluded.
\end{proof}

In order to obtain stability with respect to perturbations of the connection, \blu{the following quantitative version of Lemma~\ref{alphalemma} will be useful.  A general version of this Lemma can be found in \cite[Proposition 3.21]{MR2807133} (see also \cite[Proposition 2.10]{AGS10} for the case $f_l=f_r$).}

\begin{lemma}\label{Elemmadiffconnect}
Let $(A,B)$, $(A',B')$
be two connections.
Then, for any couple of  pairs  $(u_l,u_r)$, $(v_l,v_r)\in \R^2$  that verify
\begin{equation}
    I^{AB}(u_l,u_r) \leq 0, \qquad I^{A^{\prime}B^{\prime}}(v_l,v_r) \leq 0\,,
\end{equation}
and~\eqref{lemmaRH},
it holds true
\begin{equation}
\label{eq:alpha-est2}
    \alpha(u_l,u_r,v_l, v_r) \leq  2\left| f_r(B^{\prime})-f_r(B)\right|.
\end{equation}
\end{lemma}
\begin{proof}
With no loss of generality assume that $B^{\prime} > B$.
Then, applying Lemma~\ref{lem:traceseq2}, one deduces that $B^{\prime} > B$, together with~\eqref{lemmaRH}
and 
$I^{A^{\prime}B^{\prime}}(v_l,v_r) \leq 0$,
implies that
one of the following two holds:
\begin{enumerate}
    \item $I^{AB}(v_l,v_r) \leq 0$,
    \smallskip
    \item $(v_l, v_r) = (A^{\prime},B^{\prime}) $.
\end{enumerate}
If (1) holds, then by Lemma \ref{alphalemma} we have $\alpha(u_l,u_r,v_l,v_r) \leq 0$,
and therefore~\eqref{eq:alpha-est2}
is verified. Otherwise, (2) holds. In this case, we can add and subtract the non positive quantity $I^{AB}(u_l,u_r) $,
and rewrite $\alpha$ as 
\begin{equation}
    \alpha(u_l,u_r,A^{\prime}, B^{\prime}) =
    \alpha_r(u_r) -\alpha_l(u_l) + I^{AB}(u_l,u_r) \leq  \alpha_r(u_r) -\alpha_l(u_l),
\end{equation}
where 
\begin{equation*}
\begin{aligned}
    \alpha_r(u_r) &\doteq  \mathrm{sgn}(u_r-B^{\prime})\cdot(f_r(u_r)- f_r(B^{\prime})) - \mathrm{sgn}(u_r-B) \cdot (f_r(u_r)- f_r(B)),
    \\
    \noalign{\smallskip}
    \alpha_l(u_l) &\doteq  \mathrm{sgn}(u_l-A^{\prime}) \cdot (f_l(u_l)- f_l(A^{\prime})) - \mathrm{sgn}(u_l-A)\cdot (f_l(u_l)- f_l(A)).
\end{aligned}
    \end{equation*}
We provide  separately an estimate on $\alpha_r(u_r)$ and on $\alpha_l(u_l)$. We consider first the term  $\alpha_r$, and we distinguish three cases.
\begin{enumerate}
    \item $u_r > B^{\prime}$. Then one has
    $$
        \alpha_r(u_r) = f_r(u_r) -f_r(B^{\prime})-f_r(u_r)+f_r(B) = f_r(B)-f_r(B^{\prime}).
    $$
    \item $u_r \in [B,B^{\prime}]$. Observe that, applying Lemma~\ref{lem:traceseq2}
    and relying on~\eqref{lemmaRH} and 
$I^{AB}(u_l,u_r) \leq 0$,
    we deduce
    $f_r(u_r)\geq f_r(B)$.
    Then one has
    $$
    \begin{aligned}
        \alpha_r(u_r) &= -f_r(u_r) +f_r(B^{\prime})-f_r(u_r)+f_r(B) 
        \\
        \noalign{\smallskip}
        &= (f_r(B)+f_r(B^{\prime})-2f_r(u_r)) \leq  (f_r(B^{\prime})-f_r(B)).
    \end{aligned}
        $$
    \item $u_r < B$. Then one has
    $$
    \alpha_r(u_r) = -f_r(u_r) +f_r(B^{\prime})+f_r(u_r)-f_r(B) \leq f_r(B^{\prime})-f_r(B).
    $$
\end{enumerate}
In every case, we obtain 
\begin{equation}
\label{eq:est-alpha-3}
    \alpha_r(u_r) \leq \left| f_r(B^{\prime})-f_r(B)\right|.
\end{equation}
Analogously, and thanks to~\eqref{lemmaRH}, we can prove that
\begin{equation}
\label{eq:est-alpha-4}
    \alpha_l(u_l) \geq  -\left| f_l(A^{\prime})-f_l(A)\right| = -\left| f_r(B^{\prime})-f_r(B)\right|
\end{equation}
which in turn, together with~\eqref{eq:est-alpha-3}, implies 
\begin{equation}
    \alpha(u_l,u_r,A^{\prime},B^{\prime}) \leq 2 \left| f_r(B^{\prime})-f_r(B)\right|,
\end{equation}
and this concludes the proof of the lemma. 
\end{proof}
\medskip

\begin{proof}[Proof of Theorem \ref{theoremsemigroup}-$(iv)$-$(v)$]
Set 
\begin{equation}
    u(x,t) \doteq \sabpt u_0 (x), \qquad v(x,t) \doteq  \mc S_t^{[A^{\prime}B^{\prime}]+} 
    u_0(x).
\end{equation}
Relying on property (2)
of Definition~\ref{defiAB},
with standard doubling of variable arguments (e.g. see~\cite[\S 6.3]{bressan2000hyperbolic})
one obtains that, for every
non-negative test function $\phi\in\mathcal{C}^1_c$ with compact support 
contained in $]-\infty,0[\,\times\,]0,+\infty[$, it holds true
\begin{equation}
\label{kruz+1}
\int_{-\infty}^0\int_0^{\infty}  \big\{|u-v| \phi_t + \mathrm{sgn}(u-v)\left(f_l(u) - f_l(v)\right) \phi_x \big\}
\dif x \dif t \geq 0, 
\end{equation}
and, for for every
non-negative test function $\phi\in\mathcal{C}^1_c$ with compact support 
contained in
$]0,+\infty[\,\times$ $]0,+\infty[$, it holds true
\begin{equation}
\label{kruz-1}
\int_0^{\infty}\int_0^{\infty} \big\{
|u-v| \phi_t + \mathrm{sgn}(u-v)\left(f_r(u) - f_r(v)\right) \phi_x
\big\}\dif x \dif t \geq 0.
\end{equation}
Hence, with the same arguments, one deduces that,  for every
non-negative test function $\phi\in\mathcal{C}^1_c$ with compact support 
contained in
$\R \times$ $]0,+\infty[$, it holds true
%
\begin{equation}\label{leqE}
\int_0^{+\infty} \int_{-\infty}^{+\infty} \big\{\left|u-v\right|\phi_t+\mathrm{sgn}(u-v)(f(x,u)-f(x,v))\phi_x\big\} \dif x \dif t  \geq -E\,, 
\end{equation}
where $E$ is the extra boundary term
at $x=0$ (due to the fact that,
differently from~\eqref{kruz+1}-\eqref{kruz-1}, $\phi$ will not vanish in general at $x=0$ )
given by
$$E =  \int_0^{+\infty}\big[\mathrm{sgn}(u(x,t)-v(x,t))(f(x,u(x,t))-f(x,v(x,t)))\big]_{x=0-}^{x=0+}\,\phi(0,t) \dif t\,,
$$
with $[\cdot]_{x=0-}^{x=0+}$ denoting the limit from the right minus the limit from the left at $x = 0$. 
Observe that, 
letting $u_l,u_r$ denote the one-sided limit of $u$ in $x=0$
as in~\eqref{traces},
and denoting
$v_l,v_r$, the corresponding ones 
for $v$,
recalling~\eqref{alpha} we can rewrite the quantity $E$ as 
\begin{equation}
\label{eq:E-2}
    E = \int_0^{+\infty} \alpha(u_l(t),u_r(t),v_l(t),v_r(t)) \phi(0,t) \dif t\,.
\end{equation}
On the other hand, 
since $u_l, u_r$, and
$v_l, v_r$
satisfy the Rankine-Hugoniot condition \eqref{RHtraces},
together with the inequality
\eqref{interfaceentropy}
related to the $(A,B)$, and 
$(A',B')$ connection, respectively,
applying Lemma~\ref{Elemmadiffconnect}
we deduce that it holds true
\begin{equation}
\label{eq:alpha-est3}
    \alpha\big(u_l(t(,u_r(t),v_l(t), v_r(t)\big) \leq  2\left| f_r(B^{\prime})-f_r(B)\right|
    \qquad  \text{for a.e. $t>0$\,.}
\end{equation}
Thus, combining~\eqref{leqE} with~\eqref{eq:E-2}, \eqref{eq:alpha-est2}, we find
\begin{equation}
\label{eq:kruz-3}
         \int_0^{+\infty} \int_{-\infty}^{+\infty} \big\{\left|u-v\right|\phi_t+\mathrm{sgn}(u-v)(f(x,u)-f(x,v))\phi_x\big\} \dif x \dif t  \geq -2 |f_r(B)-f_r(B^{\prime})| \int_0^{+\infty} \phi(0,t) \dif t\,.
\end{equation}
Now fix $\tau>\tau_0>0$,  $R>0$,
and consider the trapezoid
$\Omega\doteq \{(x,t)\, : \ \tau_0\leq t\leq \tau,\, |x|\leq R+L (\tau-t)\}$,
where $L\doteq \sup_{|z|\leq M}  \max \{|f'_l(z)|, |f'_r(z)|\}$, with
$M$ being a uniform ${\bf L^\infty}$ bound
for $u$ and $v$.
Then, by a standard technique
(e.g. see~\cite[\S 6.3]{bressan2000hyperbolic}), one can construct a sequence of test functions $\phi_n\in\mathcal{C}^1_c$,
with compact support 
contained in
$\R \times$ $]0,+\infty[$, that approximate the characteristic function of $\Omega$ when $n\to\infty$. 
Employing~\eqref{eq:kruz-3} with
$\phi_n$, and letting $n\to \infty$ we obtain
\begin{equation}
\label{eq:kruz-4}
    \int_{|x|\leq R}\big|u(x,\tau)-v(x,\tau)\big| \dif x \leq
    \int_{|x|\leq R+L (\tau-\tau_0)}\big|u(x,\tau_0)-v(x,\tau_0)\big| 
    \dif x+2 (\tau-\tau_0)|\big|f_r(B)-f_r(B^{\prime})\big|\,.
\end{equation}
Relying on the ${\bf L^1}$-continuity 
of $u$ and $v$ at $\tau_0=0$
(property (2) of Definition~\ref{defiAB}),
and letting $R\to \infty$
in~\eqref{eq:kruz-4},
we obtain the estimate of Theorem~\ref{theoremsemigroup}-$(iv)$ for $t=\tau$.

To establish property (v) of Theorem~\ref{theoremsemigroup} 
observe that, if $(A,B)$ is a non critical connection, then by Lemma~\ref{BVbound}-(i)
one has 
$\mc S^{AB}_t u_0\in BV_{\mr{loc}}(\R)$ for all $t>0$, and for any $u_0\in {\bf L}^\infty(\R)$. Therefore, in this case, relying on this property  
we immediately recover the ${\bf L}^1_{\mr{loc}}$-Lipschitz 
continuity of \linebreak  $t\mapsto \mc S^{AB}_t u_0$
by standard arguments (e.g. see~\cite[proof of Theorem 9.4]{bressan2000hyperbolic}).
\blu{On the other hand}, in the case of 
a critical connection $(A,B)$,
we derive the ${\bf L}^1_{loc}$-Lipschitz 
continuity of  $t\mapsto \mc S^{AB}_t u_0$
applying  Lemma~\ref{BVbound}-(ii)
and following 
the same arguments in~\cite[proof of Theorem~4.3.1] {Dafermoscontphysics}.
\end{proof}
\medskip

\begin{proof}[Proof of Corollary~\ref{cor:fluxtraces-stab}]
Relying on Theorem \ref{theoremsemigroup}-$(iv)$ we deduce that
\begin{equation}
\label{eq:app-conv1}
    u_n(\cdot, t) 
    \quad\ \rightarrow \quad u(\cdot, t)
    \qquad\text{in}\quad
    \mathbf L^1_{\blu{\mr{loc}}}(\R)\qquad \forall~t\geq 0\,,
\end{equation}
which in turn implies that there exists $\overline x>0$ such that
\begin{equation}
\label{eq:app-conv2}
    f_r(u_n(\overline x,\cdot)) 
    \quad\ \rightarrow \quad f_r(u(\overline x,\cdot))
    \qquad\text{in}\quad
    \mathbf L^1_{\mr{loc}}([0,+\infty[)\,.
\end{equation}
Then, observe that by Definition~\ref{defiAB}
$u_n$ and $u$ are entropy weak solutions 
of $u_t + f_r(u)_x=0$ on
$]0,+\infty[\,\times \,[0,+\infty[\,$.
Hence, by a general property of weak solutions 
(e.g. see~\cite[Remark~4.2]{bressan2000hyperbolic}),
for every fixed $s>0$
one has
\begin{equation}
\label{eq:app-conv3}
    \begin{aligned}
    \int_0^s f_r(u_r(s))ds&=\int_0^s f_r(u(\overline x,\cdot))ds+\int_0^{\overline x}
    u(z,T)dz - \int_0^{\overline x} 
    u(z,0)(z)dz\,,
    %
    \\
    \noalign{\smallskip}
    \int_0^s f_r(u_{n,r}(s))ds&=\int_0^s f_r(u_n(\overline x,\cdot))ds+\int_0^{\overline x}
    u_n(z,T)dz - \int_0^{\overline x} 
    u_n(z,0)
    dz\qquad \forall~n\,.
      \end{aligned}
\end{equation}
Since we are assuming that 
$\{u_{n}(\cdot ,0)\}_n$ converges to $u(\cdot,0)$ in ${\bf L}^1_{\blu{ \mr{ loc}}}$,
taking the limit as $n\to \infty$ in~\eqref{eq:app-conv3}
we deduce from~\eqref{eq:app-conv1}, \eqref{eq:app-conv2}, \eqref{eq:app-conv3} 
{ by standard arguments} that
\begin{equation}
    f_r(u_{n,r})
    \ \ \rightharpoonup \ \
     f_r(u_{r})
    \quad\text{\blu{weakly} in}\quad
    \mathbf L^1(\mathbb R^+).
\end{equation}
With entirely similar arguments one derives also the other convergence in~\eqref{eq:flux-trace-conv1}.
\end{proof}

\medskip
\section{Preclusion of 
rarefactions emanating from the interface}\label{app:no-rarefaction}

A distinctive feature of the structure of $AB$-entropy solutions is the fact that
no rarefaction wave can emerge at positive times from the interface $x=0$.
This property was established in~\cite{adimurthi2020exact}
exploiting an
explicit representation formula for $AB$-entropy solutions
a la Lax-Ole\v{\i}nik.
A different, rather technical proof, based on a detailed analysis of the structure of $AB$-entropy solutions was derived in~\cite{anconachiri}, under the additional assumption that the traces 
of the solution at $x=0$
admit one sided limits. 
Here, 
we provide 
a much simpler proof  that establishes this fact 
in the case
of a non critical connection
$(A,B)$, 
and for a $BV_{\ms{loc}}$ $AB$-entropy solution.
The proof relies on the properties of solutions of Riemann problems and on a blow-up argument. \blu{Namely, the key point is to show that Riemann-type initial data from which rarefaction waves emerge are not attainable by an $AB$-entropy solution at any positive time $t > 0$. 
Next, 
 by contradiction and performing a blow-up analysis, we prove that if 
 a rarefaction emerges from an $AB$-entropy solution at some time $\overline t>0$, then there exists a Riemann-type 
 datum $\overline u$ that generates a rarefaction 
 and which is attainable by
 an $AB$-entropy solution at time $\overline t$.
 
One can recover this property of preclusion of rarefactions 
emanating from the interface 
(for any $AB$-entropy solution and general connections)
as a byproduct of the characterization 
 of attainable profiles
$\omega\in \mc A^{[AB]}(T)$ provided by Theorems~\ref{thm:attprofiles}, \ref{thm:attprofilescrit}, \ref{thm:attprofiles2}, \ref{thm:attprofiles3}
(see Remark~\ref{rem:norar}).
}

\begin{defi}
We say that 
an $AB$-entropy solution $u(x,t)$ to \eqref{conslaw} has a \textit{rarefaction fan  emerging at the right (at the left)
from the interface} $x=0$ at time $\bar t$, if 
{ 
there exists
$\delta > 0$ and two continuity points $0 < x_1 < x_2$ for $u(\cdot, \overline t+ \delta)$ such that 
$$
x_1 - \delta f_r^\prime(u(x_1, \overline t + \delta)) = x_2 - \delta f_r^\prime(u(x_2, \overline t + \delta)) = 0.
$$
}
\end{defi}

\noindent
\blu{Notice that Definition B.1 does no require to know that the solution $u$ admits one-sided limits at $x=0$, and it
is invariant with respect to the scaling
$(x,t) \to (\rho x, \overline t+ \rho(t-\overline t))$, $\rho>0$. This definition is equivalent to say  that there exists an outgoing rarefaction fan emerging at time $\overline t$, at the right, if there exist two distinct genuine characteristics located in $\{x>0\}$ for times $t\in\,]\overline t, \overline t+\delta]$, $\delta>0$,
that emerge from the point $(0, \overline t)$.}
\begin{prop}\label{prop:K}
Let $(A, B)$ be a connection,
consider a Riemann data
\begin{equation}
    \label{eq:riemm-data}
    \overline u = \begin{cases}
u^-, & x <0, \\
u^+, & x > 0\,,
\end{cases}
\end{equation}
and assume that
the solution $\mc S^{[AB]+}_t \overline u(x)$ contains a rarefaction wave
located in the left  halfplane $\{x\leq 0\}$, 
or in the right one
$\{x\geq 0\}$.
Then for every $T > 0$ it holds $\overline u \notin \mc A^{[AB]}(T)$.  
\end{prop}
\begin{proof}
By contradiction, suppose that $\mc S^{[AB]+}_t \overline u(x)$ contains a rarefaction wave
located in $\{x\geq 0\}$,
and assume that
$\overline u \in \mc A^{[AB]}(T)$, i.e. that there exists an
$AB$-entropy solution solution $u(x,t)$ of~\eqref{conslaw},\eqref{initdat}, such that $u(\cdot\,, T) = \overline u$.
Then, by uniqueness, one has
$u(x,T+t)= \mc S^{[AB]+}_{t} \overline u(x)$ for all $x\in\R$, $t\geq 0$.
Since $\mc S^{[AB]+}_t \overline u(x)$ 
is a solution of a Riemann problem
containing a rarefaction
with nonnegative characteristic speeds, and because of the admissibility conditions~\eqref{ABtraces}, 
it then follows that the right trace $u_r$ of $u$ at $x=0$ satisfies
 $B \leq u_r(t) < u(T, x) = u^+$ for every $t > T$, $x > 0$. Tracing the backward characteristics from points $(T, x)$, $x > 0$, we find that $u_r(t) = u^+> B$ for every $t \in \,]0,T[$\,. Therefore,
 because of the admissibility conditions~\eqref{ABtraces},
 one has $u_l(t) = \pi^{r}_{l,+}(u^+)$ 
 (with $\pi^r_{l,+}$ defined as in~\eqref{pimap-def}),
 for every $t \in \,]0,T[$\,. 
 Then, letting $\xi_-, \xi_+$ denote the minimal
 and maximal backward characteristics starting
 at $(0,T)$, we deduce that $f'_l(u^-)=\dot\xi_-(T)>
 \dot\xi_+(T)=f'_l(\pi^{l}_{r,+}(u^+))$,
 which in turn implies $u^- > \pi^{r}_{l,+}(u^+)$, $\pi^{l}_{r,+}(u^-)>u^+$. Observe now that the $AB$-entropy solution of a Riemann problem with
 initial data~\eqref{eq:riemm-data} satisfying $u^- > \pi^{l}_{r,+}(u^+)$, and $u^+>B$, consists of a single shock
 located in the halfplane $\{x\geq 0\}$, 
 and connecting the left state $\pi^{l}_{r,+}(u^-)$ with the right state $u^+$. This is in contrast with the assumption made on $\mc S^{[AB]+}_t \overline u(x)$, thus completing the proof.
 
\end{proof}

\begin{prop}\label{prop:norare}
Let $(A, B)$ be a \blu{non critical} connection,
and let $u$
be an $AB$-entropy solution to~\eqref{conslaw}
that satisfies $u(\cdot, t)\in BV_{\mr{loc}}(\R)$, for all $t>0$.
Then $u$ does not contain rarefaction waves emerging from the interface $x=0$ at times $\overline t>0$.
\end{prop}

\begin{proof}
 Assume by contradiction that the solution $u$ has a rarefaction wave,
 say located in $\{x\geq 0\}$, which emerges from the interface at some time $\overline t>0$.
Let $0<\overline \rho<\overline t/3$, and
for any $\rho>0$, set
\begin{equation}
\label{eq:Ir-def}
    I_\rho\doteq \{x\in\R\, : \, |x|\leq \rho \}\,.
\end{equation}
Observe that the domain of dependence 
of $u(x,t)$, for $(x,t)\in I_{\overline \rho}\times [\overline t -\overline \rho, \overline t+\overline \rho]$, is the trapezoid $\Omega\doteq 
\{(x,t)\, : \, |x|\leq \overline \rho + \Lambda\cdot (\overline t+\overline \rho-t),\, t\in [\overline t-2\overline \rho, \overline t+ \overline \rho]\}$, where $\Lambda \doteq \sup_{|z|\leq M}  \max \{|f'_l(z)|, |f'_r(z)|\}$, with
$M$ being a uniform ${\bf L^\infty}$ bound
for $u$. Therefore, since 
the total variation of $u(t,\cdot)$  on $I_{l_t}$,
$l_t\doteq |x|\leq \overline \rho + \Lambda\cdot (\overline t+\overline \rho-t)$, is bounded, and because $(A,B)$
is a \blu{non critical} connection, we can invoke the
uniform BV bounds on $AB$-entropy solutions
established in~\cite[Lemma 8]{Garavellodiscflux}
(see also~
\cite[Theorem 2.13-(iii)]{MR2743877})
to deduce that
\begin{equation}
\label{eq:blup-est1}
    \mr{Tot.Var.}(u(\cdot, t),\, I_{\overline \rho})\leq \overline C \big(M+  \mr{Tot.Var.}(u(\cdot, \overline t-2 \overline \rho), I_{(1+2\Lambda)\overline \rho})\big)
    \qquad \forall~t\in 
    \overline t +
    I_{\overline{\rho}}\,,
\end{equation}
for some constant $\overline C>0$.
Next, consider the 
blow-up of $u$ at the point $(0,\overline t\,)$:
\begin{equation}
\label{eq:blowup-def}
u_\rho(x,t) \doteq u(\rho x, \overline  t +  \rho(t-\overline t)) \qquad x\in\R, \ t\geq 0\,,
\end{equation}
with $0<\rho<\overline \rho/\overline t$,
and observe that
it holds true
\begin{equation}
\label{eq:blup-est2}
    \mr{Tot.Var.} (u_\rho(\cdot,t), I_{\overline \rho/\rho}) \leq \sup_{\tau\, \in\, \overline t + I_{\overline \rho}}
      \mr{Tot.Var.} (u(\cdot,\tau), I_{\overline \rho})
      \qquad\forall~0\leq t<\overline t+\frac{\overline \rho}{\rho}\,.
\end{equation}
Combining~\eqref{eq:blup-est1}, \eqref{eq:blup-est2}, we find
a uniform bound on the total variation of 
$u_r(\cdot, t)$ on the 
interval $I_{\overline \rho/\rho}$, for all $t<\overline t+{\overline \rho}/{\rho}$,
and $0<\rho<\overline \rho/\overline t$.
Moreover,
observe that 
because of the finite
speed of propagation $\Lambda$,
by standard arguments
(e.g. see~\cite[\S 7.4]{bressan2000hyperbolic})
one deduces that
\begin{equation}
    \|u_\rho(\cdot,t)-u_\rho(\cdot,s)\|_{{\bf L^1}(I_{\overline \rho/\rho})}\leq \overline \Lambda\cdot  (t-s)\qquad\forall~0\leq s<t<\overline t+\frac{\overline \rho}{\rho}\,,
\end{equation}
for all $0<\rho<\overline \rho/\overline t$, and for some constant $\overline \Lambda$.
Notice that the sets
$I_{\overline \rho/\rho}\times [0, \overline t+ \overline \rho/\rho[$\,
invade $\R\times [0,+\infty[$\, as $\rho\to 0$.
Therefore we can apply 
Helly's compactness theorem~\cite[Theorem 2.4]{bressan2000hyperbolic} to the sequence $\{u_\rho\}_{0<\rho<\overline \rho/\overline t}$, 
and deduce the existence
of a function 
$v \in {\bf L^\infty}
( \mathbb R\times [0,+\infty[)$,
so that, up to a subsequence,
$u_\rho(\cdot, t)$
converges to $v(\cdot, t)$ in ${\bf L^1}_{\mr{loc}}$, as $\rho\to 0$, for all $t>0$.
By Definition~\ref{defiAB}
it follows that
also $v$ is an $AB$-entropy solution
of~\eqref{conslaw}-\eqref{initdat}, with $u_0\doteq v(\cdot, 0)$.
Notice that
\begin{equation}
    \lim_{\rho\to 0}u_\rho(x,\overline t\,)=
    \overline u(x)\doteq \begin{cases}
       u(0+,\overline t\,)\ \ &\text{if}\quad x>0\,,
        \\
        \noalign{\smallskip}
        u(0-,\overline t\,)\ \ &\text{if}\quad x<0\,,
    \end{cases}
\end{equation}
and thus we find 
$$
v(\cdot ,\overline t\,) = \overline u\,,
$$
which implies 
\begin{equation}\label{eq:contradiction}
\overline u \in \mc A^{[AB]}(\,\overline t)\,.
\end{equation}
On the other hand, observe that the rarefaction wave which emerges in the solution $u$ at \linebreak time~$\overline t$ is preserved by the blow-ups
$u_\rho$ in~\eqref{eq:blowup-def}, because 
it is self similar for the scaling \linebreak $(x,t) \mapsto (\rho x, \bar t +\rho(t-\bar t))$. Therefore there will be a rarefaction wave emerging at time $\overline t$, and located in $\{x\geq 0\}$, also in the solution~$v$. This in turn implies
that the solution $\mc S_t^{AB} \overline u(x)$ to the Riemann problem with initial datum $\overline u$
contains a rarefaction emerging at 
$t = 0$ and located in $\{x\geq 0\}$, since $\mc S_t^{AB} \overline u(x) = v(\bar t+t, x)$ for all $x\in\R$, $t\geq 0$.
This, together with~\eqref{eq:contradiction},
is in contradiction with Proposition \ref{prop:K}, thus completing the proof.
\end{proof}

\blu{\begin{remark}
    \label{rem:norar}
    In the case of a general connection $(A,B)$,
    relying on the characterization of 
    $\mc A^{[AB]}(t)$,  $t>0$,
    provided by Theorems~\ref{thm:attprofiles}, \ref{thm:attprofilescrit}, \ref{thm:attprofiles2}, \ref{thm:attprofiles3},
    we can show that 
    no rarefaction can emerge from the interface $x=0$ at any time $\overline t>0$ 
    for any  $AB$-entropy solution $u$ to~\eqref{conslaw}, 
    as follows.
    Suppose, by contradiciton, that a rarefaction is generated in
    a time interval $[\,\overline t, \overline t+\delta]$, for some $\delta>0$,
    and that lies in the semiplane $\{x\geq 0\}$.
    In particular this means that there exist 
    two genuine characteristics $\xi_1, \xi_2: [\,\bar t, \bar t+\delta\,] \to \ [0,+\infty[\,$,  such that 
$\xi_1(\,\bar t\,) = \xi_2(\,\bar t\,) = 0$,
    $\overline x_1\doteq \xi_1(\,\bar t+\delta) < \overline x_2\doteq\xi_2(\,\bar t+\delta)$.
    We may also assume that $\xi'_i=f'_r(\omega(\overline x_i))$, $i=1,2$.
    Let $\omega\doteq u(\cdot, \overline t+\delta)$, $\ms R\doteq \ms R[\omega, f_r]$ (see def~\eqref{eq:LR-def}), and consider the time $\tau(x)=(\overline t+\delta)-x/f'_r(\omega(x))$, $x\in\,]0, \ms R[\,$, at which the characteristic starting at $(x,\overline t+\delta)$, with slope $f'_r(\omega(x))$
    impacts the interface $x=0$.
    Notice that $\tau(\overline x_1)=\tau(\overline x_2)=\overline t$.
     Moreover, thanks to  Lemma 4.4 in~\cite{anconachiri}, 
     the Ole\v{\i}nik estimates satisfied by $\omega$ (because of condition~(i) or (i)' of Theorems~\ref{thm:attprofiles}, \ref{thm:attprofilescrit}, \ref{thm:attprofiles2}, \ref{thm:attprofiles3}) imply the strict monotonicity of the map $x\to \tau(x)$, $x\in\,]0, \ms R[\,$. In turn, the strict monotonicity of $\tau$ implies 
     $\tau(\overline x_1)\neq \tau(\overline x_2)$, thus contradicting the assumption 
     $\tau(\overline x_1)=\tau(\overline x_2)=\overline t$.
\end{remark}
}
\medskip

\section{Semicontinuity properties of solutions to 
convex conservation laws
}\label{app:uplwsmicsolns}

Solutions to conservation laws with 
convex flux enjoy a lower and upper semicontinuity
property  
 with respect to the ${\bf L}^1$
convergence
as stated in the following

\begin{lemma}\label{lemma:chara}
Given a uniformly convex map $f$, 
and $T>0$, let 
  $\{u_n\}_n$ be 
  a sequence of 
  entropy weak solutions 
  of 
  \begin{equation}
  \label{eq:conlaw-qp}
u_t+f(u)_x=0 
\quad\ x >0, \quad t \in [0,T],
  \end{equation}
  that admit a
  strong trace $u_n(0+,t)=\lim_{x\to 0+}u_n(x,t)$
  at $x=0$,
  for all $t\in [0,T]$,
  and let $u$ be 
  an
  entropy weak solution 
  of~\eqref{eq:conlaw-qp} 
that admits a
  strong trace 
  $u(0+,t)=\lim_{x\to 0+}u(x,t)$
  at $x=0$,
  for all $t\in [0,T]$.
Assume that $\{u_n\}_n$ are uniformly bounded
in ${\bf L}^\infty$, that
\begin{equation}
\label{eq:conlaw-qp-conv1}
    u_n(\cdot,t)\ \ \rightarrow\ \
    u(\cdot, t)
\qquad\text{in}\quad
 \mathbf L^1_{\blu{\ms loc}}(\,]0,+\infty[),\qquad \forall~t\in[0,T]\,,
\end{equation}
and that
\begin{equation}
\label{eq:conlaw-qp-conv2}
    f(u_n(0+,\cdot))\ \ \rightharpoonup\ \
    f(u(0+, \cdot))
\qquad\text{\blu{weakly} \ in}\quad
    \mathbf L^1([0, T])\,.
\end{equation}
  Then, 
  for every $x\geq 0$, 
  it holds true
    \begin{align}\label{eq:liminf}
         u(x+, T) &\leq \liminf_{\substack{n \to \infty \\ y \to x, \, \blu{y>0}}} \; u_n(y+, T)\,.
    \end{align}
    If we assume that $u_n, u$, are entropy weak solutions of $u_t+ f'(u)_x=0$ on  $x<0,\, t\in [0,T]$, and that
    the convergences~\eqref{eq:conlaw-qp-conv1}, \eqref{eq:conlaw-qp-conv2},
    hold in  $\mathbf L^1_{\blu{\ms loc}}(\,]-\infty,0[)$,
    and for the left traces in $x=0$, respectively, then
    for every $x\leq 0$, it holds true
    \begin{align}\label{eq:limsup}
         u(x-,T) &\geq  \limsup_{\substack{n \to \infty \\ y\to x,\, \blu{y<0}}} \;  u_n(y-, T)\,.
    \end{align}

\end{lemma}
\begin{proof}
We will establish only the inequality~\eqref{eq:liminf}, the proof of \eqref{eq:limsup} being entirely similar.
Given $x\geq 0$, $T>0$, consider a sequence $\{y_n\}_n$,
$y_n>0$, converging to $x$, and such that 
\begin{equation}
\label{eq:liminf-sequence}
    \lim_n u_n(y_n+, T)=\liminf_{\substack{n \to \infty \\ y \to x}} \; u_n(y+, T)\,.
\end{equation}
Let $\vartheta_n^+ :\ ]\tau_n, T]\to \ ]0,+\infty[\,$, $\tau_n\geq 0$,
denote the maximal backward characteristic 
for $u_n$ starting from~$(y_n,T)$,
with the property that either $\tau_n=0$, or 
$\lim_{t\to\tau_n} \vartheta_n^+(t)=0$.
By possibly taking a subsequence, we can assume that either $\tau_n=0$ for all $n$, or that
$\lim_{t\to\tau_n} \vartheta_n^+(t)=0$ for all $n$.
We recall that a maximal backward characteristic
for $u_n$
passing through  $(y_n,T)$, $y>0$,
is a genuine (shock free)
characteristics
whose trajectory 
is a segment with constant slope $f'(u_n(y_n+,T))$
(e.g. see~\cite{Dafermoscontphysics}).
Notice that $\{\vartheta_n^+\}_n$ is a sequence of
Lipschitz continuous functions with a uniform Lipschitz constant
$\sup_{|u|\leq M} f'(u)$
($M$ being a uniform ${\bf L}^\infty$ bound on $u_n$), defined on uniformly bounded intervals $]\tau_n, T]$.
Hence, by Ascoli-Arzel\`a Theorem we can assume that, up to a subsequence,  $\{\vartheta_n^+\}_n$ converges uniformly to some Lipschitz continuous function $\vartheta:\ ]\tau, T]\to \ ]0,+\infty[\,$, such that
\begin{equation}
    \begin{aligned}
        &\tau=0,
        \qquad
        \text{if}\qquad 
        \tau_n=0\quad \forall~n\,,
        \\
        \noalign{\smallskip}
        &\tau=\lim_{n\to\infty} \tau_n\,,\quad 
        \lim_{t\to\tau} \vartheta(t)=0,
        \qquad \text{if}\quad
        \lim_{t\to\tau_n} \vartheta_n^+(t)=0\quad \forall~n\,,
    \end{aligned}
\end{equation}
and such that $\vartheta(T)=x$.
 By a general property of
characteristics, the uniform limit of genuine characteristics is also a genuine characteristic. 
This can be easily verified 
in this context 
observing that the trajectory of
    a genuine characteristic passing through a point $(y,t)$, $y>0$, is 
    a segment connecting $(y,t)$ with the point 
    $(0,\tau(y,t))$ or 
    with the point $(z(y,t),0)$, where $\tau (y,t)$ and $z(y,t)$
    denotes the points of minimum for the functionals involved in 
    the Lax-Ole\v{\i}nik
    representation formula of solutions
    for the boundary value problem (see~\cite{LeFlochboundary}), and that such functionals are ${\bf L}^1$ continuous with respect to the initial datum and 
    { and weakly continuous in ${\bf L}^1$} with respect to the flux-trace  of the solution at $x=0$.
    Therefore it follows that 
    $\vartheta$ is a genuine characteristic with constant slope $\vartheta'$
    satisfying
    \begin{equation}
    \label{eq:lim-charact-slope}
        \vartheta'=\lim_{n\to\infty} (\vartheta^+_n)'
        =\lim_{n\to\infty} 
        f'(u_n(y_n+,T)),
        \qquad\qquad
        \vartheta'\geq f'(u(x+,T))\,.
    \end{equation}
    Since $f'$ is increasing,
    we deduce from~\eqref{eq:lim-charact-slope} that
\begin{equation}
    \lim_n u_n(y_n+, T)\geq 
    u(x+,T)\,,
\end{equation}
which, together with~\eqref{eq:liminf-sequence}, yields~\eqref{eq:liminf}.
    

\end{proof}

\bibliography{references}
\bibliographystyle{plain}

\end{document}